\documentclass[12pt]{amsart}
\usepackage{amsmath,amsfonts,amssymb,amsthm,amsxtra,amscd}
\usepackage{mathrsfs}
\usepackage{graphicx}
\numberwithin{equation}{section}

\theoremstyle{plain} 
\newtheorem{thm}{Theorem}[section] 
\newtheorem{prop}[thm]{Proposition} 
\newtheorem{lemma}[thm]{Lemma} 
\newtheorem{cor}[thm]{Corollary}

\theoremstyle{definition}

\newtheorem{remark}[thm]{Remark}

\newcommand{\C}{{\mathbb C}}
 
\newcommand{\p}{{\mathbb P}}

\newcommand{\z}{{\mathbb Z}} 
\newcommand{\pj}{{{\mathbb P}^1}}
\newcommand{\pii}{{{\mathbb P}^2}}
\newcommand{\piii}{{{\mathbb P}^3}}
\newcommand{\piv}{{{\mathbb P}^4}}
\newcommand{\pv}{{{\mathbb P}^5}}

\newcommand{\sce}{\mathscr{E}}
\newcommand{\scf}{\mathscr{F}} 

\newcommand{\sco}{\mathscr{O}} 
\newcommand{\sch}{\mathscr{H}}
\newcommand{\sci}{\mathscr{I}}

\newcommand{\scl}{\mathscr{L}}
\newcommand{\scp}{\mathscr{P}}

\newcommand{\fm}{{\mathfrak m}}

\newcommand{\h}{\text{h}}
\newcommand{\tH}{\text{H}} 

\newcommand{\izo}{\overset{\sim}{\rightarrow}} 
\newcommand{\Izo}{\overset{\sim}{\longrightarrow}} 
\newcommand{\ra}{\rightarrow} 
\newcommand{\lra}{\longrightarrow} 
\newcommand{\xra}{\xrightarrow}
\newcommand{\vb}{\, \vert \, } 
\newcommand{\prim}{{\prime}} 
\newcommand{\secund}{{\prime \prime}}
\newcommand{\Ker}{\text{Ker}\, }
\newcommand{\Cok}{\text{Coker}\, }

\newcommand{\e}{\varepsilon}

\begin{document}

\title[Space curves and vector bundles]{Locally Cohen-Macaulay space curves 
defined by cubic equations and globally generated vector bundles}

\author[Anghel,~Coand\u{a}~and~Manolache]{Cristian~Anghel,~Iustin~Coand\u{a}~
and~Nicolae~Manolache}
\address{Institute of Mathematics of the Romanian Academy, P.O. Box 1--764, 
RO--014700, Bucharest, Romania}
\email{Iustin.Coanda@imar.ro~Cristian.Anghel@imar.ro~Nicolae.Manolache@imar.ro}




\begin{abstract}
Let $E$ be a globally generated vector bundle on the projective 3-space, with 
$c_1$ at least 4, which is not the quotient of a larger vector bundle by a 
trivial subbundle. We show that if $E(-c_1 + 3)$ has a non-zero global 
section then, except for some cases that can be listed, $E$ admits as a direct
summand a line bundle of degree at least $c_1 - 3$.   
This reduces the problem of the classification of globally generated 
vector bundles with $c_1$ at most 7 on the projective 3-space to the case 
where the associated rank 2 reflexive sheaf is stable, a case which requires 
quite different methods. The proof uses the description of the monads of the 
ideal sheaves of all locally Cohen-Macaulay space curves defined by cubic 
equations. We then decide which of the bundles with the above mentioned 
properties extend to higher dimensional projective spaces. This part uses, 
among other things, a result of Barth and Ellencwajg [Lecture Notes Math. 
683 (1978), Springer Verlag, 1--24] asserting the non-existence of stable 
rank 2 vector bundles on the projective 4-space with $c_1 = 0$, $c_2 = 3$. We 
provide a different proof of this fact, based on a theorem of Mohan Kumar, 
Peterson and Rao [Manuscripta Math. 112 (2003), 183--189]. Finally, we recover 
quickly, using the previously mentioned results, the classification of 
globally generated vector bundles with $c_1 = 4$ on the projective $n$-space 
for $n$ at least 4, which is part of the main result of our previous paper 
[arXiv:1305.3464]. We provide, in the appendices to the paper, graded free 
resolutions for the homogeneous ideals and for the graded structural algebras 
of all nonreduced locally Cohen-Macaulay space curves of degree at most 4. We 
use only a small part of the content of the appendices but we decided to 
give the complete list because it might be useful in some other contexts, too 
(compare with the papers of Nollet [Ann. Sci. \'{E}cole Norm. Sup. (4) 30 
(1997), 367--384] and of Nollet and Schlesinger [Compos. Math. 139 (2003), 
169--196]).           
\end{abstract}

\maketitle
\tableofcontents

\section*{Introduction}

Motivated, in part, by some geometric applications (cf., for example, the 
papers of Manivel and Mezzetti \cite{mm}, Sierra and Ugaglia \cite{su0},  
Fania and Mezzetti \cite{fm}, Huh \cite{hu}) the systematic study of globally 
generated vector bundles on projective spaces was initiated by Sierra and 
Ugaglia \cite{su} who classified the bundles of this kind with the first Chern 
class $c_1 \leq 2$. Then Ellia \cite{e} determined the Chern classes of 
globally generated vector bundles on $\pii$, Chiodera and Ellia \cite{ce} 
determined the Chern classes of the globally generated vector bundles of rank 
2 on $\piii$ with $c_1 \leq 5$, while Anghel and Manolache \cite{am} and, 
independently, Sierra and Ugaglia \cite{su2}, classified the globally 
generated vector bundles on $\p^n$ with $c_1 = 3$. 

The present paper is a sequel to our work \cite{acm}, where we classified 
the globally generated vector bundles on $\p^n$ with $c_1 = 4$. We are here 
(mainly) concerned with the classification of globally generated vector 
bundles on $\piii$. (From the point of view of the methods used in \cite{acm}, 
this is a critical case: on one hand, the study of globally generated 
vector bundles on $\pii$ presents some special features and may be treated 
separatedly. On the other hand, once one has a classification of globally 
generated vector bundles on $\piii$ one can try to decide which of these 
bundles extend, as a globally generated vector bundle, to higher dimensional 
projective spaces. This approach has limited efficiency, but we are unaware  
of a direct method for studying globally generated bundles on 
$\p^n$, $n \geq 4$.) One natural way to investigate this kind of 
bundles is to relate them to rank 2 reflexive sheaves, which were intensively 
studied. If $E$ is a globally generated vector bundle on $\piii$, of rank 
$r$ and Chern classes $c_i = c_i(E)$, $i = 1,\, 2,\, 3$, then $r-2$ general 
global sections of $E$ define an exact sequence$\, :$ 
\[
0 \lra (r-2)\sco_\p \lra E \lra \sce^\prim \lra 0
\]  
where $\sce^\prim$ is a globally generated rank 2 reflexive sheaf with 
$c_i(\sce^\prim) = c_i$, $i = 1,\, 2,\, 3$. The classification problem splits, 
now, into two cases : (1) $\sce^\prim$ stable; (2) $\sce^\prim$ non-stable. 
We are concerned, in this paper, with the latter case. Let 
$\sce^\prim_{\text{norm}}$ denote the reflexive sheaf 
$\sce^\prim\left(-\left[\frac{c_1+1}{2}\right]\right)$, which has the first 
Chern class either $0$ or $-1$. Then $\sce^\prim$ is non-stable if and only if 
$\tH^0(\sce^\prim_{\text{norm}}) \neq 0$ and this happens, of course, if and only 
if $\tH^0\left(E\left(-\left[\frac{c_1+1}{2}\right]\right)\right) \neq 0$. 

Now, it is easy to show that if $\tH^0(E(-c_1)) \neq 0$ then $E \simeq 
\sco_\p(c_1) \oplus (r-1)\sco_\p$ (this is a particular case of a result of 
Sierra \cite{s}). Sierra and Ugaglia \cite[Prop.~2.2]{su2} (for $c_1 = 3$) 
and Anghel et al. \cite[Prop.~2.4]{acm} (for $c_1 \geq 2$) described the 
globally generated vector bundles $E$ with $\tH^0(E(-c_1)) = 0$ and 
$\tH^0(E(-c_1+1)) \neq 0$. Moreover, we described in \cite[Prop.~2.9]{acm} the 
globally generated vector bundles with $c_1 \geq 3$, $\tH^0(E(-c_1+1)) = 0$ 
and $\tH^0(E(-c_1+2)) \neq 0$. These results suffice, already, to classify 
globally generated vector bundles on $\piii$ with $c_1 \leq 3$ (see 
\cite[Remark~2.12]{acm}). 

The main objective of the present paper is to classify globally generated 
vector bundles $E$ on $\piii$ with $c_1 \geq 4$, $\tH^0(E(-c_1+2)) = 0$ and 
$\tH^0(E(-c_1+3)) \neq 0$ (the similar problem on $\pii$ has been settled in 
\cite[Prop.~3.2]{acm}). The associated rank 2 reflexive sheaf $\sce^\prim$ 
introduced above satisfies, also, the conditions $\tH^0(\sce^\prim(-c_1+2)) 
= 0$, $\tH^0(\sce^\prim(-c_1+3)) \neq 0$, hence any non-zero global section of 
$\sce^\prim(-c_1+3)$ defines an exact sequence$\, :$ 
\[
0 \lra \sco_\p(c_1-3) \lra \sce^\prim \lra \sci_Z(3) \lra 0
\]    
with $Z$ a locally Cohen-Macaulay (CM, for short) closed subscheme of $\piii$, 
of pure codimension 2, locally complete intersection (l.c.i., for short) 
except at finitely many points, with $\sci_Z(3)$ globally generated. 
One deduces an exact sequence$\, :$ 
\begin{equation}\label{E:o(c1-3)oeiz(3)}
0 \lra \sco_\p(c_1-3) \oplus (r-2)\sco_\p \lra E \lra \sci_Z(3) \lra 0\, .
\end{equation}

Now, assume we know a \emph{Horrocks monad} $B^\bullet$ of $\sci_Z(3)$, i.e., 
a three terms complex of sheaves$\, :$ 
\begin{equation}\label{E:monadiz(3)}
0 \lra B^{-1} \overset{\displaystyle d^{-1}}{\lra} B^0 
\overset{\displaystyle d^0}{\lra} B^1 \lra 0
\end{equation}
with the $B^i$s direct sums of invertible sheaves $\sco_\p(j)$, $j \in \z$, 
such that $\sch^0(B^\bullet) \simeq \sci_Z(3)$ and $\sch^i(B^\bullet) = 0$ for 
$i \neq 0$. Then there exists a morphism 
$(\phi ,\, \psi) : B^{-1} \ra \sco_\p(c_1-3) \oplus (r-2)\sco_\p$ 
such that$\, :$ 
\begin{equation}\label{E:monade}
0 \lra B^{-1} \xra{\begin{pmatrix} \phi\\ d^{-1}\\ \psi \end{pmatrix}} 
\sco_\p(c_1-3) \oplus B^0 \oplus (r-2)\sco_\p 
\xra{\displaystyle (0\, ,\, d^0,\, 0)} B^1 \lra 0
\end{equation}
is a monad for $E$. 

Thus, we are left with the problem of determining the monads of the ideal 
sheaves of locally CM space curves $Z$ with $\sci_Z(3)$ globally generated. 
There are two ways to get such a monad. The first one is based on the 
observation that if$\, :$ 
\[
0 \lra L_2 \overset{\displaystyle d_2}{\lra} L_1 
\overset{\displaystyle d_1}{\lra} L_0 \lra \tH^0_\ast(\sco_Z) \lra 0
\]
is a free resolution of the graded $S$-module $\tH^0_\ast(\sco_Z) := 
\bigoplus_{i \in \z}\tH^0(\sco_Z(i))$ (here $S$ is the 
projective coordinate ring $k[x_0,x_1,x_2,x_3]$ of $\piii$) then, 
removing the direct summand 
$S$ of $L_0$ corresponding to the generator $1 \in \tH^0(\sco_Z)$ of 
$\tH^0_\ast(\sco_Z)$ and sheafifying, one gets a monad for $\sci_Z$$\, :$ 
\[
0 \lra {\widetilde L}_2 \overset{\displaystyle {\widetilde d}_2}{\lra} 
{\widetilde L}_1 \overset{\displaystyle {\widetilde d}_1^{\, \prim}}{\lra} 
{\widetilde L}_0^\prim \lra 0\, .
\]  
The second way is based on a result, 
attributed by Peskine and Szpiro \cite[Prop.~2.5]{ps} to D. Ferrand,  
asserting that if $Z$ is linked to a curve $Z^\prim$ by a complete intersection 
of type $(a,b)$ , if 
\[
0 \lra A_2 \overset{\displaystyle d_2}{\lra} A_1 
\overset{\displaystyle d_1}{\lra} A_0 \lra \sci_{Z^\prim} \lra 0
\]
is a resolution of $\sci_{Z^\prim}$ with direct sums of invertible sheaves and 
if $\psi : \sco_\p(-a) \oplus \sco_\p(-b) \ra A_0$ is a morphism lifting 
$\sco_\p(-a) \oplus \sco_\p(-b) \ra \sci_{Z^\prim}$ then$\, :$ 
\[
0 \lra A_0^\vee \xra{\begin{pmatrix} \psi^\vee\\ d_1^\vee \end{pmatrix}} 
\sco_\p(a) \oplus \sco_\p(b) \oplus A_1^\vee 
\xra{\displaystyle (0\, ,\, d_2^\vee)} A_2^\vee \lra 0
\] 
is a monad for $\sci_Z(a+b)$. One can also see that, conversely, if 
\[
0 \lra B^{-1} \overset{\displaystyle d^{-1}}{\lra} B^0 
\overset{\displaystyle d^0}{\lra} B^1 \lra 0
\]
is a monad for $\sci_Z$ and if $\rho : \sco_\p(-a)\oplus \sco_\p(-b) \ra B^0$ 
is a morphism lifting $\sco_\p(-a)\oplus \sco_\p(-b) \ra \sci_Z$ such that 
$d^0\circ \rho = 0$ then there exists an exact sequence$\, :$ 
\[
0 \lra B^{1\vee} \xra{\displaystyle d^{0\vee}} B^{0\vee} 
\xra{\begin{pmatrix} \rho^\vee\\ d^{-1\vee} \end{pmatrix}} 
\sco_\p(a)\oplus \sco_\p(b) \oplus B^{-1\vee} \lra \sci_{Z^\prim}(a+b) \lra 0\, .
\] 
Notice that if one applies the functor $\tH^0_\ast(-)$ to this exact sequence 
(twisted by $-a-b$) one gets a graded free resolution of the homogeneous 
ideal $I(Z^\prim)\subset S$ of $Z^\prim$.

Now, let $d$ be the degree of such a curve $Z$. Of course, $d \leq 9$. 
The description of the possible monads for $\sci_Z(3)$ when $d \geq 5$, which 
we accomplish in Section~\ref{S:deg5} using the concrete descrition of the 
stable rank 2 reflexive sheaves on $\piii$ with $c_1 = -1$ and $c_2 \leq 2$ 
(due to Hartshorne \cite{ha}, Hartshorne and Sols \cite{hs}, Manolache 
\cite{m1}, and Chang \cite{ch}), turns out to be easy. We determine, in 
Section \ref{S:deg4}, the monads of $\sci_Z(3)$ for $d \leq 4$, using a case 
by case analysis. We rely heavily, in this analysis, on the results from 
Appendix~\ref{A:multilines} where we describe, using the theory of 
B\u{a}nic\u{a} and Forster \cite{bf}, the generators of the homogeneous ideals 
and the monads of the ideal sheaves of the multiple lines in $\piii$ of degree 
$\leq 4$, and on the results from Appendix~\ref{A:multicomponent} where we do 
the same thing for the reducible curves of degree $\leq 4$ having a multiple 
line as a component.  

Finally, we use, in Section~\ref{S:ggvb}, the classification of the monads 
of the locally CM curves $Z$ with $\sci_Z(3)$ globally generated to prove the 
main result of this paper which is the following$\, :$   

\begin{thm}\label{T:main}
Let $E$ be a globally vector bundle on $\piii$ with $c_1 \geq 4$ and such that 
${\fam0 H}^i(E^\vee) = 0$, $i = 0,\, 1$. If ${\fam0 H}^0(E(-c_1+2)) = 0$ and 
${\fam0 H}^0(E(-c_1+3)) \neq 0$ then one of the following holds$\, :$ 
\begin{enumerate}
\item[(i)] $E \simeq \sco_\p(c_1-3) \oplus F$, where $F$ is a globally 
generated vector bundle with $c_1(F) = 3$$\, ;$ 
\item[(ii)] $c_1 = 4$ and $E \simeq F(2)$, where $F$ is a nullcorrelation 
bundle or a $2$-instanton$\, ;$ 
\item[(iii)] $c_1 = 4$ and, up to a linear change of coordinates, E is the 
kernel of the epimorphism  
\[
(x_0\, ,\, x_1\, ,\, x_2\, ,\, x_3^2) : 3\sco_\p(2)\oplus \sco_\p(1) \lra 
\sco_\p(3)\, ; 
\]
\item[(iv)] $c_1 = 4$ and, up to a linear change of coordinates, $E$ is the 
kernel of the epimorphism 
\[
(x_0\, ,\, x_1\, ,\, x_2^2\, ,\, x_2x_3\, ,\, x_3^2) : 2\sco_\p(2)\oplus 
3\sco_\p(1) \lra \sco_\p(3)\, ;
\]
\item[(v)] $c_1 = 4$ and, up to a linear change of coordinates, $E$ is the 
cohomology of the monad  
\[
\sco_\p(-1) \xra{\begin{pmatrix} s\\ u \end{pmatrix}} 
2\sco_\p(2) \oplus 2\sco_\p(1) \oplus 4\sco_\p 
\xra{\displaystyle (p\, ,\, 0)} \sco_\p(3)
\]
where $\sco_\p(-1) \xra{s} 2\sco_\p(2) \oplus 2\sco_\p(1) \xra{p} \sco_\p(3)$ 
is a subcomplex of the Koszul complex defined by $x_0,\, x_1,\, x_2^2,\, 
x_3^2$ and $u : \sco_\p(-1) \ra 4\sco_\p$ is defined by $x_0,\, x_1,\, x_2,\, 
x_3$$\, ;$ 
\item[(vi)] $c_1 = 5$ and $E \simeq F(3)$, where $F$ is a stable rank $2$  
vector bundle with $c_1(F) = -1$ and $c_2(F) = 2$$\, ;$ 
\item[(vii)] $c_1 = 6$ and $E \simeq F(3)$, where $F$ is a strictly   
semistable rank $2$ vector bundle with $c_1(F) = 0$ and $c_2(F) = 3$$\, .$  
\end{enumerate} 
\end{thm}

Some comments are in order. As originally noticed by Sierra and Ugaglia, the 
condition $\tH^i(E^\vee) = 0$, $i = 0,\, 1$, is not too restrictive : any other 
globally generated vector bundle can be obtained from a bundle satisfying this 
additional condition by quotienting a trivial subbundle and, then, by adding a 
trivial direct summand. As we said at the beginning of the introduction, the 
bundles $F$ appearing in item (i) were classified in \cite{am} and \cite{su2}, 
independently.   

Finally, we extend, in Prop.~\ref{P:mainn4}, the classification from  
Theorem~\ref{T:main} to higher dimensional projective spaces. 
In order to show that the bundle appearing in item (vii) of 
Thm.~\ref{T:main} cannot be extended to $\piv$ we need a result of 
Barth and Ellencwajg \cite[Thm.~4.2]{be} asserting that there is no stable 
rank 2 vector bundle on $\piv$ with Chern classes $c_1 = 0$, $c_2 = 3$. 
We provide, in Thm.~\ref{T:be}, a different proof of this fact, 
based on the results of Mohan Kumar, Peterson and Rao \cite{kpr}.  

We use Prop.~\ref{P:mainn4} to recover, in a different manner, in 
Thm.~\ref{T:c1=4n4}, the classification of globally generated vector 
bundles with $c_1 = 4$ on $\p^n$, $n \geq 4$, which is part of the main 
result of our previous paper \cite{acm}. It is very likely that 
Prop.~\ref{P:mainn4} can be used to get a classification of globally 
generated vector bundles with $c_1 = 5$ on $\p^n$, $n \geq 5$. We intend to 
come back to this subject in another paper.      

\vskip2mm

\noindent
{\bf Note.}\quad We provide, in the appendices to the paper, graded free 
resolutions for the homogeneous ideals and for the graded structural algebras 
of all nonreduced locally Cohen-Macaulay curves in $\piii$ of degree at 
most 4. We need only a small part of these results for the proof of 
Thm.~\ref{T:main} but we decided to give (at least in e-print form) the 
complete (and very long) list because it might be useful in some other 
contexts, too.  
Our approach is almost algorithmic and, at least, some aspects of the 
description of quadruple structures on a line should be new. There are 
overlaps with the papers of Nollet \cite{n} and Nollet and Schlessinger 
\cite{ns} but our emphasis is on the description of individual curves and not 
on their deformations. 

\vskip2mm

\noindent
{\bf Acknowledgements.}\quad N. Manolache expresses his thanks to the 
Institute of Mathematics, Oldenburg University, especially to Udo Vetter, 
for warm hospitality during the preparation of this paper. 

\vskip2mm

\noindent
{\bf Notation.}\quad (i) We denote by $S = k[x_0, \ldots ,x_n]$ the projective 
coordinate ring of the projective $n$-space $\p^n$ over an algebraically 
closed field $k$ of characteristic 0. If $\scf$ is a coherent sheaf on 
$\p^n$, we denote by $\tH^i_\ast(\scf)$ the graded $S$-module 
$\bigoplus_{l \in \z}\tH^i(\scf(l))$. 

(ii) If $X$ is a closed subscheme of $\p^n$, we denote by 
$\sci_X \subset \sco_\p$ its ideal sheaf. If $Y$ is a closed subscheme of 
$X$, we denote by $\sci_{Y,X} \subset \sco_X$ the ideal sheaf defining $Y$ as 
a closed subscheme of $X$. In other words, $\sci_{Y,X} = \sci_Y/\sci_X$. If 
$\scf$ is a coherent sheaf on $\p^n$, we put $\scf_X := \scf 
\otimes_{\sco_\p} \sco_X$ and $\scf \vb X := i^\ast\scf$, where $i : X \ra \p^n$ 
is the inclusion. 

(iii) By a point of a quasi-projective scheme $X$ we always mean a 
\emph{closed point}. If $\scf$ is a coherent sheaf on $X$ and $x \in X$, we 
denote by $\scf(x)$ the \emph{reduced stalk} $\scf_x/\fm_x\scf_x$ of $\scf$ 
at $x$, where $\fm_x$ is the maximal ideal of $\sco_{X,x}$. 

(iv) We frequently write ``CM'' for ``Cohen-Macaulay'' and ``l.c.i.'' for 
``locally complete intersection''.

\section{Curves of degree at least 5}\label{S:deg5}  

\begin{prop}\label{P:zinci23} 
Let $Z$ be a locally CM space curve of degree $d \geq 4$. If $\sci_Z(3)$ is 
globally generated and ${\fam0 H}^0(\sci_Z(2)) \neq 0$ then one of the 
following holds$\, :$ 
\begin{enumerate}
\item[(i)] $Z$ is a complete intersection of type $(2,3)$$\, ;$ 
\item[(ii)] $d = 5$ and $\sci_Z(3)$ admits a resolution of the form$\, :$ 
\[
0 \lra 2\sco_\p(-1) \lra \sco_\p(1) \oplus 2\sco_\p \lra \sci_Z(3) \lra 0\, ;
\]
\item[(iii)] $d = 4$ and $\sci_Z(3)$ admits a resolution of the form$\, :$ 
\[
0 \lra \sco_\p \oplus \sco_\p(-1) \lra 2\sco_\p(1) \oplus \sco_\p \lra 
\sci_Z(3) \lra 0\, ;
\]
\item[(iv)] $d = 4$ and $\sci_Z(3)$ admits a resolution of the form$\, :$ 
\[
0 \lra 3\sco_\p \lra \sco_\p(1)\oplus \Omega_\p(2) \lra \sci_Z(3) \lra 0\, .
\]
\end{enumerate}
\end{prop}

\begin{proof}
$Z$ is contained in a complete intersection of type $(2,3)$ hence $d \leq 6$. 

\vskip2mm
 
\noindent
{\bf Case 1.}\quad $d = 6$. 

\vskip2mm

\noindent
In this case $Z$ is a complete intersection of type $(2,3)$. 

\vskip2mm

\noindent
{\bf Case 2.}\quad $d = 5$. 

\vskip2mm

\noindent
In this case, $Z$ is linked to a line by a complete intersection of type 
$(2,3)$ hence, by Ferrand's result about liaison, 
$\sci_Z(3)$ admits a resolution as in item (ii) of the statement. 

\vskip2mm
 
\noindent
{\bf Case 3.}\quad $d = 4$. 

\vskip2mm

\noindent
In this case, $Z$ is linked to a curve $X$ of degree 2. One has to divide 
this case into subcases. 

\vskip2mm

\noindent
{\bf Subcase 3.1.}\quad $X$ \emph{is a complete intersection of type} $(1,2)$. 

\vskip2mm

\noindent
In this subscase, by Ferrand's result about liaison, 
$\sci_Z(3)$ admits a resolution as in item (iii) of the statement. 

\vskip2mm

\noindent
{\bf Subcase 3.2.}\quad $X$ \emph{is the disjoint union of two lines}. 

\vskip2mm

\noindent
In this subcase, $\sci_X$ admits a resolution of the form 
(see Lemma~\ref{L:zcupw})$\, :$ 
\[
0 \lra \sco_\p(-4) \lra 4\sco_\p(-3) \lra 4\sco_\p(-2) \lra \sci_X \lra 0
\]  
hence, by Ferrand's result, $\sci_Z(3)$ is the cohomology of a monad of the 
form$\, :$ 
\[
0 \lra 3\sco_\p \lra 5\sco_\p(1) \lra \sco_\p(2) \lra 0
\]
Since the kernel of any epimorphism $5\sco_\p \ra \sco_\p(1)$ is isomorphic 
to $\sco_\p \oplus \Omega_\p(1)$, it follows that $\sci_Z(3)$ admits a 
resolution as in item (iv) of the statement. 

\vskip2mm

\noindent
{\bf Subcase 3.3}\quad $X$ \emph{is a double structure on a line} $L$.  

\vskip2mm

\noindent
$X$ is defined by an epimorphism $\sci_L/\sci_L^2 \ra \sco_L(l)$, for some 
$l \geq -1$ (see Subsection~\ref{SS:p2}).  
From the \emph{fundamental exact sequence of liaison} (recalled in 
\cite[Remark~2.6]{acm})$\, :$ 
\[
0 \lra \sco_\p(-5) \lra \sco_\p(-2)\oplus \sco_\p(-3) \lra \sci_Z \lra 
\omega_X(-1) \lra 0
\]
it follows that $\sci_Z(3)$ is globally generated iff $\omega_X(2)$ is 
globally generated. But, as we recalled in Subsection~\ref{SS:p2} of 
Appendix~\ref{A:multilines}, $\omega_X \simeq \sco_X(-l-2)$. One deduces that 
$l \leq 0$. If $l = -1$ then $X$ is the divisor $2L$ on a plane $H \supset L$ 
and this situation has been treated in Subcase 3.1, and if $l = 0$ then, as 
we recalled in Subsection~\ref{SS:p2}, $\sci_X$ admits a resolution having 
the same (numerical) shape as that of the disjoint union of two lines hence, 
by the proof of Subcase 3.2, $\sci_Z(3)$ admits a resolution as in item (iv) 
of the statement.   
\end{proof} 

We assume, from now and until the end of this section, that $Z$ is a locally 
CM space curve of degree $d \geq 5$, such that $\sci_Z(3)$ is globally 
generated and $\tH^0(\sci_Z(2)) = 0$. $Z$ is linked by a complete intersection 
of type $(3,3)$ to a locally CM curve $Z^\prim$ of degree $d^\prim = 9-d$. The 
fundamental exact sequence of liaison$\, :$ 
\[
0 \lra \sco_\p(-6) \lra 2\sco_\p(-3) \lra \sci_Z \lra \omega_{Z^\prim}(-2) 
\lra 0
\] 
shows that the condition $\sci_Z(3)$ globally generated is equivalent to 
$\omega_{Z^\prim}(1)$ being globally generated, and the condition 
$\tH^0(\sci_Z(2)) = 0$ is equivalent to $\tH^0(\omega_{Z^\prim}) = 0$, which by 
Serre duality is equivalent to $\tH^1(\sco_{Z^\prim}) = 0$ hence to 
$\tH^2(\sci_{Z^\prim}) = 0$. 

\begin{prop}\label{P:zprimlci}
Using the above hypotheses and notation, assume that $Z^\prim$ is l.c.i. except 
at finitely many points. Then one of the following holds$\, :$ 
\begin{enumerate} 
\item[(i)] $d = 7$ and $\sci_Z(3)$ admits a resolution of the form$\, :$ 
\[
0 \lra \sco_\p(-1)\oplus \sco_\p(-2) \lra 3\sco_\p \lra \sci_Z(3) \lra 0\, ;
\]
\item[(ii)] $d = 6$ and $\sci_Z(3)$ admits a resolution of the form$\, :$ 
\[
0 \lra 3\sco_\p(-1) \lra 4\sco_\p \lra \sci_Z(3) \lra 0\, ;
\]
\item[(iii)] $d = 5$ and $\sci_Z(3)$ admits a resolution of the form$\, :$ 
\[
0 \lra 2\sco_\p \oplus \sco_\p(-1) \lra \Omega_\p(2) \oplus \sco_\p \lra 
\sci_Z(3) \lra 0\, ;
\]
\item[(iv)] $d = 5$ and $\sci_Z(3)$ admits a resolution of the form$\, :$ 
\[
0 \lra \sco_\p(1) \oplus \sco_\p \oplus \sco_\p(-1) \lra K \oplus \sco_\p  
\lra \sci_Z(3) \lra 0
\]
where $K$ is the kernel of an epimorphism $2\sco_\p(2)\oplus 2\sco_\p(1) \ra 
\sco_\p(3)$. 
\end{enumerate}
\end{prop}

\begin{proof}
Assume, firstly, that $\tH^0(\sci_{Z^\prim}(1)) \neq 0$. Then $Z^\prim$ is a 
complete intersection of type $(1,d^\prim)$. In this case $\omega_{Z^\prim} 
\simeq \sco_{Z^\prim}(d^\prim - 3)$ hence the two conditions $\omega_{Z^\prim}(1)$ 
globally generated and $\tH^0(\omega_{Z^\prim}) = 0$ imply that $d^\prim = 2$. 
In this case, by Ferrand's result about liaison, $\sci_Z(3)$ admits a 
resolution as in item (i) of the statement. 

\vskip2mm 

Assume, from now on, that $\tH^0(\sci_{Z^\prim}(1)) = 0$. Since $Z^\prim$ is 
l.c.i. except at finitely many points and since $\omega_{Z^\prim}(1)$ is 
globally generated, a general global section of this sheaf generates it 
except at finitely many points, hence it defines an extension$\, :$ 
\[
0 \lra \sco_\p \lra \scf^\prim(2) \lra \sci_{Z^\prim}(3) \lra 0
\]
where $\scf^\prim$ is a rank 2 \emph{reflexive sheaf} with Chern classes 
$c_1(\scf^\prim) = -1$ and $c_2(\scf^\prim) = d^\prim - 2$. Our assumption 
$\tH^0(\sci_{Z^\prim}(1)) = 0$ implies that $\tH^0(\scf^\prim) = 0$, hence 
$\scf^\prim$ is stable. Now, Hartshorne's results 
\cite[Cor.~3.3]{ha} and \cite[Thm.~8.2(d)]{ha} imply that 
$c_2(\scf^\prim) \geq 1$ hence $c_2(\scf^\prim) \in \{1,\, 2\}$ because 
$d^\prim \leq 4$ and that, moreover, if $c_2(\scf^\prim) = 1$ then 
$c_3(\scf^\prim) = 1$ and if $c_2(\scf^\prim) = 2$ then $c_3(\scf^\prim) \in 
\{0,\, 2,\, 4\}$. 

\vskip2mm

\noindent
{\bf Case 1.}\quad $c_2(\scf^\prim) = 1$. 

\vskip2mm

\noindent
In this case, by Hartshorne \cite[Lemma~9.4]{ha}, $\scf^\prim$ can be realised 
as an extension$\, :$ 
\[
0 \lra \sco_\p(-1) \lra \scf^\prim \lra \sci_L \lra 0
\] 
where $L$ is a line. One deduces that $\scf^\prim$ and, then, $\sci_{Z^\prim}$ 
have resolutions of the form$\ :$ 
\begin{gather*}
0 \lra \sco_\p(-2) \lra 3\sco_\p(-1) \lra \scf^\prim \lra 0\\
0 \lra 2\sco_\p(-3) \lra 3\sco_\p(-2) \lra \sci_{Z^\prim} \lra 0\, . 
\end{gather*}
Using Ferrand's result about liaison, one gets that $\sci_Z(3)$ admits a 
resolution as in item (ii) of the statement. 

\vskip2mm

\noindent
{\bf Case 2.}\quad $c_2(\scf^\prim) = 2$, $c_3(\scf^\prim) = 4$. 

\vskip2mm

\noindent
In this case, by Hartshorne \cite[Lemma~9.6]{ha}, $\scf^\prim$ can be realised 
as an extension$\, :$ 
\[
0 \lra \sco_\p(-1) \lra \scf^\prim \lra \sci_Y \lra 0
\]
where $Y$ is a plane curve of degree 2. But in this case $\tH^2(\sci_Y(-1))  
\neq 0$, which implies $\tH^2(\scf^\prim(-1)) \neq 0$, 
hence $\tH^2(\sci_{Z^\prim}) 
\neq 0$ and this contradicts our hypothesis that $\tH^0(\sci_Z(2)) = 0$ (see 
the discussion preceeding the statement of the proposition). 

\vskip2mm

\noindent
{\bf Case 3.}\quad $c_2(\scf^\prim) = 2$, $c_3(\scf^\prim) = 2$. 

\vskip2mm

\noindent
In this case, a result of Chang \cite[Lemma~2.4]{ch} implies that $\scf^\prim$ 
can be realized as an extension$\, :$ 
\[
0 \lra \sco_\p(-1) \lra \scf^\prim \lra \sci_Y \lra 0 
\]
where either $Y$ is the union of two disjoint lines or it is a double structure 
on a line $L$ defined by an epimorphism $\sci_L/\sci_L^2 \ra \sco_L$. In 
both cases, $Y$ admits a resolution of the form (see Subsection~\ref{SS:p2} 
and Lemma~\ref{L:zcupw})$\, :$ 
\[
0 \lra \sco_\p(-4) \overset{\displaystyle d_2}{\lra}  
4\sco_\p(-3) \overset{\displaystyle d_1}{\lra}  
4\sco_\p(-2) \lra \sci_Y \lra 0
\] 
One gets resolutions for $\scf^\prim$ and $\sci_{Z^\prim}$ 
of the form$\, :$ 
\begin{gather*}
0 \lra \sco_\p(-4) \overset{\displaystyle d_2}{\lra}  
4\sco_\p(-3) \xra{\begin{pmatrix} \ast\\ d_1 \end{pmatrix}} 
\sco_\p(-1)\oplus 4\sco_\p(-2) \lra \scf^\prim \lra 0\\
0 \lra \sco_\p(-5) \xra{\begin{pmatrix} 0\\ d_2(-1)\end{pmatrix}} 
\sco_\p(-3) \oplus 4\sco_\p(-4) \lra \sco_\p(-2) 
\oplus 4\sco_\p(-3) \lra \sci_{Z^\prim} \lra 0
\end{gather*} 
Using Ferrand's result about liaison, one deduces that $\sci_Z(3)$ has a 
monad of the form$\, :$ 
\[
0 \lra 4\sco_\p \oplus \sco_\p(-1) \lra 4\sco_\p(1) \oplus 3\sco_\p 
\xra{\displaystyle (d_2^\vee(-2)\, ,\, 0)} \sco_\p(2) \lra 0
\]  
In this way, one gets an exact sequence$\, :$ 
\[
0 \lra 4\sco_\p \oplus \sco_\p(-1) \overset{\displaystyle \alpha}{\lra}  
\Omega_\p(2) \oplus 3\sco_\p \lra \sci_Z(3) \lra 0\, .
\] 
The component $4\sco_\p \ra 3\sco_\p$ of $\alpha$ must have rank at least 2 
because any monomorphism $3\sco_\p \ra \Omega_\p(2)$ degenerates along a 
surface in $\piii$. It follows that at least 2 direct summands $\sco_\p$ 
from the above resolution of $\sci_Z(3)$ must cancel (see 
Remark~\ref{R:cancellation} from Appendix~\ref{A:multilines}), hence 
$\sci_Z(3)$ has a resolution as in item (iii) of the statement. 

\vskip2mm

\noindent
{\bf Case 4.}\quad $c_2(\scf^\prim) = 2$, $c_3(\scf^\prim) = 0$. 

\vskip2mm

\noindent
In this case, according to the results of Hartshorne and Sols 
\cite[Prop.~2.1]{hs} or of Manolache \cite[Cor.~2]{m1}, $\scf^\prim$ can be 
realised as an extension$\, :$ 
\[
0 \lra \sco_\p(-1) \lra \scf^\prim \lra \sci_Y \lra 0
\] 
where $Y$ is a double structure on a line $L$ defined by an epimorphism 
$\sci_L/\sci_L^2 \ra \sco_L(1)$. According to the results recalled in 
Subsection~\ref{SS:p2}, $\sci_Y$ has a resolution of the form$\, :$ 
\[
0 \lra \sco_\p(-5) \overset{\displaystyle d_2}{\lra} 
2\sco_\p(-3) \oplus 2\sco_\p(-4) \overset{\displaystyle d_1}{\lra} 
3\sco_\p(-2) \oplus \sco_\p(-3) \lra \sci_Y \lra 0\, .
\] 
One deduces that $\scf^\prim$ and, then, $\sci_{Z^\prim}$ have resolutions of 
the form$\, :$ 
\begin{gather*}
0 \ra \sco_\p(-5) \overset{\displaystyle d_2}{\lra} 
2\sco_\p(-3) \oplus 2\sco_\p(-4) 
\xra{\begin{pmatrix} \ast\\ d_1 \end{pmatrix}} 
\sco_\p(-1) \oplus 3\sco_\p(-2) \oplus \sco_\p(-3) \lra \scf^\prim \ra 0\\
0 \ra \sco_\p(-6) \xra{\begin{pmatrix} 0\\ d_2(-1) \end{pmatrix}} 
\sco_\p(-3) \oplus 2\sco_\p(-4) \oplus 2\sco_\p(-5) 
\ra 
\sco_\p(-2) \oplus 3\sco_\p(-3) \oplus \sco_\p(-4)\\ 
\lra \sci_{Z^\prim} \lra 0\, . 
\end{gather*} 
Now, the result of Ferrand about liaison implies that $\sci_Z(3)$ has a monad 
of the form$\, :$ 
\[  
0 \lra \sco_\p(1) \oplus 3\sco_\p \oplus \sco_\p(-1) \lra 
2\sco_\p(2) \oplus 2\sco_\p(1) \oplus 3\sco_\p 
\xra{\displaystyle (d_2^\vee(-2)\, ,\, 0)} \sco_\p(3) \lra 0\, . 
\]
Denoting by $K$ the kernel of the epimorphism $d_2^\vee(-2) : 2\sco_\p(2) 
\oplus 2\sco_\p(1) \ra \sco_\p(3)$, one gets an exact sequence$\, :$ 
\[
0 \lra \sco_\p(1) \oplus 3\sco_\p \oplus \sco_\p(-1) 
\overset{\displaystyle \alpha}{\lra} K \oplus 3\sco_\p \lra \sci_Z(3) 
\lra 0\, .
\]
The component $3\sco_\p \ra 3\sco_\p$ of $\alpha$ must have rank at least 2, 
because any monomorphism $\sco_\p(1) \oplus 2\sco_\p \ra K$ degenerates along 
a surface in $\piii$. One deduces that at least two direct summands $\sco_\p$ 
from the above resolution of $\sci_Z$ cancel, 
hence $\sci_Z(3)$ admits a resolution as in item (iv) of the statement. 
\end{proof} 

\begin{prop}\label{P:zprimcm}
Using the hypotheses and notation established before Prop.~\ref{P:zprimlci}, 
assume that $Z^\prim$ is not l.c.i. except finitely many points. Then 
$d = 6$ and $\sci_Z(3)$ admits a resolution of the form$\, :$ 
\[
0 \lra 3\sco_\p(-1) \lra 4\sco_\p \lra \sci_Z(3) \lra 0\, .
\]
\end{prop} 

\begin{proof} 
There are only three possibilities for $Z^\prim$.  

\vskip2mm

\noindent
{\bf Case 1.}\quad $Z^\prim = L^{(1)} =$ \emph{first infinitesimal neighbourhood 
of a line} $L \subset \piii$. 

\vskip2mm

\noindent
$Z^\prim$ is the subscheme of $\piii$ defined by $\sci_L^2$, hence 
$\sci_{Z^\prim}$ admits a resolution of the form$\, :$ 
\[
0 \lra 2\sco_\p(-3) \lra 3\sco_\p(-2) \lra \sci_{Z^\prim} \lra 0
\]
which, by dualization, provides a resolution of $\omega_{Z^\prim}$ 
\[
0 \lra \sco_\p(-4) \lra 3\sco_\p(-2) \lra 2\sco_\p(-1) \lra \omega_{Z^\prim} 
\lra 0\, .
\]
One sees that, indeed, $\omega_{Z^\prim}(1)$ is globally generated and that 
$\tH^0(\omega_{Z^\prim}) = 0$. Applying Ferrand result about liaison to the 
above resolution of $\sci_{Z^\prim}$ one gets that $\sci_Z(3)$ admits a 
resolution as in the statement. 

\vskip2mm

\noindent
{\bf Case 2.}\quad $Z^\prim = L^{(1)} \cup L^\prim$, \emph{with} $L$ 
\emph{and} $L^\prim$ \emph{distinct lines}. 

\vskip2mm

\noindent
If $L\cap L^\prim = \emptyset$ then $\omega_{Z^\prim}(1)$ is not globally 
generated because $\omega_{L^\prim}(1) \simeq \sco_{L^\prim}(-1)$. 
If $L \cap L^\prim 
\neq \emptyset$ then, choosing conveniently the homogeneous coordinates on 
$\piii$, one may assume that $L$ has equations $x_2 = x_3 = 0$ and 
$L^\prim$ has equations $x_1 = x_3 = 0$. Then $I(Z^\prim) = (x_2x_3,\, x_3^2,\, 
x_1x_2^2)$ (see Lemma~\ref{L:l1(1)cupl} with $c = 0$). 
$Z^\prim$ is linked, via the complete intersection defined 
by $x_3^2$ and $x_1x_2^2$, to $L\cup L^\prim$. The fundamental exact sequence 
of liaison$\, :$ 
\[
0 \lra \sco_\p(-5) \lra \sco_\p(-2) \oplus \sco_\p(-3) \lra 
\sci_{L\cup L^\prim} \lra \omega_{Z^\prim}(-1) \lra 0
\] 
and the fact that $\tH^0(\sci_{L\cup L^\prim}(1)) \neq 0$ implies that 
$\tH^0(\omega_{Z^\prim}) \neq 0$, which \emph{contradicts} our hypothesis (see 
the discussion before Prop.~\ref{P:zprimlci}). 

\vskip2mm

\noindent
{\bf Case 3.}\quad $Z^\prim$ \emph{is a thick structure of degree} 4 
\emph{on a line} $L$. 

\vskip2mm

\noindent
In this case, by the results recalled in Subsection~\ref{SS:thick} of 
Appendix~\ref{A:multilines}, one has an exact sequence$\, :$ 
\[
0 \lra \sco_L(l) \lra \sco_{Z^\prim} \lra \sco_{L^{(1)}} \lra 0
\]
for some $l \geq -2$. Applying $\sce xt^2_{\sco_\p}(-,\omega_\p)$, one gets an 
exact sequence$\, :$ 
\[
0 \lra \omega_{L^{(1)}} \lra \omega_{Z^\prim} \lra \omega_L(-l) \lra 0\, .
\]
But $\omega_L(-l) \simeq \sco_L(-l-2)$ hence $\omega_{Z^\prim}(1)$ globally 
generated implies $l \leq -1$. Since we assumed that $Z^\prim$ is not 
l.c.i. except at finitely many points, Prop.~\ref{P:geniwcm} implies that 
$l$ must be even, hence $l = -2$. Since $\text{h}^0(\sco_{L^{(1)}}) = 1$ and 
$\tH^1(\sco_L(-2)) \neq 0$, the above exact sequence involving 
$\sco_{Z^\prim}$ shows that $\tH^1(\sco_{Z^\prim}) \neq 0$, which 
\emph{contradicts} our hypothesis (see the discussion before 
Prop.~\ref{P:zprimlci}).     
\end{proof}

\section{Curves of degree at most 4}\label{S:deg4} 

We want to classify, in this section, the locally CM curves $Z$ in $\piii$, 
l.c.i. except at finitely many points, of degree $d \leq 4$, with $\sci_Z(3)$ 
globally generated. More precisely, we want to list the monads of the ideal 
sheaves of these curves. According to Prop.~\ref{P:zinci23}, we can assume,  
in the case $d = 4$, that $\tH^0(\sci_Z(2)) = 0$.  

\begin{lemma}\label{L:h0iz(3)geq5}
If $d = 4$ and $\sci_Z(3)$ is globally generated then ${\fam0 h}^0(\sci_Z(3)) 
\geq 5$. 
\end{lemma} 

\begin{proof} The assertion is clear if $\tH^0(\sci_Z(2)) \neq 0$. Assume, 
from now on, that $\tH^0(\sci_Z(2)) = 0$. Since $Z$ is not a complete 
intersection of type (3,3) it follows that $\text{h}^0(\sci_Z(3)) \geq 3$. 

If $\text{h}^0(\sci_Z(3)) = 3$ then one has an exact sequence$\, :$ 
\[
0 \lra E \lra 3\sco_\p \lra \sci_Z(3) \lra 0
\] 
with $E$ a rank 2 vector bundle. Let us recall the following fact$\, :$ if 
$W$ is a closed subscheme of $\piii$ of dimension $\leq 1$ and if 
$\text{deg}\, W$ is defined by $\chi(\sco_W(t)) = t\text{deg}\, W + 
\chi(\sco_W)$, $\forall \, t\in \z$, (such that $\text{deg}\, W = 0$ if 
$\dim W \leq 0$) then$\, :$ 
\[
c_1(\sci_W(t)) = t\  \text{and}\  c_2(\sci_W(t)) = \text{deg}\, W\, .
\]
(\emph{Indeed}, it suffices to compute $c_1(\sco_W(t))$ and $c_2(\sco_W(t))$. 
If $H \subset \piii$ is a general plane then the latter Chern classes are 
(numerically) equal to the corresponding Chern classes of $\sco_{H\cap W}(t)$ 
on $H \simeq \pii$ and these can be computed as in \cite[Lemma~2.7]{ha}.) 

One gets that $c_1(E) = -3$ and $c_2(E) = 5$ but this is not possible because 
the Chern classes of a rank 2 vector bundle on $\piii$ must satisfy the 
relation $c_1c_2 \equiv 0\, (\text{mod}\  2)$ (by the Riemann-Roch formula). 
It thus remains that $\text{h}^0(\sci_Z(3)) \geq 4$. 

Assume, finally, that $\text{h}^0(\sci_Z(3)) = 4$. Eliminating this case turns 
out to be more complicated. One has an exact sequence$\, :$ 
\[
0 \lra F \lra 4\sco_\p \lra \sci_Z(3) \lra 0
\]    
where $F$ is a rank 3 vector bundle with Chern classes $c_1(F) = -3$ and 
$c_2(F) = 5$. By our assumption, $\tH^i(F) = 0$, $i = 0\, ,1$.  
Now, $Z$ is directly linked by a complete intersection of type $(3,3)$ to a 
curve $Z^\prim$ of degree 5. By Ferrand's result about liaison, one gets an 
exact sequence$\, :$ 
\begin{equation}\label{E:fdual} 
0 \lra 2\sco_\p \lra F^\vee \lra \sci_{Z^\prim}(3) \lra 0\, .
\end{equation}
One cannot have $\tH^0(\sci_{Z^\prim}(2)) \neq 0$ because this would imply that 
$Z^\prim$ is directly linked to a line by a complete intersection of type 
(2,3) which would imply in turn, by applying twice Ferrand's result about 
liaison, that $\tH^0(\sci_Z(2)) \neq 0$. It follows that 
$\tH^0(F^\vee(-1)) = 0$. 

On the other hand, $\tH^0(F(1)) \neq 0$. \emph{Indeed}, 
let $H \subset \piii$ be a plane cutting $Z$ properly. $H \cap Z$ is a 
0-dimensional subscheme of length 4 of $H$ hence $\h^1(\sci_{H \cap Z,H}) = 3$. 
One deduces, from the Lemma of Le Potier (recalled in 
\cite[Lemma~1.22]{acm}), that $\tH^1(\sci_{H \cap Z,H}(l)) = 0$ for $l \geq 3$. 
One gets, in particular, that $\h^0(\sci_{H \cap Z,H}(4)) = 15 - 4 = 11$ hence, 
using the exact sequence $\, :$ 
\[
0 \lra \sci_Z(3) \lra \sci_Z(4) \lra \sci_{H \cap Z,H}(4) \lra 0\, ,
\]  
$\h^0(\sci_Z(4)) \leq 11 + 4 = 15$. One deduces, from the exact sequence 
defining $F$, that $\h^0(F(1)) \geq 4 \times 4 - 15 = 1$.  

Consider, now, the normalized rank 3 vector bundle $G := F(1)$. It has $c_1(G) 
= 0$ and $c_2(G) = 2$. Moreover, $\tH^i(G(-1)) = 0$, $i = 0\, ,1$, 
$\tH^0(G^\vee) = 0$ and $\tH^0(G) \neq 0$. 
Since $\tH^0(G(-1)) = 0$, the scheme of zeroes $W$ of a non-zero global 
section of $G$ must have dimension $\leq 1$. One has an exact sequence$\, :$ 
\begin{equation}\label{E:fgdualiw}
0 \lra \scf \lra G^\vee \lra \sci_W \lra 0
\end{equation}
where $\scf$ is a rank 2 reflexive sheaf with $c_1(\scf) = 0$. Since 
$\tH^0(\scf) = 0$, $\scf$ is stable. It follows that $c_2(\scf) \geq 1$ (for 
results about stable rank 2 reflexive sheaves one may consult Hartshorne 
\cite{ha}, especially Sect. 7). On the other hand, one gets, using the exact 
sequence \eqref{E:fgdualiw}, that$\, :$ 
\[
c_2(\scf) + \text{deg}\, W = c_2(G^\vee) = 2\, .
\]  
If $c_2(\scf) = 1$ then $\scf$ is a nullcorrelation bundle. In particular, it 
is locally free. One deduces, from the exact sequence \eqref{E:fgdualiw}, 
that $W$ is locally CM of pure codimension 2 in $\piii$. Since 
$\text{deg}\, W = 1$, $W$ is a line. $\scf$ being a nullcorrelation bundle 
one has $\text{h}^2(\scf(-3)) = 1$. It follows that $\tH^2(G^\vee(-3)) \neq 0$ 
hence, by Serre duality, $\tH^1(G(-1)) \neq 0$, a \emph{contradiction}. 

It remains that $c_2(\scf) = 2$. In this case $\text{deg}\, W = 0$, i.e., 
$\dim W \leq 0$. Dualizing the exact sequence \eqref{E:fgdualiw} and taking 
into account that $\scf^\vee \simeq \scf$ (because $c_1(\scf) = 0$), one gets 
an exact sequence$\, :$ 
\begin{equation}\label{E:opgf}
0 \lra \sco_\p \lra G \lra \scf \lra 0\, .
\end{equation}
The possible \emph{spectra} of $\scf$ are $(0,0)$, $(-1,0)$ 
and $(-1,-1)$. For the 
first two spectra one has $\tH^1(\scf(-1)) \neq 0$ which is not possible 
because $\tH^1(G(-1)) = 0$. It remains that the spectrum of $\scf$ is 
$(-1,-1)$. In particular, $c_3(\scf) = 4$, hence, by \cite[Prop.~2.6]{ha}, 
$\text{h}^0(\sce xt^1_{\sco_\p}(\scf,\sco_\p)) = 4$. Dualizing the exact sequence 
\eqref{E:opgf}, one gets that $\sce xt^1_{\sco_\p}(\scf,\sco_\p) \simeq \sco_W$. 
Now, by Chang \cite[Lemma~2.9]{ch}, $\scf$ admits a resolution of the 
form$\, :$ 
\[
0 \lra 2\sco_\p(-2) \overset{\displaystyle \phi}{\lra} 4\sco_\p(-1) \lra 
\scf \lra 0\, .
\]
Dualizing this exact sequence one gets the following presentation$\, :$ 
\[
4\sco_\p(1) \xra{\displaystyle \phi^\vee} 2\sco_\p(2) \lra 
\sce xt^1_{\sco_\p}(\scf,\sco_\p) \lra 0\, .
\]
Since $\sce xt^1_{\sco_\p}(\scf,\sco_\p) \simeq \sco_W$, the properties of 
Fitting ideals show that $W$ is defined, as a closed subscheme of $\piii$, 
by the $2\times 2$ minors of the $2\times 4$ matrix of linear forms 
defining $\phi^\vee$. Considering the Eagon-Northcott complex associated 
to $\phi^\vee(-2) : 4\sco_\p(-1) \ra 2\sco_\p$, one gets the following 
resolution of $\sco_W$$\, :$ 
\[
0 \lra 3\sco_\p(-4) \lra 8\sco_\p(-3) \lra 6\sco_\p(-2) \lra \sco_\p \lra 
\sco_W \lra 0\, .
\] 
One deduces that $\tH^0(\sci_W(1)) = 0$. One gets, now, from the exact 
sequence \eqref{E:fgdualiw}, that $\tH^0(\scf(1)) \izo \tH^0(G^\vee(1))$. It 
follows that the dependence locus of any two global sections of $G^\vee(1)$ 
contains a surface in $\piii$. Since $G^\vee(1) \simeq F^\vee$, this 
\emph{contradicts} the existence of the exact sequence \eqref{E:fdual}.  
\end{proof}    

\begin{lemma}\label{L:yconnected} 
Let $Y$ be a reduced, connected curve of degree $3$ in $\piii$. Then either 
$Y$ is a complete intersection of type $(1,3)$ or $Y$ is directly linked to 
a line by a complete intersection of type $(2,2)$.  
\end{lemma}

\begin{proof} 
Assume that $\tH^0(\sci_Y(1)) = 0$. Then one can choose three noncolinear 
points $P_1$, $P_2$, $P_3$ of $Y$ such that the plane $H$ determined by them 
intersects $Y$ properly. In this case the map $\tH^0(\sco_\p(1)) \ra 
\tH^0(\sco_{\{P_1,P_2,P_3\}}(1))$ is surjective, hence the map 
$\tH^0(\sco_Y(1)) \ra \tH^0(\sco_{H\cap Y}(1))$ is surjective, too. Using the 
exact sequences$\, :$ 
\[
0 \lra \sco_Y(i) \lra \sco_Y(i+1) \lra \sco_{H\cap Y}(i+1) \lra 0
\] 
one derives (when $i = 0$) that $\text{h}^0(\sco_Y(1)) = 4$ and that 
$\tH^1(\sco_Y(i)) \ra \tH^1(\sco_Y(i+1))$ is injective for $i \geq 0$, hence 
$\tH^1(\sco_Y) = 0$. Using the exact sequence$\, :$ 
\[
0 \lra \sci_Y(1) \lra \sco_\p(1) \lra \sco_Y(1) \lra 0
\]
one deduces that $\tH^1(\sci_Y(1)) = 0$. Since $\tH^2(\sci_Y) \simeq 
\tH^1(\sco_Y) = 0$ it follows that $\sci_Y$ is 2-regular, hence $\sci_Y(2)$ 
is globally generated. 
\end{proof}

\begin{lemma}\label{L:h0iz(2)neq0} 
Let $Z$ be a reduced, connected curve of degree $4$ in $\piii$. Then 
${\fam0 H}^0(\sci_Z(2)) \neq 0$. 
\end{lemma}

\begin{proof}
Let $H\subset \piii$ be a plane intersecting $Z$ properly. Using the 
exact sequences$\, :$ 
\[
0 \lra \sco_Z(i) \lra \sco_Z(i+1) \lra \sco_{H\cap Z}(i+1) \lra 0\, ,
\] 
for $i = 0$ and 1, one deduces that$\, :$ 
\[
\text{h}^0(\sco_Z(2)) \leq \text{h}^0(\sco_Z(1)) + \text{h}^0(\sco_{H\cap Z}(2)) 
\leq \text{h}^0(\sco_Z) + \text{h}^0(\sco_{H\cap Z}(1)) + 
\text{h}^0(\sco_{H\cap Z}(2)) = 9\, .
\]
Using, now, the exact sequence$\, :$ 
\[
0 \lra \sci_Z(2) \lra \sco_\p(2) \lra \sco_Z(2) \lra 0
\]
and the fact that $\text{h}^0(\sco_\p(2)) = 10$ one derives that 
$\text{h}^0(\sci_Z(2)) \geq 1$. 
\end{proof}

\begin{lemma}\label{L:xcupxprim} 
Let $X$ and $X^\prim$ be two disjoint locally CM curves of degree $2$ in 
$\piii$ and let $Z = X\cup X^\prim$. Then $Z$ admits a $4$-secant hence   
$\sci_Z(3)$ is not globally generated. 
\end{lemma}

\begin{proof}
$X$ (and, similarly, $X^\prim$) can be a complete intersection of type (1,2), 
the union of two disjoint lines $L$ and $L^\prim$, or a double structure on a 
line $L$. In the first case, $X$ is contained in a plane $H$ and every line 
in $H$ intersecting $X$ properly is a 2-secant to $X$. In the second case, 
every line joining a point of $L$ and a point of $L^\prim$ is a 2-secant to 
$X$. In the third case, assume that $L$ is the line of equations $x_2 = x_3 
= 0$. As it is well known, there exist an integer $l \geq -1$ and two coprime 
polynomials $a,\, b \in k[x_0,x_1]_{l+1}$ such that the holomogeneous ideal 
$I(X) \subset S$ of $X$ is generated by$\, :$ 
\[
F_2 = -bx_2 + ax_3,\, x_2^2,\, x_2x_3,\, x_3^2
\]
(this result is recalled in Subsection~\ref{SS:p2} from 
Appendix~\ref{A:multilines}). 
The surface $\Sigma \subset \piii$ of equation $F_2 = 0$ 
is nonsingular along $L$. If $L_1$ is another line meeting $L$ in a point 
$z$ then $L_1$ is a 2-secant to $X$, i.e., $\text{deg}(L_1\cap X) = 2$ if and 
only if $L_1$ is contained in the projective tangent plane $\text{T}_z\Sigma$ 
to $\Sigma$ at $z$. This tangent plane has equation$\, :$ 
\[
-b(z)x_2 + a(z)x_3 = 0\, .
\]  

We will show that $X$ and $X^\prim$ have a common 2-secant which is, 
consequently, a 4-secant to $Z$. We split the proof of this fact into 
several cases according to the nature of $X$ and $X^\prim$. 

\vskip2mm 

\noindent
{\bf Case 1.}\quad $X$ \emph{is a double structure on the line} $L$ 
\emph{of equations} $x_2 = x_3 = 0$ \emph{and}  
$X^\prim$ \emph{is a double structure on the line} $L^\prim$ \emph{of equations} 
$x_0 = x_1 = 0$. 

\vskip2mm

\noindent
Consider the morphism $\phi : L \ra L^\prim$ defined by $\{\phi(z)\} = 
\text{T}_z\Sigma \cap L^\prim$. One has$\, :$ 
\[
\phi((z_0:z_1:0:0)) = (0:0:a(z_0,z_1):b(z_0,z_1))\, .
\] 
For $z\in L$, the line joining $z$ to $\phi(z) \in L^\prim$ is a 2-secant to 
$X$. Consider the similar morphism $\phi^\prim : L^\prim \ra L$. The composite 
morphism $\phi^\prim \circ \phi : L \ra L$ has a fixed point $z$, hence, if 
$z^\prim = \phi(z)$ then $\phi^\prim(z^\prim) = z$. The line joining $z$ to 
$z^\prim$ is a 2-secant to both $X$ and $X^\prim$.  

\vskip2mm 

\noindent
{\bf Case 2.}\quad $X$ \emph{is a double structure on the line} $L$ 
\emph{of equations} $x_2 = x_3 = 0$ \emph{and} $X^\prim$ \emph{is a complete 
intersection of type} (1,2). 

\vskip2mm

\noindent
Let $H \subset \piii$ be the plane containing $X^\prim$, let $\{z\} = H\cap L$ 
and let $L_1 = \text{T}_z\Sigma \cap H$. Then $L_1$ is a 2-secant to both 
$X$ and $X^\prim$.  

\vskip2mm

\noindent
{\bf Case 3.}\quad $X$ \emph{is a double structure on the line} $L$ 
\emph{of equations} $x_2 = x_3 = 0$ \emph{and} 
$X^\prim = L^\prim \cup L^\secund$, \emph{where} $L^\prim$ 
\emph{and} $L^\secund$ \emph{are disjoint lines not intersecting} $L$. 

\vskip2mm

\noindent 
Consider the morphism $\phi : L \ra L^\prim$ defined in Case 1. Consider 
also the morphism $\psi : L \ra L^\prim$ defined by$\, :$ 
\[
\{\psi(z)\} = \text{span}(\{z\} \cup L^\secund) \cap L^\prim\, .
\]
$\psi$ is, actually, an isomorphism. Then $\psi^{-1}\circ \phi : L \ra L$ has 
a fixed point $z$. Let $w = \phi(z) = \psi(z)$. The line joining $z$ and $w$ 
is a 2-secant to $X$ and meets both $L^\prim$ and $L^\secund$, hence it is a 
2-secant to $X^\prim$, too. 

\vskip2mm 

\noindent 
{\bf Case 4.}\quad $X$ \emph{and} $X^\prim$ \emph{are both complete 
intersections of type} (1,2). 

\vskip2mm 

\noindent 
$X$ is contained in a plane $H$ and $X^\prim$ is contained in a plane $H^\prim$. 
The line $H\cap H^\prim$ is a 2-secant to both $X$ and $X^\prim$. 

\vskip2mm

\noindent 
{\bf Case 5.}\quad $X$ \emph{is a complete intersection of type} (1,2) 
\emph{and} $X^\prim$ \emph{is the union of two disjoint lines} $L^\prim$ 
\emph{and} $L^\secund$ \emph{not intersecting} $X$. 

\vskip2mm 

\noindent 
$X$ is contained in a plane $H$ and $L^\prim$ (resp., $L^\secund$) intersects 
$H$ in a point $z^\prim$ (resp., $z^\secund$). The line joining $z^\prim$ and 
$z^\secund$ is a 2-secant to both $X$ and $X^\prim$. 

\vskip2mm

\noindent 
{\bf Case 6.}\quad $Z$ \emph{is the union of four mutually disjoint lines}. 

\vskip2mm

\noindent 
This case is well-known$\, :$ three of the lines are contained in a unique 
nonsingular quadric surface $Q\subset \piii$. Let $z$ be a point from the 
intersection of $Q$ with the fourth line. Then the line contained in $Q$, 
passing through $z$, and not belonging to the ruling of $Q$ to which the first 
three lines belong intersects all the four lines.  
\end{proof}

\begin{lemma}\label{L:ycuplprim} 
Let $Y$ be a locally CM curve of degree $3$ in $\piii$, 
$L^\prim\subset \piii$ a line not intersecting $Y$, and $Z = Y\cup L^\prim$. 
If $\sci_Z(3)$ is globally generated then $Y$ is directly linked to a line 
by a complete intersection of type $(2,2)$. 
\end{lemma}

\begin{proof}
Lemma~\ref{L:xcupxprim} implies that $Y$ is connected. $Y$ cannot 
be a complete intersection of type (1,3) because, in that case, any line 
contained in the plane $H\supset Y$ and passing through the intersection 
point of $H$ and $L^\prim$ is a 4-secant to $Z$. If $Y$ is reduced then, by 
Lemma~\ref{L:yconnected}, $Y$ is directly linked to a line by a complete 
intersection of type (2,2). 

Consequently, we can assume that $Y$ is connected and nonreduced. There are 
two cases to be considered. 

\vskip2mm

\noindent 
{\bf Case 1.}\quad $Y$ \emph{is a triple structure on a line}. 

\vskip2mm

\noindent 
In this case, since the first infinitesimal neighbourhood of a line in 
$\piii$ is directly linked to a line by a complete intersection of type (2,2), 
we can assume that $Y$ is a quasiprimitive triple structure on the line 
$L_1$ of equations $x_2 = x_3 = 0$ and that $L^\prim$ is the line of 
equations $x_0 = x_1 = 0$. We use the notation and results from 
Subsection~\ref{SS:triplecupaline} of Appendix~\ref{A:multicomponent}. 

If $l = -1$ and $m = 1$ then, by Lemma~\ref{L:l=-1m=1deg3}, $Y$ is directly 
linked to $L_1$ by a complete intersection of type $(2,2)$. If $l = -1$ and 
$m \neq 1$ then, using the notation from Lemma~\ref{L:l=-1deg3}, consider a 
divisor $\lambda \in k[x_0, x_1]_1$ of $v_2$. The line of equations $F_2 = 
\lambda = 0$ is a 3-secant to $Y$ hence a 4-secant to $Z = Y \cup L^\prim$. 

If $l = m = 0$ then $Y$ is the divisor $3L_1$ on 
a nonsingular quadric surface $Q\supset L_1$. If $P\in Q\cap L^\prim$ then the 
line contained in $Q$, passing through $P$ and intersecting $L_1$ is a 
4-secant to $Z$. 

Finally, if $l = 0$ and $m\geq 1$ or if $l\geq 1$ then, by 
Prop.~\ref{P:geniycupl1prim}(b) and by Lemma~\ref{L:h0iz(3)geq5}, $\sci_Z(3)$ 
is not globally generated. 

\vskip2mm

\noindent 
{\bf Case 2.}\quad $Y$ \emph{is the union of a double structure on a line  
and of another line intersecting it}. 

\vskip2mm 

\noindent 
In this case, we can assume that $Y = X\cup L$, where $X$ is a double 
structure on the line $L_1$ of equations $x_2 = x_3 = 0$ and $L$ is the line 
of equations $\ell = x_3 = 0$, $\ell = x_1 + cx_2$, $c\neq 0$, and that 
$L^\prim$ is the line of equations $x_0 = x_1 = 0$. We use the results 
stated in Prop.~\ref{P:genixcuplcupl1prim} and the notation appearing there. 

If $x_1 \mid b$ and $l = -1$ then, by the proof of 
Prop.~\ref{P:genixcuplcupl1prim}(a), $Y$ is a complete intersection of type 
(1,3) and this case was excluded above. 

If $x_1 \mid b$ and $l = 0$ then, by the proof of 
Prop.~\ref{P:genixcuplcupl1prim}(b), $Y$ is directly linked to a line by a 
complete intersection of type (2,2). 

If $x_1 \mid b$ and $l \geq 1$ then, by Prop.~\ref{P:genixcuplcupl1prim}(c), 
choosing a divisor $\lambda \in k[x_0,x_1]_1$ of $b_1$, the line of equations 
$\lambda = x_3 = 0$ is a 4-secant to $Z$. 

If $x_1 \nmid b$ and $l = -1$ then, by the proof of 
Prop.~\ref{P:genixcuplcupl1prim}(d), $Y$ is directly linked to a line by a 
complete intersection of type (2,2). 

Finally, if $x_1 \nmid b$ and $l \geq 0$ then, by 
Prop.~\ref{P:genixcuplcupl1prim}(e), choosing a divisor $\lambda \in 
k[x_0,x_1]_1$ of $b$, the line of equations $\lambda = x_3 = 0$ is a 4-secant 
to $Z$.         
\end{proof}

\begin{lemma}\label{L:d4zquasiprimline} 
Assume that $Z$ is a quasiprimitive structure of degree $4$ on the line 
$L \subset \piii$ of equations $x_2 = x_3 = 0$. If ${\fam0 H}^0(\sci_Z(2)) = 0$ 
then $\sci_Z(3)$ is not globally generated.  
\end{lemma} 

\begin{proof} 
We use Prop.~\ref{P:geniz} and Prop.~\ref{P:genizprim} from 
Appendix~\ref{A:multilines} and the notation from that appendix. One has  
$l \geq 0$ because, in the case $l = -1$, $a$ and $b$ are 
constants, not simultaneously 0, $F_2 = -bx+ay$, and $G_1 = F_2^2$, 
hence $\tH^0(\sci_Z(2)) \neq 0$. The conclusion follows, now, from 
Lemma~\ref{L:h0iz(3)geq5} using Prop.~\ref{P:geniz} and 
Prop.~\ref{P:genizprim}.   
\end{proof}

\begin{remark}
We take this opportunity to complete the argument for assertion 1.2(iii) from 
the paper of Gruson and Skiti \cite{gs} by showing that 
if $Z$ is a primitive structure of degree 4 on a line $L\subset \piii$ 
(let's say, of equations $x_2 = x_3 = 0$), with $l = 0$ and such that 
$\tH^0(\sci_Z(2)) = 0$, then $Z$ admits two 4-secants. 

\emph{Indeed}, the conditions $l = 0$ and $\tH^0(\sci_Z(2)) = 0$ imply that  
$Z$ doesn't satisfy the hypothesis of Prop.~\ref{P:geniz} but it satisfies 
the hypothesis of Prop.~\ref{P:genizprim}.    
One may assume, now,  that $a = x_0$ and $b = x_1$ hence$\, :$ 
\begin{gather*}
F_2 = -x_1x_2 + x_0x_3\, ,\  F_3 = F_2 + v_0x_2^2 + v_1x_2x_3 + v_2x_3^2\, ,\\ 
F_{40} = x_0F_3 + w_{00}x_2^3 + w_{01}x_2^2x_3 + w_{02}x_2x_3^2 + w_{03}x_3^3\, ,\\  
F_{41} = x_1F_3 + w_{10}x_2^3 + w_{11}x_2^2x_3 + w_{12}x_2x_3^2 + w_{13}x_3^3\, ,\\ 
G_1 = x_2F_3\, ,\  G_2 = x_3F_3\, .    
\end{gather*}  
The surface $Q \subset \piii$ of equation $F_3 = 0$ is a smooth quadric.
Relation \eqref{E:wijgamma} becomes$\, :$ 
\begin{gather*}
-x_1(w_{00}x_2^3 + w_{01}x_2^2x_3 + w_{02}x_2x_3^2 + w_{03}x_3^3) + 
x_0(w_{10}x_2^3 + w_{11}x_2^2x_3 + w_{12}x_2x_3^2 + w_{13}x_3^3)\\  
= (\gamma_0x_2^2 + \gamma_1x_2x_3 + \gamma_2x_3^2)(-x_1x_2 + x_0x_3) 
\end{gather*}
Since $w_{ij}$ and $\gamma_i$ are constants, one gets that$\, :$ 
\[
F_{40} = x_0F_3 + x_3(\gamma_0x_2^2 + \gamma_1x_2x_3 + \gamma_2x_3^2)\, ,\  
F_{41} = x_1F_3 + x_2(\gamma_0x_2^2 + \gamma_1x_2x_3 + \gamma_2x_3^2)\, .  
\]
The divisor $\{\gamma_0x_2^2 + \gamma_1x_2x_3 + \gamma_2x_3^2 = 0\} \cap Q$ on 
$Q$ contains $2L$ hence it is of the form $2L + L_1^\prim + L_2^\prim$, where 
$L_1^\prim,\, L_2^\prim$ are lines from the other ruling of $Q$. Both of these 
lines are 4-secants to $Z$. 
\end{remark}

\begin{prop}\label{P:z} 
Let $Z$ be a locally CM curve in $\piii$ of degree $4$. 
If ${\fam0 H}^0(\sci_Z(2)) = 0$ and $\sci_Z(3)$ is 
globally generated then $\sci_Z(3)$ admits a monad of the form$\, :$ 
\[
0 \lra \sco_\p(1) \oplus 2\sco_\p \lra 2\sco_\p(2) \oplus 3\sco_\p(1)   
\lra \sco_\p(3) \lra 0\, .   
\]  
\end{prop}

\begin{proof}
If $Z$ is not connected then the conclusion of the proposition follows, even 
without assuming that $\tH^0(\sci_Z(2)) = 0$, from Lemma~\ref{L:ycuplprim}, 
taking into account Lemma~\ref{L:xcupxprim} (see the first method for getting 
monads stated in the Introduction).  

Assume, from now on, that $Z$ is, moreover, connected.  
By Lemma~\ref{L:d4zquasiprimline}, $Z$ cannot be a quasiprimitive structure 
of degree 4 on a line. The hypothesis $\tH^0(\sci_Z(2)) = 0$ 
eliminates, also, many other cases as, for example, 
the case where $Z$ is reduced (by Lemma~\ref{L:h0iz(2)neq0}) or the case 
where $Z$ is the union of a double line and of two other (simple) lines, 
both of them intersecting the double line. It remains, actually, to analyse 
the following four cases$\, :$ 

\vskip2mm 

\noindent 
{\bf Case 1.}\quad $Z$ \emph{is a thick structure of degree $4$ on a line} 
$L$. 

\vskip2mm 

\noindent 
We may assume that $L$ is the line of equations $x_2 = x_3 = 0$.  
We use, in this case, the results and notation from Subsection~\ref{SS:thick} 
of Appendix~\ref{A:multilines} (with $W = Z$). Since 
\[
\sci_Z/\sci_L^3 \simeq \sco_L(-m-2) \oplus \sco_L(-n-2) 
\] 
it follows that if $\sci_Z(3)$ is globally generated then $m \leq 1$ and 
$n \leq 1$. On the other hand, since the generators $F$ and $G$ of the 
homogeneous ideal $I(Z) \subset S$ from Prop.~\ref{P:geniw} have degrees 
$m+2$ and $n+2$, respectively, and since $\tH^0(\sci_Z(2)) = 0$ it follows 
that, actually, $m = n = 1$. Since $m + n = l + 2$, one gets that $l = 0$.  
Prop.~\ref{P:reshow} implies, now, that $\sci_Z(3)$ satisfies the conclusion 
of the proposition.  

\vskip2mm 

\noindent 
{\bf Case 2.}\quad $Z$ \emph{is the union of a triple structure $Y$ on a line 
and of another line intersecting it}. 

\vskip2mm 

\noindent 
We may assume that $Y$ is a triple structure on the line $L_1$ of equations 
$x_2 = x_3 = 0$ and that the second line is the line $L_2$ of equations 
$x_1 = x_3 = 0$. We use, in this case, the results and notation from 
Subsection~\ref{SS:triplecupaline} of Appendix~\ref{A:multicomponent}. 
The hypothesis $\tH^0(\sci_Z(2)) = 0$ implies that $Y$ cannot be the first 
infinitesimal neighbourhood of $L_1$ in $\piii$ hence it is a 
quasiprimitive triple structure on $L_1$. 

Now, Lemma~\ref{L:h0iz(3)geq5} and Prop.~\ref{P:geniycupl2} imply that 
$l \leq 0$. Since $x_3F_2 \in I(Z)$ and $\tH^0(\sci_Z(2)) = 0$ it follows that 
$l \geq 0$. Consequently, $l = 0$. 

In the cases (a) and (c) of Prop.~\ref{P:geniycupl2} one has, on one hand, 
that $l + m + 2 = \text{deg}\, F_3 \geq 3$ (because $\tH^0(\sci_Z(2)) = 0$) 
and, on the other hand, since $\sco_{L_1}(-l-m-2)$ is a quotient of $\sci_Z$, 
that $-l-m-2+3 \geq 0$. It follows that, in this cases, $l = 0$ and $m = 1$. 

In the cases (b) and (d) of Prop.~\ref{P:geniycupl2} one has, on one hand, 
that $l + m + 3 = \text{deg}\, (x_1F_3) \geq 3$ (because 
$\tH^0(\sci_Z(2)) = 0$) and, on the other hand, since $\sco_{L_1}(-l-m-3)$ is a 
quotient of $\sci_Z$, that $-l-m-3+3 \geq 0$. It follows that, in this cases, 
$l = 0$ and $m = 0$. We analyse, now, each one of the cases occuring in 
Prop.~\ref{P:geniycupl2}. 

\vskip2mm 

\noindent
$\bullet$\quad 
In case (a), one must have $b = x_1$ hence $b_1 = 1$. It follows that$\, :$ 
\[
x_2F_2 = -x_1x_2^2 + ax_2x_3\  \text{and} \  x_1x_2^3 = -x_2\cdot x_2F_2 + 
a\cdot x_2^2x_3\, , 
\]          
hence $I(Z)$ is generated by polynomials of degree 3. Using 
Prop.~\ref{P:reshoycupl2}(a) one sees that $\sci_Z(3)$ satisfies the conlusion 
of the proposition (since $b_1 = 1$, one can cancel a direct summand $S(-1)$ 
from each of the first two terms of the resolution of 
$\tH^0_\ast(\sco_{Y \cup L_2})$).  

\vskip2mm 

\noindent 
$\bullet$\quad 
In case (b), essentially the same argument shows that $I(Z)$ is generated by 
polynomials of degree 3 and that $\sci_Z(3)$ satisfies the conclusion of the 
proposition. 

\vskip2mm 

\noindent 
$\bullet$\quad 
In case (c), one must have $p = x_1$. By Lemma~\ref{L:bbprimv0}(c), one can 
assume that $v_0 =c_0x_1$, for some $c_0 \in k$, hence$\, :$ 
\[
F_3 = x_1(-bx_2+ax_3) + c_0x_1x_2^2 + v_1x_2x_3 + v_2x_3^2\, .
\]   
If $\sci_Z(3)$ is globally generated then$\, :$ 
\[
x_1x_2^3 \in (F_3,\, x_3F_2,\, x_2^2x_3,\, x_2x_3^2,\, x_3^3)^{\text{sat}}\, . 
\]
Let $S^\prim = S/Sx_3 = k[x_0,x_1,x_2]$. It follows that, working in 
$S^\prim$$\, :$ 
\[
x_1x_2^3 \in (S^\prim x_1x_2(-b+c_0x_2))^{\text{sat}}\, . 
\]
But $S^\prim x_1x_2(-b+c_0x_2)$ is already saturated in $S^\prim$ and 
$x_1x_2^3 \notin S^\prim x_1x_2(-b+c_0x_2)$. It follows that $\sci_Z(3)$ is not 
globally generated in case (c). 

\vskip2mm 

\noindent 
$\bullet$\quad 
In case (d), $p = 1$, $F_3 = F_2 + v_0x_2^2 +v_1x_2x_3 + v_2x_3^2$, and the same 
kind of argument as that used in case (c) shows that $\sci_Z(3)$ cannot be 
globally generated (because $x_1x_2^3 \notin S^\prim x_1x_2(-b+v_0x_2)$).  

\vskip2mm

\noindent 
{\bf Case 3.}\quad $Z$ \emph{is the union of a double structure $X$ on a line, 
of another line intersecting it, and of a third line intersecting the second  
line but not the double line}. 

\vskip2mm

\noindent 
We may assume that $X$ is a double structure on the line $L_1$ of equations 
$x_2 = x_3 = 0$, that the second line is the line $L_2$ of equations 
$x_1 = x_3 = 0$, and that the third line is the line $L_1^\prim$ of equations 
$x_0 = x_1 = 0$. We use, in this case, the results and notation from 
Subsection~\ref{SS:doublecuptwolines}. Under the hypothesis of 
Prop.~\ref{P:genixcupl2cupl1prim}(b) one must have $l = 1$ (because $F_2 \in 
I(Z)$ and $\text{deg}\, F_2 = l + 2$ and because $\sco_{L_1}(-l-2)$ is a 
quotient of $\sci_Z$). Similarly, under the hypothesis of 
Prop.~\ref{P:genixcupl2cupl1prim}(c), one must have $l = 0$. 
Prop.~\ref{P:reshoxcupl2cupl1prim} shows, now, that $\sci_Z(3)$ satisfies the 
conclusion of the proposition. 

\vskip2mm 

\noindent 
{\bf Case 4.}\quad $Z$ \emph{is the union of a double structure $X$ on a line 
and of a (nonsingular) conic $C$ intersecting it}. 

\vskip2mm

\noindent 
Since $\tH^0(\sci_Z(2)) = 0$, the plane containing the conic must not contain 
the support of the double line. One may assume, in this case, that $X$ is a 
double structure on the line $L_1$ of equations $x_2 = x_3 = 0$ and that 
$C$ is the conic of equations $x_1 = x_0x_3 - x_2^2 = 0$. One can use, now, 
an argument similar to that used in Case 3, based on  
Prop.~\ref{P:genixcupc2} and Prop.~\ref{P:reshoxcupc2} from 
Subsection~\ref{SS:doublecupconic}.     
\end{proof} 

\begin{prop}\label{P:y}
Let $Y$ be a locally CM curve in $\piii$ of degree $3$. If $\sci_Y(3)$ is 
globally generated then one of the following holds$\, :$ 
\begin{enumerate}
\item[(i)] $Y$ is a complete intersection of type $(1,3)$$\, ;$ 
\item[(ii)] $Y$ is directly linked to a line by a complete intersection of 
type $(2,2)$$\, ;$ 
\item[(iii)] $\sci_Y(3)$ admits a monad of the form$\, :$ 
\[
0 \lra \sco_\p(1)\oplus \sco_\p \lra 3\sco_\p(2)\oplus \sco_\p(1) \lra 
\sco_\p(3) \lra 0\, ;
\]
\item[(iv)] $\sci_Y(3)$ admits a monad of the form$\, :$ 
\[
0 \lra \sco_\p(i+1)\oplus 2\sco_\p(1) \lra 2\sco_\p(i+2)\oplus 4\sco_\p(2) 
\lra \sco_\p(i+3) \oplus \sco_\p(3) \lra 0\, , 
\]
with $i = 0$ or $1$$\, ;$
\item[(v)] $\sci_Y(3)$ admits a monad of the form$\, :$ 
\[
0 \lra \sco_\p(3)\oplus\sco_\p(2)\oplus \sco_\p(1) \lra 
2\sco_\p(4)\oplus 2\sco_\p(3)\oplus 2\sco_\p(2) \lra 
\sco_\p(5)\oplus \sco_\p(4) \lra 0\, . 
\]
\end{enumerate} 
\end{prop} 

\begin{proof}
We split the proof into several cases. 

\vskip2mm 

\noindent 
{\bf Case 1.}\quad $Y$ \emph{is not connected}. 

\vskip2mm 

\noindent 
In this case $Y$ is the union of a curve $X$ of degree 2 and of a line 
$L^\prim$ not intersecting $X$. If $X$ is a complete intersection of type 
$(1,2)$ then $\sci_Y(3)$ satisfies the condition (iii) from the statement. 
If $X$ is the union of two disjoint lines then $\sci_Y(3)$ satisfies 
condition (iv) from the statement with $i = 0$. 

It remains to consider the situation where $X$ is a double structure on a 
line $L$ defined by an epimorphism $\sci_L/\sci_L^2 \simeq 2\sco_L(-1) \ra 
\sco_L(l)$ with $l \geq 0$ (for $l = -1$, $X$ is a complete intersection of 
type $(1,2)$). We may assume that $L$ has equations $x_2 = x_3 = 0$ and that 
$L^\prim$ has equations $x_0 = x_1 = 0$. In this case, using 
Prop.~\ref{P:genixcupl1prim}, one sees easily that one must have $l \leq 1$. 
It follows, from the results stated in Subsection~\ref{SS:p2}, that 
$\sci_Y(3)$ satisfies condition (iv) from the statement with $i = l$.  

\vskip2mm

\noindent 
{\bf Case 2.}\quad $Y$ \emph{is reduced and connected}. 

\vskip2mm 

\noindent 
In this case, one applies Lemma~\ref{L:yconnected}. 

\vskip2mm 

\noindent 
{\bf Case 3.}\quad $Y$ \emph{is the union of a double structure $X$ on a line 
and of another line intersecting it}. 

\vskip2mm

\noindent 
We may assume that $X$ is a double structure on the line $L_1$ of equations 
$x_2 = x_3 = 0$ and that the other line is the line $L_2$ of equations 
$x_1 = x_3 = 0$. We use, in this case, the results and notation from 
Subsection~\ref{SS:doublecupaline}. 

Under the hypothesis of Prop.~\ref{P:genixcupl}(a) (with $c = 0$), $\sci_Y(3)$ 
is globally generated if and only if $l \in \{-1,\, 0,\, 1\}$. If $l = -1$, 
then $\text{deg}\, F_2 = 1$ hence $Y$ is a complete intersection of type 
$(1,3)$. If $l = 0$ then $b = x_1$ and $I(Y) = (-x_1x_2 + ax_3,\, x_2x_2,\, 
x_3^2)$ hence $Y$ is directly linked to $L_2$ by the complete intersection 
defined by $-x_1x_2 + ax_3$ and $x_3^2$.   
If $l = 1$ then, by Prop.~\ref{P:reshoxcupl}(a), 
$\sci_Y(3)$ satisfies condition (iii) from the statement. 

On the other hand, under the hypothesis of Prop.~\ref{P:genixcupl}(b) 
(with $c = 0$), $\sci_Y(3)$ is globally generated if and only if 
$l \in \{-1,\, 0\}$. If $l = -1$ then $b = 1$ and $I(Y) = (-x_1x_2 + ax_1x_3,
\, x_2x_2,\, x_3^2)$ and $Y$ is directly linked to $L_2$ by the complete 
intersection defined by $-x_1x_2 + ax_1x_3$ and $x_3^2$.  
If $l = 0$ then, by Prop.~\ref{P:reshoxcupl}(b), 
$\sci_Y(3)$ satisfies condition (iii) from the statement. 

\vskip2mm 

\noindent 
{\bf Case 4.}\quad $Y$ \emph{is a triple structure on a line} $L$. 

\vskip2mm

\noindent 
We may assume that $L$ has equations $x_2 = x_3 = 0$. If $Y$ is 
the first infinitesimal neighbourhood of $L$ in $\piii$ then $Y$ is directly 
linked to $L$ by a complete intersection of type $(2,2)$. Consequently, we 
may assume that $Y$ is a \emph{quasiprimitive} triple structure on $L$. We 
use, in this case, the results and notation from Subsection~\ref{SS:q3}. 
Since $\sco_L(-l-m-2)$ is a quotient of $\sci_Y$ (see Remark~\ref{R:rhoeta2})  
it follows that if $\sci_Y(3)$ is globally generated then $-l-m-2+3 \geq 0$, 
i.e., $l+m \leq 1$. Conversely, if $l+m \leq 1$ then $I(Y)$ is generated by 
polynomials of degree 3 (for $l \leq 0$ this is clear, while for $l = 1$ and 
$m = 0$ $Y$ is a \emph{primitive} triple structure on $L$ hence $I(Y)$ is 
generated by $F_3$, $x_2^3$, $x_2^2x_3$, $x_2x_3^2$, $x_3^3$). 

If $l = 1$ and $m = 0$ (i.e., if $Y$ is the divisor $3L$ on a cubic surface 
$\Sigma \supset L$ nonsingular along $L$) then $\sci_Y(3)$ satisfies condition 
(v) from the statement. 

If $l = 0$ and $m = 1$ then $\sci_Y(3)$ satisfies condition (iv) from the 
statement with $i = 1$. 

If $l = 0$ and $m = 0$ then $\sci_Y(3)$ satisfies condition (iv) from the 
statement with $i = 0$. 

If $l = -1$ and $m = 2$ then $a$, $b$ are constants, at least one non-zero, 
hence $\sci_Y(3)$ satisfies condition (iii) from the statement (one can cancel 
some direct summands in the resolution of $\tH^0_\ast(\sco_Y)$).  

If $l = -1$ and $m = 1$ then, by Lemma~\ref{L:l=-1m=1deg3}, $Y$ is directly 
linked to $L$ by a complete intersection of type $(2,2)$. 

Finally, if $l = -1$ and $m = 0$ then $Y$ is a complete intersection of type 
$(1,3)$.        
\end{proof}

\begin{prop}\label{P:x} 
Let $X$ be a locally CM curve in $\piii$ of degree $2$. If $\sci_X(3)$ is 
globally generated then either $X$ is a complete intersection of type $(1,2)$ 
or $\sci_X(3)$ admits a monad of the form$\, :$ 
\[
0 \lra \sco_\p(i+1)\oplus \sco_\p(1) \lra 2\sco_\p(i+2)\oplus 2\sco_\p(2) 
\lra \sco_\p(i+3) \lra 0
\]
with $i = 0$ or $1$.  
\end{prop}

\begin{proof}
If $X$ is not a complete intersection of type $(1,2)$ then either $X$ is the 
union of two disjoint lines (in which case $\sci_X(3)$ admits, 
by Lemma~\ref{L:zcupw}, 
a monad as in the statement with $i = 0$) or it is a double 
structure on a line $L$. We may assume that $L$ has equations $x_2 = x_3 = 0$. 
We use, in this case, the results and notation from Subsection~\ref{SS:p2}. 
$\sci_X(3)$ globally generated and $X$ not a complete intersection of type 
$(1,2)$ turn out to be equivalent to $l \in \{0,\, 1\}$. 
It follows that $\sci_X(3)$ admits a monad of the form from the statement 
with $i = l$.   
\end{proof}

\section{Globally generated vector bundles}\label{S:ggvb}

In this section we prove Theorem~\ref{T:main} using the results from the 
Propositions \ref{P:zinci23}, \ref{P:zprimlci}, \ref{P:zprimcm}, 
\ref{P:z}, \ref{P:y} and \ref{P:x}. So, let $E$ be a globally generated 
vector bundle on $\piii$ with $c_1 \geq 4$ such that $\tH^i(E^\vee) = 0$, 
$i = 0,\, 1$. Assume that $\tH^0(E(-c_1+2)) = 0$ and $\tH^0(E(-c_1+3)) \neq 
0$. Recall, from the Introduction, the exact sequence 
\eqref{E:o(c1-3)oeiz(3)} and the monad \eqref{E:monade} of $E$ deduced from 
the monad \eqref{E:monadiz(3)} of $\sci_Z(3)$. Let us assume that, 
moreover$\, :$ 
\[
B^{-1} = B^{-1}_+ \oplus m_{-1}\sco_\p \oplus B^{-1}_- \  \text{and} \  
B^0 = B^0_+ \oplus m_0\sco_\p \, ,
\] 
with $\tH^0(B_+^{-1\vee}) = 0$, $\tH^0(B_-^{-1}) = 0$, $\tH^0(B_+^{0\vee}) = 0$, 
and that $\tH^0(B^{1\vee}) = 0$. Let $d_+^{-1} : B_+^{-1} \ra B_+^0$ be the 
restriction of $d^{-1} : B^{-1} \ra B^0$ and $d_+^0 : B_+^0 \ra B^1$ the 
restriction of $d^0$. 
The next lemma will allow us to handle all the cases occuring in the proof 
of Theorem~\ref{T:main}. 

\begin{lemma}\label{L:monad} 
Under the above hypotheses and notation, $r = 2 - m_0 + m_{-1} + 
{\fam0 h}^0(B_-^{-1\vee})$ and one has an exact sequence$\, :$ 
\[
0 \lra B_+^{-1} \xra{\begin{pmatrix} \phi_+\\ d_+^{-1} \end{pmatrix}} 
\sco_\p(c_1-3)\oplus F \lra E \lra 0
\]
where $F$ is the cohomology of a monad$\, :$ 
\[
0 \lra B_-^{-1} \xra{\begin{pmatrix} 0\\ u \end{pmatrix}} 
B_+^0 \oplus m\sco_\p \xra{\displaystyle (d_+^0\, ,\, \ast)} B^1 \lra 0 
\]
with $u : B_-^{-1} \ra m\sco_\p$ the dual of the evaluation morphism  
${\fam0 H}^0(B_-^{-1\vee})\otimes_k\sco_\p \ra B_-^{-1\vee}$. 
\end{lemma}  

\begin{proof}
Put $B_{\leq 0}^{-1} = m_{-1}\sco_\p\oplus B_-^{-1}$. The monad \eqref{E:monade} 
of $E$ from the Introduction can be written in the following form$\, :$ 
\[
0 \lra B_+^{-1}\oplus B_{\leq 0}^{-1} 
\xra{\begin{pmatrix} \phi_+ & \alpha\\ d_+^{-1} & \beta\\ 0 & \gamma 
\end{pmatrix}} \sco_\p(c_1-3) \oplus B_+^0 \oplus m^\prim \sco_\p 
\xra{\displaystyle (0\, ,\, d_+^0\, ,\, \rho)} B^1 \lra 0
\]
where $m^\prim = m_0 + r - 2$. One deduces an exact sequence$\, :$ 
\[
0 \lra B_+^{-1} \lra G \lra E \lra 0 
\] 
where $G$ is the cohomology of the monad$\, :$ 
\begin{equation}\label{E:monadg} 
0 \lra B_{\leq 0}^{-1} \xra{\begin{pmatrix} \alpha\\ \beta\\ \gamma 
\end{pmatrix}} 
\sco_\p(c_1-3) \oplus B_+^0 \oplus m^\prim \sco_\p  
\xra{\displaystyle (0\, ,\, d_+^0\, ,\, \rho)} B^1 \lra 0\, . 
\end{equation}
Since $\tH^i(E^\vee) = 0$, $i = 0,\, 1$, it follows that $\tH^i(G^\vee) = 0$, 
$i = 0,\, 1$, and this is equivalent to the fact that$\, :$ 
\[
\tH^0(\gamma^\vee) : \tH^0(m^\prim \sco_\p) \lra \tH^0(B_{\leq 0}^{-1\vee}) 
\] 
is an isomorphism. Up to an automorphism of $m^\prim \sco_\p$, one can assume 
that $\gamma$ is the dual of the evaluation morphism$\, :$ 
\[
\tH^0(B_{\leq 0}^{-1\vee})\otimes_k\sco_\p \lra B_{\leq 0}^{-1\vee}\, .
\]
It follows that there exist $\alpha^\prim : m^\prim \sco_\p \ra \sco_\p(c_1 - 3)$ 
and $\beta^\prim : m^\prim \sco_\p \ra B_+^0$ such that $\alpha = \alpha^\prim 
\circ \gamma$ and $\beta = \beta^\prim \circ \gamma$.  
One deduces that the monad \eqref{E:monadg} is isomorphic to$\, :$ 
\[
0 \lra B_{\leq 0}^{-1} 
\xra{\begin{pmatrix} 0\\ 0\\ \gamma \end{pmatrix}} 
\sco_\p(c_1-3) \oplus B_+^0 \oplus m^\prim \sco_\p 
\xra{\displaystyle (0\, ,\, d_+^0\, ,\, \rho^\prim)} B^1 \lra 0 
\]
with $\rho^\prim = \rho + d_+^0\circ \beta^\prim$. It suffices, now, to cancel 
the direct summand $m_{-1}\sco_\p$ of $B_{\leq 0}^{-1}$ and the corresponding 
direct summand of $m^\prim \sco_\p$. One gets, in particular, that the integer 
$m$ from the statement is equal to $m^\prim - m_{-1} = m_0 - m_{-1} + r - 2$. 
On the other hand, $m = \h^0(B_-^{-1\vee})$ whence the formula for $r$. 
\end{proof}

\begin{lemma}\label{L:koszul} 
Let $F$ be the kernel of an epimorphism $\e : \bigoplus_{i=0}^3\sco_\piii(-d_i) 
\ra \sco_\piii$, with $1 \leq d_0 \leq \cdots \leq d_3$. Then $F(t)$ is globally 
generated if and only if $t \geq d_2 + d_3$.   
\end{lemma}

\begin{proof}
Using the Koszul complex associated to $\e$ one sees easily that 
$F(d_2 + d_3)$ is globally generated. 

On the other hand, let $f_i \in \tH^0(\sco_\piii(d_i))$, $i = 0, \ldots ,3$, be 
the polynomials defining $\e$. Restricting $F$ to the complete intersection 
$C$ defined by $f_0$ and $f_1$, one gets$\, :$ 
\[
F \vb C \simeq \sco_C(-d_0) \oplus \sco_C(-d_1) \oplus \sco_C(-d_2-d_3) 
\]
from which one deduces that $F(d_2 + d_3 - 1)$ is not globally generated. 
\end{proof}

\begin{proof}[Proof of Theorem~\ref{T:main}] 
We split the proof into several cases according to the form of the monad of 
$\sci_Z(3)$, where $Z$ is the curve occuring in the exact sequence 
\eqref{E:o(c1-3)oeiz(3)} from the Introduction. 

\vskip2mm 

\noindent 
{\bf Case 0.}\quad $Z = \emptyset$.  

\vskip2mm 

\noindent 
In this case $E\simeq \sco_\p(c_1-3)\oplus \sco_\p(3)$. 

\vskip2mm

\noindent 
{\bf Case 1.}\quad $Z$ \emph{is a line}.  

\vskip2mm

\noindent 
In this case, by Lemma~\ref{L:monad}, one has an exact sequence$\, :$ 
\[
0 \lra \sco_\p(1) \lra \sco_\p(c_1-3)\oplus 2\sco_\p(2) \lra E \lra 0\, .
\] 
Since we are on $\piii$, this is possible only if $c_1 = 4$ and 
$E \simeq 2\sco_\p(2)$. But, then, $\tH^0(E(-c_1+2)) \neq 0$, which contradicts 
our hypothesis, hence this case \emph{cannot occur}. 

\vskip2mm

\noindent 
{\bf Case 2.}\quad $Z$ \emph{is a complete intersection of type $(a,b)$, with 
$1\leq a \leq b \leq 3$ and $a+b \geq 3$}. 

\vskip2mm

\noindent 
In this case, by Lemma~\ref{L:monad}, $\sco_\p(c_1-3)$ is a direct summand 
of $E$. 

\vskip2mm
 
\noindent 
{\bf Case 3.}\quad $Z$ \emph{is directly linked to a line by a complete 
intersection of type} $(2,2)$.  

\vskip2mm

\noindent 
In this case, $\sci_Z(3)$ admits a resolution of the form$\, :$ 
\[
0 \lra 2\sco_\p \lra 3\sco_\p(1) \lra \sci_Z(3) \lra 0
\]    
hence, by Lemma~\ref{L:monad}, $\sco_\p(c_1-3)$ is a direct summand of $E$. 

\vskip2mm

\noindent 
{\bf Case 4.}\quad $Z$ \emph{is as in Prop.}~\ref{P:zinci23}(ii)-(iv), 
\emph{or as in Prop.}~\ref{P:zprimlci}(i)-(iii), 
\emph{or as in Prop.}~\ref{P:zprimcm}.  

\vskip2mm

\noindent 
In this case, by Lemma~\ref{L:monad}, $\sco_\p(c_1-3)$ is a direct summand 
of $E$. \emph{Indeed}, it suffices to notice that if $Z$ is as in 
Prop.~\ref{P:zinci23}(iv) then $\sci_Z(3)$ admits a monad of the form$\, :$ 
\[
0 \lra 3\sco_\p \lra 5\sco_\p(1) \lra \sco_\p(2) \lra 0
\]
and if $Z$ is as in Prop.~\ref{P:zprimlci}(iii) then $\sci_Z(3)$ admits a 
monad of the form$\, :$ 
\[
0 \lra 2\sco_\p\oplus \sco_\p(-1) \lra 4\sco_\p(1)\oplus \sco_\p \lra 
\sco_\p(2) \lra 0\, .
\] 

\noindent 
{\bf Case 5.}\quad $Z$ \emph{is as $Y$ in Prop.}~\ref{P:y}(iv)-(v) 
\emph{or as $X$ in the second part of the conclusion of Prop.}~\ref{P:x}. 

\vskip2mm 

\noindent 
In all these cases it follows, from Lemma~\ref{L:monad}, that $E$ has rank 2 
hence it can be realized as an extension$\, :$ 
\[
0 \lra \sco_\p(c_1-3) \lra E \lra \sci_Z(3) \lra 0\, .
\]  
One deduces that $Z$ is l.c.i. and that $\omega_Z \simeq \sco_Z(2-c_1)$. Since 
$\chi(\omega_Z) = -\chi(\sco_Z)$ one deduces, from Riemann-Roch on $Z$, that 
\[
(c_1-2)\text{deg}\, Z = 2\chi(\sco_Z)\, . 
\]
Notice, also, that $c_2 = \text{deg}\, Z + 3(c_1-3)$ and that 
$\chi(\sco_Z)$ depends only on the numerical shape of the 
monad of $\sci_Z(3)$.  

\vskip2mm

\noindent 
{\bf Subcase 5.1.}\quad $Z$ \emph{as $Y$ in Prop.}~\ref{P:y}(iv) \emph{with} 
$i = 0$. 

\vskip2mm

\noindent 
In this subcase, $\text{deg}\, Z = 3$ and, in order to compute 
$\chi(\sco_Z)$, one may assume that $Z$ is a triple structure on a line $L$ 
with $l = 0$ and $m = 0$ (notation as in Subsection~\ref{SS:q3}), hence 
with $\sco_Z \simeq 3\sco_L$ as an $\sco_L$-module. One gets $\chi(\sco_Z) = 
3$, hence $c_1 = 4$ and $c_2 = 6$. It follows that $c_1(E(-2)) = 0$,  
$c_2(E(-2)) = c_2 -2c_1 +4 = 2$ and, since $\tH^0(E(-2)) = 0$, $E$ is stable. 

\vskip2mm 

\noindent 
{\bf Subcase 5.2.}\quad $Z$ \emph{as $Y$ in Prop.}~\ref{P:y}(iv) \emph{with} 
$i = 1$. 

\vskip2mm 

\noindent 
In this subcase, $\text{deg}\, Z = 3$ and, in order to compute 
$\chi(\sco_Z)$, one may assume that $Z$ is a triple structure on a line $L$ 
with $l = 0$ and $m = 1$, hence with $\sco_Z \simeq 2\sco_L\oplus \sco_L(1)$  
as an $\sco_L$-module. It follows that $\chi(\sco_Z) = 4$ hence $3\cdot 
(c_1-2) = 2\cdot 4$, hence this subcase \emph{cannot occur}. 

\vskip2mm 

\noindent 
{\bf Subcase 5.3.}\quad $Z$ \emph{as $Y$ in Prop.}~\ref{P:y}(v). 

\vskip2mm

\noindent 
In this subcase, $\text{deg}\, Z = 3$ and, in order to compute 
$\chi(\sco_Z)$, one may assume that $Z$ is a triple structure on a line $L$ 
with $l = 1$ and $m = 0$, hence with $\sco_Z \simeq \sco_L\oplus \sco_L(1) 
\oplus \sco_L(2)$  as an $\sco_L$-module. It follows that $\chi(\sco_Z) = 6$ 
hence $c_1 = 6$ and $c_2 = 12$ hence $c_1(E(-3)) = 0$ and $c_2(E(-3)) = 3$. 
The hypothesis $\tH^0(E(-4)) = 0$ and $\tH^0(E(-3)) \neq 0$ shows that 
$E(-3)$ is properly semistable. 

We notice, at this point, that if $Z$ is a triple structure on a line $L$ 
with $l = 1$ and $m = 0$ then $\omega_Z \simeq \sco_Z(-4)$ (by 
\cite[Prop.~2.3~and~\S~3.2]{bf}) hence a global section of $\omega_Z(4)$ 
generating everywhere this sheaf defines an extension$\, :$ 
\[
0 \lra \sco_\p \lra F \lra \sci_Z \lra 0
\] 
with $F$ a properly semistable rank 2 vector bundle with $c_1(F) = 0$ and 
$c_2(F) = 3$. Notice also that, conversely, if $F$ is a properly semistable 
rank 2 vector bundle with these Chern classes then the zero locus $Z$ of the 
unique nonzero global section of $F$ is a triple structure on a line $L$ with 
$l = 1$ and $m = 0$. \emph{Indeed}, $\text{deg}\, Z = 3$ and $\omega_Z \simeq 
\sco_Z(-4)$. $Z$ cannot be reducible because, in that case, one would have 
$Z = X \cup L$, with $L$ a line and $X$ a curve of degree 2 such that the 
scheme $D = X\cap L$ is 0-dimensional (or empty) and in this case it is 
well-known that $\omega_Z \vb L \simeq \omega_L \otimes \sco_L(D)$ and this 
would contradict the fact that $\omega_Z \simeq \sco_Z(-4)$. It remains that 
$Z$ is irreducible. $Z$ cannot be reduced hence $Z$ is a triple structure on 
a line $L$. Since $Z$ is l.c.i. it follows that $m = 0$ (by 
\cite[\S~3.3]{bf}) and the condition $\omega_Z \simeq \sco_Z(-4)$ implies, 
now, that $l = 1$ (by \cite[Prop.~2.3]{bf}). 

\vskip2mm

\noindent 
{\bf Subcase 5.4.}\quad $Z$ \emph{as $X$ in Prop.}~\ref{P:x} \emph{with} 
$i = 0$.   

\vskip2mm

\noindent 
In this subcase $\text{deg}\, Z = 2$ and $Z$ is either the union of two 
disjoint lines or a double structure on a line $L$ with $l = 0$ (notation as 
in Subsection~\ref{SS:p2}). It follows that $\chi(\sco_Z) = 2$ hence $c_1 = 4$ 
and $c_2 = 5$ hence $c_1(E(-2)) = 0$ and $c_2(E(-2)) = 1$. Since, by 
hypothesis, $\tH^0(E(-2)) = 0$, $E(-2)$ is stable. It is, actually, a 
nullcorrelation bundle. 

\vskip2mm

\noindent 
{\bf Subcase 5.5.}\quad $Z$ \emph{as $X$ in Prop.}~\ref{P:x} \emph{with} 
$i = 1$. 

\vskip2mm 

\noindent 
In this subcase, $\text{deg}\, Z = 2$ and $Z$ is a double structure on a line 
$L$ with $l = 1$. It follows that $\chi(\sco_Z) = 3$ hence $c_1 = 5$ and 
$c_2 = 8$. One gets that $c_1(E(-3)) = -1$ and $c_2(E(-3)) = 2$. The 
hypothesis $\tH^0(E(-3)) = 0$ implies that $E(-3)$ is stable. These 
bundles were studied by Hartshorne and Sols \cite{hs} and, independently, 
by Manolache \cite{m1}. 

\vskip2mm 

\noindent 
{\bf Case 6.}\quad $Z$ \emph{is as $Y$ in Prop.}~\ref{P:y}(iii). 

\vskip2mm 

\noindent 
In this case, by Lemma~\ref{L:monad}, one has an exact sequence$\, :$ 
\[
0 \lra \sco_\p(1) \xra{\begin{pmatrix} \phi_+\\ d_+^{-1} \end{pmatrix}} 
\sco_\p(c_1-3) \oplus F \lra E \lra 0
\]        
where $F$ is defined by an exact sequence$\, :$ 
\[
0 \lra F \lra 3\sco_\p(2) \oplus \sco_\p(1) 
\overset{\displaystyle d_+^0}{\lra} \sco_\p(3) \lra 0\, .
\]
Since any global section of $F(-1)$ vanishes along a line in $\piii$, it 
follows that $c_1 = 4$ and that $\phi_+$ is an isomorphism hence that 
$E\simeq F$. Now, up to a linear change of coordinates, one can assume that 
the first three components of $d_+^0$ are $x_0,\, x_1,\, x_2$ and then, 
modulo an automorphism of $3\sco_\p(2)\oplus \sco_\p(1)$ invariating 
$3\sco_\p(2)$, one can assume that the fourth component of $d_+^0$ is $x_3^2$. 

\vskip2mm

\noindent 
{\bf Case 7.}\quad $Z$ \emph{is as in Prop.}~\ref{P:z}. 

\vskip2mm 

\noindent 
In this case, an argument similar to that used in Case 6, shows that $c_1 = 4$ 
and that one has an exact sequence$\, :$ 
\[
0 \lra E \lra 2\sco_\p(2) \oplus 3\sco_\p(1) 
\overset{\displaystyle d_+^0}{\lra} \sco_\p(3) \lra 0\, .
\]
Since, by Lemma~\ref{L:koszul}, the kernel of an epimorphism 
$\sco_\p(2) \oplus 3\sco_\p(1) \ra \sco_\p(3)$ is not globally generated, 
the first two components of $d_+^0$ must be linearly independent. Up to a 
linear change of coordinates, one can assume that they are $x_0$ and $x_1$. 
Let $L\subset \piii$ be the line of equations $x_0 = x_1 = 0$. Since 
$E \vb L$ is globally generated, it follows that $\tH^0((d_+^0 \vb L)(-1))$ 
induces an isomorphism $\tH^0(3\sco_L) \izo \tH^0(\sco_L(2))$ (because, 
otherwise, $E \vb L \simeq 2\sco_L(2) \oplus \sco_L(1) \oplus \sco_L(-1)$). 
Now, modulo an automorphism of $2\sco_\p(2) \oplus 3\sco_\p(1)$ invariating 
$2\sco_\p(2)$, one can assume that the last three components of $d_+^0$ are 
$x_2^2$, $x_2x_3$ and $x_3^2$. 

\vskip2mm

\noindent 
{\bf Case 8.}\quad $Z$ \emph{is as in Prop.}~\ref{P:zprimlci}(iv). 

\vskip2mm

\noindent 
By Lemma~\ref{L:monad}, 
one has an exact sequence $0 \ra \sco_\p(1) \ra \sco_\p(c_1 - 3) \oplus F 
\ra E \ra 0$ with $F$ the cohomology of a monad$\, :$ 
\[
0 \lra \sco_\p(-1) \xra{\begin{pmatrix} 0\\ u \end{pmatrix}} 
2\sco_\p(2) \oplus 2\sco_\p(1) \oplus 4\sco_\p 
\xra{\displaystyle (d_+^0\, ,\, v)} \sco_\p(3) \lra 0 
\]    
where $u : \sco_\p(-1) \ra 4\sco_\p$ is defined by $x_0, \ldots ,x_3$ and 
$d_+^0 : 2\sco_\p(2) \oplus 2\sco_\p(1) \ra \sco_\p(3)$ is an epimorphism (by 
the proof of Prop.~\ref{P:zprimlci}(iv)).  
Up to a linear change of coordinates and an automorphism of $2\sco_\p(2) \oplus 
2\sco_\p(1)$, one can assume that $d_+^0$ is defined by $x_0$, $x_1$, $x_2^2$, 
$x_3^2$.  Modulo an automorphism of $4\sco_\p$, one can continue to assume 
that $u$ is defined by $x_0, \ldots ,x_3$. Since $\tH^0(d_+^0)$ is obviously 
surjective, $v : 4\sco_\p \ra \sco_\p(3)$ factorizes through $d_+^0$. 
One deduces that, modulo an automorphism of $2\sco_\p(2) \oplus 2\sco_\p(1) 
\oplus 4\sco_\p$ invariating $2\sco_\p(2) \oplus 2\sco_\p(1)$ and whose 
component $4\sco_\p \ra 4\sco_\p$ is $\text{id}_{4\sco}$, one can assume 
that $v = 0$. $F$ occurs, now, as the cohomology of a monad$\, :$ 
\[
0 \lra \sco_\p(-1) \xra{\begin{pmatrix} s^\prim\\ u \end{pmatrix}} 
2\sco_\p(2) \oplus 2\sco_\p(1) \oplus 4\sco_\p 
\xra{\displaystyle (p\, ,\, 0)} \sco_\p(3) \lra 0
\] 
where $u$ is defined by $x_0, \ldots ,x_3$ and $p = d_+^0$ by $x_0$, $x_1$, 
$x_2^2$, $x_3^2$. Since the unique global section of $F(-1)$ vanishes along the 
line $x_0 = x_1 = 0$ it follows that $c_1 = 4$ and $E \simeq F$. 
Let $K = \Ker p$ and let $s : \sco_\p(-1) \ra K$ be the 
morphism induced by $(0,\, 0,\, -x_3^2,\, x_2^2) : \sco_\p(-1) \ra 2\sco_\p(2) 
\oplus 2\sco_\p(1)$. Using the Koszul complex associated to $x_0$, $x_1$, 
$x_2^2$, $x_3^2$, one sees that there exist a constant $c\in k$ and a morphism 
$\sigma : 4\sco_\p \ra K$ such that $s^\prim = cs + \sigma \circ u$. It follows 
that the above monad is isomorphic to the monad$\, :$
\[
0 \lra \sco_\p(-1) \xra{\begin{pmatrix} cs\\ u \end{pmatrix}} 
2\sco_\p(2) \oplus 2\sco_\p(1) \oplus 4\sco_\p 
\xra{\displaystyle (p\, ,\, 0)} \sco_\p(3) \lra 0\, . 
\] 
If $c = 0$ then one would get $E \simeq K \oplus \text{T}_\p(-1)$ which is 
not possible because, by Lemma~\ref{L:koszul}, $K$ is not globally generated. 
It remains that $c \neq 0$ hence one can assume, actually, that $c = 1$.  
\end{proof}

We extend, next, Theorem~\ref{T:main} to higher dimensional projective 
spaces. At a certain point of this extension we shall need the following 
result of Barth and Ellencwajg \cite[Thm.~4.2]{be}, for which we provide a 
different proof, based on the results of Mohan Kumar, Peterson and Rao 
\cite{kpr}. 

\begin{thm}[Barth,~Elencwajg]\label{T:be} 
There exists no stable rank $2$ vector bundle $E$ on $\piv$ with Chern 
classes $c_1 = 0$ and $c_2 = 3$. 
\end{thm}

\begin{proof}
Assume that such a vector bundle $E$ exists. According to Barth \cite{b}, 
the restriction $E_H$ of $E$ to a general hyperplane $H\subset \piv$ is 
stable. If $F$ is a stable rank 2 vector bundle on $\piii$ with $c_1(F) = 0$ 
and $c_2(F) = 3$ then the possible 
spectra (see \cite[Sect.~7]{ha}) of $F$ are $(0,0,0)$ and $(-1,0,1)$. In the 
former case $\tH^1(F(l)) = 0$ for $l\leq -2$ and $\h^1(F(-1)) = 3$, while in 
the latter case $\tH^1(F(l)) = 0$ for $l \leq -3$, $\h^1(F(-2)) = 1$ and 
$\h^1(F(-1)) = 3$. Moreover, in the latter case, $F$ is the cohomology of a 
monad of the form$\, :$ 
\[
0 \lra \sco_\piii(-2) \lra \sco_\piii(-1) \oplus 2\sco_\piii \oplus \sco_\piii(1) 
\lra \sco_\piii(2) \lra 0
\]  
(see Ein \cite[Thm.~3.3(a)]{ein}) hence, in particular, the graded module 
$\tH^1_\ast(F)$ is generated by $\tH^1(F(-2))$. If $F$ is properly semistable 
then, as we saw in Subcase 5.3 of the proof of Thm.~\ref{T:main}, $F$ can be 
realized as an extension$\, :$ 
\[
0 \lra \sco_\piii \lra F \lra \sci_Z \lra 0\, ,
\]
where $Z$ is a triple structure on a line $L \subset \piii$ such that 
$\sco_Z \simeq \sco_L \oplus \sco_L(1) \oplus \sco_L(2)$ as an 
$\sco_L$-algebra. One deduces that in this case $\tH^1(F(l)) = 0$ for 
$l \leq -3$, $\h^1(F(-2)) = 1$ and $\h^1(F(-1)) = 3$. 

\vskip2mm

\noindent
{\bf Claim 1.}\quad $\tH^1(E(l)) = 0$ \emph{for} $l \leq -2$. 

\vskip2mm

\noindent
\emph{Indeed}, consider a hyperplane $H \subset \piv$ such that $E_H$ is 
stable. By \cite[Lemma~1.16(a)]{acm}, $\tH^1(E(l)) = 0$ for $l \leq -3$ and 
$\tH^1(E(-2))$ injects into $\tH^1(E_H(-2))$. If $E_H$ has spectrum 
$(0,0,0)$ then $\tH^1(E_H(-2)) = 0$ hence $\tH^1(E(-2)) = 0$. If the spectrum 
of $E_H$ is $(-1,0,1)$ and $\tH^1(E(-2)) \neq 0$ then $\tH^1(E(-2)) \izo 
\tH^1(E_H(-2))$. Since, as we recalled above, the module $\tH^1_\ast(E_H)$ is 
generated by $\tH^1(E_H(-2))$ it follows that the map $\tH^1(E(l)) \ra 
\tH^1(E_H(l))$ is surjective for any $l \in \z$ hence, by 
\cite[Lemma~1.16(b)]{acm}, $\tH^2_\ast(E) = 0$. But this \emph{contradicts} 
Mohan Kumar et al. \cite[Thm.~1]{kpr} because $E$ is not decomposable. 

\vskip2mm

\noindent
{\bf Claim 2.}\quad $\tH^2(E(l)) = 0$ \emph{for} $l \geq -1$. 

\vskip2mm

\noindent
\emph{Indeed}, since $\tH^0(E(-1)) = 0$ and $\tH^1(E(-2)) = 0$, it follows that 
for every hyperplane $H \subset \piv$ one has $\tH^0(E_H(-1)) = 0$, i.e., 
$E_H$ is semistable. It follows, by what has been said at the beginning of the 
proof, that $\h^1(E_H(-1)) = 3$ and $\tH^1(E_H(l)) = 0$ for $l \leq -3$, 
$\forall \, H \subset \piii$ hyperplane. Since $E_H^\vee \simeq E_H$ one 
deduces, by Serre duality, that $\tH^2(E_H(l)) = 0$ for $l \geq -1$, 
$\forall \, H \subset \piii$ hyperplane.   
It follows that for any 
$0 \neq h \in \tH^0(\sco_\piv(1))$, the multiplication by $h : \tH^2(E(-2)) \ra 
\tH^2(E(-1))$ is surjective and that $\h^2(E(-2)) \leq \h^2(E(-1)) + 3$. 
Assume that $\tH^2(E(-1)) \neq 0$. Applying the Bilinear map lemma 
\cite[Lemma~5.1]{ha} to $\tH^2(E(-1))^\vee \times \tH^0(\sco_\piv(1)) \ra 
\tH^2(E(-2))^\vee$ one deduces that $\h^2(E(-2)) \geq \h^2(E(-1)) + 4$, which 
\emph{contradicts} a previously established inequality. It thus remains that 
$\tH^2(E(-1)) = 0$. Since for any hyperplane $H \subset \piv$ one has 
$\tH^2(E_H(l)) = 0$ for $l \geq -1$ one gets the claim. 

\vskip2mm

\noindent
{\bf Claim 3.}\quad $\tH^2_\ast(E) \simeq k(3) \oplus k(2)$. 

\vskip2mm

\noindent
\emph{Indeed}, by Serre duality and the fact that $E\simeq E^\vee$, 
$\tH^2(E(l)) = 0$ for $l \leq -4$. Using Riemann-Roch (see, for example, 
\cite[Thm.~7.3]{acm}), one gets $\h^2(E(-2)) = \chi(E(-2)) = 1$. By 
Serre duality again, $\h^2(E(-3)) = \h^2(E(-2)) = 1$. It remains to show that,  
for every $h \in \tH^0(\sco_\piv(1))$, the multipliction by $h : \tH^2(E(-3)) 
\ra \tH^2(E(-2))$ is the zero map. Recall that the cup product$\, :$ 
\[
-\cup - : \tH^2(E(-3)) \times \tH^2(E(-3)) \lra 
\tH^4((\overset{2}{\textstyle \bigwedge}E)(-6)) \simeq \tH^4(\sco_\piv(-6))
\]     
is \emph{skew-symmetric} (because $-\wedge - : E\times E \ra 
\overset{2}{\bigwedge}E$ is). If $0 \neq \xi \in \tH^2(E(-3))$ then$\, :$ 
\[
h\xi \cup \xi = h(\xi \cup \xi) = 0\, .
\] 
Since, on the other hand, the cup product $\tH^2(E(-2)) \times \tH^2(E(-3)) 
\ra \tH^4((\overset{2}{\bigwedge}E)(-5)) \simeq \tH^4(\sco_\piv(-5)) \simeq k$ 
is a perfect pairing (by Serre duality), it follows that $h\xi = 0$. 
Claim 3 is proven. 

\vskip2mm 

\noindent
At this point one can invoke Mohan Kumar et al. \cite[Thm.~2]{kpr} and deduce 
that the bundle $E$ \emph{cannot exist}. However, since our situation is very 
concrete, we shall also provide a slightly different argument. One has, by 
Riemann-Roch, $\h^1(E(-1)) = -\chi(E(-1)) = 2$ and $\h^1(E) = - \chi(E) = 6$. 
Since $\tH^2(E(-1)) = 0$, $\tH^3(E(-2)) \simeq \tH^1(E(-3))^\vee = 0$ and 
$\tH^4(E(-3)) \simeq \tH^0(E(-2))^\vee = 0$ it follows from the 
Castelnuovo-Mumford lemma (see, for example, \cite[Lemma~1.21]{acm}) 
that the graded module $\tH^1_\ast(E)$ is generated in degrees $\leq 0$.  
By Barth's restriction theorem, 
there exists a 2-plane $\Pi \subset \piv$ such that $E_\Pi$ is stable, 
i.e., such that $\tH^0(E_\Pi) = 0$. It follows that, for every hyperplane 
$H \supset \Pi$, $\tH^0(E_H) = 0$. If $h = 0$ is the equation of such a 
hyperplane, then the multiplication by $h : \tH^1(E(-1)) \ra \tH^1(E)$ is 
injective. Applying the Bilinear map lemma \cite[Lemma~5.1]{ha} to 
$\mu : \tH^1(E(-1))\otimes \tH^0(\sci_\Pi(1)) \ra \tH^1(E)$ one gets that 
$\dim \text{Im}\, \mu \geq 3$. One deduces that the graded module 
$\tH^1_\ast(E)$ has two minimal generators in degree $-1$ and at most three 
minimal generators in degree $0$, whence a surjection $2S(1)\oplus 3S \ra 
\tH^1_\ast(E)$. Since $E^\vee \simeq E$, one deduces a surjection 
$2S(1) \oplus 3S \ra \tH^1_\ast(E^\vee)$. Using Horrocks' method of 
``killing cohomology'', one gets that $E$ is the cohomology of a monad of the 
form$\, :$ 
\[
0 \lra 3\sco_\piv \oplus 2\sco_\piv(-1) \overset{\displaystyle \alpha}{\lra} 
\scp \overset{\displaystyle \beta}{\lra} 2\sco_\piv(1) \oplus 3\sco_\piv 
\lra 0\, , 
\]   
where $\scp$ is a vector bundle of rank 12, with $\tH^1_\ast(\scp) = 0$, 
$\tH^3_\ast(\scp) = 0$ and $\tH^2_\ast(\scp) \simeq \tH^2_\ast(E)$. It follows 
that $\scp \simeq \Omega_\piv^2(3) \oplus \Omega_\piv^2(2)$, hence $E$ is 
the cohomology of a monad of the form$\, :$ 
\[
0 \lra 3\sco_\piv \oplus 2\sco_\piv(-1) \overset{\displaystyle \alpha}{\lra} 
\Omega_\piv^2(3) \oplus \Omega_\piv^2(2)  
\overset{\displaystyle \beta}{\lra} 2\sco_\piv(1) \oplus 3\sco_\piv 
\lra 0\, . 
\]

\noindent
{\bf Claim 4.}\quad \emph{Such a monad cannot exist}. 

\vskip2mm

\noindent
\emph{Indeed}, $\alpha$ maps $3\sco$ into $\Omega^2(3)$ and $\beta$ maps 
$\Omega^2(3)$ into $2\sco(1)$. Consequently, the monad has a subcomplex of 
the form$\, :$ 
\[
0 \lra 3\sco \overset{\displaystyle \alpha^\prim}{\lra} 
\Omega^2(3) 
\overset{\displaystyle \beta^\prim}{\lra} 2\sco(1) \lra 0\, . 
\]
Let $F := \Cok \alpha^\prim (-1)$. $F$ is a \emph{Tango bundle} on $\piv$. It 
is well-known that $\tH^0(F^\vee) = 0$. The usual argument is the following 
one$\, :$ since $\text{rk}\, F = 3$ and $c_1(F) = 0$ one has $F^\vee \simeq 
\overset{2}{\bigwedge}F$. The second exterior power of the resolution$\, :$ 
\[
0 \lra \sco(-3) \lra 5\sco(-2) \lra 7\sco(-1) \lra F \lra 0
\] 
of $F$ is a resolution of $\overset{2}{\bigwedge}F$ of the form$\, :$ 
\[
0 \lra 5\sco(-5) \lra 
\begin{matrix} S^2(5\sco(-2))\\ \oplus\\ 7\sco(-4) \end{matrix} \lra 
35\sco(-3) \lra \overset{2}{\textstyle \bigwedge}(7\sco(-1)) \lra 
\overset{2}{\textstyle \bigwedge}F \lra 0 
\]
from which one gets that $\tH^0(\overset{2}{\bigwedge}F) = 0$. Since 
$\tH^0(F^\vee) = 0$ it follows that there is no non-zero morphism 
$\beta^\prim : \Omega^2(3) \ra 2\sco(1)$ such that $\beta^\prim \circ 
\alpha^\prim = 0$. Since there is no epimorphism $\Omega^2(2) \ra 2\sco(1) 
\oplus 3\sco$, a monad of the above form \emph{cannot exist}. 
\end{proof}

Ballico and Chiantini \cite[Prop.~3]{bc} showed, moreover, that there is no 
strictly semistable rank 2 vector bundle on $\piv$ with Chern classes 
$c_1 = 0$ and $c_2 = 3$. The key point of their proof is the following lemma, 
for which we provide a different argument.  

\begin{lemma}[Ballico,~Chiantini]\label{L:bc}
Let $Y$ be a locally complete intersection subscheme of $\piv$, of degree 
$3$, supported on a plane $\Pi \subset \piv$. Then $Y$ is a complete 
intersection of type $(1,3)$. 
\end{lemma} 

\begin{proof}
Let $S = k[x_0,x_1,x_2,x,y]$ be the projective coordinate ring of $\piv$ and 
assume that $\Pi$ is the plane of equations $x = y = 0$. Let $L \subset \piv$ 
be the (complementary) line of equations $x_0 = x_1 = x_2 = 0$ and let 
$\pi : \piv \setminus L \ra \Pi$ be the linear projection. If $\scf$ is a 
locally CM coherent $\sco_\p$-module supported on $\Pi$ then $F := 
\pi_\ast \scf$ is a locally free $\sco_\Pi$-module. The reduced stalk $F(z)$ 
of $F$ at a point $z \in \Pi$ can be described geometrically as follows$\, :$ 
let $\Pi^\prim$ be the plane spanned by $z$ and $L$. If $\Pi^\prim$ has 
equations $\ell^\prim = \ell^\secund = 0$ with $\ell^\prim ,\, \ell^\secund \in 
k[x_0,x_1,x_2]_1$ then $\ell^\prim ,\, \ell^\secund$ generate the maximal ideal 
$\fm_{\Pi,z}$ of $\sco_{\Pi,z}$ hence$\, :$ 
\[
F(z) := F_z/\fm_{\Pi,z}F_z \simeq \scf_z/(\ell^\prim ,\ell^\secund)\scf_z \simeq 
\scf \otimes_{\sco_\p} \sco_{\Pi^\prim}\, . 
\]
Notice also that if $\ell \in k[x_0,x_1,x_2]_1$ has the property that 
$\ell(z) \neq 0$ then $\ell ,\, \ell^\prim,\, \ell^\secund$ is a $k$-basis of 
$k[x_0,x_1,x_2]_1$, $\ell ,\, x,\, y$ is a $k$-basis of 
$\tH^0(\sco_{\Pi^\prim}(1))$ and $z$ is the point of $\Pi^\prim$ of equations 
$x = y = 0$. 

Let us, now, prove the lemma. Consider the rank-2 locally free 
$\sco_\Pi$-module $E := \pi_\ast(\sci_\Pi/\sci_Y)$ and the two morphisms 
$\xi,\, \eta : E \ra E(1)$ obtained by applying $\pi_\ast$ to the morphisms 
$x\cdot - ,\, y\cdot - : \sci_\Pi/\sci_Y \ra (\sci_\Pi/\sci_Y)(1)$, 
respectively. Let $z$ be a point of $\Pi$ and let $\Pi^\prim$ be as above. 
Since any locally complete intersection subscheme of $\Pi^\prim$ of degree 3, 
supported on $z$, is a complete intersection of type $(1,3)$ in $\Pi^\prim$ 
it follows that there exist linearly independent linear forms 
$x^\prim ,\, y^\prim \in k[x,y]_1$ such that $\Pi^\prim \cap Y$ is the subscheme 
of $\Pi^\prim$ whose homogeneous ideal is generated by $x^{\prim 3}$ and 
$y^\prim$. Working with ideals of the ring $\sco_{\Pi^\prim ,z}$, the observation 
from the beginning of this proof shows that$\, :$ 
\[
E(z) \simeq \frac{(x,y)}{(x^{\prim 3},y^\prim)} \  \  \text{and}\  \     
\text{Im}\, \xi(z) + \text{Im}\, \eta(z) \simeq 
\frac{(x,y)^2 + (x^{\prim 3},y^\prim)}{(x^{\prim 3},y^\prim)} 
\simeq \frac{(x^{\prim 2},y^\prim)}{(x^{\prim 3},y^\prim)} \simeq k\, .
\]
It follows that $\text{Im}\, \xi + \text{Im}\, \eta \simeq 
((\sci_\Pi^2 + \sci_Y)/\sci_Y)(1)$ is a line subbundle of 
$E(1)$ hence$\, :$ 
\[
\frac{\sci_\Pi^2 + \sci_Y}{\sci_Y} \simeq \sco_\Pi(a),\  
\frac{\sci_\Pi}{\sci_\Pi^2 + \sci_Y} \simeq \sco_\Pi(b) \  \text{and} \  
E\simeq \sco_\Pi(a) \oplus \sco_\Pi(b) 
\] 
for some $a,\, b \in \z$. Since we have epimorphisms$\, :$ 
\[
2\sco_\Pi(-1) \simeq \frac{\sci_\Pi}{\sci_\Pi^2} \lra 
\frac{\sci_\Pi}{\sci_\Pi^2 + \sci_Y}\  \text{and} \  
(x,y) : 2\frac{\sci_\Pi}{\sci_\Pi^2 + \sci_Y} \lra 
\frac{\sci_\Pi^2 + \sci_Y}{\sci_Y}(1) 
\]
it follows that $b = -1$ and $a + 1 = b$, hence $a = -2$ and 
$E \simeq \sco_\Pi(-2) \oplus \sco_\Pi(-1)$. Using the exact sequence 
$0 \ra \sci_Y \ra \sci_\Pi \ra \sci_\Pi/\sci_Y \ra 0$ one deduces, now, that 
$\tH^0(\sci_Y(1)) \neq 0$, whence the conclusion of the lemma. 
\end{proof} 

\begin{prop}\label{P:mainn4}
Let $E$ be a globally generated vector bundle on $\p^n$, $n \geq 4$, with 
$c_1 \geq 4$ and such that ${\fam0 H}^i(E^\vee) = 0$, $i = 0,\, 1$. Let 
$\Pi \subset \p^n$ be a fixed $3$-plane. Assume that 
${\fam0 H}^0(E(-c_1 + 2)) = 0$ and that ${\fam0 H}^0(E_\Pi(-c_1 + 3)) \neq 0$. 
Then one of the following holds$\, :$ 
\begin{enumerate}
\item[(i)] $E \simeq \sco_\p(c_1 - 3) \oplus E^\prim$ where $E^\prim$ is a 
globally generated vector bundle with $c_1(E^\prim) = 3$$\, ;$ 
\item[(ii)] $n = 4$, $c_1 = 4$ and $E \simeq {\fam0 T}_\piv(-1) \oplus 
\Omega_\piv(2)$$\, ;$ 
\item[(iii)] $n = 5$, $c_1 = 4$ and $E \simeq \Omega_\pv(2)$$\, :$ 
\item[(iv)] $n = 4$, $c_1 = 4$ and, up to a linear change of coordinates, 
denoting by $(C_p,\, \delta_p)_{p\geq 0}$ the Koszul complex associated to the 
epimorphism $\delta_1 : 4\sco_\piv(-1) \oplus \sco_\piv(-2) \ra \sco_\piv$ 
defined by $x_0, \ldots , x_3, x_4^2$, one has exact sequences$\, :$ 
\begin{gather*}
0 \lra \sco_\piv(-2) \xra{\delta_5(4)} \sco_\piv \oplus 4\sco_\piv(-1) 
\xra{\delta_4(4)} 4\sco_\piv(1) \oplus 6\sco_\piv \lra E^\prim \lra 0\\
0 \lra E \lra E^\prim \xra{\phi} \sco_\piv(2) \lra 0
\end{gather*} 
where $\phi : E^\prim \ra \sco_\piv(2)$ is any morphism with the property that 
${\fam0 H}^0(\phi(-1)) : {\fam0 H}^0(E^\prim(-1)) \ra {\fam0 H}^0(\sco_\piv(1))$ 
is injective $($such morphisms exist and are automatically epimorphisms$)$. 
\end{enumerate}
\end{prop} 

\noindent 
The bundle from item (iv) of Prop.~\ref{P:mainn4} appeared for the first 
time, constructed differently, in the paper of Abo, Decker and Sasakura 
\cite{ads}.  

\begin{proof}
It follows from \cite[Prop.~2.11]{acm} that $\tH^0(E_\Pi(-c_1 + 2)) = 0$. One 
has $E_\Pi \simeq G \oplus t\sco_\Pi$, with $G$ defined by an exact 
sequence $0 \ra s\sco_\Pi \ra F \ra G \ra 0$ where $F$ is a globally generated 
vector bundle on $\Pi$ such that $\tH^i(F^\vee) = 0$, $i = 0,\, 1$ (see, for 
example, \cite[Sect.~1]{acm}). It follows that $F$ is one of the bundles 
described in the statement of Theorem~\ref{T:main}. 

\vskip2mm 

\noindent
{\bf Case 1.}\quad $F \simeq \sco_\Pi(c_1-3) \oplus F^\prim$ \emph{with} 
$c_1(F^\prim) = 3$. 

\vskip2mm

\noindent
As we said in the Introduction, the bundles $F^\prim$ were classified by 
Anghel and Manolache \cite{am} and, independently, by Sierra and Ugaglia 
\cite{su2}. A concise description of this classification can be found in 
\cite[Thm.~0.1]{acm}. Except when $F^\prim$ contains $\Omega_\Pi(2)$ as a direct 
summand, \cite[Lemma~1.18]{acm} implies that $E$ is as in item (i) of the 
statement of the proposition. 

If $F^\prim \simeq \sco_\Pi(1) \oplus \Omega_\Pi(2)$ then 
\cite[Lemma~1.19]{acm} (and \cite[Remark~1.20(c)]{acm}) implies that 
$n = 4$ and $E \simeq \sco_\piv(c_1 - 3) \oplus \Omega_\piv(2)$ or $c_1 = 4$, 
$n = 5$ and $E \simeq \Omega_\pv(2)$. 

If $F^\prim \simeq \text{T}_\Pi(-1) \oplus \Omega_\Pi(2)$ then 
\cite[Lemma~1.19]{acm} (and \cite[Remark~1.20(c)]{acm}) implies that 
$n = 4$ and $E \simeq \sco_\piv(c_1 - 3) \oplus \Omega_\piv^2(3)$ or $c_1 = 4$, 
$n = 4$ and $E \simeq \text{T}_\piv(-1) \oplus \Omega_\piv(2)$. 

\vskip2mm

\noindent
{\bf Case 2.}\quad $F$ \emph{as $E$ in Theorem}~\ref{T:main}(ii). 

\vskip2mm

\noindent
This case \emph{cannot occur}. Indeed, assume the contrary. Then 
\cite[Cor.~1.5]{acm} would imply that there exists a rank 2 vector bundle 
on $\piv$ with Chern classes $c_1 = 0$ and $c_2 = 1$ or with Chern classes 
$c_1 = 0$ and $c_2 = 2$. But this would \emph{contradict} 
Schwarzenberger's congruence asserting that the Chern classes of a coherent 
sheaf on $\piv$ satisfy the relation$\, :$ 
\[
(2c_1+3)(c_3 - c_1c_2) + c_2^2 + c_2 \equiv 2c_4 \pmod{12} 
\]    
(see, for example, \cite[Cor.~7.4]{acm}). 

\vskip2mm

\noindent
{\bf Case 3.}\quad $F$ \emph{as $E$ in Theorem}~\ref{T:main}(iii). 

\vskip2mm

\noindent
This case \emph{cannot occur}. Indeed, in order to show this, one can assume 
that $n = 4$. Since $\tH^1_\ast(F^\vee) = 0$, \cite[Lemma~1.17(a)]{acm} implies 
that $t = 0$. Moreover, $s = 0$ because $c_3(F) = 2 \neq 0$. Now, applying 
\cite[Lemma~1.16(a)]{acm} and \cite[Lemma~1.14(b)]{acm} to $F^\vee$ one 
derives the existence of an exact sequence 
\[
0 \lra E \lra 3\sco_\piv(2) \oplus \sco_\piv(1) \lra \sco_\piv(3) \lra 0\, .
\] 
But \emph{this is not possible} because there is no epimorphism 
$3\sco_\piv(2) \oplus \sco_\piv(1) \ra \sco_\piv(3)$. 

\vskip2mm

\noindent
{\bf Case 4.}\quad $F$ \emph{as $E$ in Theorem}~\ref{T:main}(iv). 

\vskip2mm

\noindent
This case \emph{cannot occur}. Indeed, in order to prove this, one can assume 
that $n = 4$. It follows, as in Case 3, that $t = 0$. Moreover, $F$ has rank 
4 and $c_3(F) = 4 \neq 0$ hence $s \leq 1$. 

If $s = 0$ then, as in Case 3, there exists an exact sequence$\, :$ 
\[
0 \lra E \lra 2\sco_\piv(2) \oplus 3\sco_\piv(1) 
\overset{\displaystyle \e}{\lra} \sco_\piv(3) \lra 0\, .
\]
Let $C \subset \piv$ be the complete intersection 1-dimensional subscheme of 
$\piv$ with the property that $\sci_C(3)$ is the image of the restriction  
$2\sco_\piv(2) \oplus \sco_\piv(1) \ra \sco_\piv(3)$ of $\e$. It follows that 
$E \vb C \simeq 2\sco_C(2) \oplus \sco_C(1) \oplus \sco_C(-1)$ hence 
$E \vb C$ is not globally generated, a \emph{contradiction}. 

If $s = 1$ then $E$ is a rank 3 vector bundle on $\piv$ with Chern classes 
$c_1 = 4$, $c_2 = 7$, $c_3 = 4$. But this would \emph{contradict} 
Schwarzenberger's congruence. 

\vskip2mm

\noindent
{\bf Case 5.}\quad $F$ \emph{as $E$ in Theorem}~\ref{T:main}(v). 

\noindent
In this case, the proof of \cite[Prop.~7.11]{acm} (more precisely, the Cases 
7 and 8 of that proof) shows that $n = 4$ and that $E$ is as in item (iv) of 
the statement of our proposition. 

\vskip2mm

\noindent
{\bf Case 6.}\quad $F$ \emph{as $E$ in Theorem}~\ref{T:main}(vi). 

\vskip2mm 

\noindent
This case \emph{cannot occur} : one uses the same kind of argument as in 
Case 2. 

\vskip2mm

\noindent
{\bf Case 7.} $F$ \emph{as $E$ in Theorem}~\ref{T:main}(vii). 

\vskip2mm

\noindent
This case \emph{cannot occur}. Indeed, assume the contrary. Choose a 4-plane 
$\piv \subseteq \p^n$ containing $\Pi$. 
The last part of \cite[Cor.~1.5]{acm} would imply that 
there exists a rank 2 vector bundle $E^\prim$ on $\piv$ with Chern classes 
$c_1(E^\prim) = 0$ and $c_2(E^\prim) = 3$ such that $E^\prim \vb \Pi$ is 
strictly semistable. The result of Barth and Ellencwajg recalled above 
(Theorem~\ref{T:be}) implies that $E^\prim$ must be strictly semistable. 
But, according to Ballico and Chiantini \cite[Prop.~3]{bc}, \emph{this is not 
possible}, either. Their argument runs as follows : one must have an exact 
sequence$\, :$ 
\[
0 \lra \sco_\piv \lra E^\prim \lra \sci_Y \lra 0
\] 
where $Y$ is a locally complete intersection closed subscheme of $\piv$, of 
codimension 2 and degree 3. As we noticed in Subcase 5.3 of the above proof of 
Theorem~\ref{T:main}, if $H \subset \piv$ is a hyperplane cutting $Y$ properly 
then the subscheme $H \cap Y$ of $H \simeq \piii$ must be supported on a line. 
One deduces that $Y$ is supported on a plane in $\piv$. Lemma~\ref{L:bc} 
implies, now, that $Y$ is a complete intersection of type $(1,3)$, hence 
$E^\prim$ is a direct sum of two line bundles. But \emph{this is not possible}  
because the system of equations $a + b = 0$ and $ab = 3$ has no integer 
solutions.   
\end{proof} 

\begin{remark}\label{R:c1=0c2=3n4}
As we recalled in Case 7 of the proof of Prop.~\ref{P:mainn4}, Ballico and 
Chiantini used their Lemma~\ref{L:bc} to show, in connection with 
Thm.~\ref{T:be} of Barth and Ellencwajg, that there exists no rank 2 
vector bundle $E$ on $\piv$, with Chern classes $c_1(E) = 0$, $c_2(E) = 3$ 
such that $\tH^0(E) \neq 0$ and $\tH^0(E(-1)) = 0$. One can push this kind of 
results one step further by showing that there exists no rank 2 vector 
bundle $E$ on $\piv$ with those Chern classes such that $\tH^0(E(-1)) \neq 0$ 
and $\tH^0(E(-2)) = 0$. 

\vskip2mm

\noindent
\emph{Indeed}, assume that such a bundle exists. Then it can be realized as 
an extension$\, :$ 
\[
0 \lra \sco_\piv(1) \lra E \lra \sci_Y(-1) \lra 0
\]   
where $Y$ is a l.c.i. closed subscheme of $\piv$, of pure codimension 2, 
of degree 4, with $\omega_Y \simeq \sco_Y(-7)$. 
Let $H \subset \piv$ be a hyperplane cutting properly $Y_{\text{red}}$ 
and $X := H \cap Y$. One has an exact sequence$\, :$ 
\[
0 \lra \sco_H(1) \lra E_H \lra \sci_{X,H}(-1) \lra 0\, .  
\]
$X$ is a l.c.i. curve in $H \simeq \piii$, of degree 4, with $\omega_X 
\simeq \sco_X(-6)$. The last condition implies that $X$ cannot be reduced and 
irreducible. Moreover, $X$ cannot have a reduced component. Indeed, if this 
would be the case then $X$ would have as a component a line or a 
(nonsingular) conic and the argument used in Subcase 5.3 of the proof of 
Thm.~\ref{T:main} would show that this contradicts the condition $\omega_X 
\simeq \sco_X(-6)$. It thus remains that $X$ is either the union of two 
double lines or a double conic or a quadruple line. It follows that 
$Y_{\text{red}}$ is either$\, :$ (i) the union of two planes or (ii) a 
nonsingular quadric surface or (iii) a quadric cone or (iv) a plane.  

\vskip2mm 
 
\noindent 
In case (i) the two planes should intersect along a line (see, for example, 
\cite[Thm.~18.12]{eis}). Consequently, in the first three cases, 
$Y_{\text{red}}$ is a reduced complete intersection of type $(1,2)$ in $\piv$ 
and $Y$ is a double structure on it. Let $Y^\prim$ be the residual of 
$Y_{\text{rsd}}$ with respect to the l.c.i. scheme $Y$. One has, 
by definition$\, :$ 
\[
\sci_{Y^\prim}/\sci_Y \simeq \sch om_{\sco_Y}(\sco_{Y_{\text{red}}},\, \sco_Y) 
\simeq \omega_{Y_{\text{red}}} \otimes \omega_Y^{-1} \simeq 
\sco_{Y_{\text{red}}}(5)\, .
\] 
Using the exact sequence $0 \ra \sco_{Y_{\text{red}}}(5) \ra \sco_Y \ra 
\sco_{Y^\prim} \ra 0$ and Hilbert polynomials one deduces that 
$\text{deg}\, Y^\prim = \text{deg}\, Y - \text{deg}\, Y_{\text{red}} = 2$. 
Since $Y_{\text{red}} \subseteq Y^\prim$ and since $Y^\prim$ is locally CM it 
follows that $Y^\prim = Y_{\text{red}}$. Since there is no epimorphism$\, :$ 
\[
\sci_{Y_{\text{red}}}/\sci^2_{Y_{\text{red}}} \simeq \sco_{Y_{\text{red}}}(-1) \oplus 
\sco_{Y_{\text{red}}}(-2) \lra \sco_{Y_{\text{red}}}(5) 
\] 
none of the first three cases can occur. 

In case (iv) the results of Manolache \cite{m3} show that $Y$ is a complete 
intersection which implies that $E$ is a direct sum of line bundles. But 
this is not possible because $E$ has Chern classes $c_1 = 0$ and $c_2 = 3$. 
\end{remark}

\begin{lemma}\label{L:h0eh-1neq0} 
Let $E$ be a globally generated vector bundle on $\piv$ with Chern classes 
$c_1 = 4$ and $c_2 \leq 8$. Then there exists a hyperplane $H \subset \piv$ 
such that ${\fam0 H}^0(E_H(-1)) \neq 0$. 
\end{lemma}

\begin{proof} 
Assume the contrary, namely that $\tH^0(E_H(-1)) = 0$ for any hyperplane 
$H \subset \piv$. For each hyperplane $H \subset \piv$ 
one has an exact sequence$\, :$ 
\[
0 \lra (r-2)\sco_H \lra E_H \lra \scf(2) \lra 0
\] 
where $r = \text{rk}\, E$ and with $\scf$ a rank 2 reflexive sheaf on $H$ 
with $c_1(\scf) = 0$, $c_2(\scf) = c_2 - 4 \leq 4$ and $c_3(\scf) = c_3(E) 
=: c_3$. The condition $\tH^0(E_H(-1)) = 0$ is equivalent to 
$\tH^0(\scf(1)) = 0$. In particular, $\scf$ is \emph{stable} (that is, 
$\tH^0(\scf) = 0$). 

By \cite[Thm.~8.2(b)]{ha}, $\tH^2(\scf(l)) = 0$ for 
$l \geq c_2(\scf) - 3$. In particular, $\tH^2(\scf(1)) = 0$. Since, by 
Serre duality and the fact that $\scf^\vee \simeq \scf$, 
$\tH^3(\scf(1)) = 0$, the Riemann-Roch theorem (recalled in 
\cite[Thm.~4.5]{acm}) implies that$\, :$ 
\[
2\chi(\sco_H(1)) - 3c_2(\scf) + \frac{1}{2}c_3(\scf) = \chi(\scf(1)) = 
-\h^1(\scf(1)) \leq 0\, ,\  \text{hence}\  \frac{1}{2}c_3 \leq 3c_2 - 20\, . 
\] 
Recall that $c_3 \geq 0$ and that $c_3 \equiv 0 \pmod{2}$. 
It follows that either $c_2 = 7$ and $c_3 \leq 2$ or $c_2 = 8$ and 
$c_3 \leq 8$. Looking at the beginning of the proof of \cite[Prop.~5.1]{acm} 
and of \cite[Prop.~6.3]{acm} (where some spectra are eliminated)  
one sees that one of the following holds$\, :$ 
\begin{enumerate}
\item[(i)] $c_2 = 7$, $c_3 = 0$ and $\scf$ has spectrum $(0,0,0)$$\, ;$
\item[(ii)] $c_2 = 7$, $c_3 = 2$ and $\scf$ has spectrum $(0,0,-1)$$\, ;$ 
\item[(iii)] $c_2 = 8$, $c_3 = 0$ and $\scf$ has spectrum $(0,0,0,0)$$\, ;$ 
\item[(iv)] $c_2 = 8$, $c_3 = 0$ and $\scf$ has spectrum $(1,0,0,-1)$$\, ;$ 
\item[(v)] $c_2 = 8$, $c_3 = 2$ and $\scf$ has spectrum $(0,0,0,-1)$$\, ;$ 
\item[(vi)] $c_2 = 8$, $c_3 = 4$ and $\scf$ has spectrum $(0,0,-1,-1)$$\, ;$ 
\item[(vii)] $c_2 = 8$, $c_3 = 6$ and $\scf$ has spectrum 
$(0,-1,-1,-1)$$\, ;$ 
\item[(viii)] $c_2 = 8$, $c_3 = 8$ and $\scf$ has spectrum 
$(-1,-1,-1,-1)$$\, ;$ 
\item[(ix)] $c_2 = 8$, $c_3 = 8$ and $\scf$ has spectrum $(0,-1,-1,-2)$.       
\end{enumerate}
The cases (iii) and (iv) can be eliminated 
using \cite[Cor.~1.5]{acm} and Schwarzenberger's congruence (recalled above). 
Case (ix) cannot occur, either. \emph{Indeed}, 
one has, in this case, $\h^1(\scf(-1)) = 1$, $\h^2(\scf) = 0$ (by the 
definition of the spectrum) and $\h^1(\scf) = 2$ (by Riemann-Roch). One 
deduces, using the exact sequence of the hyperplane section, that there exist 
planes $P \subset H \simeq \piii$ such that $\tH^0(\scf_P) \neq 0$. For such a 
plane one has $\h^0(\scf_P(1)) \geq 3$. Since $\h^1(\scf) = 2$ it follows that 
$\h^0(\scf(1)) \geq 1$, a \emph{contradiction}.  

Let us assume, from now on, that one of the remaining cases holds. 

\vskip2mm

\noindent
{\bf Claim 1.}\quad $\tH^1(E(l)) = 0$ \emph{for} $l \leq -2$.  

\vskip2mm

\noindent
\emph{Indeed}, if $h$ is a non-zero linear form on $\piv$ and $H \subset 
\piv$ is the hyperplane of equation $h = 0$ then one has an exact 
sequence$\, :$ 
\[
0=\tH^0(E_H(-1)) \lra \tH^1(E(-2)) \overset{\displaystyle h}{\lra} 
\tH^1(E(-1)) \lra \tH^1(E_H(-1))\, . 
\]   
The Bilinear map lemma \cite[Lemma~5.1]{ha} implies that if $\tH^1(E(-2)) 
\neq 0$ then $\h^1(E(-1)) \geq \h^1(E(-2)) + 4$. On the other hand$\, :$ 
\[
\h^1(E(-1)) - \h^1(E(-2)) \leq \h^1(E_H(-1)) = \h^1(\scf(1)) = 
3c_2 - 20 - \frac{1}{2}c_3 \leq 3 
\]
and this is a \emph{contradiction}. It thus remains that $\tH^1(E(-2)) = 0$. 

Using the exact sequence $\tH^0(E_H(l+1)) \ra \tH^1(E(l)) \ra \tH^1(E(l+1))$ 
one shows now, by descending induction, that $\tH^1(E(l)) = 0$ for 
$l \leq -2$. 

\vskip2mm

\noindent 
{\bf Claim 2.}\quad $\tH^2(E(-3)) = 0$. 

\vskip2mm

\noindent
\emph{Indeed}, using the notation from the proof of Claim 1, one has an exact 
sequence$\, :$ 
\[
0 \lra \tH^1(E_H(-3)) \lra \tH^2(E(-4)) \overset{\displaystyle h}{\lra} 
\tH^2(E(-3)) \lra \tH^2(E_H(-3))\, .
\]
But $\tH^2(E_H(-3)) \simeq \tH^2(\scf(-1)) = 0$ (by the definition of the 
spectrum). The Bilinear map lemma \cite[Lemma~5.1]{ha} implies, now, that 
if $\tH^2(E(-3)) \neq 0$ then $\h^2(E(-4)) \geq \h^2(E(-3)) + 4$. 
On the other hand 
\[
\h^2(E(-4)) -\h^2(E(-3)) = \h^1(E_H(-3)) = \h^1(\scf(-1)) \leq 3
\] 
and this is a \emph{contradiction}. It remains that $\tH^2(E(-3)) = 0$. 

\vskip2mm

\noindent
{\bf Claim 3.}\quad $\tH^3(E(l)) = 0$ \emph{for} $l \geq -4$. 

\vskip2mm

\noindent
\emph{Indeed}, $\tH^2(E_H(l)) \simeq \tH^2(\scf(l+2)) = 0$ for $l \geq -3$ 
hence, by \cite[Lemma~1.16(b)]{acm}, $\tH^3(E(l)) = 0$ for $l \geq -4$. 

\vskip2mm

\noindent
Finally, $\tH^4(E(-5)) \simeq \tH^0(E^\vee)^\vee = 0$. Taking into account the 
above claims it follows that $E$ is $(-1)$-\emph{regular}. But this 
\emph{contradicts} the fact that $\tH^0(E(-1)) = 0$. 
\end{proof}

\begin{thm}\label{T:c1=4n4}
Let $E$ be an indecomposable globally generated vector bundle on $\p^n$, 
$n \geq 4$, with $c_1 = 4$ and such that ${\fam0 H}^i(E^\vee) = 0$, 
$i = 0,\, 1$. Then one of the following holds$\, :$ 
\begin{enumerate}
\item[(i)] $E \simeq \sco_\p(4)$$\, ;$ 
\item[(ii)] $E \simeq P(\sco_\p(4))$$\, ;$ 
\item[(iii)] $n = 4$ and $E$ as in item \emph{(iv)} of 
Prop.~\ref{P:mainn4}$\, ;$ 
\item[(iv)] $n = 5$ and $E \simeq \Omega_\pv(2)$$\, ;$ 
\item[(v)] $n = 5$ and $E \simeq \Omega_\pv^3(4)$. 
\end{enumerate}
\end{thm} 

\begin{proof}
Taking into account Lemma~\ref{L:h0eh-1neq0}, this follows from 
Prop.~\ref{P:mainn4} above and from \cite[Prop.~2.4~and~Cor.~2.5]{acm} and 
\cite[Prop.~2.11]{acm}. Notice that $c_2(\Omega_\pv(2)) = 7 < 8$, that 
$\Omega_\pv^3(4) \simeq P(\Omega_\pv(2))$ and that $c_2(\Omega_\pv^3(4)) = 
9 > 8$.   
\end{proof}

\newpage

\appendix 
\section{Multiple structures on a line}\label{A:multilines}  

We denote, in this appendix, the homogeneous coordinates on $\piii$ by 
$x_0, x_1, x, y$ hence the projective coordinate ring of $\piii$ is 
$S = k[x_0, x_1, x, y]$. Let $L \subset \piii$ be the line of equations 
$x = y = 0$. Our aim is to provide, for a locally CM space curve $Z$ of 
degree at most 4, supported on $L$, graded free resolutions for the 
homogeneous ideal $I(Z) \subset S$ and for the graded $S$-module 
$\tH^0_\ast(\sco_Z)$. The multiple structures on a smooth curve embedded in a 
threefold, up to (local) multiplicity 4, were described by B\u{a}nic\u{a} and 
Forster \cite{bf0}. Their results were published much later in a simplified, 
more conceptual, form in \cite{bf}.    

We begin by adapting, to our particular context, the results from \cite{bf}. 
Let $\pi : \piii \setminus L^\prim \ra L$ be the linear projection, where 
$L^\prim$ is the (complementary) line of equations $x_0 = x_1 = 0$. 
The functor $\pi_\ast$ induces an equivalence of categories 
between the category of coherent $\sco_\p$-modules supported on $L$ and the 
category of coherent $\sco_L$-modules $\scf$ endowed with two commuting 
twisted endomorphisms $\xi,\, \eta  : \scf \ra \scf(1)$ (corresponding to the  
$\sco_\p$-module multiplication by $x$ and $y$). Under this equivalence, 
locally CM $\sco_\p$-modules correspond to locally free $\sco_L$-modules. 
Recall that if $\scf$ is a locally free $\sco_L$-module and $\scf^\prim$ is 
an $\sco_L$-submodule then the \emph{saturation} $\scf^{\prim \, \text{sat}}$ of 
$\scf^\prim$ is defined by $\scf^{\prim \, \text{sat}}/\scf^\prim = 
(\scf/\scf^\prim)_{\text{tors}}$. If $\scf$ has, moreover, an $\sco_\p$-module 
structure as above and if  
$\scf^\prim$ an $\sco_\p$-submodule then $\scf^{\prim \, \text{sat}}$ is an 
$\sco_\p$-submodule of $\scf$. 

According to \cite{bf}, a locally CM multiple structure $Z$ on $L$ is called  
\emph{quasi-primitive} if the morphism $\sci_Z \ra \sci_L/\sci_L^2$ is 
non-zero, and it is called \emph{thick} if $\sci_Z \subseteq \sci_L^2$. We 
consider, firstly, the \emph{quasi-primitive case}. In this case, there exists 
a nonempty open subset $U$ of $L$ such that, $\forall \, z \in U$, there 
exists a system of parameters $(t,u,v)$ of $\sco_{\p,z}$ such that 
$\sci_{L,z} = (u,v)$ and $\sci_{Z,z} = (u,v^d)$, where $d$ is the degree of $Z$. 
B\u{a}nic\u{a} and Forster define the \emph{Cohen-Macaulay filtration} 
$\sco_Z \supset \sci_1 \supset \sci_2 \supset \cdots \supset \sci_d = (0)$ of 
$\sco_Z$ by $\sci_i := (\sci_{L,Z}^i)^{\text{sat}}$. This is a 
\emph{multiplicative} filtration in the sense that $\sci_i$ are ideals of 
$\sco_Z$ and $\sci_i\sci_j \subseteq \sci_{i+j}$. $\sci_i/\sci_{i+1}$ is 
annihilated by $\sci_L$ hence it is already an $\sco_L$-module and, in the 
quasi-primitive case under consideration, $\sci_i/\sci_{i+1}$ is an 
invertible $\sco_L$-module for $i =0, \ldots , d-1$.         

\subsection{Quasiprimitive structures of degree 4}\label{SS:q4} 
We describe, firstly, following \cite{bf}, the $\sco_L$-module structure of 
$\sco_Z$. 
The CM filtration has, in this case, four steps$\, :$  
\[
\sco_Z \supset \sci_1 
\supset \sci_2 \supset \sci_3 \supset (0)\, . 
\]
Moreover, there exists an epimorphism of filtered $\sco_L$-algebras $\e : 
\sco_\p/\sci_L^4 \twoheadrightarrow \sco_Z$. Recall that, as an 
$\sco_L$-module,  
\[
\sco_\p/\sci_L^4 \simeq \sco_L \oplus 2\sco_L(-1) \oplus 3\sco_L(-2) 
\oplus 4\sco_L(-3)\, . 
\] 
Now, $\sci_1/\sci_2 \simeq \sco_L(l)$, 
for some $l \in \z$. Since there exists an epimorphism $2\sco_L(-1) \simeq 
\sci_L/\sci_L^2 \ra \sci_1/\sci_2$ it follows that $l \geq -1$. Since the 
multiplication morphisms$\, :$ 
\[
\sci_1/\sci_2\otimes_{\sco_L}\sci_1/\sci_2 \lra \sci_2/\sci_3\, , \  \  
\sci_1/\sci_2\otimes_{\sco_L}\sci_2/\sci_3 \lra \sci_3 
\] 
are nonzero (due to the local structure of the $\sco_L$-algebra $\sco_Z$) it 
follows that $\sci_2/\sci_3 \simeq \sco_L(2l+m)$ with $m \geq 0$ and 
$\sci_3 \simeq \sco_L(3l+m+n)$ with $n \geq 0$. The exact sequence$\, :$ 
\[
0 \lra \sci_3 \lra \sci_2 \lra \sci_2/\sci_3 \lra 0 
\] 
splits in the category of $\sco_L$-modules, hence $\sci_2 \simeq 
\sco_L(2l+m) \oplus \sco_L(3l+m+n)$ as $\sco_L$-modules. 
The exact sequence$\, :$ 
\[
0 \lra \sci_2 \lra \sci_1 \lra \sci_1/\sci_2 \lra 0 
\] 
splits in the category of $\sco_L$-modules if $l \geq 0$. It splits, also, 
for $l = -1$ because, in this case, the epimorphism $\e$ considered 
above induces a composite epimorphism$\, :$ 
\[
2\sco_L(-1) \lra \sci_1 \lra \sci_1/\sci_2 \simeq \sco_L(-1)\, .
\]
The same kind of argument shows that $\sco_Z/\sci_1 = \sco_L$ is a direct 
summand of the $\sco_L$-module $\sco_Z$ hence$\, :$ 
\[
\sco_Z \simeq \sco_L \oplus \sco_L(l) \oplus \sco_L(2l+m) \oplus 
\sco_L(3l+m+n)\, .
\] 
The \emph{multiplicative structure} of $\sco_Z$ is defined by two 
morphisms of $\sco_L$-modules$\, :$ 
\[
\mu_{11} : \sco_L(l) \otimes \sco_L(l) \xra{\begin{pmatrix} p\\ p^\prim 
\end{pmatrix}} \begin{matrix} \sco_L(2l+m)\\ \oplus\\ \sco_L(3l+m+n) 
\end{matrix}\, ,\  \  
\mu_{12} : \sco_L(l) \otimes \sco_L(2l+m) \overset{\displaystyle q}{\lra}  
\sco_L(3l+m+n)\, ,    
\] 
with $0 \neq p \in \tH^0(\sco_L(m))$, $p^\prim \in \tH^0(\sco_L(l+m+n))$ and 
$0 \neq q \in \tH^0(\sco_L(n))$. 

\vskip2mm 

As a graded 
$\tH^0_\ast(\sco_L) = k[x_0,x_1]$-module, $\tH^0_\ast(\sco_Z)$ has a minimal set 
of generators $1 \in \tH^0(\sco_Z)$, $e_1 \in \tH^0(\sco_Z(-l))$, 
$e_2 \in \tH^0(\sco_Z(-2l-m))$, $e_3 \in \tH^0(\sco_Z(-3l-m-n))$. The 
multiplicative structure of $\tH^0_\ast(\sco_Z)$ is defined by the 
relations$\, :$ 
\begin{equation}\label{E:multiplicativestructure} 
e_1^2 = pe_2 + p^\prim e_3\, ,\   e_1e_2 = qe_3\, ,\  e_1e_3 = 0\, ,\  
e_2^2 = 0\, ,\  e_2e_3 = 0\, ,\  e_3^2 = 0\, .     
\end{equation}
The epimorphism $\sco_\piii \ra \sco_Z$ induces a morphism of graded 
$k[x_0,x_1]$-algebras $S = \tH^0_\ast(\sco_\piii) \ra 
\tH^0_\ast(\sco_Z)$ which is completely determined by the images of 
$x,\, y \in S$. Let's say that$\, :$ 
\begin{equation}\label{E:multiplicationxy} 
x \mapsto ae_1 + a^\prim e_2 + a^\secund e_3\, ,\  \  
y \mapsto be_1 + b^\prim e_2 + b^\secund e_3\, ,
\end{equation}
with $a,\, b \in \tH^0(\sco_L(l+1))$, $a^\prim,\, b^\prim \in 
\tH^0(\sco_L(2l+m+1))$ and $a^\secund,\, b^\secund \in \tH^0(\sco_L(3l+m+n+1))$. 
The images of the other monomials in the indeterminates 
$x$ and $y$ can be deduced from the multiplicative structure of 
$\tH^0_\ast(\sco_Z)$. For example, 
\[
xy \mapsto (ae_1 + a^\prim e_2 + a^\secund e_3)(be_1 + b^\prim e_2 + b^\secund e_3) 
= pabe_2 + (p^\prim ab + q(ab^\prim + a^\prim b))e_3\, . 
\] 
One deduces that the epimorphism $\e_4 : \sci_L/\sci_L^4 \ra 
\sci_1$ induced by $\e : \sco_\p/\sci_L^4 \ra \sco_Z$ is defined by the 
matrix$\, :$  
\begin{equation}\label{E:matrixe4} 
\begin{pmatrix} 
a & b & 0 & 0 & 0 & 0 & 0 & 0 & 0\\
a^\prim & b^\prim & pa^2 & pab & pb^2 & 0 & 0 & 0 & 0\\
a^\secund & b^\secund & p^\prim a^2 + 2qaa^\prim & 
p^\prim ab + q(ab^\prim + a^\prim b) & 
p^\prim b^2 + 2qbb^\prim & pqa^3 & pqa^2b & pqab^2 & pqb^3 
\end{pmatrix}
\end{equation}
Consider the following two determinants$\, :$ 
\begin{equation}\label{E:Delta12} 
\Delta_1 := 
\begin{vmatrix} 
a & b\\
a^\prim & b^\prim
\end{vmatrix} 
\, ,\  \Delta_2 := 
\begin{vmatrix}
a & b & 0\\
a^\prim & b^\prim & p\\
a^\secund & b^\secund & p^\prim
\end{vmatrix}
\end{equation}
One has $\Delta_1 \in \tH^0(\sco_L(3l+m+2))$ and $\Delta_2 \in 
\tH^0(\sco_L(4l+2m+n+2))$. One can easily prove the following$\, :$  

\begin{lemma}\label{L:epicondition} 
The $3 \times 9$ matrix \eqref{E:matrixe4} considered above defines an 
epimorphism$\, :$  
\[
\e_4 : 
2\sco_L(-1) \oplus 3\sco_L(-2) \oplus 4\sco_L(-3) \lra \sco_L(l) \oplus 
\sco_L(2l+m) \oplus \sco_L(3l+m+n) 
\]
if and only if the following three conditions are satisfied$\, :$ 
\begin{enumerate} 
\item[(i)] $a$ and $b$ have no common zero on $L$$\, ;$ 
\item[(ii)] $\Delta_1$ and $p$ have no common zero on $L$$\, ;$ 
\item[(iii)] $\Delta_2$ and $q$ have no common zero on $L$. 
\end{enumerate}
\end{lemma} 

\begin{proof} 
It is helpful to notice the following relations$\, :$  
\begin{gather*}
-b(p^\prim a^2 + 2qaa^\prim) + a(p^\prim ab + q(ab^\prim + a^\prim b)) = 
aq\Delta_1\, ,\\  
-b(p^\prim ab + q(ab^\prim + a^\prim b)) + a(p^\prim b^2 + 2qbb^\prim) = 
bq\Delta_1\, . 
\qedhere 
\end{gather*} 
\end{proof}

Our method of finding a system of generators (and, actually, even a free 
resolution) for the homogeneous ideal 
$I(Z) \subset S$ is based on the following observation : one has an exact 
sequence $0 \ra \sci_L^4 \ra \sci_Z \ra \sci_Z/\sci_L^4 \ra 0$, 
$\sci_Z/\sci_L^4$ is the kernel of the epimorphism $\e_4 : 
\sci_L/\sci_L^4 \ra \sci_1$ and $\tH^1_\ast(\sci_L^4) = 0$, whence an exact 
sequence of graded $S$-modules$\, :$ 
\begin{equation}\label{E:il4izhkere4} 
0 \lra I(L)^4 \lra I(Z) \lra \tH^0_\ast(\Ker \e_4) \lra 0\, .
\end{equation} 
It follows that if one knows the structure of $\Ker \e_4$ as an 
$\sco_L$-module (i.e., its Grothendieck decomposition as a direct sum of 
invertible sheaves $\sco_L(i)$) and one can lift the generators of the 
graded $\tH^0_\ast(\sco_L) = k[x_0,x_1]$-module  
$\tH^0_\ast(\Ker \e_4)$ to elements of $I(Z)$ then one can complete 
the system of generators of $I(L)^4$ to a system of generators of $I(Z)$. 

In order to describe the kernel of $\e_4$, we take advantage of the fact that 
$\e_4$ is a morphism of filtered $\sco_L$-modules. More precisely, we 
describe, firstly, the kernels of the epimorphisms $\e_2 : \sci_L/\sci_L^2 \ra 
\sci_1/\sci_2$ and $\e_3 : \sci_L/\sci_L^3 \ra \sci_1/\sci_3$ induced by $\e_4$. 
We use the algorithmic procedure explained in the next remark. 

\begin{remark}\label{R:kerphi} 
(a) Assume that one wants to describe the kernel of an epimorphism$\, :$ 
\[
\phi = \begin{pmatrix} \phi_1 & 0\\ \phi_{21} & \phi_2 \end{pmatrix} : 
\begin{matrix} A_1\\ \oplus\\ A_2 \end{matrix} \xra{\  \  \  \  } 
\begin{matrix} B_1\\ \oplus\\ B_2 \end{matrix}
\]
(in an abelian category). $\phi_1$ is an epimorphism and $\phi_2$ factorizes 
as$\, :$ 
\[
A_2 \xra{\displaystyle \  \pi_2 \  } A_2^\secund 
\xra{\displaystyle \  \phi_2^\secund \  } B_2 
\] 
with $\pi_2$ an epimorphism and $\phi_2^\secund$ a monomorphism. Consider the 
kernels of $\phi_1$ and $\phi_2$, that is, consider exact sequences$\, :$ 
\begin{gather*}
0 \lra A_1^\prim \xra{\displaystyle \mu_1} A_1 \xra{\displaystyle \phi_1} 
B_1 \lra 0\, ,\\
0 \lra A_2^\prim \xra{\displaystyle \mu_2} A_2 \xra{\displaystyle \pi_2} 
A_2^\secund \lra 0\, . 
\end{gather*}
Assume that one can describe the kernel of the epimorphism$\, :$ 
\[
\phi^\secund = (\phi_{21} \circ \mu_1\, ,\, \phi_2^\secund) : A_1^\prim \oplus 
A_2^\secund \lra B_2\, , 
\] 
that is, assume that one has an exact sequence$\, :$ 
\[
0 \lra K^\secund \xra{\begin{pmatrix} \psi_1^\secund\\ \psi_2^\secund 
\end{pmatrix}} \begin{matrix} A_1^\prim\\ \oplus\\ A_2^\secund \end{matrix} 
\xra{\displaystyle \  \phi^\secund \  } B_2 \lra 0\, . 
\]
and that we have at our disposal a commutative diagram with exact rows$\, :$ 
\[
\begin{CD}
0 @>>> K^\prim @>{\displaystyle \sigma}>> K @>{\displaystyle \rho}>> K^\secund 
@>>> 0\\
@. @VV{\displaystyle \psi_2^\prim}V @VV{\displaystyle \psi_2}V 
@VV{\displaystyle \psi_2^\secund}V\\ 
0 @>>> A_2^\prim @>{\displaystyle \mu_2}>> A_2 @>{\displaystyle \pi_2}>> 
A_2^\secund @>>> 0 
\end{CD}
\] 
Then one gets an exact sequence$\, :$ 
\[
0 \lra K^\prim \xra{\begin{pmatrix} \sigma\\ -\psi_2^\prim \end{pmatrix}} 
\begin{matrix} K\\ \oplus\\ A_2^\prim \end{matrix} 
\xra{\begin{pmatrix} \mu_1 \circ \psi_1^\secund \circ \rho & 0\\ 
\psi_2 & \mu_2 \end{pmatrix}} \begin{matrix} A_1\\ \oplus\\ A_2 \end{matrix} 
\xra{\displaystyle \  \phi \  } \begin{matrix} B_1\\ \oplus\\ B_2 \end{matrix} 
\lra 0 
\]
which \emph{resolves the kernel of} $\phi$. 

\vskip2mm 

\noindent
\emph{Indeed}, it suffices to notice that $\mu_1 \oplus \text{id}_{A_2}$ 
maps the kernel of 
\[
\phi^\prim = (\phi_{21} \circ \mu_1 \, ,\, \phi_2) : A_1^\prim \oplus A_2 
\lra B_2 
\]
isomorphically onto the kernel of $\phi$ and to recall the fact that if one 
has a commutative diagram with exact rows$\, :$ 
\[
\begin{CD}
0 @>>> X^\prim @>{\displaystyle u}>> X @>{\displaystyle p}>> X^\secund 
@>>> 0\\ 
@. @VV{\displaystyle f^\prim}V @VV{\displaystyle f}V 
@VV{\displaystyle f^\secund}V\\ 
0 @>>> Y^\prim @>{\displaystyle v}>> Y @>{\displaystyle q}>> Y^\secund 
@>>> 0
\end{CD}
\] 
then the sequence$\, :$ 
\[
0 \lra X^\prim \xra{\begin{pmatrix} u\\ -f^\prim \end{pmatrix}} 
\begin{matrix} X\\ \oplus\\ Y^\prim \end{matrix} 
\xra{\begin{pmatrix} p & 0\\ f & v \end{pmatrix}} 
\begin{matrix} X^\secund\\ \oplus\\ Y \end{matrix} 
\xra{\displaystyle (-f^\secund ,\, q)} Y^\secund \lra 0  
\]
is exact. 

\vskip2mm 

(b) It may happen that the morphism $\psi_2^\secund : K^\secund \ra A_2^\secund$ 
defined above lifts itself to a morphism $\psi_2 : K^\secund \ra A_2$ (that is, 
$\pi_2 \circ \psi_2 = \psi_2^\secund$). Then the above exact sequence resolving 
$\Ker \phi$ becomes$\, :$ 
\[
0 \lra  \begin{matrix} K^\secund\\ \oplus\\ A_2^\prim \end{matrix} 
\xra{\begin{pmatrix} \mu_1 \circ \psi_1^\secund & 0\\ 
\psi_2 & \mu_2 \end{pmatrix}} \begin{matrix} A_1\\ \oplus\\ A_2 \end{matrix} 
\xra{\displaystyle \  \phi \  } \begin{matrix} B_1\\ \oplus\\ B_2 \end{matrix} 
\lra 0\, . 
\]  
\end{remark}   

\noindent 
{\bf Description of $\Ker \e_2$.}\quad The morphism  
$\e_2 : 2\sco_L(-1) \ra \sco_L(l)$ is defined by the matrix $(a\, ,\, b)$ and 
one has an exact sequence$\, :$ 
\begin{equation}\label{E:ab} 
0 \lra \sco_L(-l-2) \xra{\begin{pmatrix} -b\\ a \end{pmatrix}} 
2\sco_L(-1) \xra{\displaystyle (a\, ,\, b)} \sco_L(l) \lra 0 
\end{equation} 
hence $\Ker \e_2 \simeq \sco_L(-l-2)$. Let us denote by $\nu_2$ the morphism 
$(-b\, ,\, a)^{\text{t}} : \sco_L(-l-2) \ra 2\sco_L(-1)$. 

\vskip2mm 

\noindent
{\bf Description of $\Ker \e_3$.}\quad The morphism $\e_3 : 2\sco_L(-1) 
\oplus 3\sco_L(-2) \ra \sco_L(l) \oplus \sco_L(2l+m)$ is defined by the 
matrix$\, :$ 
\[
\begin{pmatrix}
a & b & 0 & 0 & 0\\ 
a^\prim & b^\prim & pa^2 & pab & pb^2 
\end{pmatrix}\, . 
\]   
We apply Remark~\ref{R:kerphi} with $A_1 = 2\sco_L(-1)$, $A_2 = 3\sco_L(-2)$, 
$B_1 = \sco_L(l)$, $B_2 = \sco_L(2l+m)$ and with $\phi = \e_3$. The kernel of 
the component $\phi_1 : 2\sco_L(-1) \ra \sco_L(l)$ of $\e_3$ is described by 
the exact sequence \eqref{E:ab} (because $\phi_1 = \e_2$), the component 
$\phi_2 : 3\sco_L(-2) \ra \sco_L(2l+m)$ of $\e_3$ decomposes as$\, :$ 
\[
3\sco_L(-2) \xra{\displaystyle (a^2 ,\, ab\, ,\, b^2)} \sco_L(2l) 
\overset{\displaystyle p}{\lra} \sco_L(2l+m) 
\]
and its kernel is described by the exact sequence$\, :$ 
\begin{equation}\label{E:a2abb2} 
0 \ra 2\sco_L(-l-3) 
\xra{\begin{pmatrix} 
-b & 0\\
a & -b\\
0 & a 
\end{pmatrix}} 
3\sco_L(-2) 
\xra{\displaystyle (a^2 ,\, ab\, ,\, b^2)} 
\sco_L(2l) \ra 0\, . 
\end{equation}
The morphism $\e_3^\secund : \sco_L(-l-2) \oplus \sco_L(2l) \ra \sco_L(2l+m)$ 
associated to $\e_3$ as in Remark~\ref{R:kerphi} (where it is denoted by 
$\phi^\secund$) is defined by the matrix $(\Delta_1\, ,\, p)$ and its kernel is 
described by the exact sequence$\, :$ 
\[
0 \ra \sco_L(-l-m-2) \xra{\begin{pmatrix} p\\ -\Delta_1 \end{pmatrix}} 
\sco_L(-l-2) \oplus \sco_L(2l) \xra{\displaystyle (\Delta_1\, ,\, p)} 
\sco_L(2l+m) \ra 0\, . 
\]
It follows, from the exact sequence \eqref{E:a2abb2}, 
that, for $t \geq 3l + 2$, any element of $\tH^0(\sco_L(t))$ can be expressed 
as a combination of $a^2,\, ab,\, b^2$.  
Since $\Delta_1 \in \tH^0(\sco_L(3l+m+2))$ one deduces that there exist 
polynomials $v_0,\, v_1,\, v_2 \in \tH^0(\sco_L(l+m))$ such that$\, :$  
\begin{equation}\label{E:v012} 
-\Delta_1 = v_0a^2 + v_1ab + v_2b^2 \, .
\end{equation}
It follows that the morphism $\psi_2^\secund : \sco_L(-l-m-2) \ra \sco_L(2l)$ 
defined by $-\Delta_1$ lifts to the morphism $\psi_2 : \sco_L(-l-m-2) \ra 
3\sco_L(-2)$ defined by $(v_0\, ,\, v_1\, ,\, v_2)^{\text{t}}$. Applying, now, 
Remark~\ref{R:kerphi}(b) one gets an exact sequence$\, :$ 
\begin{equation}\label{E:e3nu3}
0 \lra \begin{matrix} \sco_L(-l-m-2)\\ \oplus\\ 2\sco_L(-l-3) \end{matrix}  
\overset{\displaystyle \nu_3}{\lra}  
\begin{matrix} 2\sco_L(-1)\\ \oplus\\ 3\sco_L(-2) \end{matrix}  
\overset{\displaystyle \e_3}{\lra}  
\begin{matrix} \sco_L(l)\\ \oplus\\ \sco_L(2l+m) \end{matrix} \lra 0 
\end{equation}
with $\nu_3$ defined by the matrix$\, :$ 
\[
\begin{pmatrix} 
-pb & 0 & 0\\
pa & 0 & 0\\ 
v_0 & -b & 0\\
v_1 & a & -b\\
v_2 & 0 & a
\end{pmatrix}\, . 
\]

\noindent 
{\bf Description of $\Ker \e_4$.}\quad 
Recall the definition of $\e_4$ from Lemma~\ref{L:epicondition}. We apply 
Remark~\ref{R:kerphi} with $A_1 = 2\sco_L(-1) \oplus 3\sco_L(-2)$, $A_2 = 
4\sco_L(-3)$, $B_1 = \sco_L(l) \oplus \sco_L(2l+m)$, $B_2 = \sco_L(3l+m+n)$ 
and with $\phi = \e_4$. The kernel of the component $\phi_1 : \sco_L(-2) 
\oplus 3\sco_L(-3) \ra \sco_L(l) \oplus \sco_L(2l+m)$ of $\e_4$ is described 
by the exact sequence \eqref{E:e3nu3} (because $\phi_1 = \e_3$), the component 
$\phi_2 : 4\sco_L(-3) \ra \sco_L(3l+m+n)$ of $\e_4$ decomposes as$\, :$ 
\[
4\sco_L(-3) \xra{\displaystyle (a^3,\, a^2b\, ,\, ab^2,\, b^3)} 
\sco_L(3l) \overset{\displaystyle pq}{\lra} \sco_L(3l+m+n) 
\]
and its kernel is described by the exact sequence$\, :$ 
\begin{equation}\label{E:a3a2bab2b3} 
0 \ra 3\sco_L(-l-4) 
\xra{\begin{pmatrix} 
-b & 0 & 0\\
a & -b & 0\\
0 & a & -b\\
0 & 0 & a
\end{pmatrix}} 
4\sco_L(-3) 
\xra{\displaystyle (a^3,\,  a^2b\, ,\, ab^2,\, b^3)} 
\sco_L(3l) \ra 0\, .
\end{equation}
Now, the epimorphism $\e_4^\secund : \sco_L(-l-m-2) \oplus 2\sco_L(-l-3) \oplus 
\sco_L(3l) \ra \sco_L(3l+m+n)$ associated to $\e_4$ as in 
Remark~\ref{R:kerphi} (where it is denoted by $\phi^\secund$) is defined by the 
matrix $(-\Delta_2 + qv\, ,\, aq\Delta_1\, ,\, bq\Delta_1\, ,\, pq)$, 
where$\, :$ 
\begin{equation}\label{E:defv} 
v := 2v_0aa^\prim + v_1(ab^\prim + a^\prim b) + 2v_2bb^\prim \, . 
\end{equation} 

At this point, we have to provide the 

\vskip2mm 

\noindent 
$\bullet$\quad \emph{Description of} $\Ker \e_4^\secund$.\quad 
We apply Remark~\ref{R:kerphi} with $A_1 = \sco_L(-l-m-2)$, $A_2 = 
2\sco_L(-l-3) \oplus \sco_L(3l)$, $B_1 = 0$ and $B_2 = \sco_L(3l+m+n)$. 
The component $2\sco_L(-l-3) \oplus \sco_L(3l) \ra \sco_L(3l+m+n)$ of 
$\e_4^\secund$ factorizes as$\, :$ 
\[
2\sco_L(-l-3) \oplus \sco_L(3l) \xra{\displaystyle \pi_2^\secund} \sco_L(3l+m) 
\overset{\displaystyle q}{\lra} \sco_L(3l+m+n) 
\]  
where $\pi_2^\secund$ is defined by the matrix $(a\Delta_1\, ,\, b\Delta_1\, ,\, 
p)$. One has a commutative diagram$\, :$ 
\[
\begin{CD} 
2\sco_L(-l-3) \oplus \sco_L(-m-2) @>{\displaystyle (a,b,p)}>> \sco_L(-2)\\
@V{\displaystyle \text{id} \oplus \text{id} \oplus \Delta_1}VV 
@VV{\displaystyle \Delta_1}V\\
2\sco_L(-l-3) \oplus \sco_L(3l) @>{\displaystyle \pi_2^\secund}>> \sco_L(3l+m) 
\end{CD}
\]
Since $p$ and $\Delta_1$ are coprime, it follows that $\text{id} \oplus 
\text{id} \oplus \Delta_1$ maps isomorphically the kernel of 
$(a, b, p) : 2\sco_L(-l-3) \oplus \sco_L(-m-2) \ra \sco_L(-2)$ onto 
$\Ker \pi_2^\secund$. One needs, now, the following easy$\, :$ 

\begin{lemma}\label{L:a1a2a3}
Let $K$ be the kernel of an epimorphism: 
\[
(a_1\, ,\, a_2\, ,\, a_3) : \sco_L(-i_1) \oplus \sco_L(-i_2) \oplus \sco_L(-i_3) 
\lra \sco_L\, .
\]
Assume that $i_1 \leq i_2$ and that $a_3 \neq 0$. Then $K \simeq \sco_L(-j_1) 
\oplus \sco_L(-j_2)$ with $i_1 \leq j_1 \leq j_2 \leq i_2 + i_3$. 
In particular, ${\fam0 H}^1(K(t)) = 0$ for $t \geq i_2 + i_3 - 1$.  
\end{lemma} 

\begin{proof}
$K$ is locally free of rank 2 hence one can write $K \simeq \sco_L(-j_1) 
\oplus \sco_L(-j_2)$ with $j_1 \leq j_2$. It follows from our hypothesis that 
$\tH^0(K(-i_1+1)) = 0$, hence $j_1 \geq i_1$. On the other hand, $c_1(K) = 
-i_1-i_2-i_3$ hence $j_2 = i_1+i_2+i_3-j_1 \leq i_2+i_3$. 
\end{proof}

According to this lemma, one has an exact sequence$\, :$ 
\begin{equation}\label{E:syzabp}
0 \ra \sco_L(-l_1) \oplus \sco_L(-l_2) 
\xra{\begin{pmatrix}
f_1 & f_2\\
g_1 & g_2\\
u_1 & u_2
\end{pmatrix}} 
2\sco_L(-l-1) \oplus \sco_L(-m) 
\xra{\displaystyle (a\, ,\, b\, ,\, p)} 
\sco_L \ra 0\, ,
\end{equation}
with $l+1 \leq l_1 \leq l_2 \leq l+m+1$ and $l_1 + l_2 = 2l + m +2$. One 
derives an exact sequence$\, :$ 
\begin{equation}\label{E:kerpi2secund}
0 \lra 
\begin{matrix}
\sco_L(-l_1-2)\\ \oplus\\ \sco_L(-l_2-2 ) 
\end{matrix}
\xra{\begin{pmatrix}
f_1 & f_2\\
g_1 & g_2\\
u_1\Delta_1 & u_2\Delta_1
\end{pmatrix}} 
\begin{matrix}
2\sco_L(-l-3)\\ \oplus\\ \sco_L(3l) 
\end{matrix}
\overset{\displaystyle \pi_2^\secund}{\lra}   
\sco_L(3l+m) \lra 0\, ,
\end{equation} 
which describes the kernel of the component $2\sco_L(-l-3) \oplus \sco_L(3l) 
\ra \sco_L(3l+m+n)$ of $\e_4^\secund$. 

Now, the kernel of the epimorphism $(-\Delta_2 + qv\, ,\, q) : 
\sco_L(-l-m-2) \oplus \sco_L(3l+m) \ra \sco_L(3l+m+n)$ associated to 
$\e_4^\secund$ as in Remark~\ref{R:kerphi} (where it is denoted by 
$\phi^\secund$) is described by the exact sequence$\, :$ 
\[
0 \ra \sco_L(-l-m-n-2) 
\xra{\begin{pmatrix} q\\ \Delta_2 - qv \end{pmatrix}} 
\begin{matrix}
\sco_L(-l-m-2)\\ \oplus\\ \sco_L(3l+m) 
\end{matrix}
\xra{\displaystyle (-\Delta_2 + qv\, ,\, q)} 
\sco_L(3l+m+n) \ra 0\, .  
\]
Since $\Delta_2 -qv \in \tH^0(\sco_L(4l+2m+n+2))$ and since 
\[
- l_i -2 + (l + m + n + 2) \geq -(l + m + 1) - 2 + (l + m + n + 2) = 
n - 1 \geq -1\, ,\  i = 1,\, 2\, ,    
\]
the exact sequence \eqref{E:kerpi2secund} shows that there exist polynomials 
$f,\, g \in \tH^0(\sco_L(m+n-1))$ and $w \in \tH^0(\sco_L(4l+m+n+2))$ such 
that$\, :$ 
\begin{equation}\label{E:fgw} 
\Delta_2 - qv = fa\Delta_1 + gb\Delta_1 + wp 
\end{equation}  
($v$ has been defined in \eqref{E:defv}). One gets that the morphism 
$\sco_L(-l-m-n-2) \ra \sco_L(3l+m)$ defined by $\Delta_2 - qv$ (corresponding 
to $\psi_2^\secund$ in Remark~\ref{R:kerphi}) lifts to the 
morphism $\sco_L(-l-m-n-2) \ra 2\sco_L(-l-3) \oplus \sco_L(3l)$ defined by the 
matrix $(f\, ,\, g\, ,\, w)^{\text{t}}$ (corresponding to $\psi_2$ in 
Remark~\ref{R:kerphi}(b)). Remark~\ref{R:kerphi}(b) implies, now, that one has 
an exact sequence$\, :$ 
\begin{equation}\label{E:e4secundnu4secund}
0 \lra 
\begin{matrix}
\sco_L(-l-m-n-2)\\ \oplus\\ \sco_L(-l_1-2)\\ \oplus\\ \sco_L(-l_2-2) 
\end{matrix}
\overset{\displaystyle \nu_4^\secund}{\lra}  
\begin{matrix}
\sco_L(-l-m-2)\\ \oplus\\ 2\sco_L(-l-3)\\ \oplus\\ \sco_L(3l) 
\end{matrix}
\overset{\displaystyle \e_4^\secund}{\lra}  
\sco_L(3l+m+n) \lra 0 
\end{equation}
with $\nu_4^\secund$ defined by the matrix$\, :$ 
\[
\begin{pmatrix} 
q & 0 & 0\\
f & f_1 & f_2\\
g & g_1 & g_2\\
w & u_1\Delta_1 & u_2\Delta_1
\end{pmatrix}\, . 
\] 
\emph{The description of} $\Ker \e_4^\secund$ \emph{is complete}. 

\vskip2mm 

Now, coming back to the description of $\Ker \e_4$, the component$\, :$
\[
\psi_2^\secund : \sco_L(-l-m-n-2) \oplus \sco_L(-l_1 - 2) \oplus 
\sco_L(-l_2 - 2) \lra \sco_L(3l) 
\]
of the morphism $\nu_4^\secund$ from \eqref{E:e4secundnu4secund} 
is defined by the matrix $(w\, ,\, u_1\Delta_1\, ,\, u_2\Delta_1)$. 
We would like, in order to apply the most favourable case (b) of 
Remark~\ref{R:kerphi}, to lift $\psi_2^\secund$ to a morphism$\, :$ 
\[
\psi_2 : \sco_L(-l-m-n-2) \oplus \sco_L(-l_1 - 2) \oplus 
\sco_L(-l_2 - 2) \lra 4\sco_L(-3)\, ,
\] 
that is, to express $w$, $u_1\Delta_1$, $u_2\Delta_1$ as combinations of 
$a^3, \ldots , b^3$. It follows, from the exact sequence \eqref{E:a3a2bab2b3}, 
that, for $t \geq 4l + 3$, any element of $\tH^0(\sco_L(t))$ can be expressed 
as a combination of $a^3, \ldots , b^3$. 

Since $u_i \in \tH^0(\sco_L(l_i - m))$, $\Delta_1 \in 
\tH^0(\sco_L(3l + m + 2))$ and $(l_i - m) + (3l + m + 2) \geq 
(l + 1 - m) + (3l +m + 2) = 4l + 3$, $i = 1,\, 2$, one derives the existence 
of polynomials $v_{ij} \in \tH^0(\sco_L(l_i - 1))$, $j = 0, \ldots , 3$, such 
that$\, :$  
\begin{equation}\label{E:vij} 
u_i\Delta_1 = v_{i0}a^3 + v_{i1}a^2b + v_{i2}ab^2 + v_{i3}b^3\, ,\  
i = 1,\, 2\, .
\end{equation}
On the other hand, it might happen that $w$, which belongs to 
$\tH^0(\sco_L(4l + m + n + 2))$, cannot be expressed as a combination of 
$a^3, \ldots , b^3$. Consequently, one has to consider two cases. The case 
where $w$ can be expressed as such a combination is dealt with in 
Lemma~\ref{L:kere4} below. For the opposite case, see 
Remark~\ref{R:wcombination} and the proof of Prop.~\ref{P:genizprim} below. 

\begin{lemma}\label{L:kere4} 
If the polynomial $w \in {\fam0 H}^0(\sco_L(4l+m+n+2))$ $($defined in 
\eqref{E:fgw}$)$ can be written as a combination$\, :$ 
\[
w = w_0a^3 + w_1a^2b + w_2ab^2 + w_3b^3 
\]
with $w_0, \ldots ,w_3 \in {\fam0 H}^0(\sco_L(l+m+n-1))$ then one has an exact 
sequence$\, :$ 
\[
0 \lra 
\begin{matrix}
\sco_L(-l-m-n-2)\\ \oplus\\ \sco_L(-l_1-2)\\ \oplus\\ \sco_L(-l_2-2)\\ 
\oplus\\ 3\sco_L(-l-4) 
\end{matrix}
\overset{\displaystyle \nu_4}{\lra}   
\begin{matrix}
2\sco_L(-1)\\ \oplus\\ 3\sco_L(-2)\\ \oplus\\ 4\sco_L(-3) 
\end{matrix}
\overset{\displaystyle \e_4}{\lra}   
\begin{matrix} 
\sco_L(l)\\ \oplus\\ \sco_L(2l+m)\\ \oplus\\ \sco_L(3l+m) 
\end{matrix}
\lra 0 
\] 
with $\nu_4$ defined by the matrix$\, :$ 
\[
\begin{pmatrix}
-pbq & 0 & 0 & 0 & 0 & 0\\
paq & 0 & 0 & 0 & 0 & 0\\
v_0q - bf & -bf_1 & -bf_2 & 0 & 0 & 0\\
v_1q + af - bg & af_1 - bg_1 & af_2 - bg_2 & 0 & 0 & 0\\
v_2q + ag & ag_1 & ag_2 & 0 & 0 & 0\\ 
w_0 & v_{10} & v_{20} & -b & 0 & 0\\
w_1 & v_{11} & v_{21} & a & -b & 0\\
w_2 & v_{12} & v_{22} & 0 & a & -b\\
w_3 & v_{13} & v_{23} & 0 & 0 & a  
\end{pmatrix}\, . 
\]
\end{lemma} 

\begin{proof} 
The component $\psi_2^\secund : \sco_L(-l-m-n-2) \oplus \sco_L(-l_1 - 2) \oplus 
\sco_L(-l_2 - 2) \ra \sco_L(3l)$ of the morphism $\nu_4^\secund$ from the exact 
sequence \eqref{E:e4secundnu4secund} lifts to the morphism$\, :$  
\[
\psi_2 : \sco_L(-l-m-n-2) \oplus \sco_L(-l_1 - 2) \oplus \sco_L(-l_2 - 2) \lra 
4\sco_L(-3)
\]
defined by the matrix$\, :$ 
\[
\begin{pmatrix} 
w_0 & v_{10} & v_{20}\\ 
w_1 & v_{11} & v_{21}\\ 
w_2 & v_{12} & v_{22}\\ 
w_3 & v_{13} & v_{23}
\end{pmatrix} 
\, . 
\] 
One can apply, now, Remark~\ref{R:kerphi}(b). 
\end{proof}

\begin{remark}\label{R:wcombination} 
If $m \neq 0$ or if $n \neq 0$ then $4l + m + n + 2 \geq 4l + 3$ hence, 
according to the discussion before Lemma~\ref{L:kere4}, $w$ can be written as 
a combination of $a^3,\, a^2b,\, ab^2, b^3$. This is also true when 
$m = n = 0$ and $l = -1$ because, in this case, $w \in \tH^0(\sco_L(-2)) 
= 0$.   
\end{remark} 

\begin{prop}\label{P:geniz}   
Assume that the polynomial $w$ $($defined in \eqref{E:fgw}$)$ can be written 
as a combination of $a^3, \ldots ,b^3$ as in Lemma~\ref{L:kere4} $($which 
happens automatically if $m \neq 0$ or if $n \neq 0$ or if $m = n = 0$ and 
$l = -1$ according to Remark~\ref{R:wcombination}$)$. 
Consider the polynomials $($in $S$$)$$\, :$ 
\[
F_2 = \begin{vmatrix} a & b\\ x & y\end{vmatrix}\, ,\  
F_3 = pF_2 + v_0x^2 + v_1xy + v_2y^2\, .
\]
with $v_0,\, v_1,\, v_2$ defined in \eqref{E:v012}.
Then the homogeneous ideal $I(Z) 
\subset S$ of $Z$ is generated by the following polynomials$\; :$ 
\begin{gather*} 
F_4 = qF_3 + (fx+gy)F_2 + w_0x^3 + w_1x^2y + w_2xy^2 + w_3y^3\, ,\\
G_1 = (f_1x+g_1y)F_2 + v_{10}x^3 + v_{11}x^2y + v_{12}xy^2 + v_{13}y^3\, ,\\
G_2 = (f_2x+g_2y)F_2 + v_{20}x^3 + v_{21}x^2y + v_{22}xy^2 + v_{23}y^3\, ,\\
x^2F_2,\, ,\  xyF_2\, ,\  y^2F_2\, ,\  x^4\, ,\  x^3y\, ,\  x^2y^2\, ,\  
xy^3\, ,\  y^4\, .
\end{gather*}
with $f,\, g$ defined in \eqref{E:fgw}, the $f_i$'s and the $g_i$'s defined in 
\eqref{E:syzabp} and the $v_{ij}$'s defined in \eqref{E:vij}.    
\end{prop} 

\noindent 
Notice that ${\fam0 deg}\, F_2 = l+2$, ${\fam0 deg}\, F_3 = l + m + 2$, 
${\fam0 deg}\, F_4 = l + m + n + 2$, ${\fam0 deg}\, G_1 = l_1 + 2$ and 
${\fam0 deg}\, G_2 = l_2 + 2$, with $l + 1 \leq l_1 \leq l_2 \leq l + m + 1$ 
and $l_1 + l_2 = 2l + m + 2$. 

The case where $w$ cannot be written as a combination of $a^3,\ldots ,b^3$ 
will be analysed in Prop.~\ref{P:genizprim} below. 

\begin{proof} 
One uses Lemma~\ref{L:kere4}, the exact sequence \eqref{E:il4izhkere4} and the 
fact that$\, :$ 
\[
\sci_L/\sci_L^4 = \sco_L(-1)x \oplus \sco_L(-1)y \oplus \sco_L(-2)x^2 
\oplus \cdots \oplus \sco_L(-3)y^3\, . 
\qedhere
\]
\end{proof}

\begin{remark}\label{R:vivij} 
We want to emphasize two relations between the polynomials 
$v_0,\, v_1,\, v_2 \in \tH^0(\sco_L(l+m))$ defined by relation 
\eqref{E:v012} and the polynomials $v_{ij} \in \tH^0(\sco_L(l_i-1))$, 
$i = 1,\, 2$, $0\leq j\leq 3$, defined by relation \eqref{E:vij}. 
They are needed if one wants to write down a complete system of relations 
between the generators of $I(Z)$ from Prop.~\ref{P:geniz}.  

Firstly, we notice that, by considering the Eagon-Northcott complex associated 
to the matrix appearing on the left side of the exact sequence 
\eqref{E:syzabp} as in the proof of the Hilbert-Burch theorem, one may assume 
that$\, :$ 
\begin{equation}\label{E:abpasdet} 
a = \begin{vmatrix} g_1 & g_2\\ u_1 & u_2 \end{vmatrix}\, ,\  
b = -\begin{vmatrix} f_1 & f_2\\ u_1 & u_2 \end{vmatrix}\, ,\  
p = \begin{vmatrix} f_1 & f_2\\ g_1 & g_2 \end{vmatrix}\, .
\end{equation}
Now, since $a = -g_2u_1+g_1u_2$ then, multiplying relation \eqref{E:v012} 
by $a$ and the two relations in \eqref{E:vij} by $-g_2$ and $g_1$, 
respectively, one gets$\, :$ 
\[
a(v_0a^2+v_1ab+v_2b^2) - g_2(v_{10}a^3+\ldots +v_{13}b^3) + 
g_1(v_{20}a^3+\ldots +v_{23}b^3) = 0\, .
\]
Using the exact sequence \eqref{E:a3a2bab2b3}, one derives the 
existence of polynomials $\alpha_0,\, \alpha_1,\, \alpha_2 \in 
\tH^0(\sco_L(m-1))$ such that$\, :$ 
\begin{equation}\label{E:alpha} 
\begin{pmatrix} v_0\\ v_1\\ v_2\\ 0 \end{pmatrix} -g_2
\begin{pmatrix} v_{10}\\ v_{11}\\ v_{12}\\ v_{13} \end{pmatrix} +g_1
\begin{pmatrix} v_{20}\\ v_{21}\\ v_{22}\\ v_{23} \end{pmatrix} = 
\begin{pmatrix} 
-b & 0 & 0\\
a & -b & 0\\
0 & a & -b\\
0 & 0 & a
\end{pmatrix} 
\begin{pmatrix} \alpha_0\\ \alpha_1\\ \alpha_2 \end{pmatrix}\, .    
\end{equation} 
Similarly, using the relation $b = f_2u_1 - f_1u_2$, one derives the existence 
of polynomials $\beta_0,\, \beta_1,\, \beta_2 \in \tH^0(\sco_L(m-1))$ such 
that$\, :$ 
\begin{equation}\label{E:beta} 
\begin{pmatrix} 0\\ v_0\\ v_1\\ v_2 \end{pmatrix} +f_2
\begin{pmatrix} v_{10}\\ v_{11}\\ v_{12}\\ v_{13} \end{pmatrix} -f_1
\begin{pmatrix} v_{20}\\ v_{21}\\ v_{22}\\ v_{23} \end{pmatrix} = 
\begin{pmatrix} 
-b & 0 & 0\\
a & -b & 0\\
0 & a & -b\\
0 & 0 & a
\end{pmatrix} 
\begin{pmatrix} \beta_0\\ \beta_1\\ \beta_2 \end{pmatrix}\, .    
\end{equation} 

Now, using the determinantal expression of $p$ from \eqref{E:abpasdet}, one 
gets the relations$\, :$ 
\begin{gather*}
xp - g_2(f_1x+g_1y) + g_1(f_2x+g_2y) = 0\, ,\\
yp + f_2(f_1x+g_1y) - f_1(f_2x+g_2y) = 0\, .
\end{gather*} 
On the other hand, multiplying to the left the matrix relation 
\eqref{E:alpha} by $(x^3,\, x^2y,\, xy^2,\, y^3)$ one obtains$\, :$ 
\begin{gather*} 
x(v_0x^2 + v_1xy + v_2y^2) - g_2(v_{10}x^3 + \dots + v_{13}y^3) + 
g_1(v_{20}x^3 + \dots + v_{23}y^3) =\\
(\alpha_0x^2 + \alpha_1xy + \alpha_2y^2)(-bx+ay)\, .
\end{gather*} 
One deduces, similarly, from the matrix relation \eqref{E:beta}, the following 
polynomial relation$\, :$ 
\begin{gather*} 
y(v_0x^2 + v_1xy + v_2y^2) + f_2(v_{10}x^3 + \dots + v_{13}y^3) - 
f_1(v_{20}x^3 + \dots + v_{23}y^3) =\\
(\beta_0x^2 + \beta_1xy + \beta_2y^2)(-bx+ay)\, .
\end{gather*}
One thus obtains the following relations$\, :$ 
\begin{gather*} 
xF_3 - g_2G_1 + g_1G_2 - \alpha_0x^2F_2 -\alpha_1xyF_2 -\alpha_2y^2F_2 = 0\, ,\\
yF_3 + f_2G_1 - f_1G_2 - \beta_0x^2F_2 -\beta_1xyF_2 - \beta_2y^2F_2 = 0\, .  
\end{gather*}     
\end{remark} 

\noindent 
{\bf A graded free resolution of} $I(Z)$ {\bf under the hypothesis of 
Prop.~\ref{P:geniz}.}\quad  
One has a filtration by homogeneous ideals$\, :$ 
\[
I(Z) \supset J_2 \supset J_3 \supset I(L)^4 \supset (0) 
\]
where $J_2$ is the ideal generated by $G_1,\, G_2,\, x^2F_2,\, xyF_2,\, y^2F_2,
\, x^4, \ldots ,\, y^4$ and $J_3$ is the ideal generated by $x^2F_2,\, xyF_2,\, 
y^2F_2,\, x^4, \ldots ,\, y^4$. Notice that$\, :$ 
\[
(x,\, y)I(Z) \subseteq J_2,\  (x,\, y)J_2 \subseteq J_3,\  
(x,\, y)J_3 \subseteq I(L)^4 \, . 
\] 
If $\nu_4$ is the morphism from Lemma~\ref{L:kere4} then $\tH^0_\ast(\nu_4)$ 
induces isomorphisms$\, :$   
\begin{gather*}
\tH^0_\ast(\sco_L(-l_1-2) \oplus \sco_L(-l_2-2) \oplus 3\sco_L(-l-4)) \Izo 
J_2/I(L)^4\, ,\\
\tH^0_\ast(3\sco_L(-l-4)) \Izo J_3/I(L)^4\, . 
\end{gather*}
One deduces isomorphisms of $S$-modules$\, :$ 
\begin{gather*}
I(Z)/J_2 \simeq S(L)(-l-m-n-2),\  J_2/J_3 \simeq S(L)(-l_1-2) \oplus 
S(L)(-l_2-2),\\  
J_3/I(L)^4 \simeq 3S(L)(-l-4)\, , 
\end{gather*}
where $S(L) := S/I(L)$. 
Now, using the well-known resolutions$\, :$ 
\begin{gather*} 
0 \lra S(-2) \xra{\begin{pmatrix} -y\\ x\end{pmatrix}} 2S(-1) 
\xra{\displaystyle (x\, ,\, y)} S \lra S(L) \lra 0\\
0 \ra 4S(-5) \xra{\begin{pmatrix} 
                  -y & 0 & 0 & 0\\ 
                   x & -y & 0 & 0\\
                   0 & x & -y & 0\\
                   0 & 0 & x & -y\\
                   0 & 0 & 0 & x
                   \end{pmatrix}} 
5S(-4) \xra{\displaystyle (x^4,\, x^3y\, ,\, x^2y^2,\, xy^3,\, y^4)} 
I(L)^4 \ra 0
\end{gather*} 
one deduces that $I(Z)$ has a (not necessarily minimal) graded free resolution 
of the form$\, :$ 
\[
0 \lra \begin{matrix} S(\text{--}l\text{--}m\text{--}n\text{--}4)\\ 
\oplus\\ S(-l_1-4)\\ \oplus\\ 
S(-l_2-4)\\ \oplus\\ 3S(-l-6)\end{matrix} 
\overset{\displaystyle d_2}{\lra} 
\begin{matrix} 2S(\text{--}l\text{--}m\text{--}n\text{--}3)\\ 
\oplus\\ 2S(-l_1-3)\\ \oplus\\ 2S(-l_2-3)\\ 
\oplus\\ 6S(-l-5)\\ \oplus\\ 4S(-5)\end{matrix} 
\overset{\displaystyle d_1}{\lra}  
\begin{matrix} S(\text{--}l\text{--}m\text{--}n\text{--}2)\\ 
\oplus\\ S(-l_1-2)\\ \oplus\\ 
S(-l_2-2)\\ \oplus\\ 3S(-l-4)\\ \oplus\\ 5S(-4)\end{matrix} 
\overset{\displaystyle d_0}{\lra} I(Z) \lra 0  
\] 
with $d_0$ defined by the generators of $I(Z)$ enumerated in the statement 
of Prop.~\ref{P:geniz} and with the linear parts of $d_1$ and $d_2$ deduced 
from the above resolutions of $S/I(L)$ and $I(L)^4$. The rest of the matrices 
defining $d_1$ and $d_2$ can be easily guessed. For example, $d_1$ is defined 
by the matrix of relations between the generators of $I(Z)$.  
The module of these relations is generated by the  
relations of the following form: one multiplies each generator of $I(Z)$ 
by $x$ and by $y$ and one expresses the results as combinations of the next 
generators. One gets the following matrix for $d_1$$\, :$ 
\[
\setcounter{MaxMatrixCols}{16} 
\begin{pmatrix} 
x & y & 0 & 0 & 0 & 0 & 0 & 0 & 0 & 0 & 0 & 0 & 0 & 0 & 0 & 0\\ 
\text{--}qg_2 & qf_2 & x & y & 0 & 0 & 0 & 0 & 0 & 0 & 0 & 0 & 0 & 0 & 0 & 0\\ 
qg_1 & \text{--}qf_1 & 0 & 0 & x & y & 0 & 0 & 0 & 0 & 0 & 0 & 0 & 0 & 0 & 0\\ 
\text{--}f\text{--}q\alpha_0 & \text{--}q\beta_0 & \text{--}f_1 & 0 & 
\text{--}f_2 & 0 & x & y & 0 & 0 & 0 & 0 & 0 & 0 & 0 & 0\\ 
\text{--}g\text{--}q\alpha_1 & \text{--}f\text{--}q\beta_1 & \text{--}g_1 & 
\text{--}f_1 & \text{--}g_2 & \text{--}f_2 & 0 & 0 & x & y &  0 & 0 & 
0 & 0 & 0 & 0\\ 
\text{--}q\alpha_2 & \text{--}g\text{--}q\beta_2 & 0 & \text{--}g_1 & 0 & 
\text{--}g_2 & 0 & 0 & 0 & 0 & x & y & 0 & 0 & 0 & 0\\ 
\text{--}w_0 & 0 & \text{--}v_{10} & 0 & v_{20} & 0 & b & 0 & 0 & 0 & 0 & 0 & 
\text{--}y & 0 & 0 & 0\\ 
\text{--}w_1 & \text{--}w_0 & \text{--}v_{11} & \text{--}v_{10} & 
\text{--}v_{21} & \text{--}v_{20} & \text{--}a & b & b & 0 & 0 & 0 & x & 
\text{--}y & 0 & 0\\
\text{--}w_2 & \text{--}w_1 & \text{--}v_{12} & \text{--}v_{11} & 
\text{--}v_{22} & \text{--}v_{21} & 0 & \text{--}a & \text{--}a & b & b & 0 & 
0 & x & \text{--}y & 0\\
\text{--}w_3 & \text{--}w_2 & \text{--}v_{13} & \text{--}v_{12} & 
\text{--}v_{23} & \text{--}v_{22} & 0 & 0 & 0 & \text{--}a & \text{--}a & 
b & 0 & 0 & x & \text{--}y\\
0 & \text{--}w_3 & 0 & \text{--}v_{13} & 0 & \text{--}v_{23} & 0 & 0 & 0 & 0 & 
0 & \text{--}a & 0 & 0 & 0 & x
\end{pmatrix}
\] 
Notice that the first two columns of this matrix correspond to the relations 
between the generators of $I(Z)$ from Prop.~\ref{P:geniz} deduced after 
multiplying by $q$ the relations at the end of Remark~\ref{R:vivij}. 

The matrix of $d_2$ is the matrix of relations between the columns of the 
matrix of $d_1$. It looks like this$\, :$ 
\[
\begin{pmatrix}
-y & 0 & 0 & 0 & 0 & 0\\
x & 0 & 0 & 0 & 0 & 0\\
-qf_2 & -y & 0 & 0 & 0 & 0\\ 
-qg_2 & x & 0 & 0 & 0 & 0\\
qf_1 & 0 & -y & 0 & 0 & 0 \\
qg_1 & 0 & x & 0 & 0 & 0\\
q\beta_0 & 0 & 0 & -y & 0 & 0\\
-f-q\alpha_0 & -f_1 & -f_2 & x & 0 & 0\\
f+q\beta_1 & f_1 & f_2 & 0 & -y & 0\\
-g-q\alpha_1 & -g_1 & -g_2 & 0 & x & 0\\
g+q\beta_2 & g_1 & g_2 & 0 & 0 & -y\\
-q\alpha_2 & 0 & 0 & 0 & 0 & x\\
w_0 & v_{10} & v_{20} & -b & 0 & 0\\
w_1 & v_{11} & v_{21} & a & -b & 0\\
w_2 & v_{12} & v_{22} & 0 & a & -b\\
w_3 & v_{13} & v_{23} & 0 & 0 & a
\end{pmatrix}
\]  
Again, one needs to motivate only the fact that the first column of this 
matrix is a relation between the columns of the matrix of $d_1$. For example, 
if one multiplies the 9\emph{th} row of the matrix of $d_1$ against the 
first column of the matrix of $d_2$ one gets$\, :$ 
\[
qf_2v_{12} + qg_2v_{11} - qf_1v_{22} - qg_1v_{21} + qa\alpha_0 - qa\beta_1 
- qb\alpha_1 + qb\beta_2\, . 
\tag{$\ast$}
\] 
But identifying the coefficient of $x^2y$ in the polynomial relation deduced 
in Remark~\ref{R:vivij} from the matrix relation \eqref{E:alpha} and the 
coefficient of $xy^2$ in the polynomial relation deduced in Remark~\ref{R:vivij}
from \eqref{E:beta} one obtains$\, :$ 
\begin{gather*} 
v_1 - g_2v_{11} + g_1v_{21} = -b\alpha_1 + a\alpha_0\\ 
v_1 + f_2v_{12} - f_1v_{22} = a\beta_0 - b\beta_2 
\end{gather*} 
from which one deduces that $(\ast) = 0$. 

\begin{prop}\label{P:reshoz} 
No matter whether the polynomial $w$ $($defined in \eqref{E:fgw}$)$ is a 
combination of $a^3, \dots , b^3$ as in Lemma~\ref{L:kere4} or not, the 
graded $S$-module ${\fam0 H}^0_\ast(\sco_Z)$ has a graded free resolution of the 
form$\, :$ 
\[
0 \lra 
\begin{matrix} 
S(-2)\\ \oplus\\ S(l-2)\\ \oplus\\ S(2l+m-2)\\ \oplus\\ S(3l+m+n\text{--}2)  
\end{matrix}
\overset{\displaystyle \delta_2}{\lra} 
\begin{matrix} 
2S(-1)\\ \oplus\\ 2S(l-1)\\ \oplus\\ 2S(2l+m-1)\\ \oplus\\ 
2S(3l+m+n\text{--}1) 
\end{matrix}
\overset{\displaystyle \delta_1}{\lra} 
\begin{matrix} 
S\\ \oplus\\ S(l)\\ \oplus\\ S(2l+m)\\ \oplus\\ S(3l+m+n)
\end{matrix}
\overset{\displaystyle \delta_0}{\lra} 
{\fam0 H}^0_\ast(\sco_Z) \lra 0 
\] 
with $\delta_0$ defined by the generators $1,\, e_1,\, e_2,\, e_3$ of 
${\fam0 H}^0_\ast(\sco_Z)$ considered at the beginning of this subsection and 
with $\delta_1$, $\delta_2$ defined by the matrices$\, :$ 
\[
\begin{pmatrix} 
x & y & 0 & 0 & 0 & 0 & 0 & 0\\
\text{--}a & \text{--}b & x & y & 0 & 0 & 0 & 0\\
\text{--}a^\prim & \text{--}b^\prim & \text{--}pa & \text{--}pb & x & y & 0 
& 0\\
\text{--}a^\secund & \text{--}b^\secund & \text{--}p^\prim a\text{--}qa^\prim & 
\text{--}p^\prim b\text{--}qb^\prim & \text{--}qa & \text{--}qb & x & y 
\end{pmatrix}
\, , 
\begin{pmatrix} 
\text{--}y & 0 & 0 & 0\\
x & 0 & 0 & 0\\
b & \text{--}y & 0 & 0\\
\text{--}a & x & 0 & 0\\
b^\prim & pb & \text{--}y & 0\\
\text{--}a^\prim & \text{--}pa & x & 0\\
b^\secund & p^\prim b + qb^\prim & qb & \text{--}y\\
\text{--}a^\secund & \text{--}p^\prim a\text{--}qa^\prim & \text{--}qa & x
\end{pmatrix}
\] 
\end{prop}

\begin{proof} 
$\tH^0_\ast(\sco_Z)$ admits a filtration by graded $S$-submodules with the 
successive quotiens isomorphic to $S(L):= S/I(L)$, $S(L)(l)$, $S(L)(2l+m)$ and 
$S(L)(3l+m+n)$, respectively. Using the well-known graded free resolution of 
$S(L)$ over $S$ one deduces that $\tH^0_\ast(\sco_Z)$ has a graded free 
resolution of the numerical shape from the statement. Moreover, the linear 
parts of the differentials $\delta_1$ and $\delta_2$ can be deduced from the 
resolution of $S(L)$. The rest of the matrices defining $\delta_1$ and 
$\delta_2$ can be easily guessed. For example, $\delta_1$ is defined by the 
matrix of relations between the generators $1,\, e_1,\, e_2,\, e_3$ of 
$\tH^0_\ast(\sco_Z)$. The module of these relations is generated by the 
relations of the following form$\, :$ one multiplies each generator of 
$\tH^0_\ast(\sco_Z)$ by $x$ and by $y$ and one expresses the results as 
combinations of the next generators.       
\end{proof}

\begin{remark}\label{R:cancellation}
We recall, for completeness, an algorithmic procedure for getting a minimal 
free resolution of a graded $S$-module $M$ from a non-minimal one. Consider 
a complex$\, :$ 
$F_\bullet \, : \, \cdots F_{s+1} \xra{d_{s+1}} F_s 
\xra{d_s} F_{s-1} \xra{d_{s-1}} F_{s-2} \cdots $
of graded free $S$-modules of finite rank. Assume that the $(i,j)$ entry of 
the matrix of $d_s$ is a non-zero constant $c\in k$. Performing elementary 
operations on the columns of the matrix of $d_s$, one can turn into 0 all the 
other entries of its $i$th row. One gets, in this way, a new morphism 
$d_s^{\, \prim} : F_s \ra F_{s-1}$. Then$\, :$ 
\begin{enumerate} 
\item[(i)] deleting the $j$th rank 1 direct summand of $F_s$ and the $i$th 
rank 1 direct summand of $F_{s-1}$, 
\item[(ii)] deleting the $i$th row and the $j$th column of the matrix of 
$d_s^{\, \prim}$, 
\item[(iii)] deleting the $i$th column of the matrix of $d_{s-1}$, 
\item[(iv)] deleting the $j$th row of the matrix of $d_{s+1}$, 
\end{enumerate}
one obtains a complex $F_\bullet^{\,\prim}$ such that $F_\bullet$ is isomorphic to 
the direct sum of $F_\bullet^{\, \prim}$ and of a complex of the form$\, :$  
$\cdots 0 \lra S(-a) \overset{c}{\lra} S(-a) \lra 0 \cdots \, .$
\end{remark}

\subsection{Primitive structures of degree 4}\label{SS:p4} 
We assume, in this subsection, using the notation introduced in the previous 
one, that $m = n = 0$ hence that $p,\, q \in k$. Since $p \neq 0$ and 
$q \neq 0$ one can assume that $e_2 = e_1^2$, $e_3 = e_1^3$ (and $e_1^4 = 0$), 
i.e., that $p = 1$, $p^\prim = 0$ and $q = 1$. Since $\Delta_1 \in 
\tH^0(\sco_L(3l+2))$, the polynomials $v_0,\, v_1,\, v_2 \in \tH^0(\sco_L(l))$ 
defined by relation \eqref{E:v012}$\, :$ 
\[
-\Delta_1 = v_0a^2 + v_1ab + v_2b^2 
\] 
are \emph{uniquely determined} (taking into account the exact sequence 
\eqref{E:a2abb2}) and the relations \eqref{E:vij} become$\, :$ 
\begin{gather*} 
-a\Delta_1 = v_0a^3 + v_1 a^2b + v_2 ab^2 \, ,\\
-b\Delta_1 = v_0a^2b + v_1ab^2 + v_2b^3\, .
\end{gather*}
In relation \eqref{E:fgw} one must have $f = g = 0$, and this relation 
becomes$\, :$ 
\begin{equation}\label{E:fgwprimitive} 
\Delta_2 - 2v_0aa^\prim - v_1(ab^\prim + a^\prim b) - 2v_2bb^\prim = w \, ,
\end{equation}
with $w \in \tH^0(\sco_L(4l+2))$. The exact sequence \eqref{E:syzabp} 
becomes$\, :$ 
\[
0 \lra 2\sco_L(-l-1) \xra{\begin{pmatrix}1 & 0\\ 0 & 1\\ 
\text{--}a & \text{--}b\end{pmatrix}} 
2\sco_L(-l-1) \oplus \sco_L  
\xra{\displaystyle (a\, ,\, b\, ,\, 1)} \sco_L \lra 0 \, .
\] 
In particular, $l_1 = l_2 = l + 1$. The exact sequence 
\eqref{E:e4secundnu4secund} becomes$\, :$ 
\begin{equation}\label{E:e4secundnu4secundprimitive}
0 \lra 
\begin{matrix}
\sco_L(-l-2)\\ \oplus\\ 2\sco_L(-l-3) \end{matrix}
\overset{\displaystyle \nu_4^\secund}{\lra} 
\begin{matrix}
\sco_L(-l-2)\\ \oplus\\ 2\sco_L(-l-3)\\ \oplus\\ \sco_L(3l) 
\end{matrix}
\overset{\displaystyle \e_4^\secund}{\lra} 
\sco_L(3l) \lra 0 
\end{equation} 
with $\e_4^\secund$ and $\nu_4^\secund$ defined by the matrices$\, :$ 
\[
(-w\, ,\, a\Delta_1\, ,\, b\Delta_1\, ,\, 1)\  ,\  
\begin{pmatrix} 
1 & 0 & 0\\ 
0 & 1 & 0\\ 
0 & 0 & 1\\ 
w & -a\Delta_1 & -b\Delta_1 
\end{pmatrix}
\, , 
\]  
respectively. Prop.~\ref{P:geniz} becomes, in this special case$\, :$ 

\begin{cor}\label{C:geniz} 
Assume that $m = n = 0$ and that $w$ $($defined in 
\eqref{E:fgwprimitive}$)$ can be written as a combination $w = w_0a^3 + 
w_1a^2b + w_2ab^2 + w_3b^3$ with $w_0,\ldots ,w_3 \in {\fam0 H}^0(\sco_L(l-1))$ 
$($uniquely determined, taking into account the exact sequence 
\eqref{E:a3a2bab2b3}$)$. Recalling the notation$\, :$ 
\[
F_2 = \begin{vmatrix} a & b\\ x & y \end{vmatrix}\, ,\  
F_3 = F_2 + v_0x^2 + v_1xy + v_2y^2\, ,
\] 
the homogeneous ideal $I(Z)\subset S$ is generated by$\, :$ 
\[
F_4 = F_3 + w_0x^3 + w_1x^2y + w_2xy^2 + w_3y^3\, ,\  x^4\, ,\  x^3y\, ,\  
x^2y^2\, ,\  xy^3\, ,\  y^4  
\]
hence $Z$ is the divisor $4L$ on the surface $F_4 = 0$, which is nonsingular 
along $L$. 
\end{cor}

\begin{proof} 
It suffices to notice that the polynomials $G_1$ and $G_2$ occuring in the 
statement of Prop.~\ref{P:geniz} are equal, in the particular case under 
consideration, to $xF_3$ and $yF_3$, respectively, and that these polynomials 
toghether with the polynomials $x^2F_2$, $xyF_2$, $y^2F_2$ belong to the ideal 
generated by the polynomials from the statement of the corollary. 
\end{proof}

\begin{remark}\label{R:m=0n=0l=-1} 
It follows, in particular, from Corollary~\ref{C:geniz} and 
Remark~\ref{R:wcombination}, that if $m = n = 0$ and $l = -1$ then $Z$ is the 
divisor $4L$ on the plane $H \supset L$ of equation $-bx + ay = 0$. 
\end{remark}

\noindent 
{\bf A graded free resolution of} $I(Z)$ {\bf under the hypothesis of 
Cor.~\ref{C:geniz}.}\quad   
Using Remark~\ref{R:cancellation} to cancel redundant direct 
summands of the terms of the resolution of $I(Z)$ described after 
Prop.~\ref{P:geniz} and taking into account that the polynomials 
$\alpha_i$, $\beta_i$, $i = 0,\, 1,\, 2$, introduced in Remark~\ref{R:vivij} 
are all zero, one gets that the ideal $I(Z)$ has, under the hypothesis of 
Cor.~\ref{C:geniz}, a free resolution of the form$\, :$ 
\[
0 \lra 3S(-l-6) \overset{\displaystyle d_2^{\, \prim}}{\lra}  
\begin{matrix} 4S(-l-5)\\ \oplus\\ 4S(-5) \end{matrix} 
\overset{\displaystyle d_1^{\, \prim}}{\lra}  
\begin{matrix} S(-l-2)\\ \oplus\\ 5S(-4) \end{matrix} 
\xra{\displaystyle d_0^{\, \prim}} I(Z) \lra 0
\] 
with $d_1^{\, \prim}$ and $d_2^{\, \prim}$ defined by the matrices$\, :$ 
\begin{gather*}
\begin{pmatrix}
x^3 & x^2y & xy^2 & y^3 & 0 & 0 & 0 & 0\\
b \text{--} v_0x \text{--} w_0x^2 & 0 & 0 & 0 & \text{--} y & 0 & 0 & 0\\
\text{--} a \text{--} v_1x \text{--} w_1x^2 & 
b \text{--} v_0x \text{--} w_0x^2 & 
\text{--} w_0xy & \text{--} w_0y^2 & x & \text{--} y & 0 & 0\\
\text{--} v_2x \text{--} w_2x^2 & \text{--} a \text{--} v_1x \text{--} w_1x^2 & 
b \text{--} v_0x \text{--} w_1xy & \text{--} v_0y \text{--} w_1y^2 & 
0 & x & \text{--} y & 0\\
\text{--} w_3x^2 & \text{--} v_2x \text{--} w_2x^2 & 
\text{--} a \text{--} v_1x \text{--} w_2xy & 
b \text{--} v_1y \text{--} w_2y^2 & 0 & 0 & x & \text{--} y\\
0 & \text{--} w_3x^2 & \text{--} v_2x \text{--} w_3xy & 
\text{--} a \text{--} v_2y \text{--} w_3y^2 & 0 & 0 & 0 & x 
\end{pmatrix}\, ,\\
\begin{pmatrix}
- y & 0 & 0\\
x & - y & 0\\
0 & x & - y\\
0 & 0 & x\\
- b + v_0x + w_0x^2 & 0 & 0\\
a + v_1x + w_1x^2 & - b + v_0x & 0\\
v_2x + w_2x^2 & a + v_1x & - b\\
w_3x^2 & v_2x & a
\end{pmatrix}\, .
\end{gather*}

\vskip2mm 

\noindent 
{\bf The case where $w$ cannot be written as a combination of $a^3, \dots , 
b^3$.}\quad  
Assume that $m = n = 0$, $l \geq 0$ and that the polynomial 
$w \in \tH^0(\sco_L(4l+2))$, defined in \eqref{E:fgwprimitive}, cannot be 
written as a combination of $a^3, \ldots , b^3$. Since $x_iw \in 
\tH^0(\sco_L(4l+3))$, $i = 0,\, 1$, it follows, using the exact sequence 
\eqref{E:a3a2bab2b3}, that there exist uniquely determined polynomials 
$w_{ij} \in \tH^0(\sco_L(l))$, $i = 0,\, 1$, $j = 0, \ldots ,3$, such that$\, :$ 
\begin{equation}\label{E:wij} 
x_iw = w_{i0}a^3 + w_{i1}a^2b + w_{i2}ab^2 + w_{i3}b^3\, ,\  i = 0,\, 1\, . 
\end{equation}

\begin{prop}\label{P:genizprim} 
Assume that $m = n = 0$, $l \geq 0$ and that the polynomial $w$ $($defined in  
\eqref{E:fgwprimitive}$)$ cannot be written as a combination of $a^3, \ldots , 
b^3$ as in Cor.~\ref{C:geniz}. Then the homogeneous ideal $I(Z) \subset S$ 
is generated by the polynomials$\, :$ 
\begin{gather*} 
F_{40} = x_0F_3 + w_{00}x^3 + w_{01}x^2y + w_{02}xy^2 + w_{03}y^3\, ,\\
F_{41} = x_1F_3 + w_{10}x^3 + w_{11}x^2y + w_{12}xy^2 + w_{13}y^3\, ,\\ 
G_1 = xF_3\, ,\  G_2 = yF_3\, ,\  x^4\, ,\  x^3y\, ,\  x^2y^2\, ,\  xy^3\, ,
\  y^4 
\end{gather*} 
with the $w_{ij}$'s as in \eqref{E:wij}.  
\end{prop}

\begin{proof}
We describe, firstly, under the hypothesis of the proposition, the kernel of 
the epimorphism $\e_4 : 2\sco_L(-1) \oplus 3\sco_L(-2) \oplus 4\sco_L(-3) 
\ra \sco_L(l) \oplus \sco_L(2l) \oplus \sco_L(3l)$. We intend to apply, for 
this purpose, Remark~\ref{R:kerphi}. The component $\psi_2^\secund : 
\sco_L(-l-2) \oplus 2\sco_L(-l-3) \ra \sco_L(3l)$ of the morphism 
$\nu_4^\secund$ from the exact sequence \eqref{E:e4secundnu4secundprimitive} is 
defined by the matrix $(w\, ,\, -a\Delta_1\, , -b\Delta_1)$. One has a 
commutative diagram with exact rows$\, :$ 
\[
\begin{CD}
0 @>>> \sco_L(-l-4) @>{\displaystyle \sigma}>> \begin{matrix} 2\sco_L(-l-3)\\ 
\oplus\\ 2\sco_L(-l-3) \end{matrix} @>{\displaystyle \rho}>> 
\begin{matrix} \sco_L(-l-2)\\ \oplus\\ 2\sco_L(-l-3) \end{matrix} @>>> 0\\
@. @VV{\displaystyle \psi_2^\prim}V @VV{\displaystyle \psi_2}V 
@VV{\displaystyle \psi_2^\secund}V\\ 
0 @>>> 3\sco_L(-l-4) @>>> 4\sco_L(-3) @>>> \sco_L(3l) @>>> 0 
\end{CD}
\] 
with the exact sequence \eqref{E:a3a2bab2b3} as bottom row, with $\rho$, 
$\sigma$, $\psi_2$ defined, respectively, by the matrices$\, :$ 
\[
\begin{pmatrix} 
x_0 & x_1 & 0 & 0\\ 
0 & 0 & 1 & 0\\ 
0 & 0 & 0 & 1 
\end{pmatrix} \  ,\  
\begin{pmatrix} 
-x_1\\ x_0\\ 0\\ 0 
\end{pmatrix}\  ,\  
\begin{pmatrix}
w_{00} & w_{10} & v_0 & 0\\ 
w_{01} & w_{11} & v_1 & v_0\\ 
w_{02} & w_{12} & v_2 & v_1\\ 
w_{03} & w_{13} & 0 & v_2 
\end{pmatrix}
\]
and with $\psi_2^\prim$ defined by a matrix of the form $(\gamma_0\, ,\, 
\gamma_1\, ,\, \gamma_2)^{\text{t}}$, with $\gamma_0,\, \gamma_1,\, \gamma_2 \in 
k$. The commutativity of the left square of the above diagram is equivalent 
to the matrix relation$\, :$ 
\begin{equation}\label{E:wijgamma012}
-x_1\begin{pmatrix} w_{00}\\ w_{01}\\ w_{02}\\ w_{03} \end{pmatrix} 
+ x_0\begin{pmatrix} w_{10}\\ w_{11}\\ w_{12}\\ w_{13} \end{pmatrix} = 
\begin{pmatrix} 
-b & 0 & 0\\
a & -b & 0\\
0 & a & -b\\
0 & 0 & a
\end{pmatrix} 
\begin{pmatrix} \gamma_0\\ \gamma_1\\ \gamma_2 \end{pmatrix}\, . 
\end{equation} 
It follows that at least one of the constants $\gamma_0,\, \gamma_1,\, 
\gamma_2$ must be non-zero because, otherwise, $x_0$ would divide each of the 
polynomials $w_{00},\ldots ,w_{03}$ and this would contradict our assumption 
that $w$ cannot be written as a combination of $a^3,\ldots ,b^3$. 

Applying, now, Remark~\ref{R:kerphi}(a) one gets an exact sequence$\, :$ 
\begin{equation}\label{E:kere4primitive} 
0 \lra \sco_L(-l-4) \overset{\displaystyle \kappa_4}{\lra} 
\begin{matrix} 2\sco_L(-l-3)\\ \oplus\\ 2\sco_L(-l-3)\\ \oplus\\ 3\sco_L(-l-4) 
\end{matrix} \overset{\displaystyle \nu_4}{\lra} 
\begin{matrix} 2\sco_L(-1)\\ \oplus\\ 3\sco_L(-2)\\ \oplus\\ 4\sco_L(-3) 
\end{matrix} \overset{\displaystyle \e_4}{\lra} 
\begin{matrix} \sco_L(l)\\ \oplus\\  \sco_L(2l)\\ \oplus\\ \sco_L(3l) 
\end{matrix} \lra 0   
\end{equation} 
with $\nu_4$ and $\kappa_4$ defined, repectively, by the matrices$\, :$ 
\[
\begin{pmatrix} 
-bx_0 & -bx_1 & 0 & 0 & 0 & 0 & 0\\
ax_0 & ax_1 & 0 & 0 & 0 & 0 & 0\\
v_0x_0 & v_0x_1 & -b & 0 & 0 & 0 & 0\\
v_1x_0 & v_1x_1 & a & -b & 0 & 0 & 0\\
v_2x_0 & v_2x_1 & 0 & a & 0 & 0 & 0\\
w_{00} & w_{10} & v_0 & 0 & -b & 0 & 0\\
w_{01} & w_{11} & v_1 & v_0 & a & -b & 0\\
w_{02} & w_{12} & v_2 & v_1 & 0 & a & -b\\
w_{03} & w_{13} & 0 & v_2 & 0 & 0 & a
\end{pmatrix}\ ,\  
\begin{pmatrix} -x_1\\ x_0\\ 0\\ 0\\ -\gamma_0\\ -\gamma_1\\ -\gamma_2 
\end{pmatrix}\, .  
\]
It follows that the columns of the matrix defining $\nu_4$ generate the graded 
$\tH^0_\ast(\sco_L)$-module $\tH^0_\ast(\Ker \e_4)$. 

Now, using the exact sequence \eqref{E:il4izhkere4} and the fact that$\, :$ 
\[
\sci_L/\sci_L^4 = \sco_L(-1)x \oplus \sco_L(-1)y \oplus \sco_L(-2)x^2 
\oplus \cdots \oplus \sco_L(-3)y^3 
\] 
one deduces that $I(Z)$ is generated by the polynomials from the statement  
and by $x^2F_2$, $xyF_2$, $y^2F_2$. The latter ones, however, 
belong to the ideal generated by $G_1$, $G_2$, $x^4, \ldots ,\, y^4$. 
\end{proof} 

\noindent 
{\bf A graded free resolution of} $I(Z)$ {\bf under the hypothesis of 
Prop.~\ref{P:genizprim}.}\quad 
The morphism $\nu_4$ from the exact sequence \eqref{E:kere4primitive} maps 
$2\sco_L(-l-3) \oplus 3\sco_L(-l-4)$ into $3\sco_L(-2) \oplus 4\sco_L(-3)$ and 
the induced morphism ${\overline \nu}_4 : 2\sco_L(-l-3) \oplus 3\sco_L(-l-4) 
\ra 3\sco_L(-2) \oplus 4\sco_L(-3)$ is defined by the matrix$\, :$ 
\[
\begin{pmatrix} 
-b & 0 & 0 & 0 & 0\\
a & -b & 0 & 0 & 0\\
0 & a & 0 & 0 & 0\\
v_0 & 0 & -b & 0 & 0\\
v_1 & v_0 & a & -b & 0\\
v_2 & v_1 & 0 & a & -b\\
0 & v_2 & 0 & 0 & a 
\end{pmatrix}\, . 
\]
Applying the Snake Lemma to the diagram$\, :$ 
\[
\begin{CD}
0 @>>> 0 @>>> \begin{matrix} 2\sco_L(-l-3)\\ \oplus\\ 3\sco_L(-l-4) 
\end{matrix} 
@= \begin{matrix} 2\sco_L(-l-3)\\ \oplus\\ 3\sco_L(-l-4) \end{matrix} 
@>>> 0\\ 
@. @VVV @VV{\displaystyle \text{incl}}V @VV{\displaystyle {\overline \nu}_4}V\\ 
0 @>>> \sco_L(-l-4) @>{\displaystyle \kappa_4}>> 
\begin{matrix} 4\sco_L(-l-3)\\ \oplus\\ 3\sco_L(-l-4) \end{matrix} 
@>{\displaystyle \nu_4}>> \Ker \e_4 @>>> 0   
\end{CD}
\]
one gets an exact sequence$\, :$ 
\begin{equation}\label{E:etanu4bar} 
0 \lra 2\sco_L(-l-3) \oplus 3\sco_L(-l-4) 
\overset{\displaystyle {\overline \nu}_4}{\lra} 
\Ker \e_4 \overset{\displaystyle \eta}{\lra} \sco_L(-l-2) \lra 0 
\end{equation} 
where $\eta$ maps the elements$\, :$ 
\[
(-bx_i,\, ax_i,\, v_0x_i,\, v_1x_i,\, v_2x_i,\, w_{i0},\, w_{i1},\, w_{i2},\, 
w_{i3})^{\text{t}} ,\  i = 0,\, 1\, , 
\]
of $\tH^0((\Ker \e_4)(l+3))$ into the elements $x_0,\, x_1$ of 
$\tH^0(\sco_L(1))$. 

It is easy to check that the morphism ${\overline \nu}_4$ considered above can 
be identified with the morphism$\, :$ 
\[
F_3 \cdot - : (\sci_L/\sci_L^3)(-l-2) \lra \sci_L^2/\sci_L^4 \, . 
\]
Applying, now, $\tH^0_\ast(-)$ to the exact sequence \eqref{E:etanu4bar} 
one gets an exact sequence of graded $S$-modules$\, :$ 
\begin{equation}\label{E:h0kere4} 
0 \lra (I(L)/I(L)^3)(-l-2) \xra{\displaystyle F_3 \cdot -} 
I(Z)/I(L)^4 \lra S(L)_+(-l-2) \lra 0 
\end{equation}
where $S(L) = S/I(L)$ and $S(L)_+ = x_0S(L) + x_1S(L) = 
\bigoplus_{i\geq 1}S(L)_i$. 

We provide, now, explicit minimal free resolutions of the graded 
$S$-modules $S(L)_+$ and $I(L)/I(L)^3$. Using the exact sequence$\, :$ 
\[
0 \lra S(L)(-2) \xra{\begin{pmatrix} -x_1\\ x_0 \end{pmatrix}} 
2S(L)(-1) \xra{\displaystyle (x_0\, ,\, x_1)} S(L)_+ \lra 0\, .
\]
one sees that the tensor product of the complexes$\, :$
\[
S(-2) \xra{\begin{pmatrix} -x_1\\ x_0 \end{pmatrix}} 2S(-1)\, ,\  
S(-2) \xra{\begin{pmatrix} -y\\ x \end{pmatrix}} 2S(-1) 
\xra{\displaystyle (x\, ,\, y)} S 
\] 
is the following minimal free resolution of $S(L)_+$$\, :$ 
\[
0 \ra 
S(-4) \xra{\begin{pmatrix} \text{--}y\\ x\\ \text{--}x_1\\ x_0\end{pmatrix}} 
4S(-3) \xra{\begin{pmatrix} 
x & y & 0 & 0\\
x_1 & 0 & \text{--}y & 0\\
0 & x_1 & x & 0\\
\text{--}x_0 & 0 & 0 & \text{--}y\\
0 & \text{--}x_0 & 0 & x
\end{pmatrix}} 5S(-2) \xra{\begin{pmatrix} 
\text{--}x_1 & x & y & 0 & 0\\
x_0 & 0 & 0 & x & y
\end{pmatrix}} 2S(-1)  
\] 

As for $I(L)/I(L)^3$, one uses the standard free resolutions of $I(L)/I(L)^2 
\simeq 2S(L)(-1)$ and of $I(L)^2/I(L)^3 \simeq 3S(L)(-2)$ to get a 
non-minimal free resolution of $I(L)/I(L)^3$ over $S$ and then, cancelling 
redundant terms (see Remark~\ref{R:cancellation}), one gets the following 
minimal free resolution$\, :$ 
\[
0\ra 
3S(-4) \xra{\begin{pmatrix} 
0 & 0 & \text{--}y^2\\
\text{--}y & 0 & 0\\
x & \text{--}y & 0\\
0 & x & \text{--}y\\
0 & 0 & x
\end{pmatrix}} \begin{matrix} S(-2)\\ \oplus\\ 4S(-3) \end{matrix} 
\xra{\begin{pmatrix} 
\text{--}y & x^2 & xy & y^2 & 0\\
x & 0 & 0 & 0 & y^2
\end{pmatrix}} 2S(-1) \xra{\displaystyle (x\, ,\, y)}  
\frac{I(L)}{I(L)^3}
\]  
 
One deduces, now, from the exact sequences \eqref{E:il4izhkere4} and 
\eqref{E:h0kere4}, that $I(Z)$ admits a non-minimal free resolution of 
the form$\, :$ 
\[
0 \lra S(-l-6) \overset{\displaystyle d_3}{\lra} 
\begin{matrix} 4S(-l-5)\\ \oplus\\ 3S(-l-6) \end{matrix} 
\overset{\displaystyle d_2}{\lra}  
\begin{matrix} 6S(-l-4)\\ \oplus\\ 4S(-l-5)\\ \oplus\\ 4S(-5) \end{matrix} 
\overset{\displaystyle d_1}{\lra}  
\begin{matrix} 4S(-l-3)\\ \oplus\\ 5S(-4) \end{matrix} \lra I(Z) \lra 0\, .
\]  
This resolution has a filtration with successive quotients the above minimal 
free resolutions of $S(L)_+(-l-2)$, $(I(L)/I(L)^3)(-l-2)$ and $I(L)^4$, 
respectively. Consequently, large parts of the matrices of $d_1$, $d_2$ and 
$d_3$ are known. In order to get the entire matrix of $d_1$ one has to 
extend the known parts of its columns to relations between the generators 
of $I(Z)$ from Prop.~\ref{P:genizprim}. This presents some difficulty only 
for the first column. 
Multiplying to the left the matrix relation \eqref{E:wijgamma012} by 
$(x^3,\, x^2y,\, xy^2,\, y^3)$ one gets the following relation$\, :$ 
\begin{multline}\label{E:wijgamma}
-x_1(w_{00}x^3 + w_{01}x^2y + w_{02}xy^2 + w_{03}y^3) + 
x_0(w_{10}x^3 + w_{11}x^2y + w_{12}xy^2 + w_{13}y^3)\\  
= (\gamma_0x^2 + \gamma_1xy + \gamma_2y^2)(-bx+ay)\, ,\  \  \  \  \  \  \  \    
\  \  \  \  \  \  \  \  \  \  \  \  \  \  \  \  \  \  \  \  \  \   
\end{multline}
from which one deduces that$\, :$ 
\[
\text{--}x_1F_{40} + x_0F_{41} 
\text{--}\, (\gamma_0x^2 + \gamma_1xy + \gamma_2y^2)F_3 
+ (\gamma_0x^2 + \gamma_1xy + \gamma_2y^2)(v_0x^2 + v_1xy + v_2y^2) = 0 
\]  
Now, the matrices of $d_1$, $d_2$ and $d_3$ are the following ones$\, :$ 
\begin{equation*}
\setcounter{MaxMatrixCols}{14}
\begin{pmatrix} 
\text{-}x_1 & x & y & 0 & 0 & 0 & 0 & 0 & 0 & 0 & 0 & 0 & 0 & 0\\
x_0 & 0 & 0 & x & y & 0 & 0 & 0 & 0 & 0 & 0 & 0 & 0 & 0\\
\text{-}\gamma_0x \text{-}\gamma_1y & \text{-}x_0 & 0 & \text{-}x_1 & 0 & 
\text{-}y & x^2 & xy & y^2 & 0 & 0 & 0 & 0 & 0\\
\text{-}\gamma_2y & 0 & \text{-}x_0 & 0 & \text{-}x_1 & x & 0 & 0 & 0 & 
y^2 & 0 & 0 & 0 & 0\\
\theta_0 & \text{-}w_{00} & 0 & \text{-}w_{10} & 0 & 0 & b\, \text{-}v_0x & 
0 & 0 & 0 & \text{-}y & 0 & 0 & 0\\
\theta_1 & \text{-}w_{01} & \text{-}w_{00} & \text{-}w_{11} & 
\text{-}w_{10} & 0 & \text{-}a\text{-}v_1x & b\, \text{-}v_0x & 0 & 0 & x & 
\text{-}y & 0 & 0\\
\theta_2 & \text{-}w_{02} & \text{-}w_{01} & 
\text{-}w_{12} & \text{-}w_{11} & 0 & \text{-}v_2x & \text{-}a\text{-}v_1x 
& b\, \text{-}v_0x & \text{-}v_0y & 0 & x & \text{-}y & 0\\
\theta_3 & \text{-}w_{03} & \text{-}w_{02} & \text{-}w_{13} 
& \text{-}w_{12} & 0 & 0 & \text{-}v_2x & \text{-}a\text{-}v_1x & 
b\, \text{-}v_1y & 0 & 0 & x & \text{-}y\\
\theta_4 & 0 & \text{-}w_{03} & 0 & \text{-}w_{13} & 0 & 0 & 0 & 
\text{-}v_2x & \text{-}a\text{-}v_2y & 0 & 0 & 0 & x
\end{pmatrix}
\end{equation*}
\begin{gather*}
\begin{pmatrix} 
x & y & 0 & 0 & 0 & 0 & 0\\
x_1 & 0 & -y & 0 & 0 & 0 & 0\\
0 & x_1 & x & 0 & 0 & 0 & 0\\
-x_0 & 0 & 0 & -y & 0 & 0 & 0\\
0 & -x_0 & 0 & x & 0 & 0 & 0\\
\gamma_2y & 0 & x_0 & x_1 & 0 & 0 & -y^2\\
\gamma_0 & 0 & 0 & 0 & -y & 0 & 0\\
\gamma_1 & \gamma_0 & 0 & 0 & x & -y & 0\\
\gamma_2 & \gamma_1 & 0 & 0 & 0 & x & -y\\
0 & \gamma_2 & 0 & 0 & 0 & 0 & x\\
0 & \gamma_0v_0 & w_{00} & w_{10} & -b+v_0x & 0 & 0\\
0 & \gamma_0v_1 + \gamma_1v_0 & w_{01} & w_{11} & a+v_1x & -b+v_0x & 0\\
0 & \gamma_0v_2 + \gamma_1v_1 & w_{02} & w_{12} & v_2x & a+v_1x & -b\\
0 & \gamma_1v_2 & w_{03} & w_{13} & 0 & v_2x & a 
\end{pmatrix}\, ,\  
\begin{pmatrix} 
-y\\ x\\ -x_1\\ x_0\\ -\gamma_0\\ -\gamma_1\\ -\gamma_2
\end{pmatrix}
\end{gather*}
where $\theta_0,\ldots,\theta_4$ are defined by$\, :$
\[
(\gamma_0x^2 + \gamma_1xy + \gamma_2y^2)(v_0x^2 + v_1xy + v_2y^2) = 
\theta_0x^4 + \theta_1x^3y + \theta_2x^2y^2 + \theta_3xy^3 + \theta_4y^4 
\]
In order to check that $d_1d_2 = 0$ and $d_2d_3 = 0$ one has to use the above 
polynomial relation \eqref{E:wijgamma}. 
Using Remark~\ref{R:cancellation}, one can cancel the redundant terms of 
the above free resolution of $I(Z)$. One has to consider three cases$\, :$ 
(I) $\gamma_0 \neq 0$, (II) $\gamma_0 = 0$ and $\gamma_1 \neq 0$, (III) 
$\gamma_0 = \gamma_1 = 0$ and $\gamma_2 \neq 0$. In all the cases, one gets a 
free resolution of $I(Z)$ having the following shape$\, :$ 
\[
0 \lra  
\begin{matrix} 2S(-l-5)\\ \oplus\\ 2S(-l-6) \end{matrix} 
\overset{\displaystyle d_2^{\, \prim}}{\lra} 
\begin{matrix} 6S(-l-4)\\ \oplus\\ 2S(-l-5)\\ \oplus\\ 4S(-5) \end{matrix} 
\overset{\displaystyle d_1^{\, \prim}}{\lra}  
\begin{matrix} 4S(-l-3)\\ \oplus\\ 5S(-4) \end{matrix} \lra I(Z) \lra 0\, .
\]

\subsection{Thick structures of degree 4}\label{SS:thick} 
According to B\u{a}nic\u{a} and Forster \cite[\S~4]{bf}, if $W$ is such a 
structure supported on the line $L$ then $\sci_L^3 \subset \sci_W \subset 
\sci_L^2$ and one has an exact sequence of the form$\, :$ 
\[
0 \lra \sci_W/\sci_L^3 \lra \sci_L^2/\sci_L^3 
\overset{\displaystyle \e}{\lra} \sco_L(l) \lra 0 
\]
for some integer $l$. Since $\sci_L^2/\sci_L^3 \simeq 3\sco_L(-2)$, one 
must have $l \geq -2$. The epimorphism $\e$ is defined by three polynomials 
$p_0,\, p_1,\, p_2 \in \tH^0(\sco_L(l+2))$ having no common zero on $L$. One 
deduces an exact sequence$\, :$ 
\begin{equation}\label{E:ollowol1}
0 \lra \sco_L(l) \lra \sco_W \lra \sco_{L^{(1)}} \lra 0 
\end{equation}
where $L^{(1)}$ is the first infinitesimal neighbourhood of $L$ in $\piii$, 
defined by the ideal $\sci_L^2$. It follows that, as an $\sco_L$-module$\, :$ 
\[
\sco_W \simeq \sco_L \oplus 2\sco_L(-1) \oplus \sco_L(l)\, .
\]
The graded $\tH^0_\ast(\sco_L) = k[x_0,x_1]$-module $\tH^0_\ast(\sco_W)$ 
has a system of generators consisting of $1\in \tH^0(\sco_W)$, 
$e_0,\, e_1 \in \tH^0(\sco_W(1))$ and $e_2 \in \tH^0(\sco_W(-l))$. The 
multiplicative structure of $\tH^0_\ast(\sco_W)$ is defined by$\, :$ 
\[
e_0^2 = p_0e_2\, ,\  e_0e_1 = p_1e_2\, ,\  e_1^2 = p_2e_2\, ,\  e_0e_2 = 0\, ,\  
e_1e_2 = 0\, ,\  e_2^2 = 0\, . 
\] 
The canonical epimorphism $\sco_\piii \ra \sco_W$ induces a morphism of graded 
$k[x_0,x_1]$-algebras $S = \tH^0_\ast(\sco_\piii) \ra \tH^0(\sco_W)$ which is 
completely determined by the fact that $x \mapsto e_0$ and $y \mapsto e_1$.  
One has an exact sequence$\, :$ 
\begin{equation}\label{E:kerp0p1p2}
0 \lra \sco_L(-m)\oplus \sco_L(-n) 
\xra{\begin{pmatrix} 
f_0 & g_0\\
f_1 & g_1\\
f_2 & g_2
\end{pmatrix}} 3\sco_L 
\xra{\displaystyle (p_0\, ,\, p_1\, ,\, p_2)} \sco_L(l+2) \lra 0
\end{equation} 
with $0 \leq m \leq n$ and $m+n = l+2$. Since $\sci_W/\sci_L^3$ is annihilated 
by $\sci_L$ it is already an $\sco_L$-module. Using the above exact sequence, 
it follows that, actually$\, :$ 
\[
\sci_W/\sci_L^3 \simeq \sco_L(-m-2) \oplus \sco_L(-n-2)\, .
\]
Since $\tH^1_\ast(\sci_L^3) = 0$, one deduces an exact sequence of graded 
$S$-modules$\, :$ 
\[
0 \lra I(L)^3 \lra I(W) \lra S(L)(-m-2)\oplus S(L)(-n-2) \lra 0  
\]
with $S(L) = S/I(L)$, from which one gets immediately$\, :$ 

\begin{prop}\label{P:geniw} 
The homogeneous ideal $I(W) \subset S$ of $W$ is generated by the 
following polynomials$\, :$ 
\[
F = f_0x^2 + f_1xy + f_2y^2\, ,\  G = g_0x^2 + g_1xy + g_2y^2\, ,\  
x^3\, ,\  x^2y\, ,\  xy^2\, ,\  y^3\, .
\]
\end{prop}

Moreover, one also gets that $I(W)$ has a graded free resolution of the 
form$\, :$ 
\[
0 \lra \begin{matrix} S(-m-4)\\ \oplus\\ S(-n-4) \end{matrix} 
\overset{\displaystyle d_2}{\lra} 
\begin{matrix} 2S(-m-3)\\ \oplus\\ 2S(-n-3)\\ \oplus\\ 3S(-4)\end{matrix} 
\overset{\displaystyle d_1}{\lra} 
\begin{matrix} S(-m-2)\\ \oplus\\ S(-n-2)\\ \oplus\\ 4S(-3)\end{matrix} 
\lra I(W) \lra 0
\]
with $d_1$ and $d_2$ defined by the matrices$\, :$ 
\[
\begin{pmatrix} 
x & y & 0 & 0 & 0 & 0 & 0\\
0 & 0 & x & y & 0 & 0 & 0\\
-f_0 & 0 & -g_0 &  0 & -y & 0 & 0\\
-f_1 & -f_0 & -g_1 & -g_0 & x & -y & 0\\
-f_2 & -f_1 & -g_2 & -g_1 & 0 & x & -y\\
0 & -f_2 & 0 & -g_2 & 0 & 0 & x 
\end{pmatrix}\, ,\  
\begin{pmatrix} 
-y & 0\\
x & 0\\
0 & -y\\
0 & x\\
f_0 & g_0\\
f_1 & g_1\\
f_2 & g_2
\end{pmatrix}
\]

\begin{prop}\label{P:geniwcm} 
Assume that $W$ is properly locally CM everywhere, i.e., it is not l.c.i. 
except at finitely many points. Then $l$ is even, $l = 2l^{\, \prim}$ with 
$l^{\, \prim} \geq -1$, and there exist two polynomials $q_0,\, q_1 \in 
{\fam0 H}^0(\sco_L(l^{\, \prim}+1))$ having no common zero on $L$  
such that $I(W)$ is generated by$\, :$ 
\[
x\begin{vmatrix} q_0 & q_1\\ x & y\end{vmatrix}\, ,\  
y\begin{vmatrix} q_0 & q_1\\ x & y\end{vmatrix}\, ,\  
x^3\, ,\  x^2y\, ,\  xy^2\, ,\  y^3\, .  
\]
Notice that if $X$ is the double structure on $L$ defined by the 
epimorphism $(q_0,\, q_1) : \sci_L/\sci_L^2 \ra \sco_L(l^{\, \prim})$ 
$($see Subsection~\ref{SS:p2} below$)$ then $I(W) = I(L)I(X)$. 
\end{prop}

\begin{proof} 
By \cite[Prop.~4.3]{bf}, the hypothesis of the proposition implies that 
$p_1^2 = p_0p_2$. Since $p_0,\, p_1,\, p_2$ have no common zero on $L$, it 
follows that $p_0$ and $p_2$ have no common zero on $L$. One deduces, now, 
that $p_0$ and $p_2$ must be ``perfect squares'', i.e., that there exist 
polynomials $q_0$ and $q_1$ such that $p_0 = q_0^2$ and $p_2 = q_1^2$. Then 
$p_1 = q_0q_1$ and the exact sequence \eqref{E:kerp0p1p2} becomes$\, :$ 
\[
0 \lra 2\sco_L(-l^{\, \prim}-1) 
\xra{\begin{pmatrix} -q_1 & 0\\ q_0 & -q_1\\ 0 & q_0 \end{pmatrix}} 
3\sco_L \xra{\displaystyle (q_0^2\, ,\, q_0q_1\, ,\, q_1^2)} 
\sco_L(l+2) \lra 0\, .
\] 
One can apply, now, Prop.~\ref{P:geniw}. 
\end{proof}

Applying $\tH^0_\ast(-)$ to the exact sequence \eqref{E:ollowol1} one gets 
an exact sequence$\, :$ 
\[
0 \lra S(L)(l) \lra \tH^0_\ast(\sco_W) \lra S/I(L)^2 \lra 0 
\]
from which one gets readily the following$\, :$ 

\begin{prop}\label{P:reshow} 
The graded $S$-module ${\fam0 H}^0_\ast(\sco_W)$ admits the following free 
resolution$\, :$ 
\begin{gather*}
0 \ra 
\begin{matrix} 2S(-3)\\ \oplus\\ S(-l-2) \end{matrix} 
\xra{\begin{pmatrix}
\text{--}y & 0 & 0\\
x & \text{--}y & 0\\
0 & x & 0\\
p_1 & p_2 & \text{--}y\\
\text{--}p_0 & \text{--}p_1 & x
\end{pmatrix}} 
\begin{matrix} 3S(-2)\\ \oplus\\ 2S(-l-1) \end{matrix} 
\xra{\begin{pmatrix}
x^2 & xy & y^2 & 0 & 0\\
\text{--}p_0 & \text{--}p_1 & \text{--}p_2 & x & y
\end{pmatrix}} 
\begin{matrix} S\\ \oplus\\ S(-l)\end{matrix}   
\ra {\fam0 H}^0_\ast(\sco_W)    
\end{gather*}
\qed 
\end{prop}

\emph{We record, finally, for reference, the similar but easier results 
concerning multiple structures of degree} 2 \emph{or} 3 \emph{on the line}  
$L \subset \piii$.

\subsection{Primitive structures of degree 2}\label{SS:p2}  
Let $X$ be a locally CM curve of degree 2 supported by the line $L$. 
The CM filtration of $\sco_X$ has two steps $\sco_X \supset \sci_L/\sci_X 
\supset (0)$ with $\sci_L/\sci_X \simeq \sco_L(l)$ for some $l \geq -1$. 
One has an isomorphism of $\sco_L$-modules $\sco_X \simeq \sco_L \oplus 
\sco_L(l)$ and if $1 \in \tH^0(\sco_X)$ and $e_1 \in \tH^0(\sco_X(-l))$ are the 
corresponding generators of the graded $\tH^0_\ast(\sco_L) = 
k[x_0,x_1]$-module $\tH^0_\ast(\sco_X)$ then the multiplicative structure of 
$\tH^0_\ast(\sco_X)$ is defined by $e_1^2 = 0$. The epimorphism $\sco_\piii \ra 
\sco_X$ induces a morphism of graded $k[x_0,x_1]$-algebras $S = 
\tH^0_\ast(\sco_\piii) \ra \tH^0_\ast(\sco_X)$ which is uniquely determined by 
the images of $x,\, y \in S$. Let's say that$\, :$ 
\[
x \mapsto ae_1\, ,\  y \mapsto be_1\, ,
\]   
with $a,\, b \in \tH^0(\sco_L(l+1))$. The epimorphism $\sci_L/\sci_L^2 \ra 
\sci_L/\sci_X$ can be identified with the morphism $\e_2 : 2\sco_L(-1) \ra 
\sco_L(l)$ defined by the matrix $(a\, ,\, b)$ and the surjectivity of 
$\e_2$ is equivalent to $a$ and $b$ being coprime. 

\vskip2mm 

Now, under these hypotheses, the homogeneous ideal $I(X) \subset S$ is 
generated by $F_2 := -bx+ay,\, x^2,\, xy,\, y^2$ (i.e., $X$ is the divisor $2L$ 
on the surface $F_2 = 0$, which is nonsingular along $L$) and one has graded 
free resolutions$\, :$ 
\begin{equation*}
0 \lra S(\text{--}l\text{--}4) 
\xra{\begin{pmatrix} \text{--}y\\ x\\ \text{--}b\\ a \end{pmatrix}} 
\begin{matrix} 2S(\text{--}l\text{--}3)\\ \oplus\\ 2S(-3) \end{matrix} 
\xra{\begin{pmatrix} 
x & y & 0 & 0\\
b & 0 &\text{--}y & 0\\
\text{--}a & b & x & \text{--}y\\
0 & \text{--}a & 0 & x
\end{pmatrix}} 
\begin{matrix} S(\text{--}l\text{--}2)\\ \oplus\\ 3S(-2) \end{matrix} 
\lra I(X) \lra 0\, , 
\end{equation*}
\begin{equation*}
0 \lra \begin{matrix} S(-2)\\ \oplus\\ S(l-2) \end{matrix} 
\xra{\begin{pmatrix} 
-y & 0\\
x & 0\\
b & -y\\
-a & x
\end{pmatrix}} 
\begin{matrix} 2S(-1)\\ \oplus\\ 2S(l-1) \end{matrix} 
\xra{\begin{pmatrix}
x & y & 0 & 0\\
-a & -b & x & y
\end{pmatrix}}
\begin{matrix} S\\ \oplus\\ S(l) \end{matrix} 
\lra \tH^0_\ast(\sco_X) \lra 0\, . 
\end{equation*}

Moreover, by a general result of Ferrand~\cite{f}, one has $\omega_X \simeq 
\sco_X(-l-2)$. 

\begin{remark}\label{R:etaeta1} 
We want to describe, for later use, the successive quotients of the descending 
chain of ideal sheaves $\sci_L \supset \sci_X \supset \sci_L^2 \supset 
\sci_L\sci_X \supset \sci_L^3$. 

(i) Of course, $\sci_L/\sci_X \simeq \sco_L(l)$ whence an exact sequence$\, :$ 
\[
0 \lra \sci_X \lra \sci_L \overset{\displaystyle \pi}{\lra} \sco_L(l) 
\lra 0 
\] 
where $\pi$ maps $x$ (resp., $y$) $\in \tH^0(\sci_L(1))$ to $a$ (resp., $b$) 
$\in \tH^0(\sco_L(l+1))$. Moreover, by B\u{a}nic\u{a} and Forster 
\cite[Prop.~2.2]{bf}, the canonical morphism $(\sci_L/\sci_X)^{\otimes \, j} 
\ra \sci_L^j/\sci_L^{j-1}\sci_X$ is an isomorphism, $\forall \, j \geq 1$, hence 
$\sci_L^2/\sci_L\sci_X \simeq \sco_L(2l)$. 

\vskip2mm 

(ii) The exact sequence$\, :$ 
\[
0 \lra \sci_X/\sci_L^2 \lra \sci_L/\sci_L^2 \lra \sci_L/\sci_X \lra 0  
\] 
can be identified with the exact sequence$\, :$ 
\[
0 \lra \sco_L(-l-2) \xra{\begin{pmatrix} -b\\ a \end{pmatrix}} 2\sco_L(-1) 
\xra{\displaystyle (a\, ,\, b)} \sco_L(l) \lra 0  
\]
whence an isomorphism $\sci_X/\sci_L^2 \simeq \sco_L(-l-2)$. One deduces an 
exact sequence$\, :$ 
\begin{equation}\label{E:il2ixeta} 
0 \lra \sci_L^2 \lra \sci_X \overset{\displaystyle \eta}{\lra} \sco_L(-l-2) 
\lra 0 
\end{equation} 
where $\eta$ maps $F_2 \in \tH^0(\sci_X(l+2))$ to $1 \in \tH^0(\sco_L)$. 

\vskip2mm 

(iii) The exact sequence$\, :$ 
\[
0 \lra \sci_L\sci_X/\sci_L^3 \lra \sci_L^2/\sci_L^3 \lra \sci_L^2/\sci_L\sci_X 
\lra 0 
\] 
can be identified with the sequence$\, :$ 
\[
0 \lra 2\sco_L(-l-3) \xra{\begin{pmatrix} -b & 0\\ a & -b\\ 0 & a 
\end{pmatrix}} 3\sco_L(-2) \xra{\displaystyle (a^2,\, ab\, ,\, b^2)} 
\sco_L(2l) \lra 0\, . 
\]
One deduces an isomorphism $\sci_L\sci_X/\sci_L^3 \simeq 2\sco_L(-l-3)$, 
whence an exact sequence$\, :$ 
\begin{equation}\label{E:il3ilixeta1} 
0 \lra \sci_L^3 \lra \sci_L\sci_X \overset{\displaystyle \eta_1}{\lra} 
2\sco_L(-l-3) \lra 0 
\end{equation}
where $\eta_1$ maps $xF_2$ (resp., $yF_2$) $\in \tH^0((\sci_L\sci_X)(l+3))$ to 
$(1,0)$ (resp., $(0,1)$) $\in \tH^0(2\sco_L)$. 

Notice that the ideal sheaf $\sci_L\sci_X$ defines a thick structure $W$ of 
degree 4 on $L$ which is l.c.i. at no point of $L$ (see Prop.~\ref{P:geniwcm}). 

\vskip2mm 

(iv) In the exact sequence of $\sco_L$-modules$\, :$ 
\[
0 \lra \sci_L^2/\sci_L\sci_X \lra \sci_X/\sci_L\sci_X \lra \sci_X/\sci_L^2 
\lra 0 
\] 
one has $\sci_X/\sci_L^2 \simeq \sco_L(-l-2)$ and $\sci_L^2/\sci_L\sci_X \simeq 
\sco_L(2l)$. Since $l \geq -1$, it follows that this sequence splits. One 
gets an isomorphism$\, :$ 
\[
\sci_X/\sci_L\sci_X \simeq \sco_L(-l-2) \oplus \sco_L(2l)\, . 
\]
Under this isomorphism, the image of $F_2 \in \tH^0(\sci_X(l+2))$ into 
$\tH^0((\sci_X/\sci_L\sci_X)(l+2))$ is identified with $(1,0)$ and the image 
of $x^2$ (resp., $xy$, resp., $y^2$) $\in \tH^0(\sci_X(2))$ is identified 
with $(0,a^2)$ (resp., $(0,ab)$, resp., $(0,b^2)$). 
\end{remark}

\subsection{Quasiprimitive structures of degree 3}\label{SS:q3} 
Let $Y$ be a locally CM curve of degree 3 supported on the line $L$. $Y$ is 
a quasiprimitive multiple structure on $L$ unless it is the first 
infinitesimal neighbourhood $L^{(1)}$ of $L$ (defined by the ideal sheaf  
$\sci_L^2$). Assume, from now on, that $Y$ is quasiprimitive.    
The CM filtration of $\sco_Y$ has three steps $\sco_Y \supset \sci_L/\sci_Y 
\supset \sci_X/\sci_Y \supset (0)$, where $X$ is a double structure on $L$ 
with $\sci_L/\sci_X \simeq \sco_L(l)$ for some $l \geq -1$, and where 
$\sci_X/\sci_Y \simeq \sco_L(2l+m)$ for some $m \geq 0$. 
One has an isomorphism of $\sco_L$-modules$\, :$ 
\[
\sco_Y \simeq \sco_L \oplus \sco_L(l) \oplus \sco_L(2l+m) 
\] 
and if $1 \in \tH^0(\sco_Y)$, $e_1 \in \tH^0(\sco_Y(-l))$, $e_2 \in 
\tH^0(\sco_Y(-2l-m))$ are the corresponding generators of the graded 
$\tH^0_\ast(\sco_L) = k[x_0,x_1]$-module $\tH^0_\ast(\sco_Y)$ then the 
multiplicative structure of $\tH^0_\ast(\sco_Y)$ is defined by$\, :$ 
\[
e_1^2 = pe_2\, ,\  e_1e_2 = 0\, ,\  e_2^2 = 0 
\]
with $0 \neq p \in \tH^0(\sco_L(m))$. The canonical epimorphism $\sco_\piii \ra 
\sco_Y$ induces a morphism of graded $k[x_0,x_1]$-algebras $S = 
\tH^0_\ast(\sco_\piii) \ra \tH^0_\ast(\sco_Y)$ uniquely determined by the 
images of $x,\, y \in S$. Let's say that$\, :$ 
\[
x \mapsto ae_1 + a^\prim e_2\, ,\  y \mapsto be_1 + b^\prim e_2 
\] 
with $a,\, b \in \tH^0(\sco_L(l+1))$ and $a^\prim ,\, b^\prim \in 
\tH^0(\sco_L(2l+m+1))$. The canonical epimorhism $\sci_L/\sci_L^3 \ra 
\sci_L/\sci_Y$ can be identified, in the category of $\sco_L$-modules, with 
the morphism $\e_3 : 2\sco_L(-1) \oplus 3\sco_L(-2) \ra \sco_L(l) \oplus 
\sco_L(2l+m)$ defined by the matrix$\, :$ 
\begin{equation}\label{E:matrixe3} 
\begin{pmatrix} 
a & b & 0 & 0 & 0\\ 
a^\prim & b^\prim & pa^2 & pab & pb^2 
\end{pmatrix}
\, . 
\end{equation} 
The surjectivity of $\e_3$ is equivalent to the fact that $a$ and $b$, on one 
hand, and $\Delta_1 := ab^\prim - a^\prim b$ and $p$, on the other hand, are 
coprime. Moreover, there exist elements $v_0,\, v_1,\, v_2 \in 
\tH^0(\sco_L(l+m))$ such that$\, :$ 
\[
\Delta_1 + v_0a^2 + v_1ab + v_2b^2 = 0 
\]
(one uses the exact sequence \eqref{E:a2abb2}). 

\vskip2mm 

Now, recalling the notation $F_2 := -bx+ay$, 
the homogeneous ideal $I(Y) \subset S$ of $Y$ is generated by$\, :$ 
\[
F_3 = pF_2 + v_0x^2 + v_1xy + v_2y^2\, ,\  xF_2\, ,\  yF_2\, ,\    
x^3\, ,\  x^2y\, ,\  xy^2\, ,\  y^3\, . 
\] 
$I(Y)$ admits a graded free resolution over $S$ of the form$\, :$ 
\[
0 \lra \begin{matrix} S(\text{--}l\text{--}m\text{--}4)\\ \oplus\\ 
2S(-l-5) \end{matrix} \overset{\displaystyle d_2}{\lra} 
\begin{matrix} 2S(\text{--}l\text{--}m\text{--}3)\\ \oplus\\ 
4S(-l-4)\\ \oplus\\ 3S(-4) \end{matrix} 
\overset{\displaystyle d_1}{\lra} 
\begin{matrix} S(\text{--}l\text{--}m\text{--}2)\\ \oplus\\ 
2S(-l-3)\\ \oplus\\ 4S(-3) \end{matrix} \lra I(Y) \lra 0 
\]
with $d_1$ and $d_2$ defined by the matrices$\, :$ 
\[
\begin{pmatrix} 
x & y & 0 & 0 & 0 & 0 & 0 & 0 & 0\\
-p & 0 & x & y & 0 & 0 & 0 & 0 & 0\\
0 & -p & 0 & 0 & x & y & 0 & 0 & 0\\
-v_0 & 0 & b & 0 & 0 & 0 & -y & 0 & 0\\
-v_1 & -v_0 & -a & b & b & 0 & x & -y & 0\\
-v_2 & -v_1 & 0 & -a & -a & b & 0 & x & -y\\
0 & -v_2 & 0 & 0 & 0 & -a & 0 & 0 & x 
\end{pmatrix}\, ,\  
\begin{pmatrix} 
-y & 0 & 0\\
x & 0 & 0\\
0 & -y & 0\\
-p & x & 0\\
p & 0 & -y\\
0 & 0 & x\\
v_0 & -b & 0\\
v_1 & a & -b\\
v_2 & 0 & a
\end{pmatrix}\, .
\]

Notice that in the \emph{primitive case}, i.e., when $m = 0$ hence $p = 1$, 
$I(Y)$ is generated only by $F_3$, $x^3$, $x^2y$, $xy^2$, $y^3$ (i.e., $Y$ is 
the divisor $3L$ on the surface $F_3 = 0$, which is nonsingular along $L$) 
and one can cancel a number of direct summands of the terms of the above free 
resolution of $I(Y)$. 

The graded $S$-module $\tH^0_\ast(\sco_Y)$ admits the following free 
resolution$\, :$ 
\[
0 \lra \begin{matrix} S(-2)\\ \oplus\\ S(l-2)\\ \oplus\\ S(2l+m\text{--}2) 
\end{matrix} \overset{\displaystyle \delta_2}{\lra} 
\begin{matrix} 2S(-1)\\ \oplus\\ 2S(l-1)\\ \oplus\\ 2S(2l+m\text{--}1) 
\end{matrix} \overset{\displaystyle \delta_1}{\lra} 
\begin{matrix} S\\ \oplus\\ S(l)\\ \oplus\\ S(2l+m) \end{matrix} 
\lra \tH^0_\ast(\sco_Y) \lra 0\, , 
\]
with $\delta_1$ and $\delta_2$ defined by the matrices$\, :$ 
\[
\begin{pmatrix}
x & y & 0 & 0 & 0 & 0\\
-a & -b & x & y & 0 & 0\\
-a^\prim & -b^\prim & -pa & -pb & x & y
\end{pmatrix}\, ,\  
\begin{pmatrix} 
-y & 0 & 0\\
x & 0 & 0\\
b & -y & 0\\
-a & x & 0\\
b^\prim & pb & -y\\
-a^\prim & -pa & x
\end{pmatrix}\, .
\]

\begin{remark}\label{R:rhoeta2} 
Keeping the above hypotheses and notation, we want to describe, for later use, 
the successive quotients of the descending chain of ideal sheaves $\sci_X 
\supset \sci_Y \supset \sci_L\sci_X$.  

The right square of the commutative diagram of $\sco_\piii$-modules$\, :$ 
\[
\begin{CD}
0 @>>> \sci_Y/\sci_L^3 @>>> \sci_L/\sci_L^3 @>>> \sci_L/\sci_Y @>>> 0\\
@. @VVV @\vert @VVV\\
0 @>>> \sci_X/\sci_L^3 @>>> \sci_L/\sci_L^3 @>>> \sci_L/\sci_X @>>> 0  
\end{CD}
\]   
can be identified, in the category of $\sco_L$-modules, with$\, :$ 
\[
\begin{CD}
2\sco_L(-1) \oplus 3\sco_L(-2) @>{\displaystyle \e_3}>> 
\sco_L(l) \oplus \sco_L(2l+m)\\
@\vert @VV{\displaystyle \text{pr}_1}V\\
2\sco_L(-1) \oplus 3\sco_L(-2) @>{\displaystyle {\widetilde \e}_2}>> 
\sco_L(l) 
\end{CD}
\]
where ${\widetilde \e}_2$ is defined by the matrix $(a\, ,\, b\, ,\, 0\, ,\, 
0\, ,\, 0)$ and $\e_3$ by the matrix \eqref{E:matrixe3}. One deduces an 
explicit isomorphism $\sco_L(2l+m) \izo \sci_X/\sci_Y$, whence an exact 
sequence$\, :$ 
\[
0 \lra \sci_Y \lra \sci_X \overset{\displaystyle \rho}{\lra} 
\sco_L(2l+m) \lra 0  
\] 
where $\rho$ maps $F_2 \in \tH^0(\sci_X(l+2))$ to $\Delta_1 \in 
\tH^0(\sco_L(3l + m + 2))$ and $x^2$ (resp., $xy$, resp., $y^2$) 
$\in \tH^0(\sci_X(2))$ to $pa^2$ (resp., $pab$, resp., $pb^2$) $\in 
\tH^0(\sco_L(2l+m+2))$. 

Now, the last exact sequence induces an exact sequence$\, :$ 
\[
0 \lra \sci_Y/\sci_L\sci_X \lra \sci_X/\sci_L\sci_X 
\overset{\displaystyle {\overline \rho}}{\lra} \sco_L(2l+m) \lra 0\, .  
\]
Taking into account Remark~\ref{R:etaeta1}(iv), this exact sequence can be 
identified with the sequence$\, :$ 
\[
0 \lra \sco_L(-l-m-2) \xra{\begin{pmatrix} p\\ -\Delta_1 \end{pmatrix}} 
\sco_L(-l-2) \oplus \sco_L(2l) \xra{\displaystyle (\Delta_1\, ,\, p)} 
\sco_L(2l+m) \lra 0\, . 
\] 
One deduces an isomorphism $\sci_Y/\sci_L\sci_X \simeq \sco_L(-l-m-2)$, 
whence an exact sequence$\, :$ 
\begin{equation}\label{E:ilixiyeta2} 
0 \lra \sci_L\sci_X \lra \sci_Y \overset{\displaystyle \eta_2}{\lra} 
\sco_L(-l-m-2) \lra 0 
\end{equation}
where $\eta_2$ maps $F_3 : = pF_2 + v_0x^2 + v_1xy + v_2y^2 \in 
\tH^0(\sci_Y(l+m+2))$ to $1 \in \tH^0(\sco_L)$ (recall that $v_0a^2 + v_1ab + 
v_2b^2 = -\Delta_1$). 
\end{remark}

\section{Curves with a multiple line as a component}
\label{A:multicomponent} 

We shall denote, in this appendix, the projective coordinates on $\piii$ by 
$x_0,\, x_1,\, x_2,\, x_3$, hence the projective coordinate ring of $\piii$ is 
$S = k[x_0,x_1,x_2,x_3]$. We consider the following points and lines in 
$\piii$$\, :$ 
\begin{gather*}
P_0 = (1:0:0:0)\, ,\  P_1 = (0:1:0:0)\, ,\  P_2 = (0:0:1:0)\, ,\  
P_3 = (0:0:0:1)\\
L_1 = \overline{P_0P_1} = \{x_2 = x_3 = 0\}\, ,\  
L_1^\prim = \overline{P_2P_3} = \{x_0 = x_1 = 0\}\\
L_2 = \overline{P_0P_2} = \{x_1 = x_3 = 0\}\, ,\  
L_2^\prim = \overline{P_1P_3} = \{x_0 = x_2 = 0\}\\
L_3 = \overline{P_0P_3} = \{x_1 = x_2 = 0\}\, ,\  
L_3^\prim = \overline{P_1P_2} = \{x_0 = x_3 = 0\}\, .
\end{gather*} 

We shall need the following easy results$\, :$ 

\begin{lemma}\label{L:zcupw} 
Let $Z$ and $W$ be closed subschemes of $\piii$ such that $Z \cap W = 
\emptyset$. Assume that $Z$ is arithmetically CM of pure codimension $2$.  
Consider minimal graded free resolutions$\, :$ 
\begin{gather*} 
0 \lra A_1 \lra A_0 \lra I(Z) \lra 0\\
0 \lra B_2 \lra B_1 \lra B_0 \lra I(W) \lra 0\, .
\end{gather*}    
Then$\, :$\quad \emph{(a)} $A_\bullet \otimes_S B_\bullet$ is a minimal graded 
free resolution of the ideal $I(Z)I(W)$. 

\emph{(b)} $I(Z \cup W) = I(Z)I(W)$ if, moreover, $W$ is arithmetically CM 
of pure codimension $2$.  
\end{lemma}

\begin{proof} 
(a) Tensorizing by $\sci_W$ the exact sequence$\, :$ 
\[
0 \lra {\widetilde A}_1 \lra {\widetilde A}_0 \lra \sco_\p \lra \sco_Z \lra 0
\]
one gets an exact sequence$\, :$ 
\[
0 \lra {\widetilde A}_1\otimes\sci_W \lra {\widetilde A}_0 \otimes \sci_W 
\lra \sci_W \lra \sco_Z \lra 0\, .
\] 
But $\Ker(\sci_W \ra \sco_Z) = \sci_{Z\cup W}$, whence an exact sequence$\, :$ 
\[
0 \lra {\widetilde A}_1\otimes\sci_W \lra {\widetilde A}_0 \otimes \sci_W 
\lra \sci_{Z\cup W} \lra 0\, . 
\]
Applying $\tH^0_\ast(-)$ to this exact sequence one gets an exact 
sequence$\, :$ 
\[
0 \lra A_1\otimes_SI(W) \lra A_0\otimes_SI(W) \lra I(Z\cup W)\, .
\]   
The image of $A_0\otimes_SI(W) \ra I(Z\cup W)$ is $I(Z)I(W)$. 

(b) Since $\tH^1_\ast(\sci_W) = 0$ it follows that $A_0\otimes_SI(W) \ra 
I(Z\cup W)$ is surjective. Notice that, under the hypothesis of (b), one has 
$B_2 = 0$.  
\end{proof}

\begin{remark}\label{R:js} 
Let $R = k[x_0,x_1]$ and let $J \subset R$ be a homogeneous, $R_+$-primary 
ideal. Then $JS$ is the homogeneous ideal of an arithmetically CM subscheme 
of $\piii$ supported on the line of equations $x_0 = x_1 = 0$. 

\emph{Indeed}, $x_2$, $x_3$ is an $S/JS$-regular sequence because $S/JS \simeq 
(R/J)[x_2, x_3]$.  
\end{remark}

\begin{remark}\label{R:monomial} 
(a) If $I$ (resp., $J$) is the ideal of $S$ generated by the monomials 
$m_1, \ldots ,m_s$ (resp., $n_1, \ldots , n_t$) then $I\cap J$ is generated 
by the monomials $[m_i,n_j]$, $1\leq i \leq s$, $1\leq j \leq t$, where 
$[f,g]$ denotes the least common multiple of $f$ and $g$. 

(b) If $I$ is the ideal of $S$ generated by the monomials 
$m_1, \ldots , m_s$ then$\, :$ 
\[
((m_1, m_2):I) = \bigcap_{i=3}^s\left(S\frac{[m_i,m_1]}{m_i} + 
S\frac{[m_i,m_2]}{m_i}\right)\, . 
\]  
\end{remark} 

\begin{remark}\label{R:basicdoublelinkage} 
We recall, from Lazarsfeld and Rao \cite{lr}, the notion of \emph{basic 
double linkage} (which is a particular case of the notion of \emph{liaison 
addition} introduced by Schwartau \cite{sw}). 

Let $Y$ be a locally CM subscheme of $\piii$, of pure codimension 2, and let 
$f,\, h \in S_+$ be coprime homogeneous polynomials, of degree $a$ and $c$, 
respectively, such that $f \in I(Y)$. 
Let $Z$ be the closed subscheme of $\piii$ defined by the homogeneous 
ideal $J := Sf + I(Y)h$. Set theoretically, one has $Z = Y \cup \{f = h = 
0\}$. Consider a graded free resolution of the homogeneous ideal $I(Y) 
\subset S$$\, :$ 
\[
0 \lra K_2 \overset{\displaystyle d_2}{\lra} K_1 
\overset{\displaystyle d_1}{\lra} K_0 
\overset{\displaystyle d_0}{\lra} I(Y) \lra 0\, .
\]   
One has an exact sequence$\, :$ 
\[
0 \lra S(-a-c) \xra{\begin{pmatrix} -h\\ f \end{pmatrix}}  
S(-a) \oplus I(Y)(-c) \xra{\displaystyle (f\, ,\, h)} 
J \lra 0\, .
\]
The morphism $S(-a) \ra I(Y)$ defined by $f$ lifts to a morphism $\phi : 
S(-a) \ra K_0$. One deduces that the sequence$\, :$ 
\[
0 \ra K_2(-c) \xra{\begin{pmatrix} 0\\ d_2(-c) \end{pmatrix}} 
\begin{matrix} S(-a-c)\\ \oplus\\ K_1(-c) \end{matrix}  
\xra{\begin{pmatrix} -h & 0\\ \phi(-c) & d_1(-c) \end{pmatrix}} 
\begin{matrix} S(-a)\\ \oplus\\ K_0(-c) \end{matrix} 
\xra{\displaystyle (f\, ,\, hd_0(-c))} J \ra 0
\]
is exact hence $Z$ is locally CM and $J = I(Z)$. 
\end{remark}

\begin{lemma}\label{L:ycupt}
Let $X \subseteq Y \subseteq Z$ be closed subschemes of a scheme $P$ and let 
$T$ be another closed subscheme of $P$. Let us denote by $\phi$ the composite 
morphism $\sci_{Y \cup T} := \sci_Y \cap \sci_T \hookrightarrow \sci_Y 
\twoheadrightarrow \sci_Y/\sci_Z$. Recall that $\sco_{Z \cup T}$ embeds into 
$\sco_Z \oplus \sco_T$. Then$\, :$ 

\emph{(a)} One has exact sequences$\, :$  
\begin{gather*}
0 \lra \sci_{Z \cup T} \lra \sci_{Y \cup T} \overset{\displaystyle \phi}{\lra}  
\sci_Y/\sci_Z\, ,\\
0 \lra {\fam0 Im}\, \phi \times \{0\} \lra \sco_{Z \cup T} \lra 
\sco_{Y \cup T} \lra 0\, . 
\end{gather*} 

\emph{(b)} If $Z \cap T = \emptyset$ then $\phi$ is an epimorphism. 

\emph{(c)} If $\sci_X\sci_Y \subseteq \sci_Z$ then ${\fam0 Coker}\,  \phi$ 
is an $\sco_{X \cap T}$-module. 
\end{lemma}

\begin{proof}
(a) One has $\Ker \phi = \sci_Z \cap (\sci_Y \cap \sci_T) = \sci_Z \cap 
\sci_T =: \sci_{Z \cup T}$ whence the first exact sequence. For the second exact 
sequence, notice that the image of the composite morphism $\sci_{Y \cup T} 
\hookrightarrow \sco_P \ra \sco_Z$ coincides with $\text{Im}\, \phi$, and that 
the composite morphism $\sci_{Y \cup T} \hookrightarrow \sco_P \ra \sco_T$ is 0 
(because $\sci_{Y \cup T} \subseteq \sci_T$).   

(b) $\text{Supp}(\sci_Y/\sci_Z) \subseteq Z$ and $\phi$ is surjective   
on the open neighbourhood $P \setminus T$ of $Z$. 

(c) $\Cok \phi \simeq \sci_Y/(\sci_Z + (\sci_Y \cap \sci_T))$ is annihilated 
by $\sci_X + \sci_T =: \sci_{X \cap T}$. 
\end{proof}

\subsection{A double line union a line}\label{SS:doublecupaline} 
Let $X$ be a (primitive) double structure on the line $L_1 \subset \piii$. 
We use the results and notation from Subsection~\ref{SS:p2} (with $x$ 
replaced by $x_2$, $y$ by $x_3$ and $L$ by $L_1$).  

Recall that $\sco_X \simeq \sco_{L_1} \oplus \sco_{L_1}(l)$ as an 
$\sco_{L_1}$-module, for some $l \geq -1$, and that if $1 \in \tH^0(\sco_X)$ 
and $e_1 \in \tH^0(\sco_X(-l))$ are the corresponding generators of the graded 
$\tH^0_\ast(\sco_{L_1}) = k[x_0,x_1]$-module $\tH^0_\ast(\sco_X)$ then the 
multiplicative structure of $\tH^0_\ast(\sco_X)$ is defined by $e_1^2 = 0$ 
and the morphism of $k[x_0,x_1]$-algebras $S=\tH^0_\ast(\sco_\piii) \ra 
\tH^0_\ast(\sco_X)$ maps $x_2$ to $ae_1$ and $x_3$ to $be_1$, for some coprime 
$a,\, b \in \tH^0(\sco_{L_1}(l+1))$.

Recall also, from Remark~\ref{R:etaeta1}, that one has exact sequences$\, :$ 
\begin{gather*}
0 \lra \sci_X \lra \sci_{L_1} \overset{\displaystyle \pi}{\lra} \sco_{L_1}(l) 
\lra 0\, ,\\
0 \lra \sci_{L_1}^2 \lra \sci_X \overset{\displaystyle \eta}{\lra} 
\sco_{L_1}(-l-2) \lra 0\, ,   
\end{gather*}   
where $\pi$ maps $x_2,\, x_3 \in \tH^0(\sci_{L_1}(1))$ to 
$a,\, b \in \tH^0(\sco_{L_1}(l+1))$, respectively, and where 
$\eta$ maps $F_2 := -bx_2 + ax_3 \in \tH^0(\sci_X(l+2))$ to $1 \in 
\tH^0(\sco_{L_1})$. It follows that $I(X) = SF_2 + I(L_1)^2$. 

\vskip2mm 

Consider, now, another line $L \subset \piii$ and let $\phi$ (resp., $\psi$) 
denote the composite morphism$\, :$ 
\begin{gather*}
\sci_{L_1 \cup L} \ra \sci_{L_1} \overset{\displaystyle \pi}{\lra} 
\sco_{L_1}(l)\\  
(\text{resp.,}\  \  \sci_{X \cup L} \ra \sci_X \overset{\displaystyle \eta}{\lra} 
\sco_{L_1}(-l-2))\, . 
\end{gather*}
It follows, from Lemma~\ref{L:ycupt}, that one has exact sequences$\, :$ 
\begin{gather*}
0 \lra \sci_{L_1^{(1)} \cup L} \lra \sci_{X \cup L} 
\overset{\displaystyle \psi}{\lra} \sco_{L_1}(-l-2)\, ,\\   
0 \lra \text{Im}\, \phi \times \{0\} \lra \sco_{X \cup L} \lra \sco_{L_1 \cup L} 
\lra 0\, . 
\end{gather*}
Moreover, $\Cok \psi$ is an $\sco_{L_1 \cap L}$-module. Lemma~\ref{L:zcupw} 
implies that$\, :$  

\begin{lemma}\label{L:l1(1)cupl1prim} 
The homogeneous ideal $I(L_1^{(1)} \cup L_1^\prim)$ of $L_1^{(1)} \cup L_1^\prim$ 
admits the following graded free resolution$\, :$ 
\[
0 \lra 2S(-5) \overset{\displaystyle d_2}{\lra} 7S(-4) 
\overset{\displaystyle d_1}{\lra} 6S(-3) 
\overset{\displaystyle d_0}{\lra}  
I(L_1^{(1)} \cup L_1^\prim) \lra 0 
\]
with the differentials $d_0,\, d_1,\, d_2$ defined by the matrices$\, :$ 
\begin{gather*} 
(x_0x_2^2\, ,\,  x_0x_2x_3\, ,\,  x_0x_3^2\, ,\,  x_1x_2^2\, ,\,  x_1x_2x_3\, ,
\,  x_1x_3^2)\, ,\\
\begin{pmatrix}
-x_3 & 0    & 0    & 0    & -x_1 & 0    & 0\\ 
 x_2 & -x_3 & 0    & 0    & 0    & -x_1 & 0\\ 
 0   &  x_2 & 0    & 0    & 0    & 0    & -x_1\\ 
 0   &  0   & -x_3 & 0    & x_0  & 0    & 0\\
 0   &  0   &  x_2 & -x_3 & 0    & x_0  & 0\\
 0   &  0   & 0    &  x_2 & 0    & 0    & x_0 
\end{pmatrix}
\, ,\  
\begin{pmatrix} 
-x_1 & 0\\
 0   & -x_1\\
 x_0 &  0\\
 0   & x_0\\
 x_3 & 0\\
-x_2 & x_3\\
 0   & -x_2
\end{pmatrix}
\end{gather*}
\end{lemma}

\begin{lemma}\label{L:l=-1deg2} 
If $l = -1$ $($hence $a,\, b \in k$ and $F_2 = -bx_2 + ax_3$ is a linear 
form$)$ then $X$ is the divisor $2L_1$ on the plane $H\supset L_1$ of equation 
$F_2 = 0$. Actually, if $b \neq 0$ then $kF_2 + kx_3 = kx_2 + kx_3$ and 
$I(X) = (F_2,\, x_3^2)$. 
\qed  
\end{lemma}

\begin{prop}\label{P:genixcupl1prim} 
If $l \geq 0$ then $I(X\cup L_1^\prim) = SF_2 + (x_0,\, x_1)(x_2,\, x_3)^2$. 
\end{prop}

\begin{proof} 
By what has been said at the beginning of the subsection, one has an exact 
sequence$\, :$ 
\[
0 \lra \sci_{L_1^{(1)} \cup L_1^\prim} \lra \sci_{X\cup L_1^\prim} 
\overset{\displaystyle \psi}{\lra} \sco_{L_1}(-l-2) \lra 0\, .
\]
Our hypothesis implies that $F_2 \in I(X\cup L_1^\prim)$ from which one deduces 
that the sequence$\, :$ 
\[
0 \lra I(L_1^{(1)} \cup L_1^\prim) \lra I(X\cup L_1^\prim) 
\xra{\displaystyle \tH^0_\ast(\psi)} S(L_1)(-l-2) \lra 0
\]
is exact. Lemma~\ref{L:l1(1)cupl1prim} implies that 
$I(L_1^{(1)} \cup L_1^\prim) = (x_0,\, x_1)(x_2,\, x_3)^2$. 
Notice that the last exact sequence allows one to get a graded free 
resolution of $I(X \cup L_1^\prim)$.  
\end{proof}

\begin{lemma}\label{L:l1(1)cupl} 
Let $L$ be a line contained in the plane $x_3 = 0$ and different from $L_1$. 
It is given by equations of the form $\ell = x_3 = 0$, where $\ell = c_0x_0 
+ c_1x_1 + c_2 x_2$ and at least one of the coefficients $c_0$, $c_1$ is 
non-zero. Using a linear change of coordinates invariating $x_2$, $x_3$ and the 
vector space $kx_0 + kx_1$, one can assume that $\ell = x_1 + cx_2$, $c \in k$. 
Then the homogeneous ideal $I(L_1^{(1)} \cup L)$ of $L_1^{(1)} \cup L$ admits 
the following graded free resolution$\, :$ 
\[
0 \lra \begin{matrix} S(-3)\\ \oplus\\ S(-4) \end{matrix} 
\xra{\begin{pmatrix} 
-x_3 & -\ell x_2\\
x_2 & 0\\
0 & x_3
\end{pmatrix}} 
\begin{matrix} 2S(-2)\\ \oplus\\ S(-3) \end{matrix} 
\xra{\displaystyle (x_2x_3\, ,\, x_3^2\, ,\, \ell x_2^2)} 
I(L_1^{(1)} \cup L) \lra 0\, .
\]
\end{lemma} 

\begin{proof}
One has $I(L_1^{(1)} \cup L) = (x_2^2,\, x_2x_3,\, x_3^2) \cap (\ell,\, x_3) = 
(x_2x_3,\, x_3^2,\, \ell x_2^2)$. If $Z^\prim$ is the curve directly linked to 
$L_1^{(1)} \cup L$ by the complete intersection defined by $x_3^2$ and 
$\ell x_2^2$ then $I(Z^\prim) = (x_3,\, \ell x_2)$ hence $Z^\prim = L_1 \cup L$. 
One can apply, now, Ferrand's result about liaison. 
\end{proof}

\begin{prop}\label{P:genixcupl} 
Let $X$ be the double structure on the line $L_1$ considered at the 
beginning of this subsection and let $L$ be the line of equations $\ell = 
x_3 =0$, $\ell := x_1+cx_2$, considered in Lemma~\ref{L:l1(1)cupl}. 

\emph{(a)} If $x_1 \mid b$, i.e., if $b = x_1b_1$ then 
$I(X\cup L) = (F_2-cb_1x_2^2,\, x_2x_3,\, x_3^2,\, \ell x_2^2)$. 

\emph{(b)} If $x_1 \nmid b$ then $I(X\cup L) = (\ell F_2,\, x_2x_3,\, 
x_3^2,\, \ell x_2^2)$.   
\end{prop}

\begin{proof} 
By what has been said at the beginning of the subsection, one has an exact 
sequence$\, :$ 
\[
0 \lra \sci_{L_1^{(1)}\cup L} \lra \sci_{X\cup L} \overset{\displaystyle \psi}{\lra} 
\sco_{L_1}(-l-2)\, .
\]
Moreover, $\Cok \psi$ is an $\sco_{L_1 \cap L} = \sco_{\{P_0\}}$-module. 
Actually, since $\ell F_2 \in I(X\cup L)$ and $\psi(\ell F_2) = 
\eta(\ell F_2) = x_1 \in \tH^0(\sco_{L_1}(1))$, it follows that 
$x_1\sco_{L_1}(-l-3) \subseteq \text{Im}\, \psi \subseteq \sco_{L_1}(-l-2)$. 

\vskip2mm 

(a) Since $b = x_1b_1$, $F_2 - cb_1x_2^2$ vanishes on $L$ hence belongs to 
$I(X \cup L)$. Since $\psi(F_2 - cb_1x_2^2) = \eta(F_2 - cb_1x_2^2) = 1 \in 
\tH^0(\sco_{L_1})$, it follows that $\psi$ is an epimorphism and that one has 
an exact sequence$\, :$ 
\begin{equation}\label{E:ixcupla}
0 \lra I(L_1^{(1)} \cup L) \lra I(X \cup L) 
\xra{\displaystyle \tH^0_\ast(\psi)} S(L_1)(-l-2) \lra 0\, .
\end{equation}  

(b) We prove, firstly, the following$\, :$ 

\vskip2mm

\noindent
{\bf Claim.}\quad $\text{Im}\, \psi = x_1\sco_{L_1}(-l-3)$. 

\vskip2mm

\noindent
\emph{Indeed}, assume that $\text{Im}\, \psi = \sco_{L_1}(-l-2)$. Since 
$\tH^1_\ast(\sci_{L_1^{(1)} \cup L}) = 0$ by Lemma~\ref{L:l1(1)cupl}, there exists 
an element $f \in \tH^0(\sci_{X \cup L}(l+2))$ such that $\psi(f) = 1 \in 
\tH^0(\sco_{L_1})$. But $\psi(f) = \eta(f)$ hence  
$f = F_2 + f_0x_2^2 + f_1x_2x_3 + f_2x_3^2$, 
with $f_0,\, f_1,\, f_2 \in S_l$. Since $f \vb L = 0$ it follows that$\, :$ 
\[
-(b \vb L)x_2 + (f_0 \vb L)x_2^2 = 0\, . 
\tag{$*$}
\]
The composite map $k[x_0,x_2] \hookrightarrow S \twoheadrightarrow S(L)$ is 
bijective and, with respect to this identification, $b \vb L = 
b(x_0,-cx_2)$. Since $x_1 \nmid b$ it follows that $x_2 \nmid b(x_0,-cx_2)$ 
which \emph{contradicts} relation $(*)$. The claim is proven.  

\vskip2mm

\noindent
Now, according to the Claim, $\psi$ can be written as a composite morphism 
$\sci_{X \cup L} \xra{\psi^\prim} \sco_{L_1}(-l-3) \overset{x_1}{\lra} 
\sco_{L_1}(-l-2)$ with $\psi^\prim(\ell F_2) = 1 \in \tH^0(\sco_{L_1})$. One 
deduces the existence of an exact sequence$\, :$ 
\begin{equation}\label{E:ixcuplb}
0 \lra I(L_1^{(1)} \cup L) \lra I(X \cup L) 
\xra{\displaystyle \tH^0_\ast(\psi^\prim)} S(L_1)(-l-3) \lra 0\, .
\end{equation} 

Notice that the exact sequences \eqref{E:ixcupla} and \eqref{E:ixcuplb} can 
be used not only to describe a system of generators of $I(X \cup L)$ but also 
to get a graded free resolution of this ideal. 
\end{proof}  

\begin{prop}\label{P:reshoxcupl} 
Under the hypothesis of Prop.~\ref{P:genixcupl}$\, :$ 

\emph{(a)} If $x_1 \mid b$ then ${\fam0 H}^0_\ast(\sco_{X\cup L})$ admits the 
following graded free resolution$\, :$ 
\[
0 \ra \begin{matrix} S(-3)\\ \oplus\\ S(l\text{--}3) \end{matrix} 
\xra{\begin{pmatrix} 
\text{--}\ell x_2 & 0\\
x_3 & 0\\
\text{--}\ell b_1 & \text{--}x_3\\
a & x_2
\end{pmatrix}} 
\begin{matrix} S(-1)\\ \oplus\\ S(-2)\\ \oplus\\ 2S(l\text{--}2) \end{matrix} 
\xra{\begin{pmatrix} 
x_3 & \ell x_2 & 0 & 0\\
\text{--}b_1 & \text{--}a & x_2 & x_3
\end{pmatrix}} 
\begin{matrix} S\\ \oplus\\ S(l\text{--}1) \end{matrix} 
\xra{\displaystyle (1\, ,\, x_1e_1)}  
{\fam0 H}^0_\ast(\sco_{X\cup L}) \ra 0  
\]

\emph{(b)} If $x_1 \nmid b$ then ${\fam0 H}^0_\ast(\sco_{X\cup L})$ admits the 
following graded free resolution$\, :$ 
\[
0 \ra \begin{matrix} S(-3)\\ \oplus\\ S(l\text{--}2) \end{matrix} 
\xra{\begin{pmatrix} 
\text{--}\ell x_2 & 0\\
x_3 & 0\\
\text{--}\ell b & \text{--}x_3\\
x_1a & x_2
\end{pmatrix}} 
\begin{matrix} S(-1)\\ \oplus\\ S(-2)\\ \oplus\\ 2S(l\text{--}1) \end{matrix} 
\xra{\begin{pmatrix} 
x_3 & \ell x_2 & 0 & 0\\
\text{--}b & \text{--}x_1a & x_2 & x_3
\end{pmatrix}} 
\begin{matrix} S\\ \oplus\\ S(l) \end{matrix} 
\xra{\displaystyle (1\, ,\, e_1)} 
{\fam0 H}^0_\ast(\sco_{X\cup L}) \ra 0\, . 
\]
\end{prop} 

\begin{proof}
By what has been said at the beginning of this subsection, one has an exact 
sequence$\, :$ 
\[
0 \lra \text{Im}\, \phi \times \{0\} \lra \sco_{X \cup L} \lra \sco_{L_1 \cup L} 
\lra 0\, .
\]
Since $I(L_1 \cup L) = (x_3,\, \ell x_2)$, the image of $\phi$ coincides with 
the image of the composite morphism$\, :$ 
\[
\sco_\p(-1) \oplus \sco_\p(-2) \xra{\displaystyle (x_3\, ,\, \ell x_2)} 
\sci_{L_1} \overset{\displaystyle \pi}{\lra} \sco_{L_1}(l) 
\]
which, in turn, coincides with the image of the composite morphism$\, :$ 
\[
\sco_{L_1}(-1) \oplus \sco_{L_1}(-2) 
\xra{\begin{pmatrix} 0 & x_1\\ 1 & 0 \end{pmatrix}} 2\sco_{L_1}(-1) 
\xra{\displaystyle (a\, ,\, b)} \sco_{L_1}(l) 
\]
i.e., with the image of the morphism $(b\, ,\, x_1a) : \sco_{L_1}(-1) \oplus 
\sco_{L_1}(-2) \ra \sco_{L_1}(l)$. 

(a) In this case $\text{Im}\, \phi = x_1\sco_{L_1}(l-1)$. Since the graded 
$S$-module $\tH^0_\ast(\sco_{L_1 \cup L})$ is generated by $1 \in 
\tH^0(\sco_{L_1 \cup L})$, the morphism $\tH^0_\ast(\sco_{X \cup L}) \ra 
\tH^0_\ast(\sco_{L_1 \cup L})$ is surjective. One deduces the existence of an 
exact sequence$\, :$ 
\[
0 \lra S(L_1)(l-1) \lra \tH^0_\ast(\sco_{X \cup L}) \lra 
\tH^0_\ast(\sco_{L_1 \cup L}) \lra 0
\]
where the left morphism maps $1 \in S(L_1)$ to the element of 
$\tH^0(\sco_{X \cup L}(-l+1))$ whose image into 
$\tH^0(\sco_X(-l+1)) \oplus \tH^0(\sco_L(-l+1))$ is $(x_1e_1,0)$. 

(b) In this case $\text{Im}\, \phi = \sco_{L_1}(l)$ and one deduces, as in 
(a), the existence of an exact sequence$\, :$ 
\[
0 \lra S(L_1)(l) \lra \tH^0_\ast(\sco_{X \cup L}) \lra 
\tH^0_\ast(\sco_{L_1 \cup L}) \lra 0
\]
where the left morphism maps $1 \in S(L_1)$ to the element of 
$\tH^0(\sco_{X \cup L}(-l))$ whose image into 
$\tH^0(\sco_X(-l)) \oplus \tH^0(\sco_L(-l))$ is $(e_1,0)$.  
\end{proof}

\subsection{A triple line union a line}\label{SS:triplecupaline} 
Let $Y$ be a quasiprimitive triple structure on the line $L_1 \subset \piii$. 
We use the results and notation from Subsection~\ref{SS:q3} (with $x$ replaced 
by $x_2$, $y$ by $x_3$ and $L$ by $L_1$).  

Recall that, as an $\sco_{L_1}$-module, 
\[
\sco_Y \simeq \sco_{L_1} \oplus \sco_{L_1}(l) \oplus \sco_{L_1}(2l+m) 
\] 
for some integers $l \geq -1$ and $m \geq 0$, and that if $1 \in 
\tH^0(\sco_Y)$, $e_1 \in \tH^0(\sco_Y(-l))$, $e_2 \in \tH^0(\sco_Y(-2l-m))$ 
are the corresponding generators of the 
$\tH^0_\ast(\sco_{L_1}) = k[x_0,x_1]$-module $\tH^0_\ast(\sco_Y)$ then the 
multiplicative structure of $\tH^0_\ast(\sco_Y)$ is defined by $e_1^2 = pe_2$, 
$e_1e_2 = 0$, $e_2^2 = 0$ for some $0 \neq p \in \tH^0(\sco_{L_1}(m))$ and that 
the morphism of graded $k[x_0,x_1]$-algebras $S=\tH^0_\ast(\sco_\piii) \ra 
\tH^0_\ast(\sco_Y)$ maps $x_2$ to $ae_1 + a^\prim e_2$ and $x_3$ to $be_1 + 
b^\prim e_2$, for some $a,\, b \in \tH^0(\sco_{L_1}(l+1))$, $a^\prim,\, b^\prim \in 
\tH^0(\sco_{L_1}(2l+m+1))$. $a$ and $b$, on one hand, and $\Delta_1 := ab^\prim 
- a^\prim b$ and $p$, on the other hand, must be coprime. Moreover, there exist 
$v_0,\, v_1,\, v_2 \in \tH^0(\sco_{L_1}(l+m))$ such that 
\[
\Delta_1 + v_0a^2 + v_1ab + v_2b^2 = 0\, . 
\]

$Y$ contains, as a subscheme, the double structure $X$ on $L_1$ defined by 
the exact sequence$\, :$ 
\[
0 \lra \sci_X \lra \sci_{L_1} \overset{\displaystyle \pi}{\lra} 
\sco_L(l) \lra 0 
\]
where $\pi$ maps $x_2,\, x_3 \in \tH^0(\sci_{L_1}(1))$ to $a,\, b \in 
\tH^0(\sco_{L_1}(l+1))$, respectively, and is contained in the quadruple thick 
structure $W$ on $L_1$ defined by the ideal sheaf $\sci_{L_1}\sci_X$ (see 
Prop.~\ref{P:geniwcm}). Recall that the polynomial $F_2 := -bx_2 + ax_3$ 
belongs to $\tH^0(\sci_X(l+2))$ and that the polynomial $F_3 := pF_2 + v_0x_2^2 
+ v_1x_2x_3 + v_2x_3^2$ belongs to $\tH^0(\sci_Y(l+m+2))$. Then, according to  
Remark~\ref{R:etaeta1}(iii) and to Remark~\ref{R:rhoeta2}, one has exact 
sequences$\, :$ 
\begin{gather*}
0 \lra \sci_{L_1}^3 \lra \sci_W \overset{\displaystyle \eta_1}{\lra} 
2\sco_{L_1}(-l-3) \lra 0\, ,\\ 
0 \lra \sci_W \lra \sci_Y \overset{\displaystyle \eta_2}{\lra} 
\sco_{L_1}(-l-m-2) \lra 0\, ,\\ 
0 \lra \sci_Y \lra \sci_X \overset{\displaystyle \rho}{\lra} 
\sco_{L_1}(2l+m) \lra 0   
\end{gather*}
where $\eta_1$ maps $x_2F_2$ (resp., $x_3F_2$) $\in \tH^0(\sci_W(l+3))$ 
to $(1,0)$ (resp., $(0,1)$) $\in \tH^0(2\sco_{L_1})$, where $\eta_2$ maps 
$F_3 \in \tH^0(\sci_Y(l+m+2))$ to $1 \in \tH^0(\sco_{L_1})$, and where $\rho$ 
maps $F_2 \in \tH^0(\sci_X(l+2))$ to $\Delta_1 \in \tH^0(\sco_{L_1}(3l +m + 2))$ 
and $x_2^2,\, x_2x_3,\, x_3^2 \in \tH^0(\sci_X(2))$ to $pa^2,\, pab,\, pb^2 \in 
\tH^0(\sco_{L_1}(2l+m+2))$. It follows, in particular, that $I(W) = I(L_1)I(X)$ 
and that $I(Y) = SF_3 + I(W)$. 

\vskip2mm  

Consider, now, another line $L \subset \piii$. Let us denote by $\psi_1$, 
$\psi_2$ and $\phi$ the following composite morphisms$\, :$ 
\begin{gather*}
\sci_{W\cup L} \lra \sci_W \overset{\displaystyle \eta_1}{\lra} 
2\sco_{L_1}(-l-3)\, ,\\   
\sci_{Y\cup L} \lra \sci_Y \overset{\displaystyle \eta_2}{\lra} 
\sco_{L_1}(-l-m-2)\, ,\\
\sci_{X\cup L} \lra \sci_X \overset{\displaystyle \rho}{\lra} 
\sco_{L_1}(2l+m)\, .
\end{gather*}
It follows from Lemma~\ref{L:ycupt} that one has exact sequences$\, :$ 
\begin{gather*}
0 \lra \sci_{L_1^{(2)}\cup L} \lra \sci_{W\cup L} 
\overset{\displaystyle \psi_1}{\lra} 2\sco_{L_1}(-l-3)\, ,\\ 
0 \lra \sci_{W\cup L} \lra \sci_{Y\cup L} 
\overset{\displaystyle \psi_2}{\lra} \sco_{L_1}(-l-m-2)\, ,\\
0 \lra \text{Im}\, \phi \times \{0\} \lra \sco_{Y\cup L} \lra \sco_{X\cup L} 
\lra 0\, .
\end{gather*}
Moreover, $\Cok \psi_1$, $\Cok \psi_2$ and $\Cok \phi$ are 
$\sco_{L_1\cap L}$-modules. 

\vskip2mm 

The following lemma follows easily from Lemma~\ref{L:zcupw}. 

\begin{lemma}\label{L:l1(2)cupl1prim} 
Let $L_1^{(2)}$ denote the second infinitesimal neighbourhood of $L_1$ in 
$\piii$ defined by the ideal sheaf $\sci_{L_1}^3$. Then the homogeneous ideal 
$I(L_1^{(2)} \cup L_1^\prim)$ of $L_1^{(2)}\cup L_1^\prim$
admits the following graded free resolution$\, :$ 
\[
0 \lra 3S(-6) \overset{\displaystyle d_2}{\lra} 10S(-5) 
\overset{\displaystyle d_1}{\lra} 8S(-4) 
\overset{\displaystyle d_0}{\lra} I(L_1^{(2)} \cup L_1^\prim) \lra 0 
\]
with $d_0,\, d_1,\, d_2$ defined by the matrices$\, :$
\begin{gather*}
(x_0x_2^3\, ,\,  x_0x_2^2x_3\, ,\,  x_0x_2x_3^2\, ,\,  x_0x_3^3\, ,\,   
x_1x_2^3\, ,\,  x_1x_2^2x_3\, ,\,  x_1x_2x_3^2\, ,\,  x_1x_3^3)\, ,\\ 
\begin{pmatrix} 
-x_3 & 0     & 0   & 0    & 0   & 0   &   -x_1 & 0    & 0   &  0\\  
x_2  & -x_3  & 0   & 0    & 0   & 0   &   0    & -x_1 & 0    & 0\\ 
0   &  x_2   & -x_3 & 0   & 0   & 0    &  0    &  0   & -x_1 & 0\\    
0   &  0     & x_2  & 0   & 0   & 0    &  0    &  0   &  0  &  -x_1\\ 
0   &  0     &  0  & -x_3 &  0  & 0    &  x_0  &  0   &  0   & 0\\ 
0   &  0     &  0  &  x_2 & -x_3 & 0   &  0    &  x_0 &  0   & 0\\ 
0   &  0     &  0  &  0   & x_2  & -x_3 &  0   &  0   &  x_0 & 0\\ 
0   &  0     &  0  &  0   &  0   &  x_2 &  0   &  0   &  0   & x_0 
\end{pmatrix}
\, ,\  
\begin{pmatrix} 
-x_1 & 0    & 0\\    
 0   & -x_1 & 0\\    
 0   & 0    & -x_1\\ 
 x_0 & 0    & 0\\
 0   & x_0  & 0\\
 0   &  0   & x_0\\
 x_3 &  0   & 0\\
-x_2 &  x_3 & 0\\
 0   & -x_2 & x_3\\
 0   &  0   & -x_2
\end{pmatrix}
\end{gather*}
\end{lemma}

\begin{lemma}\label{L:l=-1deg3} 
Assume that $l = -1$ and $b \neq 0$. Recall Lemma~\ref{L:l=-1deg2}. Then$\, :$ 

\emph{(a)} One can assume that $v_0 = v_1 = 0$ hence that $F_3 = pF_2 + 
v_2x_3^2$ with $p$ and $v_2 := -\Delta_1/b^2$ coprime$\, ;$ 

\emph{(b)} $I(W) = (F_2^2,\, x_3F_2,\, x_3^3)$ and $W$ is directly linked to 
$X$ by the complete intersection of type $(2,3)$ defined by $F_2^2$ and 
$x_3^3$ hence $W$ is arithmetically CM and $I(W)$ admits the following graded 
free resolution$\, :$ 
\[
0 \lra \begin{matrix} S(-3)\\ \oplus\\ S(-4) \end{matrix} 
\xra{\begin{pmatrix} 
-x_3 & 0\\
F_2 & -x_3^2\\
0 & F_2
\end{pmatrix}} \begin{matrix} 2S(-2)\\ \oplus\\ S(-3)\end{matrix} 
\xra{\displaystyle (F_2^2\, ,\, x_3F_2\, ,\, x_3^3)} I(W) \lra 0\, .
\] 
\end{lemma} 

\begin{lemma}\label{L:l=-1m=1deg3} 
If $l = -1$ and $m = 1$ then $Y$ is directly linked to $L_1$ by a complete 
intersection of type $(2,2)$ hence it is arithmetically CM. Actually, $Y$ is 
the Weil divisor $3L_1$ on the quadric cone $F_3 = 0$. 
\end{lemma}

\begin{proof}
Assume that $b \neq 0$. Using the notation from Lemma~\ref{L:l=-1deg3}, $p$ is 
a linear form and $v_2$ is a constant. Since $p$ and $v_2$ are coprime, one has 
$v_2 \neq 0$. It follows that $I(Y) = SF_3 + I(W) = (F_3,\, F_2^2, x_3F_2)$ 
hence $Y$ is directly linked to $L_1$ by the complete intersection defined by 
$F_3$ and $F_2^2$. One gets, consequently, the following 
graded free resolution$\, :$ 
\[
0 \lra 2S(-3) \xra{\begin{pmatrix} 
F_2 & 0\\
-p & -x_3\\
-v_2x_3 & F_2
\end{pmatrix}} 3S(-2) \lra I(Y) \lra 0\, .
\qedhere
\]  
\end{proof}

\begin{prop}\label{P:geniycupl1prim} 
\emph{(a)} If $l = -1$ and $m \geq 2$ then$\, :$ 
\[
I(Y\cup L_1^\prim) = SF_3 + (x_0,\, x_1)(x_2,\, x_3)I(X)\, .
\]

\emph{(b)} If $l \geq 0$ then$\, :$ 
\[
I(Y\cup L_1^\prim) = SF_3 + Sx_2F_2 + Sx_3F_2 + 
(x_0,\, x_1)(x_2,\, x_3)^3\, . 
\]
except when $l = m = 0$ and $L_1^\prim$ is not contained in the quadric 
surface $Q$ of equation $F_3 = 0$ $($which means that at least one of the 
constants $v_0$, $v_1$, $v_2$ is non-zero$)$ in which case$\, :$
\[
I(Y\cup L_1^\prim) = Sx_0F_3 + Sx_1F_3 + Sx_2F_2 + Sx_3F_2 + 
(x_0,\, x_1)(x_2,\, x_3)^3\, . 
\]
\end{prop}

\begin{proof}
(a) One uses the exact sequence$\, :$ 
\[
0 \lra \sci_{W\cup L_1^\prim} \lra \sci_{Y\cup L_1^\prim} 
\overset{\displaystyle \psi_2}{\lra} \sco_{L_1}(-m-1) \lra 0 
\]
(look at the beginning of this subsection). $W$ is arithmetically CM  
(by Lemma~\ref{L:l=-1deg3}(b)) hence $I(W \cup L_1^\prim) = I(W)I(L_1^\prim)$ 
(by Lemma~\ref{L:zcupw}). 
Since $m \geq 2$ it follows that $F_3$ vanishes on $L_1^\prim$ hence 
$F_3 \in I(Y\cup L_1^\prim)$ hence $\psi_2(F_3) = \eta_2(F_3) = 1 \in 
\tH^0(\sco_{L_1})$. One deduces the exactness of the sequence$\, :$ 
\[
0 \lra I(W \cup L_1^\prim) \lra I(Y\cup L_1^\prim) 
\xra{\displaystyle \tH^0_\ast(\psi_2)} S(L_1)(-m-1) \lra 0\, .
\] 

(b) As we saw at the beginnig of this subsection, one has exact 
sequences$\, :$ 
\begin{gather*} 
0 \lra \sci_{L_1^{(2)}\cup L_1^\prim} \lra \sci_{W\cup L_1^\prim} 
\overset{\displaystyle \psi_1}{\lra} 2\sco_{L_1}(-l-3) \lra 0\\
0 \lra \sci_{W\cup L_1^\prim} \lra \sci_{Y\cup L_1^\prim} 
\overset{\displaystyle \psi_2}{\lra} \sco_{L_1}(-l-m-2) \lra 0\, .
\end{gather*} 
Since $l \geq 0$, $F_2$ vanishes on $L_1^\prim$ hence  
$x_2F_2,\, x_3F_2 \in I(W\cup L_1^\prim)$ hence the sequence$\, :$ 
\[
0 \lra I(L_1^{(2)} \cup L_1^\prim) \lra I(W\cup L_1^\prim) 
\xra{\displaystyle \tH^0_\ast(\psi_1)} 2S(L_1)(-l-3) \lra 0 
\]
is exact. 

On the other hand, if $l \geq 1$, or if $l = 0$ and $m \geq 1$ or 
if $l = m = 0$ and $v_0 = v_1 = v_2 = 0$ then $F_3$ vanishes on $L_1^\prim$  
hence $F_3 \in I(Y\cup L_1^\prim)$ hence the sequence$\, :$
\[
0 \lra I(W\cup L_1^\prim) \lra I(Y\cup L_1^\prim) 
\xra{\displaystyle \tH^0_\ast(\psi_2)} S(L_1)(-l-m-2) \lra 0
\]
is exact. 

Finally, if $l = m = 0$ and at least one of the constants $v_0$, $v_1$, $v_2$ 
is non-zero then $F_3$ doesn't vanish on $L_1^\prim$ hence 
$\tH^0(\sci_{Y\cup L_1^\prim}(2)) = 0$. On the other hand, $x_0F_3$ and $x_1F_3$ 
vanish on $L_1^\prim$ hence they belong to $\tH^0(\sci_{Y\cup L_1^\prim}(3))$ 
and $\psi_2(x_iF_3) = \eta_2(x_iF_3) = x_i \in \tH^0(\sco_{L_1}(1))$,  
$i = 0,\, 1$. One deduces the exactness of the sequence$\, :$ 
\[
0 \lra I(W\cup L_1^\prim) \lra I(Y\cup L_1^\prim) 
\xra{\displaystyle \tH^0_\ast(\psi_2)} S(L_1)_+(-2) \lra 0\, .
\] 

Notice that the exact sequences appearing in this proof can be used to get 
(easily) a concrete graded free resolution of $I(Y \cup L_1^\prim)$. (A 
minimal graded free resolution of $S(L_1)_+$ can be found in the discussion 
following the proof of Prop.~\ref{P:genizprim}.)  
\end{proof}

\begin{lemma}\label{L:l1(2)cupl2} 
The homogeneous ideal $I(L_1^{(2)} \cup L_2)$ of $L_1^{(2)} \cup L_2$ admits the 
following graded free resolution$\, :$ 
\[
0 \lra 
\begin{matrix} 2S(-4)\\ \oplus\\ S(-5) \end{matrix} 
\xra{\begin{pmatrix} 
\text{--}x_3 & 0 & \text{--}x_1x_2\\
x_2 & \text{--}x_3 & 0\\
0 & x_2 & 0\\
0 & 0 & x_3
\end{pmatrix}} 
\begin{matrix} 3S(-3)\\ \oplus\\ S(-4) \end{matrix} 
\xra{\displaystyle (x_2^2x_3\, ,\, x_2x_3^2\, ,\, x_3^3\, ,\, x_1x_2^3)} 
I(L_1^{(2)} \cup L_2) \lra 0\, .   
\] 
\end{lemma}

\begin{proof} 
Using Remark~\ref{R:monomial}, one gets$\, :$ 
\[
I(L_1^{(2)}\cup L_2) = (x_2^3,\, x_2^2x_3,\, x_2x_3^2,\, x_3^3) \cap (x_1,\, x_3) 
= (x_2^2x_3,\, x_2x_3^2,\, x_3^3,\, x_1x_2^3)\, .
\] 
If $Z^\prim$ is the curve directly linked to $L_1^{(2)} \cup L_2$ by the 
complete intersection defined by $x_3^3$ and $x_1x_2^3$ then, using 
Remark~\ref{R:monomial}, 
\[
I(Z^\prim) = (x_3^2,\, x_1x_2) \cap (x_3,\, x_1x_2^2) = (x_3^2,\, x_1x_2^2,\, 
x_1x_2x_3)\, .
\]
If $Z^\secund$ is the curve directly linked to $Z^\prim$ by the complete 
intersection defined by $x_3^2$ and $x_1x_2^2$ then 
$I(Z^\secund) = (x_2,\, x_3)$, i.e., $Z^\secund = L_1$. Using Ferrand's result 
about linkage one deduces that the free resolution of $I(L_1^{(2)} \cup L_2)$ 
has the numerical shape from the statement and its differentials are easy to 
guess.  
\end{proof}

\begin{lemma}\label{L:bbprimv0} 
Using the notation recalled at the beginning of this subsection$\, :$ 

\emph{(a)} If $x_1 \mid b$ then $x_1 \mid b^\prim$ if and only if 
$x_1 \mid v_0$. 

\emph{(b)} If $x_1 \nmid b$ then one can assume that $x_1 \mid b^\prim$. 

\emph{(c)} If $x_1 \nmid b$ and $m \geq 1$ then one can assume that 
$x_1 \mid v_0$. 
\end{lemma}

\begin{proof} 
(a) This follows from relation \eqref{E:v012}$\, :$ 
$ab^\prim - a^\prim b + v_0a^2 + v_1ab + v_2b^2 = 0$ (recalled at the beginning 
of this subsection).   

(b) Let $\alpha \in \tH^0(\sco_{L_1}(l+m))$. Replacing the generators 
$1,\, e_1,\, e_2$ of the $\tH^0_\ast(\sco_{L_1}) = k[x_0,x_1]$-module 
$\tH^0_\ast(\sco_Y)$ by the generators $1,\, e_1^\prim = e_1 + \alpha e_2,\, 
e_2^\prim = e_2$ one has $e_1^{\prim \, 2} = pe_2^\prim$ and the morphism of graded 
$k[x_0,x_1]$ -algebras $S = \tH^0_\ast(\sco_\piii) \ra \tH^0_\ast(\sco_Y)$ maps 
$x_2$ to $ae_1^\prim + (a^\prim - \alpha a)e_2^\prim$ and $x_3$ to 
$be_1^\prim + (b^\prim - \alpha b)e_2^\prim$.  
Consequently, $(a,\, b,\, a^\prim -\alpha a,\, b^\prim -\alpha b,\, p)$ 
and $(a,\, b,\, a^\prim ,\, b^\prim ,\, p)$ define the same triple structure on 
$L_1$.  

If $l = -1$ and $m = 0$ then $b^\prim = 0$ because it belongs to 
$\tH^0(\sco_{L_1}(-1)) = 0$. Assume, now, that $l = -1$ and $m \geq 1$ or 
that $l \geq 0$. If $x_1 \nmid b$ then, since $b^\prim$ has degree 
$2l+m+1 \geq l+1$, there exist $\alpha \in \tH^0(\sco_{L_1}(l+m))$ and 
$b_1^\prim \in \tH^0(\sco_{L_1}(2l+m))$ such that $b^\prim = \alpha b + 
b_1^\prim x_1$. 

(c) $v_0,\, v_1,\, v_2$ are (any) elements of $\tH^0(\sco_{L_1}(l+m))$ 
that satisfy relation \eqref{E:v012} recalled in the proof of (a).  
If $x_1 \nmid b$ then $x_1a$ and $b$ are coprime. Since $b^\prim$ has degree 
$2l+m+1 \geq 2l+2$ (we used, here, the hypothesis $m \geq 1$), one can find 
elements $v_0^\prim \in \tH^0(\sco_{L_1}(l+m-1))$ and $v_1^\prim \in 
\tH^0(\sco_{L_1}(l+m))$ such that$\, :$ 
\[
b^\prim = v_0^\prim x_1a + v_1^\prim b\, .
\]
One gets, similarly, elements $v_1^\secund , v_2 \in \tH^0(\sco_{L_1}(l+m))$ 
such that $a^\prim = v_1^\secund a + v_2b$. One may take $v_0 = - x_1v_0^\prim$ 
and $v_1 = -v_1^\prim + v_1^\secund$.  
\end{proof}

\begin{prop}\label{P:geniycupl2} 
Using the notation recalled at the beginning of this subsection and assuming 
that the conditions from the conclusion of Lemma~\ref{L:bbprimv0} are 
fulfilled, one has$\, :$ 

\emph{(a)} If $x_1 \mid b$ and $x_1 \mid b^\prim$ then  
$
I(Y\cup L_2) = (F_3,\, x_2F_2,\, x_3F_2,\, x_2^2x_3,\, x_2x_3^2,\, x_3^3,\, 
x_1x_2^3)\, .
$  

\emph{(b)} If $x_1 \mid b$ and $x_1 \nmid b^\prim$ then  
$
I(Y\cup L_2) = (x_1F_3,\, x_2F_2,\, x_3F_2,\, x_2^2x_3,\, x_2x_3^2,\, x_3^3,\, 
x_1x_2^3)\, .
$

\emph{(c)} If $x_1 \nmid b$ and $x_1 \mid p$ then 
$
I(Y\cup L_2) = (F_3,\, x_3F_2,\, x_1x_2F_2,\, x_2^2x_3,\, x_2x_3^2,\, x_3^3,\, 
x_1x_2^3)\, .
$

\emph{(d)} If $x_1 \nmid b$ and $x_1 \nmid p$ then  
$
I(Y\cup L_2) = (x_1F_3,\, x_3F_2,\, x_1x_2F_2,\, x_2^2x_3,\, x_2x_3^2,\, x_3^3,\, 
x_1x_2^3)\, .
$
\end{prop}

\begin{proof}
We use the  exact sequences$\, :$ 
\begin{gather*}
0 \lra \sci_{L_1^{(2)}\cup L_2} \lra \sci_{W\cup L_2} 
\overset{\displaystyle \psi_1}{\lra} 2\sco_{L_1}(-l-3)\, ,\\
0 \lra \sci_{W \cup L_2} \lra \sci_{Y \cup L_2} 
\overset{\displaystyle \psi_2}{\lra} \sco_{L_1}(-l-m-2) 
\end{gather*}
introduced at the beginning of this subsection. Since $x_1x_2F_2$, $x_3F_2$ 
and $x_1F_3$ vanish on $L_2$, the first two of them belong to $I(W \cup L_2)$ 
and the third one to $I(Y \cup L_2)$ hence$\, :$ 
\begin{gather*}
x_1\sco_{L_1}(-l-4) \oplus \sco_{L_1}(-l-3) \subseteq \text{Im}\, \psi_1 
\subseteq 2\sco_{L_1}(-l-3)\, ,\\  
\text{and} \  x_1\sco_{L_1}(-l-m-3) \subseteq \text{Im}\, \psi_2 
\subseteq \sco_{L_1}(-l-m-2)\, .
\end{gather*} 

\noindent
{\bf Claim 1.}\quad $\text{Im}\, \psi_1 = 2\sco_{L_1}(-l-3)$ \emph{if and only 
if} $x_1 \mid b$. 

\vskip2mm

\noindent
\emph{Indeed}, the ``if'' part is clear (because, in this case, $x_2F_2$ 
vanishes on $L_2$). For the ``only if'' part assume that 
$\text{Im}\, \psi_1 = 2\sco_{L_1}(-l-3)$. Since, by Lemma~\ref{L:l1(2)cupl2}, 
$\tH^1_\ast(\sci_{L_1^{(2)} \cup L_2}) = 0$ it follows that there exists 
$f \in \tH^0(\sci_{W\cup L_2}(l+3))$ such that $\psi_1(f) = (1,0) \in 
\tH^0(2\sco_{L_1})$. Since $\psi_1(f) = \eta_1(f)$ it follows that $f$ 
must be of the form$\, :$ 
\[
f = x_2F_2 + f_0x_2^3 + f_1x_2^2x_3 + f_2x_2x_3^2 + f_3x_3^3 
\]
with $f_0, \ldots ,f_3 \in S_l$. One has$\, :$ 
\[
0 = f \vb L_2 = -b(x_0,0)x_2^2 + f_0(x_0,0,x_2,0)x_2^3 
\]
hence $b(x_0,0) = 0$ hence $x_1 \mid b$. Claim 1 is proven.  

\vskip2mm

\noindent
One deduces, from Claim 1, that if $x_1 \mid b$ then one has an exact 
sequence$\, :$ 
\[
0 \lra I(L_1^{(2)} \cup L_2) \lra I(W \cup L_2)  
\xra{\displaystyle \tH^0_\ast(\psi_1)} 2S(L_1)(-l-3) \lra 0\, ,
\]
and if $x_1 \nmid b$ then $\psi_1$ factorizes as$\, :$ 
\[
\sci_{W \cup L_2} \xra{\displaystyle \psi_1^\prim} 
\sco_{L_1}(-l-4) \oplus \sco_{L_1}(-l-3) 
\xra{\displaystyle x_1 \oplus \, \text{id}} 2\sco_{L_1}(-l-3)
\]
with $\psi^\prim(x_1x_2F_2) = (1,0) \in \tH^0(\sco_{L_1} \oplus \sco_{L_1}(1))$,  
$\psi^\prim(x_3F_2) = (0,1) \in \tH^0(\sco_{L_1}(-1) \oplus \sco_{L_1})$ and one 
has an exact sequence$\, :$ 
\[
0 \lra I(L_1^{(2)} \cup L_2) \lra I(W \cup L_2)  
\xra{\displaystyle \tH^0_\ast(\psi_1^\prim)} 
S(L_1)(-l-4) \oplus S(L_1)(-l-3) \lra 0\, .
\] 

\noindent
{\bf Claim 2.}\quad \emph{If} $x_1 \mid b$ \emph{then} $\text{Im}\, \psi_2 = 
\sco_{L_1}(-l-m-2)$ \emph{if and only if} $x_1 \mid b^\prim$. 

\vskip2mm

\noindent
\emph{Indeed}, the ``if'' part is clear (Lemma~\ref{L:bbprimv0}(a) implies 
that $x_1 \mid v_0$ hence $F_3$ vanishes on $L_2$).  
For the ``only if'' part, assume that $\text{Im}\, \psi_2 = 
\sco_{L_1}(-l-m-2)$. Using the exact sequence$\, :$ 
\[
0 \lra \sci_{L_1^{(2)}\cup L_2} \lra \sci_{W\cup L_2} 
\overset{\displaystyle \psi_1}{\lra} 2\sco_{L_1}(-l-3) \lra 0 
\] 
and Lemma~\ref{L:l1(2)cupl2} one gets that 
$\tH^1(\sci_{W \cup L_2}(l+m+2)) = 0$. One deduces that there exists an 
element $f \in \tH^0(\sci_{Y \cup L_2}(l+m+2))$ such that $\psi_2(f) = 1 \in 
\tH^0(\sco_{L_1})$. Since $\psi_2(f) = \eta_2(f)$ it follows that $f$ must 
have the form$\, :$ 
\[
f = F_3 + f_0x_2F_2 + f_1x_3F_2 + g_0x_2^3 + g_1x_2^2x_3 + g_2x_2x_3^2 + 
g_3x_3^3
\]
with $f_0,\, f_1 \in S_{m-1}$ and $g_0, \ldots ,g_3 \in S_{l+m-1}$. Taking into 
account that $F_2 \vb L_2 = 0$ (because $x_1 \mid b$), one deduces that$\, :$ 
\[
0 = f \vb L_2 = v_0(x_0,0)x_2^2 + g_0(x_0,0,x_2,0)x_2^3\, . 
\]
It follows that $v_0(x_0,0) = 0$ hence $x_1 \mid v_0$ hence, by 
Lemma~\ref{L:bbprimv0}(a), $x_1 \mid b^\prim$. 

\vskip2mm 

(a) It follows from Claim 2 that one has an exact sequence$\, :$ 
\[
0 \lra I(W \cup L_2) \lra I(Y \cup L_2) 
\xra{\displaystyle \tH^0_\ast(\psi_2)} S(L_1)(-l-m-2) \lra 0\, .
\] 

(b) It follows from Claim 2 that $\psi_2$ factorizes as$\, :$ 
\[
\sci_{Y \cup L_2} \xra{\displaystyle \psi_2^\prim} \sco_{L_1}(-l-m-3) 
\overset{\displaystyle x_1}{\lra} \sco_{L_1}(-l-m-2) 
\]
with $\psi_2^\prim(x_1F_3) = 1 \in \tH^0(\sco_{L_1})$ and that one has an 
exact sequence$\, :$ 
\[
0 \lra I(W \cup L_2) \lra I(Y \cup L_2) 
\xra{\displaystyle \tH^0_\ast(\psi_2^\prim)} S(L_1)(-l-m-3) \lra 0\, .
\]

(c) Since $x_1 \mid p$ one has $m\geq 1$. Since $x_1 \nmid b$, 
Lemma~\ref{L:bbprimv0}(c) implies that one can assume that $x_1 \mid v_0$.  
It follows that $F_3 \vb L_2 = 0$ hence $F_3 \in I(Y\cup L_2)$. One deduces 
that $\text{Im}\, \psi_2 = \sco_{L_1}(-l-m-2)$ and that one has an exact 
sequence$\, :$ 
\[
0 \lra I(W \cup L_2) \lra I(Y \cup L_2) 
\xra{\displaystyle \tH^0_\ast(\psi_2)} S(L_1)(-l-m-2) \lra 0\, .
\] 

(d) Any element $f \in I(Y)$ can be written as$\, :$ 
\[
f = f_0F_3 + f_1x_2F_2 + f_2x_3F_2 + f_3x_2^3 + f_4x_2^2x_3 + f_5x_2x_3^2 + 
f_6x_3^3 
\] 
with $f_0,\, f_1,\, f_2 \in k[x_0,x_1]$, $f_3 \in k[x_0,x_1,x_2]$ and 
$f_4,\, f_5,\, f_6 \in S$. One has$\, :$ 
\begin{gather*}
f \vb L_2 = -f_0(x_0,0)p(x_0,0)b(x_0,0)x_2 + f_0(x_0,0)v_0(x_0,0)x_2^2 -\\ 
-f_1(x_0,0)b(x_0,0)x_2^2 + f_3(x_0,0,x_2)x_2^3\, . 
\end{gather*} 
Since $x_1 \nmid b$ and $x_1 \nmid p$ it follows that $b(x_0,0) \neq 0$ and 
$p(x_0,0) \neq 0$. One deduces that if $f \in I(Y \cup L_2)$ (i.e., if 
$f \vb L_2 = 0$) then, firstly, $f_0(x_0,0) = 0$ then $f_1(x_0,0) = 0$ and, 
finally, $f_3(x_0,0,x_2) = 0$. This means that $x_1 \mid f_0$, $x_1 \mid f_1$ 
and $x_1 \mid f_3$. One derives that $I(Y \cup L_2)$ is generated by the 
polynomials indicated in the statement. Moreover, it follows that $\psi_2$ 
factorizes as$\, :$ 
\[
\sci_{Y \cup L_2} \xra{\displaystyle \psi_2^\prim} \sco_{L_1}(-l-m-3) 
\overset{\displaystyle x_1}{\lra} \sco_{L_1}(-l-m-2) 
\]
with $\psi_2^\prim(x_1F_3) = 1 \in \tH^0(\sco_{L_1})$ and that one has an 
exact sequence$\, :$ 
\[
0 \lra I(W \cup L_2) \lra I(Y \cup L_2) 
\xra{\displaystyle \tH^0_\ast(\psi_2^\prim)} S(L_1)(-l-m-3) \lra 0\, .
\] 

Notice that the exact sequences appearing in the above proof can be used to 
get (easily) concrete graded free resolutions of $I(W \cup L_2)$ and of 
$I(Y \cup L_2)$. 
\end{proof}

\begin{prop}\label{P:reshoycupl2} 
With the notation recalled at the beginning of this subsection$\, :$

\emph{(a)} If $x_1 \mid b$ and $x_1 \mid b^\prim$, i.e., if $b = x_1b_1$ and   
$b^\prim = x_1b_1^\prim$, then ${\fam0 H}^0_\ast(\sco_{Y \cup L_2})$ admits 
the following graded free resolution$\, :$
\[
0 \lra \begin{matrix} S(-3)\\ \oplus\\ S(l-3)\\ \oplus\\ S(2l+m-3) 
\end{matrix} \overset{\displaystyle \delta_2}{\lra} 
\begin{matrix} S(-1)\\ \oplus\\ S(-2)\\ \oplus\\ 2S(l-2)\\ \oplus\\ 
2S(2l+m-2) \end{matrix} \overset{\displaystyle \delta_1}{\lra} 
\begin{matrix} S\\ \oplus\\ S(l-1)\\ \oplus\\ S(2l+m-1) \end{matrix} 
\overset{\displaystyle \delta_0}{\lra} {\fam0 H}^0_\ast(\sco_{Y\cup L_2}) \lra 0
\] 
with $\delta_0 = (1\, ,\, x_1e_1\, ,\, x_1e_2)$ and with $\delta_1$ and 
$\delta_2$ defined by the matrices$\, :$ 
\[
\begin{pmatrix} 
x_3 & x_1x_2 & 0 & 0 & 0 & 0\\
-b_1 & -a & x_2 & x_3 & 0 & 0\\
-b_1^\prim & -a^\prim & -pa & -pb & x_2 & x_3
\end{pmatrix}\, ,\  
\begin{pmatrix}
-x_1x_2 & 0 & 0\\
x_3 & 0 & 0\\
-b & -x_3 & 0\\
a & x_2 & 0\\
-b^\prim & pb & -x_3\\
a^\prim & -pa & x_2
\end{pmatrix}\, .
\]

\emph{(b)} If $x_1 \mid b$ and $x_1 \nmid b^\prim$ then 
${\fam0 H}^0_\ast(\sco_{Y \cup L_2})$ admits the following graded free 
resolution$\, :$ 
\[
0 \lra \begin{matrix} S(-3)\\ \oplus\\ S(l-3)\\ \oplus\\ S(2l+m-2) 
\end{matrix} \overset{\displaystyle \delta_2}{\lra} 
\begin{matrix} S(-1)\\ \oplus\\ S(-2)\\ \oplus\\ 2S(l-2)\\ \oplus\\ 
2S(2l+m-1) \end{matrix} \overset{\displaystyle \delta_1}{\lra} 
\begin{matrix} S\\ \oplus\\ S(l-1)\\ \oplus\\ S(2l+m) \end{matrix} 
\overset{\displaystyle \delta_0}{\lra} {\fam0 H}^0_\ast(\sco_{Y\cup L_2}) \lra 0
\] 
with $\delta_0 = (1\, ,\, x_1e_1\, ,\, e_2)$ and with $\delta_1$ and $\delta_2$ 
defined by the matrices$\, :$ 
\[
\begin{pmatrix} 
x_3 & x_1x_2 & 0 & 0 & 0 & 0\\
-b_1 & -a & x_2 & x_3 & 0 & 0\\
-b^\prim & -x_1a^\prim & -x_1pa & -x_1pb & x_2 & x_3
\end{pmatrix}\, ,\  
\begin{pmatrix}
-x_1x_2 & 0 & 0\\
x_3 & 0 & 0\\
-b & -x_3 & 0\\
a & x_2 & 0\\
-x_1b^\prim & x_1pb & -x_3\\
x_1a^\prim & -x_1pa & x_2
\end{pmatrix}\, .
\]

\emph{(c)} If $x_1 \nmid b$ and $x_1 \mid p$ then 
${\fam0 H}^0_\ast(\sco_{Y \cup L_2})$ admits the following graded free 
resolution$\, :$ 
\[
0 \lra \begin{matrix} S(-3)\\ \oplus\\ S(l-2)\\ \oplus\\ S(2l+m-3) 
\end{matrix} \overset{\displaystyle \delta_2}{\lra} 
\begin{matrix} S(-1)\\ \oplus\\ S(-2)\\ \oplus\\ 2S(l-1)\\ \oplus\\ 
2S(2l+m-2) \end{matrix} \overset{\displaystyle \delta_1}{\lra} 
\begin{matrix} S\\ \oplus\\ S(l)\\ \oplus\\ S(2l+m-1) \end{matrix} 
\overset{\displaystyle \delta_0}{\lra} {\fam0 H}^0_\ast(\sco_{Y\cup L_2}) \lra 0
\] 
with $\delta_0 = (1\, ,\, e_1\, ,\, x_1e_2)$ and with $\delta_1$ and $\delta_2$ 
defined by the matrices$\, :$ 
\[
\begin{pmatrix} 
x_3 & x_1x_2 & 0 & 0 & 0 & 0\\
-b & -x_1a & x_2 & x_3 & 0 & 0\\
-b_1^\prim & -a^\prim & -p_1a & -p_1b & x_2 & x_3
\end{pmatrix}\, ,\  
\begin{pmatrix}
-x_1x_2 & 0 & 0\\
x_3 & 0 & 0\\
-x_1b & -x_3 & 0\\
x_1a & x_2 & 0\\
-b^\prim & p_1b & -x_3\\
a^\prim & -p_1a & x_2
\end{pmatrix}\, .
\]

\emph{(d)} If $x_1 \nmid b$ and $x_1 \nmid p$ then 
${\fam0 H}^0_\ast(\sco_{Y \cup L_2})$ admits the following graded free 
resolution$\, :$ 
\[
0 \lra \begin{matrix} S(-3)\\ \oplus\\ S(l-2)\\ \oplus\\ S(2l+m-2) 
\end{matrix} \overset{\displaystyle \delta_2}{\lra} 
\begin{matrix} S(-1)\\ \oplus\\ S(-2)\\ \oplus\\ 2S(l-1)\\ \oplus\\ 
2S(2l+m-1) \end{matrix} \overset{\displaystyle \delta_1}{\lra} 
\begin{matrix} S\\ \oplus\\ S(l)\\ \oplus\\ S(2l+m) \end{matrix} 
\overset{\displaystyle \delta_0}{\lra} {\fam0 H}^0_\ast(\sco_{Y\cup L_2}) \lra 0
\] 
with $\delta_0 = (1\, ,\, e_1\, ,\, e_2)$ and with $\delta_1$ and $\delta_2$ 
defined by the matrices$\, :$ 
\[
\begin{pmatrix} 
x_3 & x_1x_2 & 0 & 0 & 0 & 0\\
-b & -x_1a & x_2 & x_3 & 0 & 0\\
-b^\prim & -x_1a^\prim & -pa & -pb & x_2 & x_3
\end{pmatrix}\, ,\  
\begin{pmatrix}
-x_1x_2 & 0 & 0\\
x_3 & 0 & 0\\
-x_1b & -x_3 & 0\\
x_1a & x_2 & 0\\
-x_1b^\prim & pb & -x_3\\
x_1a^\prim & -pa & x_2
\end{pmatrix}\, .
\]
\end{prop}

\begin{proof}
Recall, from the beginning of this subsection, that one has an exact 
sequence$\, :$ 
\[
0 \lra \text{Im}\, \phi \times \{0\} \lra \sco_{Y \cup L_2} \lra \sco_{X \cup L_2} 
\lra 0 
\]
with $\phi$ is the composite morphism $\sci_{X \cup L_2} \ra \sci_X 
\overset{\rho}{\lra} \sco_{L_1}(2l+m)$, where the epimorphism $\rho : \sci_X \ra 
\sco_{L_1}(2l+m)$ is characterized by$\, :$ 
\[
\rho(F_2) = \Delta_1,\  \rho(x_2^2) = pa^2,\  \rho(x_2x_3) = pab,\  
\rho(x_3^2) = pb^2\, . 
\] 
Moreover, $\Cok \phi$ is an $\sco_{L_1 \cap L_2} = \sco_{\{P_0\}}$-module which 
implies that $x_1\sco_{L_1}(2l+m-1) \subseteq \text{Im}\, \phi \subseteq 
\sco_{L_1}(2l+m)$. 

It follows, from Prop.~\ref{P:genixcupl}, that if $x_1 \mid b$ then 
$I(X \cup L_2) = (F_2,\, x_2x_3,\, x_3^2,\, x_1x_2^2)$ hence $\text{Im}\, \phi$ 
coincides with the image of the morphism$\, :$ 
\[
(\Delta_1,\, pab,\, pb^2,\, px_1a^2) : 
\sco_{L_1}(-l-2) \oplus 2\sco_{L_1}(-2) \oplus \sco_{L_1}(-3) \lra 
\sco_{L_1}(2l+m) 
\]
and if $x_1 \nmid b$ then $I(X \cup L_2) = (x_1F_2,\, x_2x_3,\, x_3^2,\, 
x_1x_2^2)$ hence $\text{Im}\, \phi$ coincides with the image of the 
morphism$\, :$ 
\[
(x_1\Delta_1,\, pab,\, pb^2,\, px_1a^2) : 
\sco_{L_1}(-l-3) \oplus 2\sco_{L_1}(-2) \oplus \sco_{L_1}(-3) \lra 
\sco_{L_1}(2l+m)\, . 
\]

Recall, also, from Prop.~\ref{P:reshoxcupl}, the graded free resolution of 
$\tH^0_\ast(\sco_{X \cup L_2})$. 

\vskip2mm

(a) In this case $x_1 \mid \Delta_1$ and $x_1 \mid b$ hence $\text{Im}\, \phi 
= x_1\sco_{L_1}(2l+m-1)$. According to Prop.~\ref{P:reshoxcupl}(a), the graded 
$S$-module $\tH^0_\ast(\sco_{X \cup L_2})$ is generated by $1 \in 
\tH^0(\sco_{X \cup L_2})$ and by $x_1e_1 \in \tH^0(\sco_{X \cup L_2}(-l+1))$. Since 
$\tH^1((\text{Im}\, \phi)(-l+1)) = 0$ it follows that the morphism  
$\tH^0_\ast(\sco_{Y \cup L_2}) \ra \tH^0_\ast(\sco_{X \cup L_2})$ is surjective. 
One deduces an exact sequence$\, :$ 
\[
0 \lra S(L_1)(2l+m-1) \lra \tH^0_\ast(\sco_{Y \cup L_2}) \lra 
\tH^0_\ast(\sco_{X \cup L_2}) \lra 0
\] 
where the left morphism maps $1 \in S(L_1)$ to the element of 
$\tH^0(\sco_{Y \cup L_2}(-2l-m+1))$ whose image into $\tH^0(\sco_Y(-2l-m+1)) 
\oplus \tH^0(\sco_{L_2}(-2l-m+1))$ is $(x_1e_2,\, 0)$. 

(b) In this case, since $x_1 \nmid \Delta_1$ it follows that $\text{Im}\, 
\phi = \sco_{L_1}(2l+m)$. One derives, as in case (a), an exact sequence$\, :$ 
\[
0 \lra S(L_1)(2l+m) \lra \tH^0_\ast(\sco_{Y \cup L_2}) \lra 
\tH^0_\ast(\sco_{X \cup L_2}) \lra 0
\] 
where the left morphism maps $1 \in S(L_1)$ to the element of 
$\tH^0(\sco_{Y \cup L_2}(-2l-m))$ whose image into $\tH^0(\sco_Y(-2l-m)) 
\oplus \tH^0(\sco_{L_2}(-2l-m))$ is $(e_2,\, 0)$. 

(c) In this case $x_1 \mid p$ hence $\text{Im}\, \phi = 
x_1\sco_{L_1}(2l+m-1)$. By Prop.~\ref{P:reshoxcupl}(a), the graded 
$S$-module $\tH^0_\ast(\sco_{X \cup L_2})$ is generated by $1 \in 
\tH^0(\sco_{X \cup L_2})$ and by $e_1 \in \tH^0(\sco_{X \cup L_2}(-l))$. 
Since $x_1 \mid p$ it follows that $m \geq 1$ which implies that 
$\tH^1((\text{Im}\, \phi)(-l)) = 0$. It follows that the morphism  
$\tH^0_\ast(\sco_{Y \cup L_2}) \ra \tH^0_\ast(\sco_{X \cup L_2})$ is surjective. 
One deduces an exact sequence$\, :$ 
\[
0 \lra S(L_1)(2l+m-1) \lra \tH^0_\ast(\sco_{Y \cup L_2}) \lra 
\tH^0_\ast(\sco_{X \cup L_2}) \lra 0
\] 
where the left morphism maps $1 \in S(L_1)$ to the element of 
$\tH^0(\sco_{Y \cup L_2}(-2l-m+1))$ whose image into $\tH^0(\sco_Y(-2l-m+1)) 
\oplus \tH^0(\sco_{L_2}(-2l-m+1))$ is $(x_1e_2,\, 0)$. 

(d) In this case, since $x_1 \nmid pb^2$ it follows that $\text{Im}\, 
\phi = \sco_{L_1}(2l+m)$. One derives an exact sequence$\, :$ 
\[
0 \lra S(L_1)(2l+m) \lra \tH^0_\ast(\sco_{Y \cup L_2}) \lra 
\tH^0_\ast(\sco_{X \cup L_2}) \lra 0
\] 
where the left morphism maps $1 \in S(L_1)$ to the element of 
$\tH^0(\sco_{Y \cup L_2}(-2l-m))$ whose image into $\tH^0(\sco_Y(-2l-m)) 
\oplus \tH^0(\sco_{L_2}(-2l-m))$ is $(e_2,\, 0)$.  
\end{proof}

\subsection{A double line union two lines}\label{SS:doublecuptwolines}

Let $X$ be a double structure on a line in $\piii$ and $\Lambda_1$, $\Lambda_2$ 
two other lines in $\piii$. We want to describe the homogeneous ideal and 
the graded structural algebra of the union $X \cup \Lambda_1 \cup \Lambda_2$. 
One has to consider the following configurations$\, :$ 
\begin{enumerate} 
\item $\Lambda_1 \cap X \neq \emptyset$, $\Lambda_2 \cap X \neq \emptyset$, 
$\Lambda_1 \cap \Lambda_2 \neq \emptyset$, 
$\Lambda_1 \cap \Lambda_2 \cap X = \emptyset$$\, ;$
\item $\Lambda_1 \cap X \neq \emptyset$, $\Lambda_2 \cap X \neq \emptyset$, 
$\Lambda_1 \cap \Lambda_2 = \emptyset$$\, ;$ 
\item $\Lambda_1 \cap \Lambda_2 \cap X \neq \emptyset$$\, ;$ 
\item $\Lambda_1 \cap X \neq \emptyset$, $\Lambda_2 \cap X = \emptyset$, 
$\Lambda_1 \cap \Lambda_2 \neq \emptyset$$\, ;$ 
\item $\Lambda_1 \cap X \neq \emptyset$, $\Lambda_2 \cap X = \emptyset$, 
$\Lambda_1 \cap \Lambda_2 = \emptyset$$\, ;$ 
\item $\Lambda_1 \cap X = \emptyset$, $\Lambda_2 \cap X = \emptyset$, 
$\Lambda_1 \cap \Lambda_2 \neq \emptyset$$\, ;$ 
\item $\Lambda_1 \cap X = \emptyset$, $\Lambda_2 \cap X = \emptyset$, 
$\Lambda_1 \cap \Lambda_2 = \emptyset$$\, ;$    
\end{enumerate} 

Up to a linear change of coordinates in $\piii$ one can assume that $X$ is a 
double structure on the line $L_1$ and it suffices to treat only the 
following cases$\, :$
\begin{enumerate}
\item[(i)] $X \cup L_2 \cup L_3^\prim$$\, ;$
\item[(ii)] $X \cup L_2 \cup L_2^\prim$$\, ;$ 
\item[(iii)] $X \cup L_2 \cup L_3$$\, ;$ 
\item[(iv)] $X \cup L_2 \cup L_1^\prim$$\, ;$ 
\item[(v)] $X \cup L \cup L_1^\prim$ where $L$ is the line of equations 
$x_1 + cx_2 = x_3 = 0$, $c \in k\setminus \{0\}$$\, ;$ 
\item[(vi)] $X \cup L_1^\prim \cup L^\prim$, where $L^\prim$ is the line of 
equations $x_0 + cx_2 = x_1 = 0$, $c \in k \setminus \{0\}$$\, ;$ 
\item[(vii)] $X \cup L_1^\prim \cup L_1^\secund$ where $L_1^\secund$ is the line of 
equations $x_0 - x_2 = x_1 - x_3 = 0$$\, ;$ $L_1$, $L_1^\prim$, $L_1^\secund$ are 
mutually disjoint lines contained in the quadric surface $Q \subset \piii$ of 
equation $x_0x_3 - x_1x_2 = 0$. 
\end{enumerate}
For the double structure $X$ on $L_1$ we use the notation from the beginning 
of Subsection~\ref{SS:doublecupaline}.

\vskip2mm 

In case (i), just make $c = 0$ in the statements and proofs of 
Lemma~\ref{L:l1(1)cupc4}, Prop.~\ref{P:genixcupc4} and 
Prop.~\ref{P:reshoxcupc4} from Subsection~\ref{SS:doublecupconic} below. 
The next four cases (ii)--(v) can be treated in an unitary manner as 
follows$\, :$ notice, fistly, that if one takes $c = 0$ in the first equation 
of $L$ one gets $L_2$. Let us denote $x_1 + cx_2$, $c \in k$, by $\ell$. 
One has $I(L_1 \cup L) = (x_3,\, \ell x_2)$. Let us denote by $\Lambda$ any of 
the lines $L_1^\prim$, $L_2^\prim$ and $L_3$.  
We have to consider two subcases. 

\vskip2mm

(I) If $x_1 \mid b$, i.e., if $b = x_1b_1$ then, according to the proof of 
Prop.~\ref{P:genixcupl}(a) and of Prop.~\ref{P:reshoxcupl}(a), 
$I(X \cup L) = (F_2 - cb_1x_2^2,\, x_2x_3,\, x_3^2,\, \ell x_2^2)$ and one has 
exact sequences$\, :$ 
\begin{gather*}
0 \lra \sci_{L_1^{(1)} \cup L} \lra \sci_{X \cup L} 
\overset{\displaystyle \psi}{\lra} \sco_{L_1}(-l-2) \lra 0\\
0 \lra \sci_{X \cup L} \lra \sci_{L_1 \cup L} 
\overset{\displaystyle \phi^\prim}{\lra} \sco_{L_1}(l-1) \lra 0  
\end{gather*}
with $\psi(F_2 - cb_1x_2^2) = 1 \in \tH^0(\sco_{L_1})$, and with 
$\phi^\prim (x_3) = b_1$, $\phi^\prim(\ell x_2) = a$. Let us denote by $\psi_1$ 
and $\phi^\prim_1$ the composite morphisms$\, :$ 
\[
\sci_{X \cup L \cup \Lambda} \lra \sci_{X \cup L} 
\overset{\displaystyle \psi}{\lra} \sco_{L_1}(-l-2) \  \text{and} \  
\sci_{L_1 \cup L \cup \Lambda} \lra \sci_{L_1 \cup L} 
\overset{\displaystyle \phi^\prim}{\lra} \sco_{L_1}(l-1)\, . 
\]
Then, by Lemma~\ref{L:ycupt}, one has exact sequences$\, :$ 
\begin{gather*}
0 \lra \sci_{L_1^{(1)} \cup L \cup \Lambda} \lra \sci_{X \cup L \cup \Lambda} 
\overset{\displaystyle \psi_1}{\lra} \sco_{L_1}(-l-2)\, ,\\
0 \lra \text{Im}\, \phi^\prim_1 \times \{0\} \lra \sco_{X \cup L \cup \Lambda} 
\lra \sco_{L_1 \cup L \cup \Lambda} \lra 0\, .  
\end{gather*}
Moreover, $\Cok \psi_1$ and $\Cok \phi^\prim_1$ are 
$\sco_{L_1 \cap \Lambda}$-modules. 

\vskip2mm

(II) If $x_1 \nmid b$ then, according to the proof of 
Prop.~\ref{P:genixcupl}(b) and of Prop.~\ref{P:reshoxcupl}(b), 
$I(X \cup L) = (\ell F_2,\, x_2x_3,\, x_3^2,\, \ell x_2^2)$ and one has exact 
sequences$\, :$ 
\begin{gather*}
0 \lra \sci_{L_1^{(1)} \cup L} \lra \sci_{X \cup L} 
\overset{\displaystyle \psi^\prim}{\lra} \sco_{L_1}(-l-3) \lra 0\\
0 \lra \sci_{X \cup L} \lra \sci_{L_1 \cup L} 
\overset{\displaystyle \phi}{\lra} \sco_{L_1}(l) \lra 0  
\end{gather*}
with $\psi^\prim (\ell F_2) = 1 \in \tH^0(\sco_{L_1})$, and with 
$\phi(x_3) = b$, $\phi(\ell x_2) = x_1a$. Let us denote by $\psi_1^\prim$ 
and $\phi_1$ the composite morphisms$\, :$ 
\[
\sci_{X \cup L \cup \Lambda} \lra \sci_{X \cup L} 
\overset{\displaystyle \psi^\prim}{\lra} \sco_{L_1}(-l-3) \  \text{and} \  
\sci_{L_1 \cup L \cup \Lambda} \lra \sci_{L_1 \cup L} 
\overset{\displaystyle \phi}{\lra} \sco_{L_1}(l)\, . 
\]
Then, by Lemma~\ref{L:ycupt}, one has exact sequences$\, :$ 
\begin{gather*}
0 \lra \sci_{L_1^{(1)} \cup L \cup \Lambda} \lra \sci_{X \cup L \cup \Lambda} 
\overset{\displaystyle \psi_1^\prim}{\lra} \sco_{L_1}(-l-3)\, ,\\
0 \lra \text{Im}\, \phi_1 \times \{0\} \lra \sco_{X \cup L \cup \Lambda} \lra 
\sco_{L_1 \cup L \cup \Lambda} \lra 0\, .  
\end{gather*}
Moreover, $\Cok \psi_1^\prim$ and $\Cok \phi_1$ are 
$\sco_{L_1 \cap \Lambda}$-modules.   

\begin{lemma}\label{L:l1(1)cupl2cupl2prim} 
The homogeneous ideal $I(W) \subset S$ of $W = L_1^{(1)} \cup L_2 \cup 
L_2^\prim$ is generated by $x_2x_3$, $x_0x_3^2$, $x_1x_2^2$ and admits the 
following graded free resolution$\, :$ 
\[
0 \lra 2S(-4) \xra{\begin{pmatrix} -x_0x_3 & -x_1x_2\\ x_2 & 0\\ 0 & x_3 
\end{pmatrix}} S(-2) \oplus 2S(-3) \lra I(W) \lra 0\, .
\]
\end{lemma}

\begin{proof} 
$I(W)$ is, by definition$\, :$ 
\[
(x_2^2,\, x_2x_3,\, x_3^2) \cap (x_1,\, x_3) \cap (x_0,\, x_2) = 
(x_2x_3,\, x_3^2,\, x_1x_2^2) \cap (x_0,\, x_2) = 
(x_2x_3,\, x_0x_3^2,\, x_1x_2^2)\, .  
\]
If $W^\prim$ is the curve directly linked to $W$ by the complete intersection 
defined by $x_0x_3^2$ and $x_1x_2^2$ then $I(W^\prim) = (x_0x_3,\, x_1x_2)$. 
One can apply, now, Ferrand's result about liaison. 
\end{proof}

\begin{prop}\label{P:genixcupl2cupl2prim}
Let $X$ be the double structure on the line $L_1$ considered at the beginning 
of Subsection~\ref{SS:doublecupaline}. 

\emph{(a)} If $x_1 \mid b$ and $x_0 \mid a$ then $I(X \cup L_2 \cup L_2^\prim) 
= SF_2 + I(L_1^{(1)} \cup L_2 \cup L_2^\prim)$. 

\emph{(b)} If $x_1 \mid b$ and $x_0 \nmid a$ then $I(X \cup L_2 \cup L_2^\prim) 
= Sx_0F_2 + I(L_1^{(1)} \cup L_2 \cup L_2^\prim)$. 

\emph{(c)} If $x_1 \nmid b$ and $x_0 \mid a$ then $I(X \cup L_2 \cup L_2^\prim) 
= Sx_1F_2 + I(L_1^{(1)} \cup L_2 \cup L_2^\prim)$. 

\emph{(d)} If $x_1 \nmid b$ and $x_0 \nmid a$ then $I(X \cup L_2 \cup L_2^\prim) 
= Sx_0x_1F_2 + I(L_1^{(1)} \cup L_2 \cup L_2^\prim)$.      
\end{prop} 

\begin{proof}
We use the exact sequences defined at the beginning of this subsection. 
Since $L_1 \cap L_2^\prim = \{P_1\}$ one deduces that$\, :$ 
\begin{gather*}
x_0\sco_{L_1}(-l-3) \subseteq \text{Im}\, \psi_1 \subseteq \sco_{L_1}(-l-2)\, ,\\
x_0\sco_{L_1}(-l-4) \subseteq \text{Im}\, \psi_1^\prim \subseteq \sco_{L_1}(-l-3) 
\, .
\end{gather*}

\noindent 
{\bf Claim 1.}\quad \emph{If} $x_1 \mid b$ \emph{then} $\text{Im}\, \psi_1 
= \sco_{L_1}(-l-2)$ \emph{if and only if} $x_0 \mid a$. 

\vskip2mm

\noindent
\emph{Indeed}, the ``if'' part is clear because if $x_0 \mid a$ then 
$F_2 \vb L_2^\prim = 0$ hence $F_2 \in I(X \cup L_2 \cup L_2^\prim)$. 

For the ``only if'' part, assume that $\text{Im}\, \psi_1 = \sco_{L_1}(-l-2)$. 
Since, by Lemma~\ref{L:l1(1)cupl2cupl2prim}, 
$\tH^1_\ast(\sci_{L_1^{(1)} \cup L_2 \cup L_2^\prim}) = 0$, it follows that there 
exists $f \in \tH^0(\sci_{X \cup L_2 \cup L_2^\prim}(l+2))$ such that $\psi_1(f) = 
1 \in \tH^0(\sco_{L_1})$. Since $\psi_1(f) = \psi(f)$ one deduces that $f$ 
must have the form$\, :$ 
\[
f = F_2 + f_0x_2x_3 + f_1 x_3^2 + f_2x_1x_2^2\, .
\] 
But $0 = f \vb L_2^\prim = a(0,x_1)x_3 + f_1(0,x_1,0,x_3)x_3^2$ implies that 
$a(0,x_1) = 0$ hence $x_0 \mid a$. 

\vskip2mm

\noindent
{\bf Claim 2.}\quad \emph{If} $x_1 \nmid b$ \emph{then} 
$\text{Im}\, \psi_1^\prim  = \sco_{L_1}(-l-3)$ \emph{if and only if} 
$x_0 \mid a$. 

\vskip2mm

\noindent
\emph{Indeed}, the ``if'' part is clear because if $x_0 \mid a$ then 
$F_2 \vb L_2^\prim = 0$ hence $x_1F_2 \in I(X \cup L_2 \cup L_2^\prim)$. 

For the ``only if'' part, assume that $\text{Im}\, \psi_1^\prim  
= \sco_{L_1}(-l-3)$. Since, by Lemma~\ref{L:l1(1)cupl2cupl2prim}, 
$\tH^1_\ast(\sci_{L_1^{(1)} \cup L_2 \cup L_2^\prim}) = 0$, it follows that there 
exists $f \in \tH^0(\sci_{X \cup L_2 \cup L_2^\prim}(l+3))$ such that 
$\psi_1^\prim(f) = 1 \in \tH^0(\sco_{L_1})$. 
Since $\psi_1^\prim(f) = \psi^\prim(f)$ one deduces that $f$ must have the 
form$\, :$ 
\[
f = x_1F_2 + f_0x_2x_3 + f_1 x_3^2 + f_2x_1x_2^2\, .
\] 
But $0 = f \vb L_2^\prim = a(0,x_1)x_1x_3 + f_1(0,x_1,0,x_3)x_3^2$ implies that 
$a(0,x_1) = 0$ hence $x_0 \mid a$. 

\vskip2mm

(a) One deduces, from Claim 1, the existence of an exact sequence$\, :$ 
\[
0 \lra I(L_1^{(1)} \cup L_2 \cup L_2^\prim) \lra I(X \cup L_2 \cup L_2^\prim) 
\xra{\displaystyle \tH^0_\ast(\psi_1)} S(L_1)(-l-2) \lra 0\, .
\]

(b) Claim 1 implies that $\psi_1$ factorizes as$\, :$ 
\[
\sci_{X \cup L_2 \cup L_2^\prim} \overset{\displaystyle \psi_2}{\lra} 
\sco_{L_1}(-l-3) \overset{\displaystyle x_0}{\lra} \sco_{L_1}(-l-2)  
\]
with $\psi_2(x_0F_2) = 1 \in \tH^0(\sco_{L_1})$. One deduces the existence of 
an exact sequence$\, :$ 
\[
0 \lra I(L_1^{(1)} \cup L_2 \cup L_2^\prim) \lra I(X \cup L_2 \cup L_2^\prim) 
\xra{\displaystyle \tH^0_\ast(\psi_2)} S(L_1)(-l-3) \lra 0\, .
\] 

(c) One deduces, from Claim 2, the existence of an exact sequence$\, :$ 
\[
0 \lra I(L_1^{(1)} \cup L_2 \cup L_2^\prim) \lra I(X \cup L_2 \cup L_2^\prim) 
\xra{\displaystyle \tH^0_\ast(\psi_1^\prim)} S(L_1)(-l-3) \lra 0\, .
\]

(d) Claim 2 implies that $\psi_1^\prim$ factorizes as$\, :$ 
\[
\sci_{X \cup L_2 \cup L_2^\prim} \overset{\displaystyle \psi_2^\prim}{\lra} 
\sco_{L_1}(-l-4) \overset{\displaystyle x_0}{\lra} \sco_{L_1}(-l-3)  
\]
with $\psi_2^\prim(x_0x_1F_2) = 1 \in \tH^0(\sco_{L_1})$. One deduces the 
existence of an exact sequence$\, :$ 
\[
0 \lra I(L_1^{(1)} \cup L_2 \cup L_2^\prim) \lra I(X \cup L_2 \cup L_2^\prim) 
\xra{\displaystyle \tH^0_\ast(\psi_2^\prim)} S(L_1)(-l-4) \lra 0\, .
\] 

Notice that the exact sequences appearing in the above proof can be used to 
get a concrete graded free resolution of $I(X \cup L_2 \cup L_2^\prim)$. 
\end{proof}

\begin{prop}\label{P:reshoxcupl2cupl2prim}
Let $X$ be the double structure on the line $L_1$ considered at the beginning 
of Subsection~\ref{SS:doublecupaline}. 

\emph{(a)} If $x_1 \mid b$ and $x_0 \mid a$, i.e., if $b = x_1b_1$ and 
$a = x_0a_0$ then the graded $S$-module 
${\fam0 H}^0_\ast(\sco_{X \cup L_2 \cup L_2^\prim})$ admits the following graded free 
resolution$\, :$ 
\[
0 \lra \begin{matrix} 2S(-3)\\ \oplus\\ S(l-4) \end{matrix} 
\overset{\displaystyle \delta_2}{\lra} 
\begin{matrix} 3S(-2)\\ \oplus\\ 2S(l-3) \end{matrix} 
\overset{\displaystyle \delta_1}{\lra} 
\begin{matrix} S\\ \oplus\\ S(l-2) \end{matrix} 
\overset{\displaystyle \delta_0}{\lra} 
{\fam0 H}^0_\ast(\sco_{X \cup L_2 \cup L_2^\prim}) \lra 0  
\]
with $\delta_0 = (1\, ,\, x_0x_1e_1)$ and with $\delta_1$ and $\delta_2$ defined 
by the matrices$\, :$ 
\[
\begin{pmatrix} x_0x_3 & x_1x_2 & x_2x_3 & 0 & 0\\
-b_1 & -a_0 & 0 & x_2 & x_3 \end{pmatrix}\, ,\  
\begin{pmatrix} -x_2 & 0 & 0\\
0 & -x_3 & 0\\
x_0 & x_1 & 0\\
-b_1 & 0 & -x_3\\
0 & -a_0 & x_2 \end{pmatrix}\, .
\]

\emph{(b)} If $x_1 \mid b$ and $x_0 \nmid a$ then the graded $S$-module 
${\fam0 H}^0_\ast(\sco_{X \cup L_2 \cup L_2^\prim})$ admits the following graded free 
resolution$\, :$ 
\[
0 \lra \begin{matrix} 2S(-3)\\ \oplus\\ S(l-3) \end{matrix} 
\overset{\displaystyle \delta_2}{\lra} 
\begin{matrix} 3S(-2)\\ \oplus\\ 2S(l-2) \end{matrix} 
\overset{\displaystyle \delta_1}{\lra} 
\begin{matrix} S\\ \oplus\\ S(l-1) \end{matrix} 
\overset{\displaystyle \delta_0}{\lra} 
{\fam0 H}^0_\ast(\sco_{X \cup L_2 \cup L_2^\prim}) \lra 0  
\]
with $\delta_0 = (1\, ,\, x_1e_1)$ and with $\delta_1$ and $\delta_2$ defined 
by the matrices$\, :$ 
\[
\begin{pmatrix} x_0x_3 & x_1x_2 & x_2x_3 & 0 & 0\\
-x_0b_1 & -a & 0 & x_2 & x_3 \end{pmatrix}\, ,\  
\begin{pmatrix} -x_2 & 0 & 0\\
0 & -x_3 & 0\\
x_0 & x_1 & 0\\
-x_0b_1 & 0 & -x_3\\
0 & -a & x_2 \end{pmatrix}\, .
\]

\emph{(c)} If $x_1 \nmid b$ and $x_0 \mid a$ then the graded $S$-module 
${\fam0 H}^0_\ast(\sco_{X \cup L_2 \cup L_2^\prim})$ admits the following graded free 
resolution$\, :$ 
\[
0 \lra \begin{matrix} 2S(-3)\\ \oplus\\ S(l-3) \end{matrix} 
\overset{\displaystyle \delta_2}{\lra} 
\begin{matrix} 3S(-2)\\ \oplus\\ 2S(l-2) \end{matrix} 
\overset{\displaystyle \delta_1}{\lra} 
\begin{matrix} S\\ \oplus\\ S(l-1) \end{matrix} 
\overset{\displaystyle \delta_0}{\lra} 
{\fam0 H}^0_\ast(\sco_{X \cup L_2 \cup L_2^\prim}) \lra 0  
\]
with $\delta_0 = (1\, ,\, x_0e_1)$ and with $\delta_1$ and $\delta_2$ defined 
by the matrices$\, :$ 
\[
\begin{pmatrix} x_0x_3 & x_1x_2 & x_2x_3 & 0 & 0\\
-b & -x_1a_0 & 0 & x_2 & x_3 \end{pmatrix}\, ,\  
\begin{pmatrix} -x_2 & 0 & 0\\
0 & -x_3 & 0\\
x_0 & x_1 & 0\\
-b & 0 & -x_3\\
0 & -x_1a_0 & x_2 \end{pmatrix}\, .
\]

\emph{(d)} If $x_1 \nmid b$ and $x_0 \nmid a$ then the graded $S$-module 
${\fam0 H}^0_\ast(\sco_{X \cup L_2 \cup L_2^\prim})$ admits the following graded free 
resolution$\, :$ 
\[
0 \lra \begin{matrix} 2S(-3)\\ \oplus\\ S(l-2) \end{matrix} 
\overset{\displaystyle \delta_2}{\lra} 
\begin{matrix} 3S(-2)\\ \oplus\\ 2S(l-1) \end{matrix} 
\overset{\displaystyle \delta_1}{\lra} 
\begin{matrix} S\\ \oplus\\ S(l) \end{matrix} 
\overset{\displaystyle \delta_0}{\lra} 
{\fam0 H}^0_\ast(\sco_{X \cup L_2 \cup L_2^\prim}) \lra 0  
\]
with $\delta_0 = (1\, ,\, e_1)$ and with $\delta_1$ and $\delta_2$ defined 
by the matrices$\, :$ 
\[
\begin{pmatrix} x_0x_3 & x_1x_2 & x_2x_3 & 0 & 0\\
-x_0b & -x_1a & 0 & x_2 & x_3 \end{pmatrix}\, ,\  
\begin{pmatrix} -x_2 & 0 & 0\\
0 & -x_3 & 0\\
x_0 & x_1 & 0\\
-x_0b & 0 & -x_3\\
0 & -x_1a & x_2 \end{pmatrix}\, .
\]
\end{prop}

\begin{proof}
$I(L_1 \cup L_2 \cup L_2^\prim) = (x_2,\, x_3) \cap (x_1,\, x_3) \cap 
(x_0,\, x_2) = (x_0x_3,\, x_1x_2,\, x_2x_3)$. 
$L_1 \cup L_2 \cup L_2^\prim$ is directly linked by the complete intersection 
defined by $x_0x_3$ and $x_1x_2$ to the curve whose homogeneous ideal is 
$(x_0,\, x_1)$, i.e., to the line $L_1^\prim$. Using Ferrand's result about 
liaison one gets a minimal graded free resolution$\, :$
\[
0 \lra 2S(-3) \xra{\begin{pmatrix} -x_2 & 0\\ 0 & -x_3\\ x_0 & x_1 
\end{pmatrix}} 3S(-2) \lra I(L_1 \cup L_2 \cup L_2^\prim) \lra 0\, .
\] 

If $x_1 \mid b$ we use the exact sequence (defined at the beginning of 
this subsection)$\, :$ 
\[
0 \lra \text{Im}\, \phi_1^\prim \times \{0\} \lra \sco_{X \cup L_2 \cup L_2^\prim} 
\lra \sco_{L_1 \cup L_2 \cup L_2^\prim} \lra 0 
\] 
with $\phi_1^\prim$ the composite morphism $\sci_{L_1 \cup L_2 \cup L_2^\prim} 
\ra \sci_{L_1 \cup L_2} \overset{\phi^\prim}{\lra} \sco_{L_1}(l-1)$. Recall, from 
the proof of Prop.~\ref{P:reshoxcupl}, that $\phi^\prim(x_3) = b_1$ and 
$\phi^\prim(x_1x_2) = a$ hence $\, :$ 
\[
\phi_1^\prim(x_0x_3) = x_0b_1,\  \phi_1^\prim(x_1x_2) = a,\  
\phi_1^\prim(x_2x_3) =0\, .
\]

If $x_1 \nmid b$ we use the exact sequence $\, :$ 
\[
0 \lra \text{Im}\, \phi_1 \times \{0\} \lra \sco_{X \cup L_2 \cup L_2^\prim} 
\lra \sco_{L_1 \cup L_2 \cup L_2^\prim} \lra 0 
\] 
with $\phi_1$ the composite morphism $\sci_{L_1 \cup L_2 \cup L_2^\prim} 
\ra \sci_{L_1 \cup L_2} \overset{\phi}{\lra} \sco_{L_1}(l)$. Recall, from the 
proof of Prop.~\ref{P:reshoxcupl}, that $\phi(x_3) = b$ and 
$\phi(x_1x_2) = x_1a$ hence$\, :$ 
\[
\phi_1(x_0x_3) = x_0b,\  \phi_1(x_1x_2) = x_1a,\  
\phi_1(x_2x_3) =0\, .
\]

(a) In this case $\text{Im}\, \phi_1^\prim = x_0\sco_{L_1}(l-2)$. Since 
$L_1 \cup L_2 \cup L_2^\prim$ is arithmetically CM, the graded $S$-module 
$\tH^0_\ast(\sco_{L_1 \cup L_2 \cup L_2^\prim})$ is generated by $1 \in 
\tH^0(\sco_{L_1 \cup L_2 \cup L_2^\prim})$. One deduces an exact sequence$\, :$ 
\[
0 \lra S(L_1)(l-2) \lra \tH^0_\ast(\sco_{X \cup L_2 \cup L_2^\prim}) \lra 
\tH^0_\ast(\sco_{L_1 \cup L_2 \cup L_2^\prim}) \lra 0
\]
where the left morphism maps $1 \in S(L_1)$ to the element of 
$\tH^0(\sco_{X \cup L_2 \cup L_2^\prim}(-l+2))$ whose image into 
$\tH^0(\sco_X(-l+2)) \oplus \tH^0(\sco_{L_2}(-l+2)) \oplus 
\tH^0(\sco_{L_2^\prim}(-l+2))$ is $(x_0x_1e_1,\, 0,\, 0)$. 

(b) In this case $\text{Im}\, \phi_1^\prim = \sco_{L_1}(l-1)$. One deduces an 
exact sequence$\, :$ 
\[
0 \lra S(L_1)(l-1) \lra \tH^0_\ast(\sco_{X \cup L_2 \cup L_2^\prim}) \lra 
\tH^0_\ast(\sco_{L_1 \cup L_2 \cup L_2^\prim}) \lra 0
\]
where the left morphism maps $1 \in S(L_1)$ to the element of 
$\tH^0(\sco_{X \cup L_2 \cup L_2^\prim}(-l+1))$ whose image into 
$\tH^0(\sco_X(-l+1)) \oplus \tH^0(\sco_{L_2}(-l+1)) \oplus 
\tH^0(\sco_{L_2^\prim}(-l+1))$ is $(x_1e_1,\, 0,\, 0)$. 

(c) In this case $\text{Im}\, \phi_1 = x_0\sco_{L_1}(l-1)$. One deduces an 
exact sequence$\, :$ 
\[
0 \lra S(L_1)(l-1) \lra \tH^0_\ast(\sco_{X \cup L_2 \cup L_2^\prim}) \lra 
\tH^0_\ast(\sco_{L_1 \cup L_2 \cup L_2^\prim}) \lra 0
\]
where the left morphism maps $1 \in S(L_1)$ to the element of 
$\tH^0(\sco_{X \cup L_2 \cup L_2^\prim}(-l+1))$ whose image into 
$\tH^0(\sco_X(-l+1)) \oplus \tH^0(\sco_{L_2}(-l+1)) \oplus 
\tH^0(\sco_{L_2^\prim}(-l+1))$ is $(x_0e_1,\, 0,\, 0)$. 

(d) In this case $\text{Im}\, \phi_1 = \sco_{L_1}(l)$. One deduces an 
exact sequence$\, :$ 
\[
0 \lra S(L_1)(l) \lra \tH^0_\ast(\sco_{X \cup L_2 \cup L_2^\prim}) \lra 
\tH^0_\ast(\sco_{L_1 \cup L_2 \cup L_2^\prim}) \lra 0
\]
where the left morphism maps $1 \in S(L_1)$ to the element of 
$\tH^0(\sco_{X \cup L_2 \cup L_2^\prim}(-l))$ whose image into 
$\tH^0(\sco_X(-l)) \oplus \tH^0(\sco_{L_2}(-l)) \oplus 
\tH^0(\sco_{L_2^\prim}(-l))$ is $(e_1,\, 0,\, 0)$. 
\end{proof}

\begin{lemma}\label{L:l1(1)cupl2cupl3} 
The homogeneous ideal $I(W) \subset S$ of $W = L_1^{(1)} \cup L_2 \cup 
L_3$ is generated by $x_2x_3$, $x_1x_2^2$, $x_1x_3^2$ and admits the 
following graded free resolution$\, :$ 
\[
0 \lra 2S(-4) \xra{\begin{pmatrix} -x_1x_2 & -x_1x_3\\ x_3 & 0\\ 0 & x_2 
\end{pmatrix}} S(-2) \oplus 2S(-3) \lra I(W) \lra 0\, .
\]
\end{lemma}

\begin{proof} 
$I(W)$ is, by definition$\, :$ 
\[
(x_2^2,\, x_2x_3,\, x_3^2) \cap (x_1,\, x_3) \cap (x_1,\, x_2) = 
(x_2x_3,\, x_3^2,\, x_1x_2^2) \cap (x_1,\, x_2) = 
(x_2x_3,\, x_1x_2^2,\, x_1x_3^2)\, .  
\]
If $W^\prim$ is the curve directly linked to $W$ by the complete intersection 
defined by $x_2x_3$ and $x_1(x_2^2 + x_3^2)$ then 
$I(W^\prim) = (x_2,\, x_3)$ (because $x_2 \cdot x_1x_2^2 = x_2 \cdot x_1(x_2^2 + 
x_3^2) - x_1x_3 \cdot x_2x_3$), i.e., $W^\prim = L_1$.  
One can apply, now, Ferrand's result about liaison. 
\end{proof}

\begin{prop}\label{P:genixcupl2cupl3}
Let $X$ be the double structure on the line $L_1$ considered at the beginning 
of Subsection~\ref{SS:doublecupaline}. Then $I(X \cup L_2 \cup L_3) = 
Sx_1F_2 + I(L_1^{(1)} \cup L_2 \cup L_3)$. 
\end{prop}

\begin{proof}
If $x_1 \mid b$ then one has an exact sequence (look at the beginning of 
this subsection)$\, :$ 
\[
0 \lra \sci_{L_1^{(1)} \cup L_2 \cup L_3} \lra \sci_{X \cup L_2 \cup L_3} 
\overset{\displaystyle \psi_1}{\lra} \sco_{L_1}(-l-2) 
\]
where $\psi_1$ is the composite morphism $\sci_{X \cup L_2 \cup L_3} \ra  
\sci_{X \cup L_2} \overset{\psi}{\lra} \sco_{L_1}(-l-2)$. Since $L_1 \cap L_3 = 
\{P_0\}$, $\Cok \psi_1$ is an $\sco_{\{P_0\}}$-module hence 
$x_1\sco_{L_1}(-l-3) \subseteq \text{Im}\, \psi_1 \subseteq \sco_{L_1}(-l-2)$. 

\vskip2mm 

\noindent
{\bf Claim.}\quad $\text{Im}\, \psi_1 = x_1\sco_{L_1}(-l-3)$. 

\vskip2mm

\noindent
\emph{Indeed}, assume that $\text{Im}\, \psi_1 = \sco_{L_1}(-l-2)$. Since 
$\tH^1_\ast(\sci_{L_1^{(1)} \cup L_2 \cup L_3}) = 0$ (by 
Lemma~\ref{L:l1(1)cupl2cupl3}) it follows that there exists $f \in 
\tH^0(\sci_{X \cup L_2 \cup L_3}(l+2))$ such that $\psi_1(f) = 1 \in 
\tH^0(\sco_{L_1})$. But $\psi_1(f) = \psi(f)$ hence, by the proof of 
Prop.~\ref{P:genixcupl}, $f$ must have the form$\, :$ 
\[
f = F_2 + f_0x_2x_3 + f_1x_3^2 + f_2x_1x_2^2\, .
\]
Since $0 = f \vb L_3 = a(x_0,0)x_3 + f_1(x_0,0,0,x_3)x_3^2$ it follows that 
$a(x_0,0) = 0$ hence $x_1 \mid a$. But this \emph{contradicts} the fact that 
$a$ and $b$ are coprime. 

\vskip2mm

\noindent
It follows from the Claim that $\psi_1$ factorizes as 
$\sci_{X \cup L_2 \cup L_3} \overset{\psi_2}{\lra} \sco_{L_1}(-l-3) 
\overset{x_1}{\lra} \sco_{L_1}(-l-2)$ with $\psi_2$ mapping $x_1F_2 \in 
\tH^0(\sci_{X \cup L_2 \cup L_3}(l+3))$ to $1 \in \tH^0(\sco_{L_1})$. One deduces  
an exact sequence$\, :$ 
\[
0 \lra I(L_1^{(1)} \cup L_2 \cup L_3) \lra I(X \cup L_2 \cup L_3) 
\xra{\displaystyle \tH^0_\ast(\psi_2)} S(L_1)(-l-3) \lra 0\, .
\]

\noindent 
$\bullet$\quad If $x_1 \nmid b$ then one has an exact sequence$\, :$ 
\[
0 \lra \sci_{L_1^{(1)} \cup L_2 \cup L_3} \lra \sci_{X \cup L_2 \cup L_3} 
\overset{\displaystyle \psi_1^\prim}{\lra} \sco_{L_1}(-l-3) 
\]
where $\psi_1^\prim$ is the composite morphism $\sci_{X \cup L_2 \cup L_3} \ra  
\sci_{X \cup L_2} \overset{\psi^\prim}{\lra} \sco_{L_1}(-l-3)$. Since $x_1F_2 
\vb L_3 = 0$ it follows that $x_1F_2 \in I(X \cup L_2 \cup L_3)$. But 
$\psi_1^\prim(x_1F_2) = \psi^\prim(x_1F_2) = 1 \in \tH^0(\sco_{L_1})$. One deduces 
that $\psi_1^\prim$ is an epimorphism and that, moreover, one has an exact 
sequence$\, :$ 
\[
0 \lra I(L_1^{(1)} \cup L_2 \cup L_3) \lra I(X \cup L_2 \cup L_3) 
\xra{\displaystyle \tH^0_\ast(\psi_1^\prim)} S(L_1)(-l-3) \lra 0\, .
\]

Notice that one can use this exact sequence (and the similar one for the case 
$x_1 \mid b$) to get a concrete graded free resolution of 
$I(X \cup L_2 \cup L_3)$. 
\end{proof}

\begin{prop}\label{P:reshoxcupl2cupl3}
Let $X$ be the double structure on the line $L_1$ considered at the beginning 
of Subsection~\ref{SS:doublecupaline}. Then the graded $S$-module 
${\fam0 H}^0_\ast(\sco_{X \cup L_2 \cup L_3})$ admits the following graded free 
resolution$\, :$ 
\[
0 \lra \begin{matrix} 2S(-3)\\ \oplus\\ S(l-3) \end{matrix} 
\overset{\displaystyle \delta_2}{\lra} 
\begin{matrix} 3S(-2)\\ \oplus\\ 2S(l-2) \end{matrix} 
\overset{\displaystyle \delta_1}{\lra} 
\begin{matrix} S\\ \oplus\\ S(l-1) \end{matrix} 
\overset{\displaystyle \delta_0}{\lra} 
{\fam0 H}^0_\ast(\sco_{X \cup L_2 \cup L_3}) \lra 0  
\]
with $\delta_0 = (1\, ,\, x_1e_1)$ and with $\delta_1$ and $\delta_2$ defined 
by the matrices$\, :$ 
\[
\begin{pmatrix} x_1x_2 & x_1x_3 & x_2x_3 & 0 & 0\\
-a & -b & 0 & x_2 & x_3 \end{pmatrix}\, ,\  
\begin{pmatrix} 
-x_3 & 0 & 0\\
x_2 & -x_2 & 0\\
0 & x_1 & 0\\
b & -b & -x_3\\ 
-a & 0 & x_2 
\end{pmatrix}\, .
\]
\end{prop}

\begin{proof}
$I(L_1 \cup L_2 \cup L_3) = (x_2,\, x_3) \cap (x_1,\, x_3) \cap (x_1,\, x_2) 
= (x_1x_2,\, x_1x_3,\, x_2x_3)$. One deduces a minimal graded free 
resolution$\, :$ 
\[
0 \lra 2S(-3) \xra{\begin{pmatrix} -x_3 & 0\\ x_2 & -x_2\\ 0 & x_1 
\end{pmatrix}} 3S(-2) \lra I(L_1 \cup L_2 \cup L_3) \lra 0\, . 
\]

\noindent 
$\bullet$\quad If $x_1 \mid b$, i.e., if $b = x_1b_1$ we use the exact sequence 
(defined at the beginning of this subsection)$\, :$ 
\[
0 \lra \text{Im}\, \phi_1^\prim \times \{0\} \lra \sco_{X \cup L_2 \cup L_3} 
\lra \sco_{L_1 \cup L_2 \cup L_3} \lra 0 
\] 
with $\phi_1^\prim$ the composite morphism $\sci_{L_1 \cup L_2 \cup L_3} 
\ra \sci_{L_1 \cup L_2} \overset{\phi^\prim}{\lra} \sco_{L_1}(l-1)$. Recall, from 
the proof of Prop.~\ref{P:reshoxcupl}, that $\phi^\prim(x_3) = b_1$ and 
$\phi^\prim(x_1x_2) = a$ hence$\, :$ 
\[
\phi_1^\prim(x_1x_2) = a,\  \phi_1^\prim(x_1x_3) = x_1b_1 = b,\  
\phi_1^\prim(x_2x_3) = 0\, . 
\]
It follows that $\phi_1^\prim$ is an epimorphism. Moreover, since $L_1 \cup 
L_2 \cup L_3$ is arithmetically CM the graded $S$-module 
$\tH^0_\ast(\sco_{L_1 \cup L_2 \cup L_3})$ is generated by $1 \in 
\tH^0(\sco_{L_1 \cup L_2 \cup L_3})$. One deduces an exact sequence$\, :$ 
\[
0 \lra S(L_1)(l-1) \lra \tH^0_\ast(\sco_{X \cup L_2 \cup L_3}) \lra 
\tH^0_\ast(\sco_{L_1 \cup L_2 \cup L_3}) \lra 0
\] 
where the left morphism maps $1 \in S(L_1)$ to the element of 
$\tH^0(\sco_{X \cup L_2 \cup L_3}(-l+1))$ whose image into $\tH^0(\sco_X(-l+1)) 
\oplus \tH^0(\sco_{L_2}(-l+1)) \oplus \tH^0(\sco_{L_3}(-l+1))$ is 
$(x_1e_1,\, 0,\, 0)$. 

\vskip2mm

\noindent 
$\bullet$\quad If $x_1 \nmid b$ we use the exact sequence $\, :$ 
\[
0 \lra \text{Im}\, \phi_1 \times \{0\} \lra \sco_{X \cup L_2 \cup L_3} \lra 
\sco_{L_1 \cup L_2 \cup L_3} \lra 0 
\] 
with $\phi_1$ the composite morphism $\sci_{L_1 \cup L_2 \cup L_3} 
\ra \sci_{L_1 \cup L_2} \overset{\phi}{\lra} \sco_{L_1}(l)$. Recall, from 
the proof of Prop.~\ref{P:reshoxcupl}, that $\phi(x_3) = b$ and 
$\phi(x_1x_2) = x_1a$ hence$\, :$ 
\[
\phi_1(x_1x_2) = x_1a,\  \phi_1(x_1x_3) = x_1b,\  
\phi_1(x_2x_3) = 0\, . 
\] 
One deduces that $\text{Im}\, \phi_1 = x_1\sco_{L_1}(l-1)$ and that one has an 
exact sequence$\, :$ 
\[
0 \lra S(L_1)(l-1) \lra \tH^0_\ast(\sco_{X \cup L_2 \cup L_3}) \lra 
\tH^0_\ast(\sco_{L_1 \cup L_2 \cup L_3}) \lra 0
\] 
where the left morphism maps $1 \in S(L_1)$ to the element of 
$\tH^0(\sco_{X \cup L_2 \cup L_3}(-l+1))$ whose image into $\tH^0(\sco_X(-l+1)) 
\oplus \tH^0(\sco_{L_2}(-l+1)) \oplus \tH^0(\sco_{L_3}(-l+1))$ is 
$(x_1e_1,\, 0,\, 0)$.  
\end{proof}

\begin{lemma}\label{L:l1(1)cupl2cupl1prim}
The homogeneous ideal $I(W) \subset S$ of $W = L_1^{(1)} \cup L_2 \cup 
L_1^\prim$ is generated by $x_0x_2x_3$, $x_0x_3^2$, $x_1x_2^2$, $x_1x_2x_3$, 
$x_1x_3^2$ and admits the following graded free resolution$\, :$ 
\[
0 \lra S(-5) \overset{\displaystyle d_2}{\lra} 5S(-4) 
\overset{\displaystyle d_1}{\lra} 5S(-3) \overset{\displaystyle d_0}{\lra} 
I(W) \lra 0 
\]
where $d_1$ and $d_2$ are defined by the matrices$\, :$ 
\[
\begin{pmatrix} 
-x_3 & 0 & 0 & -x_1 & 0\\
x_2 & 0 & 0 & 0 & -x_1\\
0 & -x_3 & 0 & 0 & 0\\
0 & x_2 & -x_3 & x_0 & 0\\
0 & 0 & x_2 & 0 & x_0
\end{pmatrix} \, ,\  
\begin{pmatrix}
-x_1\\ 0\\ x_0\\ x_3\\ -x_2
\end{pmatrix}\, .
\]
\end{lemma}

\begin{proof} 
By definition$\, :$ 
\begin{gather*}
I(W) = (x_2^2,\, x_2x_3,\, x_3^2) \cap (x_1,\, x_3) \cap (x_0,\, x_1) =\\  
(x_2x_3,\, x_3^2,\, x_1x_2^2) \cap (x_0,\, x_1) = 
(x_0x_2x_3,\, x_0x_3^2,\, x_1x_2^2,\, x_1x_2x_3,\, x_1x_3^2)\, . 
\end{gather*}
Since $I(L_1\cup L_1^\prim) = (x_0x_2,\, x_0x_3,\, x_1x_2,\, x_1x_3)$, it 
follows that$\, :$ 
\[
I(W) = Sx_1x_2^2 + I(L_1\cup L_1^\prim)x_3\, .
\]
Moreover, by Lemma~\ref{L:zcupw}, $I(L_1\cup L_1^\prim)$ admits a graded free 
resolution of the form$\, :$ 
\[
0 \lra S(-4) \lra 4S(-3) \lra 4S(-2) \lra I(L_1\cup L_1^\prim) \lra 0\, .
\]
One can apply, now, Remark~\ref{R:basicdoublelinkage}. Actually, by the 
main property of the basic double linkage, $L_1 \cup L_1^\prim$ can be 
obtained from $W$ by two direct linkages. Concretely, these two linkages are 
the following ones$\, :$ if $W^\prim$ is the curve directly linked to $W$ by 
the complete intersection defined by $x_0x_3^2$ and $x_1x_2^2$ then, by 
Remark~\ref{R:monomial}(b), 
\[
I(W^\prim) = (x_3,\, x_1x_2) \cap (x_2,\, x_0x_3) \cap (x_0,\, x_2^2) = 
(x_0x_3,\, x_0x_1x_2,\, x_1x_2^2,\, x_2^2x_3)\, . 
\] 
If $W^\secund$ is the curve directly linked to $W^\prim$ by the complete 
intersection defined by $x_0x_3$ and $x_1x_2^2$ then, by the same remark, 
$I(W^\secund) = (x_2,\, x_3) \cap (x_0,\, x_1)$ hence $W^\secund = L_1 \cup 
L_1^\prim$. 
\end{proof}

\begin{prop}\label{P:genixcupl2cupl1prim} 
Let $X$ be the double structure on the line $L_1$ considered at the beginning 
of Subsection~\ref{SS:doublecupaline}. 

\emph{(a)} If $l = -1$ and $b = 0$ then $I(X \cup L_2 \cup L_1^\prim) 
= (x_0x_3,\, x_1x_3,\, x_1x_2^2)$. 

\emph{(b)} If $l \geq 0$ and $x_1 \mid b$ then $I(X\cup L_2 \cup L_1^\prim) 
= SF_2 + I(L_1^{(1)}\cup L_2 \cup L_1^\prim)$. 

\emph{(c)} If $x_1 \nmid b$ then $I(X\cup L_2 \cup L_1^\prim) = Sx_1F_2 + 
I(L_1^{(1)}\cup L_2 \cup L_1^\prim)$.
\end{prop}

\begin{proof} 
(a) In this case $F_2 = x_3$ hence $I(X) = (x_3,\, x_2^2)$ hence 
$I(X \cup L_2 \cup L_1^\prim) = (x_3,\, x_2^2) \cap (x_1,\, x_3) \cap 
(x_0,\, x_1) = (x_0x_3,\, x_1x_3,\, x_1x_2^2)$. 

(b) We use the exact sequence$\, :$ 
\[
0 \lra \sci_{L_1^{(1)} \cup L_2 \cup L_1^\prim} \lra 
\sci_{X \cup L_2 \cup L_1^\prim} \overset{\displaystyle \psi_1}{\lra} 
\sco_{L_1}(-l-2) 
\]
defined at the beginning of this subsection. Since $L_1 \cap L_1^\prim = 
\emptyset$ it follows that $\psi_1$ is an epimorphism. Moreover, since 
$l \geq 0$ it follows that $F_2 \vb L_1^\prim = 0$ hence $F_2 \in 
I(X \cup L_2 \cup L_1^\prim)$. Since $\psi_1(F_2) = \psi(F_2) = 1 \in 
\tH^0(\sco_{L_1})$ one deduces an exact sequence$\, :$ 
\[
0 \lra I(L_1^{(1)} \cup L_2 \cup L_1^\prim) \lra 
I(X \cup L_2 \cup L_1^\prim) 
\xra{\displaystyle \tH^0_\ast(\psi_1)} S(L_1)(-l-2) \lra 0\, .
\] 

(c) We use the exact sequence$\, :$ 
\[
0 \lra \sci_{L_1^{(1)} \cup L_2 \cup L_1^\prim} \lra 
\sci_{X \cup L_2 \cup L_1^\prim} \overset{\displaystyle \psi_1^\prim}{\lra} 
\sco_{L_1}(-l-3) 
\]
defined at the beginning of this subsection. Since $L_1 \cap L_1^\prim = 
\emptyset$ it follows that $\psi_1^\prim$ is an epimorphism. Moreover, 
since $x_1F_2 \vb L_1^\prim = 0$ it follows that $x_1F_2 \in 
I(X \cup L_2 \cup L_1^\prim)$. Since $\psi_1^\prim(x_1F_2) = \psi^\prim(x_1F_2) 
= 1 \in \tH^0(\sco_{L_1})$ one deduces an exact sequence$\, :$ 
\[
0 \lra I(L_1^{(1)} \cup L_2 \cup L_1^\prim) \lra 
I(X \cup L_2 \cup L_1^\prim) 
\xra{\displaystyle \tH^0_\ast(\psi_1^\prim)} S(L_1)(-l-3) \lra 0\, .
\] 

Notice that the exact sequences deduced in the above proof can be used to get 
a concrete graded free resolution of $I(X \cup L_2 \cup L_1^\prim)$. 
\end{proof}

\begin{prop}\label{P:reshoxcupl2cupl1prim} 
Let $X$ be the double structure on the line $L_1$ considered at the beginning 
of Subsection~\ref{SS:doublecupaline}. 

\emph{(a)} If $x_1 \mid b$, i.e., if $b = x_1b_1$ then the graded $S$-module 
${\fam0 H}^0_\ast(\sco_{X\cup L_2\cup L_1^\prim})$ admits the following free 
resolution$\, :$ 
\[
0 \lra \begin{matrix} 2S(-3)\\ \oplus\\ S(l-3) \end{matrix} 
\overset{\displaystyle \delta_2}{\lra} 
\begin{matrix} 3S(-2)\\ \oplus\\ 2S(l-2) \end{matrix} 
\overset{\displaystyle \delta_1}{\lra} 
\begin{matrix} S\\ \oplus\\ S(l-1) \end{matrix} 
\overset{\displaystyle \delta_0}{\lra} 
{\fam0 H}^0_\ast(\sco_{X\cup L_2\cup L_1^\prim}) \lra 0 
\]
with $\delta_0 = (1\, ,\, x_1e_1)$ and $\delta_1$, $\delta_2$ defined by the 
matrices$\, :$
\[
\begin{pmatrix} 
x_0x_3 & x_1x_2 & x_1x_3 & 0 & 0\\
-x_0b_1 & -a & -b & x_2 & x_3
\end{pmatrix}\, ,\  
\begin{pmatrix} 
-x_1 & 0 & 0\\
0 & -x_3 & 0\\
x_0 & x_2 & 0\\
0 & b & -x_3\\
0 & -a & x_2
\end{pmatrix}\, .
\]

\emph{(b)} If $x_1 \nmid b$ then the graded $S$-module 
${\fam0 H}^0_\ast(\sco_{X\cup L_2\cup L_1^\prim})$ admits the following free 
resolution$\, :$ 
\[
0 \lra \begin{matrix} 2S(-3)\\ \oplus\\ S(l-2) \end{matrix} 
\overset{\displaystyle \delta_2}{\lra} 
\begin{matrix} 3S(-2)\\ \oplus\\ 2S(l-1) \end{matrix} 
\overset{\displaystyle \delta_1}{\lra} 
\begin{matrix} S\\ \oplus\\ S(l) \end{matrix} 
\overset{\displaystyle \delta_0}{\lra} 
{\fam0 H}^0_\ast(\sco_{X\cup L_2\cup L_1^\prim}) \lra 0 
\]
with $\delta_0 = (1\, ,\, e_1)$ and $\delta_1$, $\delta_2$ defined by the 
matrices$\, :$
\[
\begin{pmatrix} 
x_0x_3 & x_1x_2 & x_1x_3 & 0 & 0\\
-x_0b & -x_1a & -x_1b & x_2 & x_3
\end{pmatrix}\, ,\  
\begin{pmatrix} 
-x_1 & 0 & 0\\
0 & -x_3 & 0\\
x_0 & x_2 & 0\\
0 & x_1b & -x_3\\
0 & -x_1a & x_2
\end{pmatrix}\, .
\]
\end{prop}

\begin{proof} 
The homogeneous ideal of $L_1\cup L_2\cup L_1^\prim$ is$\, :$ 
\[
(x_2,\, x_3)\cap (x_1,\, x_3)\cap (x_0,\, x_1) = 
(x_0x_3,\, x_1x_2,\, x_1x_3)\, .
\]
One deduces that $L_1\cup L_2\cup L_1^\prim$ is directly linked by the complete 
intersection defined by $x_0x_3$ and $x_1x_2$ to the line $L_2^\prim$ of 
equations $x_0 = x_2 = 0$. Using Ferrand's result about liaison, one gets 
the following graded free resolution of $I(L_1\cup L_2\cup L_1^\prim)$$\, :$ 
\[
0 \lra 2S(-3) \xra{\begin{pmatrix} -x_1 & 0\\ 0 & -x_3\\ x_0 & x_2 
\end{pmatrix}} 3S(-2) \lra I(L_1\cup L_2\cup L_1^\prim) \lra 0\, .
\]
It results that $L_1\cup L_2\cup L_1^\prim$ is arithmetically CM, hence the 
graded $S$-module $\tH^0_\ast(\sco_{L_1\cup L_2\cup L_1^\prim})$ is generated by 
$1 \in \tH^0(\sco_{L_1\cup L_2\cup L_1^\prim})$. 

\vskip2mm

(a) We use the exact sequence (look at the beginning of 
this subsection)$\, :$ 
\[
0 \lra \text{Im}\, \phi_1^\prim \times \{0\} \lra \sco_{X \cup L_2 \cup L_1^\prim} 
\lra \sco_{L_1 \cup L_2 \cup L_1^\prim} \lra 0 
\] 
with $\phi_1^\prim$ the composite morphism $\sci_{L_1 \cup L_2 \cup L_1^\prim} 
\ra \sci_{L_1 \cup L_2} \overset{\phi^\prim}{\lra} \sco_{L_1}(l-1)$. Since 
$L_1 \cap L_1^\prim = \emptyset$, $\phi_1^\prim$ is an epimorphism. One deduces 
an exact sequence$\, :$ 
\[
0 \lra S(L_1)(l-1) \lra \tH^0_\ast(\sco_{X \cup L_2 \cup L_1^\prim}) \lra 
\tH^0_\ast(\sco_{L_1 \cup L_2 \cup L_1^\prim}) \lra 0 
\]
where the left morphism maps $1 \in S(L_1)$ to the element of 
$\tH^0(\sco_{X \cup L_2 \cup L_1^\prim}(-l+1))$ whose image into 
$\tH^0(\sco_X(-l+1)) \oplus \tH^0(\sco_{L_2}(-l+1)) \oplus 
\tH^0(\sco_{L_1^\prim}(-l+1))$ is $(x_1e_1,\, 0,\, 0)$. 

(b) We use the exact sequence $\, :$ 
\[
0 \lra \text{Im}\, \phi_1 \times \{0\} \lra \sco_{X \cup L_2 \cup L_1^\prim} \lra 
\sco_{L_1 \cup L_2 \cup L_1^\prim} \lra 0 
\] 
with $\phi_1$ the composite morphism $\sci_{L_1 \cup L_2 \cup L_1^\prim} 
\ra \sci_{L_1 \cup L_2} \overset{\phi}{\lra} \sco_{L_1}(l)$.  Since 
$L_1 \cap L_1^\prim = \emptyset$, $\phi_1$ is an epimorphism. One deduces 
an exact sequence$\, :$ 
\[
0 \lra S(L_1)(l) \lra \tH^0_\ast(\sco_{X \cup L_2 \cup L_1^\prim}) \lra 
\tH^0_\ast(\sco_{L_1 \cup L_2 \cup L_1^\prim}) \lra 0 
\]
where the left morphism maps $1 \in S(L_1)$ to the element of 
$\tH^0(\sco_{X \cup L_2 \cup L_1^\prim}(-l))$ whose image into 
$\tH^0(\sco_X(-l)) \oplus \tH^0(\sco_{L_2}(-l)) \oplus 
\tH^0(\sco_{L_1^\prim}(-l))$ is $(e_1,\, 0,\, 0)$. 
\end{proof}  

\begin{lemma}\label{L:l1(1)cuplcupl1prim} 
Let $L\subset \piii$ be the line of equations $\ell = x_3 = 0$, where 
$\ell = x_1 + cx_2$, $c \neq 0$. Then the homogeneous ideal 
$I(W) \subset S$ of $W = L_1^{(1)} \cup L\cup L_1^\prim$ is generated by 
$x_0x_2x_3$, $x_0x_3^2$, $x_1x_2x_3$, $x_1x_3^2$, $x_0\ell x_2^2$, $x_1\ell x_2^2$ 
and admits the following graded free resolution$\, :$  
\[
0 \lra S(-5) \oplus S(-6) \overset{\displaystyle d_2}{\lra} 
4S(-4) \oplus 3S(-5) \overset{\displaystyle d_1}{\lra} 
4S(-3) \oplus 2S(-4) \overset{\displaystyle d_0}{\lra} I(W) \lra 0 
\]
with $d_1$ and $d_2$ defined by the matrices$\, :$ 
\[
\begin{pmatrix} 
-x_1 & 0 & -x_3 & 0 & 0 & -\ell x_2 & 0\\ 
0 & -x_1 & x_2 & 0 & 0 & 0 & 0\\
x_0 & 0 & 0 & -x_3 & 0 & 0 & -\ell x_2\\
0 & x_0 & 0 & x_2 & 0 & 0 & 0\\
0 & 0 & 0 & 0 & -x_1 & x_3 & 0\\
0 & 0 & 0 & 0 & x_0 & 0 & x_3 
\end{pmatrix}\  ,\  
\begin{pmatrix} 
-x_3 & -\ell x_2\\
x_2 & 0\\
x_1 & 0\\
-x_0 & 0\\
0 & x_3\\
0 & x_1\\
0 & -x_0 
\end{pmatrix}\, . 
\]
\end{lemma} 

\begin{proof} 
According to Lemma~\ref{L:l1(1)cupl}, $L_1^{(1)}\cup L$ is arithmetically CM. 
It follows, now, from Lemma~\ref{L:zcupw}, that 
$I(W) = I(L_1^{(1)}\cup L)I(L_1^\prim)$ and that the 
tensor product of the minimal graded free resolutions of $I(L_1^{(1)}\cup L)$ 
and of $I(L_1^\prim)$ is a minimal graded free resolution of $I(W)$.   
\end{proof}

\begin{prop}\label{P:genixcuplcupl1prim} 
Let $X$ be the double structure on the line $L_1$ considered at the beginning 
of Subsection~\ref{SS:doublecupaline} and let $L\subset \piii$ be the line of 
equations $\ell = x_3 = 0$, where $\ell = x_1 + cx_2$, $c \neq 0$. 

\emph{(a)} If $x_1 \mid b$ and $l = -1$ then $I(X\cup L\cup L_1^\prim) = 
(x_0,\, x_1)(x_3,\, \ell x_2^2)$. 

\emph{(b)} If $x_1 \mid b$ and $l = 0$ then $I(X\cup L\cup L_1^\prim) = 
(x_0,\, x_1)(-\ell x_2 + ax_3,\, x_2x_3,\, x_3^2)$. 

\emph{(c)} If $x_1 \mid b$, i.e., if $b = x_1b_1$, and $l \geq 1$ then 
$I(X\cup L\cup L_1^\prim) = S(F_2-cb_1x_2^2) + I(L_1^{(1)}\cup L\cup L_1^\prim)$. 

\emph{(d)} If $x_1 \nmid b$ and $l = -1$ then $I(X\cup L\cup L_1^\prim) = 
(x_0,\, x_1)(\ell F_2,\, x_2x_3,\, x_3^2)$. 

\emph{(e)} If $x_1 \nmid b$ and $l \geq 0$ then 
$I(X\cup L\cup L_1^\prim) = S\ell F_2 + I(L_1^{(1)}\cup L\cup L_1^\prim)$.   
\end{prop}

\begin{proof} 
(a) If $l = -1$ the condition $x_1 \mid b$ means that $b = 0$, hence 
$F_2 = x_3$, hence $I(X) = (x_3,\, x_2^2)$. One deduces that $I(X\cup L) = 
(x_3,\, \ell x_2^2)$ and one can apply, now, Lemma~\ref{L:zcupw}. 

(b) If $l = 0$ and $x_1 \mid b$ then one can assume that $b = x_1$, 
hence $F_2 = -x_1x_2 + ax_3$. It follows, from Prop.~\ref{P:genixcupl},  
that$\, :$ 
\[
I(X\cup L) = (-x_1x_2 + ax_3 - cx_2^2,\, x_2x_3,\, x_3^2,\, \ell x_2^2) = 
(-\ell x_2 + ax_3,\, x_2x_3,\, x_3^2)\, . 
\] 
One deduces that $X\cup L$ is directly linked by the complete intersection 
defined by $-\ell x_2 + ax_3$ and $x_3^2$ to the line $L$, hence it is 
arithmetically CM. One can apply, now, Lemma~\ref{L:zcupw}. 

(c) We use the exact sequence$\, :$ 
\[
0 \lra \sci_{L_1^{(1)}\cup L\cup L_1^\prim} \lra \sci_{X\cup L\cup L_1^\prim} 
\overset{\displaystyle \psi_1}{\lra} \sco_{L_1}(-l-2) 
\]
defined at the beginning of this subsection. $L_1 \cap L_1^\prim = \emptyset$ 
implies that $\psi_1$ is an epimorphism.  
Since $l \geq 1$, the element $F_2 - cb_1x_2^2$ of $I(X\cup L)$ belongs to 
$I(X\cup L\cup L_1^\prim)$ and this implies that the sequence$\, :$ 
\[
0 \lra I(L_1^{(1)}\cup L\cup L_1^\prim) \lra I(X\cup L\cup L_1^\prim) 
\xra{\displaystyle \tH^0_\ast(\psi_1)} S(L_1)(-l-2) \lra 0 
\] 
is exact. 

(d) If $l = -1$ the condition $x_1 \nmid b$ is equivalent to $b\neq 0$. 
One can assume that $b = -1$, hence $F_2 = x_2 + ax_3$, hence 
$I(X) = (x_2 + ax_3,\, x_3^2)$. It follows that $I(X\cup L) = 
((x_1+cx_2)(x_2+ax_3),\, x_2x_3,\, x_3^2)$. One deduces that $X\cup L$ is 
directly linked by the complete intersection defined by 
$(x_1+cx_2)(x_2+ax_3)$ and $x_3^2$ to $L$, hence it is arithmetically CM. 
One can apply, now, Lemma~\ref{L:zcupw}. 

(e) If $x_1 \nmid b$ one gets, as in the proof of (c), an exact 
sequence$\, :$
\[
0 \lra \sci_{L_1^{(1)}\cup L\cup L_1^\prim} \lra \sci_{X\cup L\cup L_1^\prim} 
\overset{\displaystyle \psi_1^\prim}{\lra} \sco_{L_1}(-l-3) \lra 0\, . 
\]  
Since $l \geq 0$, the generator $\ell F_2$ of $I(X\cup L)$ (see 
Prop.~\ref{P:genixcupl}) belongs to $I(X\cup L\cup L_1^\prim)$ and this implies 
that the sequence$\, :$ 
\[
0 \lra I(L_1^{(1)}\cup L\cup L_1^\prim) \lra I(X\cup L\cup L_1^\prim) 
\xra{\displaystyle \tH^0_\ast(\psi_1^\prim)} S(L_1)(-l-3) \lra 0 
\] 
is exact.  

Notice that using the exact sequences from the above proof one can get a 
concrete graded free resolution of $I(X \cup L \cup L_1^\prim)$. 
\end{proof}

The next lemma follows immediately from Lemma~\ref{L:zcupw}. 

\begin{lemma}\label{L:l1(1)cupl1primcuplprim}
Let $L^\prim \subset \piii$ be the line of equations $\ell_0 = x_1 = 0$, 
where $\ell_0 = x_0 + cx_2$, $c \neq 0$.   
Then the homogeneous ideal $I(W) \subset S$ of  
$W = L_1^{(1)} \cup L_1^\prim \cup L^\prim$ is generated by  
$x_1x_2^2$, $x_1x_2x_3$, $x_1x_3^2$, $x_0\ell_0x_2^2$, $x_0\ell_0x_2x_3$, 
$x_0\ell_0x_3^2$ 
and admits the following minimal graded free resolution$\, :$ 
\[
0 \lra 2S(-6) \overset{\displaystyle d_2}{\lra} 
2S(-4) \oplus 5S(-5)  
\overset{\displaystyle d_1}{\lra} 
3S(-3) \oplus 3S(-4)  
\overset{\displaystyle d_0}{\lra} 
I(W) \lra 0 
\] 
with $d_1$ and $d_2$ defined by the matrices$\, :$ 
\[
\begin{pmatrix} 
-x_3 & 0 & 0 & 0 & -x_0\ell_0 & 0 & 0\\
x_2 & -x_3 & 0 & 0 & 0 & -x_0\ell_0 & 0\\
0 & x_2 & 0 & 0 & 0 & 0 & -x_0\ell_0\\
0 & 0 & -x_3 & 0 & x_1 & 0 & 0\\
0 & 0 & x_2 & -x_3 & 0 & x_1 & 0\\
0 & 0 & 0 & x_2 & 0 & 0 & x_1
\end{pmatrix}\, ,\  
\begin{pmatrix}
x_0\ell_0 & 0\\
0 & x_0\ell_0\\
-x_1 & 0\\
0 & -x_1\\
-x_3 & 0\\
x_2 & -x_3\\
0 & x_2
\end{pmatrix}\, . 
\]
\end{lemma}

\begin{prop}\label{P:genixcupl1primcuplprim} 
Let $X$ be the double structure on the line $L_1$ considered at the beginning 
of Subsection~\ref{SS:doublecupaline} and let $L^\prim \subset \piii$ be the 
line of equations $\ell_0 = x_1 = 0$, where $\ell_0 = x_0 + cx_2$, $c \neq 0$. 

\emph{(a)} If $l = -1$ then $I(X \cup L_1^\prim \cup L^\prim) = I(X)I(L_1^\prim 
\cup L^\prim)$. 

\emph{(b)} If $l = 0$ then $I(X \cup L_1^\prim \cup L^\prim) = S\ell_0F_2 + 
Sx_1F_2 + I(L_1^{(1)} \cup L_1^\prim \cup L^\prim)$. 

\emph{(c)} If $l \geq 1$ then, writting $F_2 = x_0^2F_2^\prim + x_1F_2^\secund$, 
one has 
\[
I(X \cup L_1^\prim \cup L^\prim) = S(F_2 + cx_0x_2F_2^\prim) + 
I(L_1^{(1)} \cup L_1^\prim \cup L^\prim)\, . 
\]  
\end{prop}

\begin{proof}
Tensorizing by $\sci_{L_1^\prim \cup L^\prim}$ the exact sequence$\, :$ 
\[
0 \lra \sci_{L_1^{(1)}} \lra \sci_X \overset{\displaystyle \eta}{\lra} 
\sco_{L_1}(-l-2) \lra 0 
\]
one gets an exact sequence$\, :$ 
\[
0 \lra \sci_{L_1^{(1)} \cup L_1^\prim \cup L^\prim} \lra 
\sci_{X \cup L_1^\prim \cup L^\prim} 
\overset{\displaystyle \psi}{\lra} \sco_{L_1}(-l-2) \lra 0\, . 
\]
It follows from Lemma~\ref{L:l1(1)cupl1primcuplprim} that 
$\tH^1(\sci_{L_1^{(1)} \cup L_1^\prim \cup L^\prim}(i)) = 0$ for $i \geq 3$ hence 
$\tH^0(\psi(i))$ is surjective for $i \geq 3$. 

Recall, also, from Prop.~\ref{P:genixcupl1prim}, that if $l \geq 0$ then 
$I(X \cup L_1^\prim) = SF_2 + (x_0,\, x_1)(x_2,\, x_3)^2$.  

(a) In this case, by Lemma~\ref{L:l=-1deg2}, $X$ is the divisor $2L_1$ on 
the plane $H\subset \piii$ of equation $F_2 = 0$ and the result follows from 
Lemma~\ref{L:zcupw}. By the same lemma, one can get a graded free resolution 
of $I(X \cup L_1^\prim \cup L^\prim)$.  

(b) In this case $\tH^0(\sci_{X \cup L_1^\prim \cup L^\prim}(2)) = 0$. \emph{Indeed}, 
$F_2 = -bx_2 + ax_3$ (with $a,\, b \in k[x_0,x_1]_1$) does not vanish 
identically on the plane $\{x_1 = 0\}$ because $a$ and $b$ are coprime.  
On the other hand, $F_2$ vanishes on $L_1^\prim \subset \{x_1 = 0\}$ and in 
$P_0 \in \{x_1 = 0\}\setminus L^\prim$ hence it cannot vanish on $L^\prim$. 

One deduces, now, an exact sequence$\, :$ 
\[
0 \lra I(L_1^{(1)} \cup L_1^\prim \cup L^\prim) \lra 
I(X \cup L_1^\prim \cup L^\prim)  
\xra{\displaystyle \tH^0_\ast(\psi)} S(L_1)_+(-2) \lra 0\, . 
\]
It remains to notice that $\ell_0F_2$ and $x_1F_2$ belong 
to $\tH^0(\sci_{X \cup L_1^\prim \cup L^\prim}(3))$ and that they are mapped by 
$\psi$ to $x_0 \in \tH^0(\sco_{L_1}(1))$ and $x_1 \in \tH^0(\sco_{L_1}(1))$, 
respectively. 
Notice, also, that using the above exact sequence and the resolution of 
$S(L_1)_+$ that can be found in the discussion following 
Prop.~\ref{P:genizprim}, one can get a graded free resolution of 
$I(X \cup L_1^\prim \cup L^\prim)$. 

(c) In this case one has an exact sequence$\, :$
\[
0 \lra I(L_1^{(1)} \cup L_1^\prim \cup L^\prim) \lra 
I(X \cup L_1^\prim \cup L^\prim) 
\xra{\displaystyle \tH^0_\ast(\psi)} S(L_1)(-l-2) \lra 0\, . 
\] 
It remains to notice that, by Prop.~\ref{P:genixcupl1prim},  
$F_2 + cx_0x_2F_2^\prim = x_0\ell_0F_2^\prim + x_1F_2^\secund$ 
belongs to $\tH^0(\sci_{X \cup L_1^\prim \cup L^\prim}(l+2))$ and 
that it is mapped by $\psi$ to $1 \in \tH^0(\sco_{L_1})$. Notice, also, that 
using the above exact sequence one can get a graded free resolution of 
$I(X \cup L_1^\prim \cup L^\prim)$. 
\end{proof}

\begin{lemma}\label{L:l1cupl1primcupl1secund}
Let $L_1^\secund \subset \piii$ be the line of equations $\ell_0 = \ell_1 = 0$, 
where $\ell_0 = x_0 - x_2$ and $\ell_1 = x_1 - x_3$. $L_1$, $L_1^\prim$ and 
$L_1^\secund$ are mutually disjoint and are contained in the quadric 
$Q\subset \piii$ of equation $x_0x_3 - x_1x_2 = 0$. Then the homogeneous ideal 
$I(Y)$ of $Y = L_1\cup L_1^\prim \cup L_1^\secund$ is generated by 
$x_0x_3 - x_1x_2$, $x_0\ell_0x_2$, $x_0\ell_0x_3$, $x_1\ell_1x_2$ and 
$x_1\ell_1x_3$ and admits the following graded free resolution$\, :$ 
\[
0 \lra 2S(-5) \overset{\displaystyle d_2}{\lra} 6S(-4) 
\overset{\displaystyle d_1}{\lra} S(-2) \oplus 4S(-3) \lra I(Y) \lra 0
\]
where $d_1$ and $d_2$ are defined by the matrices$\, :$ 
\[
\begin{pmatrix} 
-x_0\ell_0 & x_0x_1 & -x_1\ell_1 & 0 & \ell_0x_3 + x_1x_2 & 0\\
-x_1 & 0 & 0 & -x_3 & 0 & 0\\
x_0 & -x_1 & 0 & x_2 & -x_3 & 0\\
0 & x_0 & -x_1 & 0 & x_2 & -x_3\\
0 & 0 & x_0 & 0 & 0 & x_2
\end{pmatrix}\, ,\  
\begin{pmatrix} 
-x_3 & 0\\
x_2 & -x_3\\
0 & x_2\\
x_1 & 0\\
-x_0 & x_1\\
0 & -x_0
\end{pmatrix}\, .
\]
\end{lemma}

\begin{proof}
Tensorizing by $\sci_{L_1\cup L_1^\prim}$ the exact sequence$\, :$ 
\[
0 \lra \sco_\p(-2) \xra{\begin{pmatrix} -\ell_1\\ \ell_0 \end{pmatrix}} 
2\sco_\p(-1) \xra{\displaystyle (\ell_0\, ,\, \ell_1)} \sco_\p \lra 
\sco_{L_1^\secund} \lra 0 
\]
one gets an exact sequence$\, :$ 
\[
0 \lra \sci_{L_1\cup L_1^\prim}(-2) \lra 2\sci_{L_1\cup L_1^\prim}(-1) \lra 
\sci_Y \lra 0\, .
\]
Since $\tH^1(\sci_{L_1\cup L_1^\prim}(i)) = 0$ for $i\geq 1$ it follows that 
$I(Y)$ coincides with $I(L_1\cup L_1^\prim)I(L_1^\secund)$ in degrees $\geq 3$. 
On the other hand, it is well known that $I(Y)_2 = k(x_0x_3 - x_1x_2)$. 
Consequenly$\, :$ 
\[
I(Y) = S(x_0x_3 - x_1x_2) + (x_0x_2,\, x_0x_3,\, x_1x_2,\, x_1x_3)(\ell_0,\, 
\ell_1)\, .
\]
Noticing that $q := x_0x_3 - x_1x_2 = \ell_0x_3 - \ell_1x_2$ one gets$\, :$ 
\begin{gather*} 
\ell_0x_1x_2 = x_0\ell_0x_3 - \ell_0q\, ,\  \ell_0x_1x_3 = x_1\ell_1x_2 + 
x_1q\, ,\\
x_0\ell_1x_2 = x_0\ell_0x_3 - x_0q\, ,\  x_0\ell_1x_3 = x_1\ell_1x_2 + 
\ell_1q\, . 
\end{gather*} 
It follows that $I(Y)$ is (minimally) generated by the elements from the 
statement. 

Let, now, $X^\prim$ be the divisor $2L_1^\prim$ on $Q$. Using the Segre 
isomorphism $\p^1\times \p^1 \Izo Q$ one gets exact sequences$\, :$ 
\begin{gather*} 
0 \lra \sco_\p(-2) \overset{\displaystyle q}{\lra} \sci_Y \lra 
\sco_Q(-3,0) \lra 0\, ,\\
0 \lra \sco_\p(-2) \overset{\displaystyle q}{\lra} \sci_{L_1\cup X^\prim} 
\lra \sco_Q(-3,0) \lra 0\, .
\end{gather*}
One deduces that $I(Y)$ and $I(L_1\cup X^\prim)$ admit graded free resolutions 
of the same numerical shape. Using Prop.~\ref{P:genixcupl1prim}, the last 
exact sequence from its proof, and Remark~\ref{R:cancellation}, it follows 
that $I(L_1\cup X^\prim)$ is generated by $x_0x_3 - x_1x_2$, $x_0^2x_2$, 
$x_0^2x_3$, $x_1^2x_2$, $x_1^2x_3$ and admits the following graded free 
resolution$\, :$ 
\[
0 \lra 2S(-5) \overset{\displaystyle d_2^{\, \prim}}{\lra} 6S(-4) 
\overset{\displaystyle d_1^{\, \prim}}{\lra} S(-2) \oplus 4S(-3) 
\lra I(L_1\cup X^\prim) \lra 0
\]
where $d_1^{\, \prim}$ and $d_2^{\, \prim}$ are defined by the matrices$\, :$ 
\[
\begin{pmatrix} 
-x_0^2 & x_0x_1 & -x_1^2 & 0 & x_0x_3 + x_1x_2 & 0\\
-x_1 & 0 & 0 & -x_3 & 0 & 0\\
x_0 & -x_1 & 0 & x_2 & -x_3 & 0\\
0 & x_0 & -x_1 & 0 & x_2 & -x_3\\
0 & 0 & x_0 & 0 & 0 & x_2
\end{pmatrix}\, ,\  
\begin{pmatrix} 
-x_3 & 0\\
x_2 & -x_3\\
0 & x_2\\
x_1 & 0\\
-x_0 & x_1\\
0 & -x_0
\end{pmatrix}\, .
\]
One can easily guess, now, a similar graded free resolution of $I(Y)$, 
which is the one from the statement. 
\end{proof}

\begin{lemma}\label{L:l1(1)cupl1primcupl1secund}
Using the notation from the statement of Lemma~\ref{L:l1cupl1primcupl1secund}, 
the homogeneous ideal $I(W) \subset S$ of $W = L_1^{(1)}\cup L_1^\prim \cup 
L_1^\secund$ is generated by$\, :$ 
\[
x_2(x_0x_3 - x_1x_2)\, ,\, x_3(x_0x_3 - x_1x_2)\, ,\, x_0\ell_0x_2^2\, ,\, 
x_0\ell_0x_2x_3\, ,\, x_0\ell_0x_3^2\, ,\, x_1\ell_1x_2x_3\, ,\, x_1\ell_1x_3^2  
\]
and admits the following graded free resolution$\, :$ 
\[
0 \lra 3S(-6) \overset{\displaystyle d_2}{\lra} S(-4)\oplus 8S(-5) 
\overset{\displaystyle d_1}{\lra} 2S(-3)\oplus 5S(-4) \lra I(W) \lra 0 
\]
where $d_1$ and $d_2$ are defined by the matrices$\, :$ 
\[
\begin{pmatrix} 
\text{--}x_3 & 0 & 0 & 0 & 0 & \text{--}x_0\ell_0 & 0 & 0 & 0\\
x_2 & 0 & 0 & \ell_0x_3 + x_1x_2 & 0 & 0 & \text{--}x_0\ell_0 & x_0x_1 & 
\text{--}x_1\ell_1\\
0 & \text{--}x_3 & 0 & 0 & 0 & \text{--}x_1 & 0 & 0 & 0\\
0 & x_2 & \text{--}x_3 & 0 & 0 & x_0 & \text{--}x_1 & 0 & 0\\
0 & 0 & x_2 & \text{--}x_3 & 0 & 0 & x_0 & \text{--}x_1 & 0\\
0 & 0 & 0 & x_2 & \text{--}x_3 & 0 & 0 & x_0 & \text{--}x_1\\
0 & 0 & 0 & 0 & x_2 & 0 & 0 & 0 & x_0 
\end{pmatrix}\, ,\, 
\begin{pmatrix} 
x_0\ell_0 & 0 & 0\\
x_1 & 0 & 0\\
-x_0 & x_1 & 0\\
0 & -x_0 & x_1\\
0 & 0 & -x_0\\
-x_3 & 0 & 0\\
x_2 & -x_3 & 0\\
0 & x_2 & -x_3\\
0 & 0 & x_2
\end{pmatrix}\, .
\]
\end{lemma}

\begin{proof} 
Tensorizing by $\sci_{L_1^\prim \cup L_1^\secund}$ the exact sequence$\, :$ 
\[
0 \lra \sco_\p(-2) \xra{\begin{pmatrix} -x_3\\ x_2 \end{pmatrix}} 
2\sci_{L_1}(-1) \xra{\displaystyle (x_2\, ,\, x_3)} \sco_\p \lra 
\sco_{L_1^{(1)}} \lra 0 
\]
one gets an exact sequence$\, :$ 
\[
0 \lra \sci_{L_1^\prim \cup L_1^\secund}(-2) 
\xra{\begin{pmatrix} -x_3\\ x_2 \end{pmatrix}}  
2\sci_{L_1 \cup L_1^\prim \cup L_1^\secund}(-1) 
\xra{\displaystyle (x_2\, ,\, x_3)} 
\sci_W \lra 0\, .
\]
Since $\tH^1(\sci_{L_1^\prim \cup L_1^\secund}(i)) = 0$ for $i\geq 1$ 
and since $\tH^0(\sci_W(2)) = 0$ it follows that$\, :$  
\[
I(W) = I(L_1)I(L_1\cup L_1^\prim \cup L_1^\secund)\, . 
\] 
Recall that $x_0x_3 - x_1x_2 = \ell_0x_3 - \ell_1x_2$. Using the relation$\, :$ 
\[
x_1\ell_1x_2^2 = x_0\ell_0x_3^2 - x_1\cdot x_2(\ell_0x_3 - \ell_1x_2) - 
\ell_0\cdot x_3(x_0x_3 - x_1x_2)  
\]
one derives that $I(W)$ is (minimally) generated by the elements from the 
statement.  

On the other hand, one deduces an exact sequence$\, :$ 
\[
0 \lra I(L_1^\prim \cup L_1^\secund)(-2) 
\xra{\begin{pmatrix} -x_3\\ x_2 \end{pmatrix}} 
2I(L_1\cup L_1^\prim \cup L_1^\secund)(-1) 
\xra{\displaystyle (x_2\, ,\, x_3)} 
I(W) \lra 0\, .  
\]
By Lemma~\ref{L:zcupw}, $I(L_1^\prim \cup L_1^\secund)$ is generated by 
$x_0\ell_0$, $x_0\ell_1$, $x_1\ell_0$, $x_1\ell_1$ and admits the following 
graded free resolution$\, :$ 
\[
0 \lra S(-4) \overset{\displaystyle d_2^{\, \prim}}{\lra} 4S(-3) 
\overset{\displaystyle d_1^{\, \prim}}{\lra} 4S(-2) \lra 
I(L_1^\prim \cup L_1^\secund) \lra 0 
\]
with $d_1^{\, \prim}$ and $d_2^{\, \prim}$ defined by the matrices$\, :$ 
\[
\begin{pmatrix}
-x_1 & 0 & -\ell_1 & 0\\
0 & -x_1 & \ell_0 & 0\\
x_0 & 0 & 0 & -\ell_1\\
0 & x_0 & 0 & \ell_0
\end{pmatrix}\, ,\  
\begin{pmatrix}
-\ell_1\\ \ell_0\\ x_1\\ -x_0
\end{pmatrix}\, .
\]
Using the resolution of $I(L_1\cup L_1^\prim \cup L_1^\secund)$ from 
Lemma~\ref{L:l1cupl1primcupl1secund}, the morphism$\, :$ 
\[
I(L_1^\prim \cup L_1^\secund)(-2) 
\xra{\begin{pmatrix} -x_3\\ x_2 \end{pmatrix}} 
2I(L_1\cup L_1^\prim \cup L_1^\secund)(-1) 
\]
induces, in an obvious manner, a morphism between the resolutions of its 
source and of its target. The mapping cone of this morphism between 
resolutions is a resolution of $I(W)$. This resolution is not minimal, but 
using Remark~\ref{R:cancellation} one can get from it the minimal free 
resolution from the statement.    
\end{proof}

\begin{lemma}\label{L:s(l1)geq2} 
The graded $S$-submodule $S(L_1)_{\geq 2} := 
\bigoplus_{i \geq 2}S(L_1)_i$ of $S(L_1) := 
{\fam0 H}^0_\ast(\sco_{L_1})$ admits the following minimal graded free 
resolution$\, :$ 
\[
0 \ra 2S(-5) \overset{\displaystyle d_3}{\lra} 7S(-4) 
\overset{\displaystyle d_2}{\lra} 8S(-3) 
\overset{\displaystyle d_1}{\lra} 3S(-2) 
\overset{\displaystyle d_0}{\lra} S(L_1)_{\geq 2} \ra 0 
\]
with $d_0$, $d_1$, $d_2$, $d_3$ defined by the matrices$\, :$ 
\begin{gather*}
(x_0^2\, ,\, x_0x_1\, ,\, x_1^2)\, ,\  
\begin{pmatrix} 
-x_1 & 0 & x_2 & x_3 & 0 & 0 & 0 & 0\\
x_0 & -x_1 & 0 & 0 & x_2 & x_3 & 0 & 0\\
0 & x_0 & 0 & 0 & 0 & 0 & x_2 & x_3 
\end{pmatrix}\, ,\\
\begin{pmatrix}
x_2 & x_3 & 0 & 0 & 0 & 0 & 0\\
0 & 0 & x_2 & x_3 & 0 & 0 & 0\\
x_1 & 0 & 0 & 0 & -x_3 & 0 & 0\\
0 & x_1 & 0 & 0 & x_2 & 0 & 0\\
-x_0 & 0 & x_1 & 0 & 0 & -x_3 & 0\\
0 & -x_0 & 0 & x_1 & 0 & x_2 & 0\\
0 & 0 & -x_0 & 0 & 0 & 0 & -x_3\\
0 & 0 & 0 & -x_0 & 0 & 0 & x_2
\end{pmatrix}\, ,\  
\begin{pmatrix}
-x_3 & 0\\
x_2 & 0\\
0 & -x_3\\
0 & x_2\\
-x_1 & 0\\
x_0 & -x_1\\
0 & x_0
\end{pmatrix}\, .
\end{gather*}
\end{lemma}

\begin{proof}
Using the exact sequence$\, :$ 
\[
0 \ra 2S(L_1)(-3) 
\xra{\begin{pmatrix} -x_1 & 0\\ x_0 & -x_1\\ 0 & x_0 \end{pmatrix}} 
3S(L_1)(-2) \xra{\displaystyle (x_0^2\, ,\, x_0x_1\, ,\, x_1^2)} 
S(L_1)_{\geq 2} \ra 0\, . 
\]
one sees that the tensor product of the complexes$\, :$
\[
2S(-3) 
\xra{\begin{pmatrix} -x_1 & 0\\ x_0 & -x_1\\ 0 & x_0 \end{pmatrix}} 
3S(-2)\, ,\  
S(-2) \xra{\begin{pmatrix} -x_3\\ x_2 \end{pmatrix}} 2S(-1) 
\xra{\displaystyle (x_2\, ,\, x_3)} S 
\]
is a minimal graded free resolution of $S(L_1)_{\geq 2}$ over $S$.  
\end{proof}

\begin{prop}\label{P:genixcupl1primcupl1secund} 
Let $X$ be the double structure on the line $L_1$ considered at the beginning 
of Subsection~\ref{SS:doublecupaline}.  
Let $L_1^\secund \subset \piii$ be the line of equations $\ell_0 = \ell_1 = 0$, 
where $\ell_0 = x_0 - x_2$ and $\ell_1 = x_1 - x_3$. $L_1$, $L_1^\prim$ and 
$L_1^\secund$ are mutually disjoint and are contained in the quadric 
$Q\subset \piii$ of equation $x_0x_3 - x_1x_2 = 0$.  

\emph{(a)} If $l = -1$ then$\, :$ 
\[
I(X \cup L_1^\prim \cup L_1^\secund) = Sx_0\ell_0F_2 + Sx_1\ell_0F_2 +  
Sx_1\ell_1F_2 + I(L_1^{(1)} \cup L_1^\prim \cup L_1^\secund)\, . 
\]   

\emph{(b)} If $l = 0$ and $\{F_2 = 0\} \neq Q$ then$\, :$ 
\[
I(X \cup L_1^\prim \cup L_1^\secund) = S\ell_0F_2 + S\ell_1F_2 +  
I(L_1^{(1)} \cup L_1^\prim \cup L_1^\secund)\, . 
\] 

\emph{(c)} If $l = 0$ and $\{F_2 = 0\} = Q$ or if $l \geq 1$ then, 
writting $F_2 = x_0F_2^\prim + x_1F_2^\secund$, one has$\, :$ 
\[
I(X \cup L_1^\prim \cup L_1^\secund) = S(F_2 - x_2F_2^\prim - x_3F_2^\secund)  
+ I(L_1^{(1)} \cup L_1^\prim \cup L_1^\secund)\, . 
\]
\end{prop}

\begin{proof}
Tensorizing by $\sci_{L_1^\prim \cup L_1^\secund}$ the exact sequence$\, :$ 
\[
0 \lra \sci_{L_1^{(1)}} \lra \sci_X \overset{\displaystyle \eta}{\lra} 
\sco_{L_1}(-l-2) \lra 0\, , 
\] 
with $\eta$ mapping $F_2$ to $1 \in \tH^0(\sco_{L_1})$ and $x_2^2,\, x_2x_3,\, 
x_3^2$ to 0, one gets an exact sequence$\, :$ 
\[
0 \lra \sci_{L_1^{(1)} \cup L_1^\prim \cup L_1^\secund} \lra 
\sci_{X \cup L_1^\prim \cup L_1^\secund} 
\overset{\displaystyle \psi}{\lra} \sco_{L_1}(-l-2) \lra 0\, .
\]
Since, by Lemma~\ref{L:l1(1)cupl1primcupl1secund}, 
$\tH^1(\sci_{L_1^{(1)} \cup L_1^\prim \cup L_1^\secund}(i)) = 0$ for $i \geq 3$ it 
follows that $\tH^0(\psi(i))$ is surjective for $i \geq 3$. On the other 
hand$\, :$ 
\[
\tH^0(\sci_{X \cup L_1^\prim \cup L_1^\secund}(2)) \subseteq 
\tH^0(\sci_{L_1 \cup L_1^\prim \cup L_1^\secund}(2)) = k(x_0x_3 - x_1x_2)\, . 
\]

(a) In this case, by Lemma~\ref{L:l=-1deg2}, $X$ is the divisor $2L_1$ in 
the plane $H$ of equation $F_2 = 0$. Since $H \cap Q$ is the union of $L_1$ 
and of a different line intersecting it, it follows that $X$ is not contained 
in $Q$ hence $\tH^0(\sci_{X \cup L_1^\prim \cup L_1^\secund}(2)) = 0$. One deduces 
an exact sequence$\, :$ 
\[
0 \lra I(L_1^{(1)} \cup L_1^\prim \cup L_1^\secund) \lra 
I(X \cup L_1^\prim \cup L_1^\secund) 
\xra{\displaystyle \tH^0_\ast(\psi)} S(L_1)_{\geq 2}(-1) \lra 0\, .
\] 
Since $x_0\ell_0F_2,\, x_1\ell_0F_2,\, x_1\ell_1F_2$ belong to 
$I(X \cup L_1^\prim \cup L_1^\secund)_3$ and are mapped by $\psi$ to 
the elements $x_0^2,\, x_0x_1,\, x_1^2$ of $S(L_1)_2$,  
one gets the assertion from the statement. 

Notice that by using the last exact sequence and Lemma~\ref{L:s(l1)geq2} one 
can get a non-minimal graded free resolution of 
$I(X \cup L_1^\prim \cup L_1^\secund)$. This resolution has, actually, excessive 
length 3 but using Remark~\ref{R:cancellation} one can get from it a minimal 
free resolution (of length 2). 

(b) Since, by Prop.~\ref{P:genixcupl1prim}, $I(X \cup L_1^\prim) = SF_2 + 
(x_0,\, x_1)(x_2,\, x_3)^2$, our hypothesis implies that one has  
$\tH^0(\sci_{X \cup L_1^\prim \cup L_1^\secund}(2)) = 0$. One deduces an exact 
sequence$\, :$ 
\[
0 \lra I(L_1^{(1)} \cup L_1^\prim \cup L_1^\secund) \lra 
I(X \cup L_1^\prim \cup L_1^\secund) 
\xra{\displaystyle \tH^0_\ast(\psi)} S(L_1)_+(-2) \lra 0\, .
\] 
It remains to notice that $\ell_0F_2$ and $\ell_1F_2$ belong to 
$\tH^0(\sci_{X \cup L_1^\prim \cup L_1^\secund}(3))$ and that $\psi(\ell_iF_2) = 
\eta(\ell_iF_2) = x_i \in \tH^0(\sco_{L_1}(1))$, $i = 0,\, 1$. 

Notice, also, that using the last exact sequence and the resolution of 
$S(L_1)_+$ from the discussion following Prop.~\ref{P:genizprim} one can get a 
non-minimal graded free resolution of $I(X \cup L_1^\prim \cup L_1^\secund)$ of 
length 3. Using Remark~\ref{R:cancellation} one can get from it a minimal 
free resolution (of length 2). 

(c) In this case one gets an exact sequence$\, :$  
\[
0 \lra I(L_1^{(1)} \cup L_1^\prim \cup L_1^\secund) \lra 
I(X \cup L_1^\prim \cup L_1^\secund) 
\xra{\displaystyle \tH^0_\ast(\psi)} S(L_1)(-l-2) \lra 0 
\]
and it suffices to notice that 
$F_2 - x_2F_2^\prim - x_3F_2^\secund = \ell_0F_2^\prim + \ell_1F_2^\secund$ 
belongs to $\tH^0(\sci_{X \cup L_1^\prim \cup L_1^\secund}(l+2))$ and that 
$\psi(F_2 - x_2F_2^\prim - x_3F_2^\secund)  = 
\eta(F_2 - x_2F_2^\prim - x_3F_2^\secund) = 1 \in \tH^0(\sco_{L_1})$. 
Notice, also, that using the last exact sequence one can get a graded free 
resolution of $I(X \cup L_1^\prim \cup L_1^\secund)$. 
\end{proof}

\subsection{A double line union a conic}\label{SS:doublecupconic} 
Let $X$ be the double structure on the line $L_1$ considered at the beginning 
of Subsection~\ref{SS:doublecupaline} and let $C$ be a (nonsingular) conic 
in $\piii$. One has to consider the following four possibilities$\, :$ 
\begin{enumerate}
\item[(i)] $L_1 \cap C = \emptyset$$\, ;$ 
\item[(ii)] $L_1 \cap C \neq \emptyset$ and $L_1$ is not contained in the 
plane of $C$$\, ;$ 
\item[(iii)] $L_1$ tangent to $C$$\, ;$ 
\item[(iv)] $L_1 \cap C$ consists of two points. 
\end{enumerate}

We recall the following elementary$\, :$ 

\begin{lemma}\label{L:eqofaconic} 
Let $Q_0,\, Q_1,\, Q_2$ be the coordinate points $(1:0:0)$, $(0:1:0)$, 
$(0:0:1)$ of $\pii$. Consider a $($nonsingular$)$ conic $C \subset \pii$, 
containing $Q_0$ and $Q_1$ and such that ${\fam0 T}_{Q_0}C = 
\overline{Q_0Q_2} = \{x_1 = 0\}$ and ${\fam0 T}_{Q_1}C = \overline{Q_1Q_2} 
= \{x_0 = 0\}$. Then the equation of $C$ is of the form $c_0x_0x_1 + 
c_1x_2^2 = 0$ with $c_0,\, c_1 \in k\setminus \{0\}$. 
\qed 
\end{lemma} 

In case (i) one can assume, up to a linear change of coordinates in $\piii$, 
that $C$ is contained in the plane $\{x_1 = 0\}$ spanned by $P_0,\, P_2,\, 
P_3$, that $P_2,\, P_3 \in C$ and that $\text{T}_{P_2}C = \overline{P_0P_2} = 
L_2$, $\text{T}_{P_3}C = \overline{P_0P_3} = L_3$. Lemma~\ref{L:eqofaconic} 
implies that $C$ is defined by equations of the form 
$x_1 = x_0^2 + cx_2x_3 = 0$ with $c \in k\setminus \{0\}$. 

\vskip2mm

In case (ii) one can assume, up to a linear change of coordinates in $\piii$, 
that $C$ is contained in the plane $\{x_1 = 0\}$ spanned by $P_0,\, P_2,\, 
P_3$, that $P_0,\, P_3 \in C$ and that $\text{T}_{P_0}C = \overline{P_0P_2} = 
L_2$, $\text{T}_{P_3}C = \overline{P_2P_3} = L_1^\prim$. Lemma~\ref{L:eqofaconic} 
implies that $C$ is defined by equations of the form 
$x_1 = x_0x_3 + cx_2^2 = 0$ with $c \in k\setminus \{0\}$. Notice that if 
one takes $c = 0$ then $C$ becomes $L_2 \cup L_1^\prim$.  

\vskip2mm

In case (iii) one can assume, up to a linear change of coordinates in $\piii$, 
that $C$ is contained in the plane $\{x_3 = 0\}$ spanned by $P_0,\, P_1,\, 
P_2$, that $P_0,\, P_2 \in C$ and that $\text{T}_{P_0}C = \overline{P_0P_1} = 
L_1$, $\text{T}_{P_2}C = \overline{P_1P_2} = L_3^\prim$. Lemma~\ref{L:eqofaconic} 
implies that $C$ is defined by equations of the form 
$x_3 = x_1^2 + cx_0x_2 = 0$ with $c \in k\setminus \{0\}$. 

\vskip2mm

In case (iv) one can assume, up to a linear change of coordinates in $\piii$, 
that $C$ is contained in the plane $\{x_3 = 0\}$ spanned by $P_0,\, P_1,\, 
P_2$, that $P_0,\, P_1 \in C$ and that $\text{T}_{P_0}C = \overline{P_0P_2} = 
L_2$, $\text{T}_{P_1}C = \overline{P_1P_2} = L_3^\prim$. Lemma~\ref{L:eqofaconic} 
implies that $C$ is defined by equations of the form 
$x_3 = x_0x_1 + cx_2^2 = 0$ with $c \in k\setminus \{0\}$.   

\vskip2mm 
 
The following lemma follows immediately from Lemma~\ref{L:zcupw}. 

\begin{lemma}\label{L:l1(1)cupc1}
Let $C \subset \piii$ be the conic of equations $x_1 = x_0^2 + cx_2x_3 = 0$, 
$c \in k\setminus \{0\}$. Put $q_1 := x_0^2 + cx_2x_3$. Then  
$I(L_1^{(1)} \cup C)$ is generated by  
$x_1x_2^2$, $x_1x_2x_3$, $x_1x_3^2$, $x_2^2q_1$, $x_2x_3q_1$, $x_3^2q_1$ 
and admits the following minimal graded free resolution$\, :$ 
\[
0 \lra 2S(-6) \overset{\displaystyle d_2}{\lra} 
2S(-4) \oplus 5S(-5)  
\overset{\displaystyle d_1}{\lra} 
3S(-3) \oplus 3S(-4)  
\overset{\displaystyle d_0}{\lra} 
I(L_1^{(1)} \cup C) \lra 0 
\] 
with $d_1$ and $d_2$ defined by the matrices$\, :$ 
\[
\begin{pmatrix} 
-x_3 & 0 & 0 & 0 & -q_1 & 0 & 0\\
x_2 & -x_3 & 0 & 0 & 0 & -q_1 & 0\\
0 & x_2 & 0 & 0 & 0 & 0 & -q_1\\
0 & 0 & -x_3 & 0 & x_1 & 0 & 0\\
0 & 0 & x_2 & -x_3 & 0 & x_1 & 0\\
0 & 0 & 0 & x_2 & 0 & 0 & x_1
\end{pmatrix}\, ,\  
\begin{pmatrix}
q_1 & 0\\
0 & q_1\\
-x_1 & 0\\
0 & -x_1\\
-x_3 & 0\\
x_2 & -x_3\\
0 & x_2
\end{pmatrix}\, . 
\]
\end{lemma}

\begin{prop}\label{P:genixcupc1}
Let $X$ be the double structure on the line $L_1$ considered at the beginning 
of Subsection~\ref{SS:doublecupaline} and let $C$ be the conic from 
Lemma~\ref{L:l1(1)cupc1}. 

\emph{(a)} If $l = -1$ then $I(X \cup C) = I(X)I(C)$. 

\emph{(b)} If $l = 0$ then, writting $F_2 = x_0F_2^\prim + x_1F_2^\secund$, 
one has $I(X \cup C) = S(q_1F_2^\prim + x_0x_1F_2^\secund) + Sx_1F_2 + 
I(L_1^{(1)} \cup C)$. 

\emph{(c)} If $l \geq 1$ then, writting $F_2 = x_0^2F_2^\prim + x_1F_2^\secund$, 
one has $I(X \cup C) = S(q_1F_2^\prim + x_1F_2^\secund) + I(L_1^{(1)} \cup C)$.   
\end{prop}

\begin{proof}
Tensorizing by $\sci_C$ the exact sequence$\, :$ 
\[
0 \lra \sci_{L_1^{(1)}} \lra \sci_X \overset{\displaystyle \eta}{\lra} 
\sco_{L_1}(-l-2) \lra 0 
\]
one gets an exact sequence$\, :$ 
\[
0 \lra \sci_{L_1^{(1)} \cup C} \lra \sci_{X \cup C} 
\overset{\displaystyle \psi}{\lra} \sco_{L_1}(-l-2) \lra 0\, . 
\]
It follows from Lemma~\ref{L:l1(1)cupc1} that 
$\tH^1(\sci_{L_1^{(1)} \cup C}(i)) = 0$ for $i \geq 3$ hence $\tH^0(\psi(i))$ is 
surjective for $i \geq 3$. 

(a) In this case, by Lemma~\ref{L:l=-1deg2}, $X$ is the divisor $2L_1$ on 
the plane $H\subset \piii$ of equation $F_2 = 0$ and the result follows from 
Lemma~\ref{L:zcupw}. By the same lemma, one can get a graded free resolution 
of $I(X \cup C)$.  

(b) In this case $\tH^0(\sci_{X \cup C}(2)) = 0$. \emph{Indeed}, assume that 
there exists a non-zero $f \in \tH^0(\sci_{X \cup C}(2))$. 
Since $\tH^0(\sci_X(1)) = 0$,  
$f$ cannot vanish identically on the plane $\{x_1 = 0\}$. On the other 
hand, $f$ vanishes in $P_0 \in \{x_1 = 0\}\setminus C$ hence it cannot vanish 
on $C$ which is a \emph{contradiction}. 

One deduces, now, an exact sequence$\, :$ 
\[
0 \lra I(L_1^{(1)} \cup C) \lra I(X \cup C) 
\xra{\displaystyle \tH^0_\ast(\psi)} S(L_1)_+(-2) \lra 0\, . 
\]
It remains to notice that $q_1F_2^\prim + x_0x_1F_2^\secund = x_0F_2 + 
cx_2x_3F_2^\prim$ and $x_1F_2$ belong 
to $\tH^0(\sci_{X \cup C}(3))$ and that they are mapped by $\psi$ to 
$x_0 \in \tH^0(\sco_{L_1}(1))$ and $x_1 \in \tH^0(\sco_{L_1}(1))$, respectively. 
Notice, also, that using the above exact sequence and the resolution of 
$S(L_1)_+$ that can be found in the discussion following 
Prop.~\ref{P:genizprim}, one can get a graded free resolution of $I(X \cup C)$. 

(c) In this case one has an exact sequence$\, :$
\[
0 \lra I(L_1^{(1)} \cup C) \lra I(X \cup C) 
\xra{\displaystyle \tH^0_\ast(\psi)} S(L_1)(-l-2) \lra 0\, . 
\] 
It remains to notice that $q_1F_2^\prim + x_1F_2^\secund = F_2 + cx_2x_3F_2^\prim$  
belongs to $\tH^0(\sci_{X \cup C}(l+2))$ and that it is mapped by $\psi$ to 
$1 \in \tH^0(\sco_{L_1})$. Notice, also, that using the above exact sequence 
one can get a graded free resolution of $I(X \cup C)$. 
\end{proof}

\begin{lemma}\label{L:l1(1)cupc2} 
Let $C\subset \piii$ be the conic of equations $x_1 = x_0x_3 + cx_2^2 = 0$, 
$c \in k \setminus \{0\}$. Put $q_2 := x_0x_3 + cx_2^2$. 
Then $I(L_1^{(1)} \cup C)$ is generated by 
$x_2q_2$, $x_3q_2$, $x_1x_2^2$, $x_1x_2x_3$, 
$x_1x_3^2$ and admits the following minimal graded free resolution$\, :$ 
\[
0 \lra S(-5) \overset{\displaystyle d_2}{\lra} 5S(-4) 
\overset{\displaystyle d_1}{\lra} 5S(-3) \overset{\displaystyle d_0}{\lra} 
I(L_1^{(1)} \cup C) \lra 0 
\]
where $d_1$ and $d_2$ are defined by the matrices$\, :$ 
\[
\begin{pmatrix} 
-x_3 & 0 & 0 & -x_1 & 0\\
x_2 & 0 & 0 & 0 & -x_1\\
0 & -x_3 & 0 & cx_2 & 0\\
0 & x_2 & -x_3 & x_0 & cx_2\\
0 & 0 & x_2 & 0 & x_0
\end{pmatrix} \, ,\  
\begin{pmatrix}
-x_1\\ cx_2\\ x_0\\ x_3\\ -x_2
\end{pmatrix}\, .
\] 
\end{lemma}

\begin{proof} 
A homogeneous polynomial $f \in S$ belongs to $I(L_1^{(1)} \cup C)$ if and only 
if it belongs to $I(L_1^{(1)}) = (x_2,\, x_3)^2$ and vanishes on $C$. Consider 
the subalgebra $R = k[x_0,x_2,x_3]$ of $S$. For $f \in I(L_1^{(1)})$, one can 
write$\, :$ 
\[
f = f_1 + g \  \text{with} \  f_1 \in Rx_2^2 + Rx_2x_3 + Rx_3^2 \  
\text{and} \  g \in x_1I(L_1^{(1)})\, .   
\]
$f_1$ can be identified with the restriction of $f$ to the plane $\{x_1 = 0\} 
\supset C$. $f$ vanishes on $C$ if and only if $f_1$ belongs to$\, :$ 
\[
Rq_2 \cap (Rx_2^2 + Rx_2x_3 + Rx_3^2) = (Rx_2 + Rx_3)q_2\, . 
\] 
One deduces that $I(L_1^{(1)} \cup C)$ is generated by the elements from the 
statement. 
 
Now, let $W_1$ be the curve directly linked to $L_1^{(1)} \cup C$ by the 
complete intersection defined by $x_3q_2$ and $x_1x_2^2$. We assert that$\, :$ 
\[
I(W_1) = (x_0x_3,\, x_0x_1x_2,\, x_1x_2^2,\, x_2^2x_3)\, .
\] 
\emph{Indeed}, let $J$ denote the ideal from the right hand side of the 
equality we want to prove. It is easy to check that$\, :$ 
\[
J \cdot I(L_1^{(1)} \cup C) \subseteq (x_3q_2,\, x_1x_2^2)  
\]
hence $J \subseteq I(W_1)$. On the other hand, as we noticed in the last part 
of the proof of Lemma~\ref{L:l1(1)cupl2cupl1prim}, $J = I(W^\prim)$ where 
$W^\prim$ is the curve directly linked to $L_1^{(1)} \cup L_2 \cup L_1^\prim$ by 
the complete intersection defined by $x_0x_3^2$ and $x_1x_2^2$. Since 
$\text{deg}\, W_1 = 4 = \text{deg}\, W^\prim$ and since $W_1 \subseteq W^\prim$ 
it follows that $W_1 = W^\prim$ hence $I(W_1) = J$.  

As we saw in the last part of the proof of Lemma~\ref{L:l1(1)cupl2cupl1prim},  
$W^\prim$ (hence $W_1$) is directly linked to $L_1 \cup L_1^\prim$ by the 
complete intersection defined by $x_0x_3$ and $x_1x_2^2$.  
Since $\sci_{L_1\cup L_1^\prim}$ admits a resolution of the form$\, :$ 
\[
0 \lra \sco_\p(-4) \lra 4\sco_\p(-3) \lra 4\sco_\p(-2) \lra 
\sci_{L_1\cup L_1^\prim} \lra 0
\]
one deduces, from Ferrand's result about liaison, that $\sci_{W^\prim}$ has a 
monad of the form$\, :$ 
\[
0 \lra 3\sco_\p(-3) \lra 5\sco_\p(-2) \lra \sco_\p(-1) \lra 0
\]
and, then, that $\sci_W$ admits a resolution of the form$\, :$
\[
0 \lra \sco_\p(-5) \lra 5\sco_\p(-4) \lra 5\sco_\p(-3) \lra \sci_W \lra 0\, .
\]
One can easily determine the differentials of this resolution. 
\end{proof}  

\begin{prop}\label{P:genixcupc2} 
Let $X$ be the double structure on the line $L_1$ considered at the beginning 
of Subsection~\ref{SS:doublecupaline} and let $C\subset \piii$ be the conic 
from Lemma~\ref{L:l1(1)cupc2}. 

\emph{(a)} If $x_1 \mid b$ and $l = -1$ then $I(X\cup C) = 
(x_0x_3 + cx_2^2,\, x_1x_3)$. 

\emph{(b)} If $x_1 \mid b$ and $l \geq 0$ then $I(X\cup C) = 
S\left(F_2 + c\frac{a(x_0,0)}{x_0}x_2^2\right) + I(L_1^{(1)}\cup C)$. 

\emph{(c)} If $x_1 \nmid b$ then $I(X\cup C) = Sx_1F_2 + I(L_1^{(1)}\cup C)$.  
\end{prop}

\begin{proof} 
Using Lemma~\ref{L:ycupt}, one deduces from the exact sequence$\, :$ 
\[
0 \lra \sci_{L_1^{(1)}} \lra \sci_X \overset{\displaystyle \eta}{\lra} 
\sco_{L_1}(-l-2) \lra 0 
\] 
an exact sequence$\, :$ 
\[
0 \lra \sci_{L_1^{(1)} \cup C} \lra \sci_{X \cup C} 
\overset{\displaystyle \psi}{\lra} \sco_{L_1}(-l-2) 
\]
where $\psi$ is the composite morphism$\, :$ 
\[
\sci_{X \cup C} \lra \sci_X  
\overset{\displaystyle \eta}{\lra} \sco_{L_1}(-l-2)\, .
\]
Moreover, $\Cok \psi$ is an $\sco_{L_1 \cap C} = \sco_{\{P_0\}}$-module hence 
$x_1\sco_{L_1}(-l-3) \subseteq \text{Im}\, \psi \subseteq \sco_{L_1}(-l-2)$.

(a) If $l = -1$ then the condition $x_1 \mid b$ means that $b = 0$, hence 
$F_2 = x_3$, hence $I(X) = (x_3,\, x_2^2)$. It follows that$\, :$ 
\begin{gather*}
I(X\cup C) = (x_3,\, x_2^2) \cap (x_1,\, x_0x_3 + cx_2^2) =\\   
(x_3,\, x_0x_3 + cx_2^2) \cap (x_1,\, x_0x_3 + cx_2^2)
= (x_0x_3 + cx_2^2,\, x_1x_3)\, .
\end{gather*}

(b) If $x_1 \mid b$ and $l \geq 0$ then $F_2 + c\frac{a(x_0,0)}{x_0}x_2^2 \in 
I(X\cup C)$ (because $a(x_0,0)x_3 + c\frac{a(x_0,0)}{x_0}x_2^2 = 
\frac{a(x_0,0)}{x_0}q_2$).  
Since $\psi \left(F_2 + c\frac{a(x_0,0)}{x_0}x_2^2\right) 
= \eta \left(F_2 + c\frac{a(x_0,0)}{x_0}x_2^2\right) = 1 \in 
\tH^0(\sco_{L_1})$, one deduces that $\psi$ is an epimorphism and that 
the sequence$\, :$ 
\[
0 \lra I(L_1^{(1)}\cup C) \lra I(X\cup C) 
\xra{\displaystyle \tH^0_\ast(\psi)} S(L_1)(-l-2) \lra 0
\]
is exact. Notice that using this exact sequence one can get a graded free 
resolution of $I(X \cup C)$.  

(c) We show, firstly, that$\, :$ 

\vskip2mm

\noindent
{\bf Claim.}\quad $\text{Im}\, \psi = x_1\sco_{L_1}(-l-3)$. 

\vskip2mm

\noindent 
\emph{Indeed}, assume that $\text{Im}\, \psi = \sco_{L_1}(-l-2)$. Then, 
for $m >> 0$, there exists $f \in \tH^0(\sci_{X \cup C}(l + 2 + m))$ such that 
$\psi(f) = x_0^m \in \tH^0(\sco_{L_1}(m))$. Since $\psi(f) = \eta(f)$ it 
follows that$\, :$ 
\[
f = x_0^mF_2 + f_0x_2^2 + f_1x_2x_3 + f_2x_3^2 
\] 
with $f_0,\, f_1,\, f_2 \in S_{l+m}$. On the other hand, since $f \in 
\tH^0(\sci_C(l + 2 +m))$ it follows that$\, :$ 
\[
f = g_1x_1 + g_2(x_0x_3 + cx_2^2) 
\]
with $g_1 \in S_{l+m+1}$ and $g_2 \in S_{l+m}$. Restricting the two different 
formulae of $f$ to the line $L_2$ of equations $x_1 = x_3 = 0$ one 
gets$\, :$ 
\[
-x_0^mb(x_0,0)x_2 + (f_0 \vb L_2)x_2^2 = c(g_2 \vb L_2)x_2^2\, .
\]
Since $x_1 \nmid b$ one has $b(x_0,0) \neq 0$ and this leads to a 
\emph{contradiction}. 

\vskip2mm

\noindent
It remains that $\psi$ factorizes as$\, :$ 
\[
\sci_{X \cup C} \overset{\displaystyle \psi^\prim}{\lra} \sco_{L_1}(-l-3) 
\overset{\displaystyle x_1}{\lra} \sco_{L_1}(-l-2) 
\]
with $\psi^\prim$ an epimorphism. Since $\psi^\prim(x_1F_2) = 1 \in 
\tH^0(\sco_{L_1})$ one deduces that the sequence$\, :$ 
\[
0 \lra I(L_1^{(1)} \cup C) \lra I(X \cup C) 
\xra{\displaystyle \tH^0_\ast(\psi^\prim)} S(L_1)(-l-3) \lra 0
\]
is exact from which the descrition of $I(X \cup C)$ from the statement 
follows. Notice, also, that using this exact sequence one can get a graded 
free resolution of $I(X \cup C)$.  
\end{proof}

\begin{prop}\label{P:reshoxcupc2} 
Under the hypothesis of Prop.~\ref{P:genixcupc2}$\, :$ 

\emph{(a)} If $x_1 \mid b$, i.e., if $b = x_1b_1$ then the graded $S$-module 
${\fam0 H}^0_\ast(\sco_{X\cup C})$ admits the following free resolution$\, :$ 
\[
0 \lra \begin{matrix} 2S(-3)\\ \oplus\\ S(l-3) \end{matrix} 
\overset{\displaystyle \delta_2}{\lra} 
\begin{matrix} 3S(-2)\\ \oplus\\ 2S(l-2) \end{matrix} 
\overset{\displaystyle \delta_1}{\lra} 
\begin{matrix} S\\ \oplus\\ S(l-1) \end{matrix} 
\overset{\displaystyle \delta_0}{\lra} 
{\fam0 H}^0_\ast(\sco_{X\cup C}) \lra 0 
\]
where $\delta_0 = (1\, ,\, x_1e_1)$ and $\delta_1$, $\delta_2$ are defined by 
the matrices$\, :$ 
\[
\begin{pmatrix} 
x_0x_3 + cx_2^2 & x_1x_2 & x_1x_3 & 0 & 0\\
-x_0b_1 & -a & -b & x_2 & x_3 
\end{pmatrix}\, ,\  
\begin{pmatrix} 
-x_1 & 0 & 0\\
cx_2 & -x_3 & 0\\
x_0 & x_2 & 0\\
ca & b & -x_3\\
0 & -a & x_2
\end{pmatrix}\, . 
\]

\emph{(b)} If $x_1 \nmid b$ then the graded $S$-module 
${\fam0 H}^0_\ast(\sco_{X\cup C})$ admits the following free resolution$\, :$ 
\[
0 \lra \begin{matrix} 2S(-3)\\ \oplus\\ S(l-2) \end{matrix} 
\overset{\displaystyle \delta_2}{\lra} 
\begin{matrix} 3S(-2)\\ \oplus\\ 2S(l-1) \end{matrix} 
\overset{\displaystyle \delta_1}{\lra} 
\begin{matrix} S\\ \oplus\\ S(l) \end{matrix} 
\overset{\displaystyle \delta_0}{\lra} 
{\fam0 H}^0_\ast(\sco_{X\cup C}) \lra 0 
\]
where $\delta_0 = (1\, ,\, e_1)$ and $\delta_1$, $\delta_2$ are defined by 
the matrices$\, :$ 
\[
\begin{pmatrix} 
x_0x_3 + cx_2^2 & x_1x_2 & x_1x_3 & 0 & 0\\
-x_0b & -x_1a & -x_1b & x_2 & x_3 
\end{pmatrix}\, ,\  
\begin{pmatrix} 
-x_1 & 0 & 0\\
cx_2 & -x_3 & 0\\
x_0 & x_2 & 0\\
cx_1a & x_1b & -x_3\\
0 & -x_1a & x_2
\end{pmatrix}\, . 
\]
\end{prop} 

\begin{proof} 
One can easily show, 
as at the beginning of the proof of Lemma~\ref{L:l1(1)cupc2},  
that $I(L_1\cup C) =  (x_0x_3 + cx_2^2,\, x_1x_2,\, x_1x_3)$. 
$L_1\cup C$ is directly linked by the 
complete intersection defined by $x_0x_3 + cx_2^2$ and $x_1x_2$ to the line 
whose homogeneous ideal is $(x_0,\, x_2)$, i.e., to the line $L_2^\prim$. 
Using Ferrand's result about liaison, 
one gets the following free resolution$\, :$ 
\[
0 \lra 2S(-3) \xra{\begin{pmatrix} -x_1 & 0\\ cx_2 & -x_3\\ x_0 & x_2 
\end{pmatrix}} 3S(-2) 
\xra{\displaystyle (x_0x_3 + cx_2^2\, ,\, x_1x_2\, ,\, x_1x_3)} 
I(L_1\cup C) \lra 0\, .
\]
In particular, $L_1 \cup C$ is arithmetically CM hence the graded $S$-module 
$\tH^0_\ast(\sco_{L_1 \cup C})$ is generated by $1 \in \tH^0(\sco_{L_1 \cup C})$. 

Now, by Lemma~\ref{L:ycupt}, one has an exact sequence$\, :$  
\[
0 \lra \text{Im}\, \phi \times \{0\} \lra \sco_{X\cup C} \lra \sco_{L_1\cup C} 
\lra 0 
\]
where $\phi$ is the composite morphism$\, :$ 
\[
\sci_{L_1 \cup C} \lra \sci_{L_1} \lra \sci_{L_1}/\sci_{L_1}^2 \simeq 
2\sco_{L_1}(-1) \xra{\displaystyle (a\, ,\, b)} \sco_{L_1}(l)\, . 
\] 
Moreover, $\Cok \phi$ is an $\sco_{L_1 \cap C} = \sco_{\{P_0\}}$-module hence 
$x_1\sco_{L_1}(l-1) \subseteq \text{Im}\, \phi \subseteq \sco_{L_1}(l)$. 
Notice that$\, :$ 
\[
\phi(x_0x_3 + cx_2^2) = x_0b\, ,\  \phi(x_1x_2) = x_1a\, ,\  
\phi(x_1x_3) = x_1b\, .
\]

(a) In this case $\text{Im}\, \phi = x_1\sco_{L_1}(l-1)$ and one has an exact 
sequence$\, :$ 
\[
0 \lra S(L_1)(l-1) \lra \tH^0_\ast(\sco_{X \cup C}) \lra 
\tH^0_\ast(\sco_{L_1 \cup C}) \lra 0
\] 
where the left morphism maps $1 \in S(L_1)$ to the element of 
$\tH^0(\sco_{X \cup C}(-l+1))$ whose image into $\tH^0(\sco_X(-l+1)) \oplus 
\tH^0(\sco_C(-l+1))$ is $(x_1e_1,\, 0)$. 

(b) In this case $\text{Im}\, \phi = \sco_{L_1}(l)$ and one has an exact 
sequence$\, :$ 
\[
0 \lra S(L_1)(l) \lra \tH^0_\ast(\sco_{X \cup C}) \lra 
\tH^0_\ast(\sco_{L_1 \cup C}) \lra 0
\] 
where the left morphism maps $1 \in S(L_1)$ to the element of 
$\tH^0(\sco_{X \cup C}(-l))$ whose image into $\tH^0(\sco_X(-l)) \oplus 
\tH^0(\sco_C(-l))$ is $(e_1,\, 0)$. 
\end{proof}

\begin{lemma}\label{L:l1(1)cupc3}
Let $C \subset \piii$ be the conic of equations $x_3 = x_1^2 + cx_0x_2 = 0$, 
$c \in k\setminus \{0\}$. Put $q_3 := x_1^2 + cx_0x_2$. Then  
$I(L_1^{(1)} \cup C) = (x_2x_3,\, x_3^2,\, x_2^2q_3)$  
and admits the following minimal graded free resolution$\, :$ 
\[
0 \lra \begin{matrix} S(-3)\\ \oplus\\ S(-5) \end{matrix} 
\xra{\begin{pmatrix} -x_3 & -x_2q_3\\ x_2 & 0\\ 0 & x_3 \end{pmatrix}} 
\begin{matrix} 2S(-2)\\ \oplus\\ S(-4) \end{matrix} 
\lra I(L_1^{(1)} \cup C) \lra 0\, .
\]
\end{lemma} 

\begin{proof}
Let $f = f_0x_2^2 + f_1x_2x_3 + f_2x_3^2$ be a homogeneous element of 
$I(L_1^{(1)})$. Since $x_3$ vanishes on $C$ it follows that if $f$ vanishes on 
$C$ then $f_0x_2^2$ vanishes on $C$ hence $f_0$ vanishes on $C$ hence 
$f_0 \in (x_3,\, q_3)$. One deduces that $I(L_1^{(1)} \cup C)$ is generated by 
the elements from the statement. 

Now, if $W$ is the curve directly linked to $L_1^{(1)} \cup C$ by the complete 
intersection defined by $x_3^2$ and $x_2^2q_3$ then $I(W) = (x_3,\, x_2q_3)$ 
hence $W = L_1 \cup C$. Using Ferrand's result about liaison one gets the 
graded free resolution from the statement.    
\end{proof}

\begin{prop}\label{P:genixcupc3} 
Let $X$ be the double structure on the line $L_1$ considered at the beginning 
of Subsection~\ref{SS:doublecupaline} and let $C\subset \piii$ be the conic 
from Lemma~\ref{L:l1(1)cupc3}. 

\emph{(a)} If $x_1^2 \mid b$, i.e., if $b = x_1^2b_2$ then$\, :$ 
\[
I(X \cup C) = S(F_2 - cx_0b_2x_2^2) + I(L_1^{(1)} \cup C)\, . 
\]

\emph{(b)} If $x_1 \mid b$ but $x_1^2 \nmid b$, i.e., if $b = x_1b_1$ 
with $x_1 \nmid b_1$ then$\, :$ 
\[
I(X \cup C) = S(x_1F_2 - cx_0b_1x_2^2) + I(L_1^{(1)} \cup C)\, . 
\]

\emph{(c)} If $x_1 \nmid b$ then $I(X \cup C) = Sq_3F_2 + I(L_1^{(1)} \cup C)$. 
\end{prop}

\begin{proof}
According to Lemma~\ref{L:ycupt} there is an exact sequence$\, :$ 
\[
0 \lra \sci_{L_1^{(1)} \cup C} \lra \sci_{X \cup C} 
\overset{\displaystyle \psi}{\lra} \sco_{L_1}(-l-2) 
\]
where $\psi$ is the composite morphism$\, :$ 
\[
\sci_{X \cup C} \lra \sci_X \overset{\displaystyle \eta}{\lra} 
\sco_{L_1}(-l-2)\, . 
\]
Moreover, $\Cok \psi$ is an $\sco_{L_1 \cap C}$-module, hence 
$x_1^2\sco_{L_1}(-l-4) \subseteq \text{Im}\, \psi \subseteq \sco_{L_1}(-l-2)$. 

\vskip2mm

\noindent
{\bf Claim 1.}\quad \emph{If} $\text{Im}\, \psi = \sco_{L_1}(-l-2)$ 
\emph{then} $x_1^2 \mid b$. 

\vskip2mm

\noindent
\emph{Indeed}, since $\tH^1_\ast(\sci_{L_1^{(1)} \cup C}) = 0$ (by 
Lemma~\ref{L:l1(1)cupc3}) there exists $f \in \tH^0(\sci_{X \cup C}(l+2))$ such 
that $\psi(f) = 1 \in \tH^0(\sco_{L_1})$. Since $\psi(f) = \eta(f)$ it follows 
that$\, :$ 
\[
f = F_2 + f_0x_2^2 + f_1x_2x_3 + f_2x_3^2\, . 
\]
On the other hand, $f \in \tH^0(\sci_C(l+2))$ hence$\, :$ 
\[
f = g_1x_3 + g_2(x_1^2 + cx_0x_2)\, . 
\]
Restricting the two different expressions of $f$ to the plane $H$ of 
equation $x_3 = 0$ one gets$\, :$ 
\[
-bx_2 + (f_0 \vb H)x_2^2 = (g_2 \vb H)(x_1^2 + cx_0x_2)\, . 
\]
One deduces that $g_2 \vb H = x_2g_2^\prim$ with $g_2^\prim \in k[x_0,x_1,x_2]$ 
hence$\, :$ 
\[
-b + (f_0 \vb H)x_2 = g_2^\prim(x_1^2 + cx_0x_2)\, .
\]
Restricting this relation to the line $L_1$ of equations $x_2 = x_3 = 0$ one 
gets $-b = (g_2^\prim \vb L_1)x_1^2$ whence the Claim. 

\vskip2mm

\noindent
{\bf Claim 2.}\quad \emph{If} $x_1\sco_{L_1}(-l-3) \subseteq \text{Im}\, \psi$ 
\emph{then} $x_1 \mid b$. 

\vskip2mm

\noindent
\emph{Indeed}, since $\tH^1_\ast(\sci_{L_1^{(1)} \cup C}) = 0$ (by 
Lemma~\ref{L:l1(1)cupc3}) there exists $f \in \tH^0(\sci_{X \cup C}(l+3))$ such 
that $\psi(f) = x_1 \in \tH^0(\sco_{L_1}(1))$. Since $\psi(f) = \eta(f)$ it 
follows that$\, :$ 
\[
f = x_1F_2 + f_0x_2^2 + f_1x_2x_3 + f_2x_3^2\, . 
\]
On the other hand, $f \in \tH^0(\sci_C(l+2))$ hence$\, :$ 
\[
f = g_1x_3 + g_2(x_1^2 + cx_0x_2)\, . 
\]
Restricting the two different expressions of $f$ to the plane $H$ of 
equation $x_3 = 0$ one gets$\, :$ 
\[
-bx_1x_2 + (f_0 \vb H)x_2^2 = (g_2 \vb H)(x_1^2 + cx_0x_2)\, . 
\]
One deduces that $g_2 \vb H = x_2g_2^\prim$ with $g_2^\prim \in k[x_0,x_1,x_2]$ 
hence$\, :$ 
\[
-bx_1 + (f_0 \vb H)x_2 = g_2^\prim(x_1^2 + cx_0x_2)\, .
\]
Restricting this relation to the line $L_1$ of equations $x_2 = x_3 = 0$ one 
gets $-b = (g_2^\prim \vb L_1)x_1$ whence the Claim. 

\vskip2mm

(a) In this case, $F_2 - cx_0b_2x_2^2 = -q_3b_2x_2 + ax_3$ belongs to 
$\tH^0(\sci_{X \cup C}(l+2))$ and $\psi(F_2 - cx_0b_2x_2^2) = 
\eta(F_2 - cx_0b_2x_2^2) = 1 \in \tH^0(\sco_{L_1})$. It follows that $\psi$ is 
an epimorphism and that one has an exact sequence$\, :$ 
\[
0 \lra I(L_1^{(1)} \cup C) \lra I(X \cup C) 
\xra{\displaystyle \tH^0_\ast(\psi)} S(L_1)(-l-2) \lra 0
\]   
from which one gets the generators of $I(X \cup C)$ from the statement. 
Using this exact sequence one can also get a graded free resolution of 
$I(X \cup C)$. 

(b) In this case, by Claim 1, $\text{Im}\, \psi \subseteq 
x_1\sco_{L_1}(-l-3)$. On the other hand, 
$x_1F_2 - cx_0b_1x_2^2 = -q_3b_1x_2 + ax_1x_3$ 
belongs to $\tH^0(\sci_{X \cup C}(l+3))$ 
and $\psi(x_1F_2 - cx_0b_1x_2^2) = \eta(x_1F_2 - cx_0b_1x_2^2) = x_1 \in 
\tH^0(\sco_{L_1}(1))$. It follows that $\psi$ factorizes as$\, :$  
\[
\sci_{X \cup C} \overset{\displaystyle \psi^\prim}{\lra} \sco_{L_1}(-l-3) 
\overset{\displaystyle x_1}{\lra} \sco_{L_1}(-l-2)
\]
with $\psi^\prim$ an epimorphism and that one has an exact sequence$\, :$ 
\[
0 \lra I(L_1^{(1)} \cup C) \lra I(X \cup C) 
\xra{\displaystyle \tH^0_\ast(\psi^\prim)} S(L_1)(-l-3) \lra 0
\]   
from which one gets the generators of $I(X \cup C)$ from the statement. 
Using this exact sequence one can also get a graded free resolution of 
$I(X \cup C)$.  

(c) In this case, by Claim 2, $\text{Im}\, \psi = x_1^2\sco_{L_1}(-l-4)$. 
$q_3F_2$ belongs to $\tH^0(\sci_{X \cup C}(l+4))$ and $\psi(q_3F_2) = \eta(q_3F_2) 
= x_1^2 \in \tH^0(\sco_{L_1}(2))$. It follows that $\psi$ factorizes as$\, :$  
\[
\sci_{X \cup C} \overset{\displaystyle \psi^\secund}{\lra} \sco_{L_1}(-l-4) 
\overset{\displaystyle x_1^2}{\lra} \sco_{L_1}(-l-2)
\]
with $\psi^\secund$ an epimorphism and that one has an exact sequence$\, :$ 
\[
0 \lra I(L_1^{(1)} \cup C) \lra I(X \cup C) 
\xra{\displaystyle \tH^0_\ast(\psi^\secund)} S(L_1)(-l-4) \lra 0
\]   
from which one gets the generators of $I(X \cup C)$ from the statement. 
Using this exact sequence one can also get a graded free resolution of 
$I(X \cup C)$.  
\end{proof}

\begin{prop}\label{P:reshoxcupc3} 
Under the hypothesis of Prop.~\ref{P:genixcupc3}$\, :$ 

\emph{(a)} If $x_1^2 \mid b$, i.e., if $b = x_1^2b_2$ then the graded $S$-module 
${\fam0 H}^0_\ast(\sco_{X\cup C})$ admits the following free resolution$\, :$ 
\[
0 \lra \begin{matrix} S(-4)\\ \oplus\\ S(l-4) \end{matrix} 
\overset{\displaystyle \delta_2}{\lra} 
\begin{matrix} S(-1)\\ \oplus\\ S(-3)\\ \oplus\\ 2S(l-3) \end{matrix} 
\overset{\displaystyle \delta_1}{\lra} 
\begin{matrix} S\\ \oplus\\ S(l-2) \end{matrix} 
\overset{\displaystyle \delta_0}{\lra} 
{\fam0 H}^0_\ast(\sco_{X \cup C}) \lra 0 
\]
with $\delta_0 = (1\, ,\, x_1^2e_1)$ and with $\delta_1$ and $\delta_2$ defined 
by the matrices$\, :$ 
\[
\begin{pmatrix} 
x_3 & x_2q_3 & 0 & 0\\
-b_2 & -a & x_2 & x_3 
\end{pmatrix}\, ,\  
\begin{pmatrix}
-x_2q_3 & 0\\
x_3 & 0\\
-q_3b_2 & -x_3\\
a & x_2
\end{pmatrix}\, .
\]

\emph{(b)} If $x_1 \mid b$ but $x_1^2 \nmid b$, i.e., if $b = x_1b_1$ 
with $x_1 \nmid b_1$ then the graded $S$-module 
${\fam0 H}^0_\ast(\sco_{X\cup C})$ admits the following free resolution$\, :$ 
\[
0 \lra \begin{matrix} S(-4)\\ \oplus\\ S(l-3) \end{matrix} 
\overset{\displaystyle \delta_2}{\lra} 
\begin{matrix} S(-1)\\ \oplus\\ S(-3)\\ \oplus\\ 2S(l-2) \end{matrix} 
\overset{\displaystyle \delta_1}{\lra} 
\begin{matrix} S\\ \oplus\\ S(l-1) \end{matrix} 
\overset{\displaystyle \delta_0}{\lra} 
{\fam0 H}^0_\ast(\sco_{X \cup C}) \lra 0 
\]
with $\delta_0 = (1\, ,\, x_1e_1)$ and with $\delta_1$ and $\delta_2$ defined 
by the matrices$\, :$ 
\[
\begin{pmatrix} 
x_3 & x_2q_3 & 0 & 0\\
-b_1 & -x_1a & x_2 & x_3 
\end{pmatrix}\, ,\  
\begin{pmatrix}
-x_2q_3 & 0\\
x_3 & 0\\
-q_3b_1 & -x_3\\
x_1a & x_2
\end{pmatrix}\, .
\]

\emph{(c)} If $x_1 \nmid b$ then the graded $S$-module 
${\fam0 H}^0_\ast(\sco_{X\cup C})$ admits the following free resolution$\, :$ 
\[
0 \lra \begin{matrix} S(-4)\\ \oplus\\ S(l-2) \end{matrix} 
\overset{\displaystyle \delta_2}{\lra} 
\begin{matrix} S(-1)\\ \oplus\\ S(-3)\\ \oplus\\ 2S(l-1) \end{matrix} 
\overset{\displaystyle \delta_1}{\lra} 
\begin{matrix} S\\ \oplus\\ S(l) \end{matrix} 
\overset{\displaystyle \delta_0}{\lra} 
{\fam0 H}^0_\ast(\sco_{X \cup C}) \lra 0 
\]
with $\delta_0 = (1\, ,\, e_1)$ and with $\delta_1$ and $\delta_2$ defined 
by the matrices$\, :$ 
\[
\begin{pmatrix} 
x_3 & x_2q_3 & 0 & 0\\
-b & -x_1^2a & x_2 & x_3 
\end{pmatrix}\, ,\  
\begin{pmatrix}
-x_2q_3 & 0\\
x_3 & 0\\
-q_3b & -x_3\\
x_1^2a & x_2
\end{pmatrix}\, . 
\]
\end{prop} 

\begin{proof}
One has $I(L_1 \cup C) = (x_3,\, x_2q_3)$ whence an exact sequence$\, :$ 
\[
0 \lra S(-4) \xra{\begin{pmatrix} -x_2q_3\\ x_3 \end{pmatrix}} 
\begin{matrix} S(-1)\\ \oplus\\ S(-3) \end{matrix} 
\xra{\displaystyle (x_3\, ,\, x_2q_3)} S \lra 
\tH^0_\ast(\sco_{L_1 \cup C}) \lra 0\, .
\]
Now, by Lemma~\ref{L:ycupt}, one has an exact sequence$\, :$  
\[
0 \lra \text{Im}\, \phi \times \{0\} \lra \sco_{X\cup C} \lra \sco_{L_1\cup C} 
\lra 0 
\]
where $\phi$ is the composite morphism$\, :$ 
\[
\sci_{L_1 \cup C} \lra \sci_{L_1} \lra \sci_{L_1}/\sci_{L_1}^2 \simeq 
2\sco_{L_1}(-1) \xra{\displaystyle (a\, ,\, b)} \sco_{L_1}(l)\, . 
\] 
Moreover, $\Cok \phi$ is an $\sco_{L_1 \cap C}$-module hence 
$x_1^2\sco_{L_1}(l-2) \subseteq \text{Im}\, \phi \subseteq \sco_{L_1}(l)$. 
Notice that$\, :$ 
\[
\phi(x_3) = b\, ,\  \phi(x_2q_3) = x_1^2a\, .
\]

(a) In this case $\text{Im}\, \phi = x_1^2\sco_{L_1}(l-2)$ and one has an 
exact sequence$\, :$ 
\[
0 \lra S(L_1)(l-2) \lra \tH^0_\ast(\sco_{X \cup C}) \lra 
\tH^0_\ast(\sco_{L_1 \cup C}) \lra 0 
\]
where the left morphism maps $1 \in S(L_1)$ to the element of 
$\tH^0(\sco_{X \cup C}(-l+2))$ whose image into $\tH^0(\sco_X(-l+2)) \oplus 
\tH^0(\sco_C(-l+2))$ is $(x_1^2e_1,\, 0)$. 

(b) In this case $\text{Im}\, \phi = x_1\sco_{L_1}(l-1)$ and one has an 
exact sequence$\, :$ 
\[
0 \lra S(L_1)(l-1) \lra \tH^0_\ast(\sco_{X \cup C}) \lra 
\tH^0_\ast(\sco_{L_1 \cup C}) \lra 0 
\]
where the left morphism maps $1 \in S(L_1)$ to the element of 
$\tH^0(\sco_{X \cup C}(-l+1))$ whose image into $\tH^0(\sco_X(-l+1)) \oplus 
\tH^0(\sco_C(-l+1))$ is $(x_1e_1,\, 0)$. 

(c) In this case $\text{Im}\, \phi = \sco_{L_1}(l)$ and one has an 
exact sequence$\, :$ 
\[
0 \lra S(L_1)(l) \lra \tH^0_\ast(\sco_{X \cup C}) \lra 
\tH^0_\ast(\sco_{L_1 \cup C}) \lra 0 
\]
where the left morphism maps $1 \in S(L_1)$ to the element of 
$\tH^0(\sco_{X \cup C}(-l))$ whose image into $\tH^0(\sco_X(-l)) \oplus 
\tH^0(\sco_C(-l))$ is $(e_1,\, 0)$.   
\end{proof}

\begin{lemma}\label{L:l1(1)cupc4}
Let $C \subset \piii$ be the conic of equations $x_3 = x_0x_1 + cx_2^2 = 0$, 
$c \in k\setminus \{0\}$. Put $q_4 := x_0x_1 + cx_2^2$. Then  
$I(L_1^{(1)} \cup C) = (x_2x_3,\, x_3^2,\, x_2^2q_4)$  
and admits the following minimal graded free resolution$\, :$ 
\[
0 \lra \begin{matrix} S(-3)\\ \oplus\\ S(-5) \end{matrix} 
\xra{\begin{pmatrix} -x_3 & -x_2q_4\\ x_2 & 0\\ 0 & x_3 \end{pmatrix}} 
\begin{matrix} 2S(-2)\\ \oplus\\ S(-4) \end{matrix} 
\lra I(L_1^{(1)} \cup C) \lra 0\, .
\]
\end{lemma} 

\begin{proof}
Let $f = f_0x_2^2 + f_1x_2x_3 + f_2x_3^2$ be a homogeneous element of 
$I(L_1^{(1)})$. Since $x_3$ vanishes on $C$ it follows that if $f$ vanishes on 
$C$ then $f_0x_2^2$ vanishes on $C$ hence $f_0$ vanishes on $C$ hence 
$f_0 \in (x_3,\, q_4)$. One deduces that $I(L_1^{(1)} \cup C)$ is generated by 
the elements from the statement. 

Now, if $W$ is the curve directly linked to $L_1^{(1)} \cup C$ by the complete 
intersection defined by $x_3^2$ and $x_2^2q_4$ then $I(W) = (x_3,\, x_2q_4)$ 
hence $W = L_1 \cup C$. Using Ferrand's result about liaison one gets the 
graded free resolution from the statement.    
\end{proof} 

\begin{prop}\label{P:genixcupc4} 
Let $X$ be the double structure on the line $L_1$ considered at the beginning 
of Subsection~\ref{SS:doublecupaline} and let $C\subset \piii$ be the conic 
from Lemma~\ref{L:l1(1)cupc4}. 

\emph{(a)} If $x_0x_1 \mid b$, i.e., if $b = x_0x_1b_2$ then$\, :$ 
\[
I(X \cup C) = S(F_2 - cb_2x_2^3) + I(L_1^{(1)} \cup C)\, . 
\]

\emph{(b)} If $x_0 \mid b$ but $x_1 \nmid b$, i.e., if $b = x_0b_0$  
with $x_1 \nmid b_0$ then$\, :$ 
\[
I(X \cup C) = S(x_1F_2 - cb_0x_2^3) + I(L_1^{(1)} \cup C)\, . 
\]

\emph{(c)} If $x_1 \mid b$ but $x_0 \nmid b$, i.e., if $b = x_1b_1$  
with $x_0 \nmid b_1$ then$\, :$ 
\[
I(X \cup C) = S(x_0F_2 - cb_1x_2^3) + I(L_1^{(1)} \cup C)\, . 
\]

\emph{(d)} If $x_0 \nmid b$ and $x_1 \nmid b$ then 
$I(X \cup C) = Sq_4F_2 + I(L_1^{(1)} \cup C)$. 
\end{prop}

\begin{proof}
According to Lemma~\ref{L:ycupt} there is an exact sequence$\, :$ 
\[
0 \lra \sci_{L_1^{(1)} \cup C} \lra \sci_{X \cup C} 
\overset{\displaystyle \psi}{\lra} \sco_{L_1}(-l-2) 
\]
where $\psi$ is the composite morphism$\, :$ 
\[
\sci_{X \cup C} \lra \sci_X \overset{\displaystyle \eta}{\lra} 
\sco_{L_1}(-l-2)\, . 
\]
Moreover, $\Cok \psi$ is an $\sco_{L_1 \cap C} = \sco_{\{P_0,P_1\}}$-module, hence 
$x_0x_1\sco_{L_1}(-l-4) \subseteq \text{Im}\, \psi \subseteq \sco_{L_1}(-l-2)$. 

\vskip2mm

\noindent
{\bf Claim 1.}\quad \emph{If} $\text{Im}\, \psi = \sco_{L_1}(-l-2)$ 
\emph{then} $x_0x_1 \mid b$. 

\vskip2mm

\noindent
\emph{Indeed}, since $\tH^1_\ast(\sci_{L_1^{(1)} \cup C}) = 0$ (by 
Lemma~\ref{L:l1(1)cupc4}) there exists $f \in \tH^0(\sci_{X \cup C}(l+2))$ such 
that $\psi(f) = 1 \in \tH^0(\sco_{L_1})$. Since $\psi(f) = \eta(f)$ it follows 
that$\, :$ 
\[
f = F_2 + f_0x_2^2 + f_1x_2x_3 + f_2x_3^2\, . 
\]
On the other hand, $f \in \tH^0(\sci_C(l+2))$ hence$\, :$ 
\[
f = g_1x_3 + g_2(x_0x_1 + cx_2^2)\, . 
\]
Restricting the two different expressions of $f$ to the plane $H$ of 
equation $x_3 = 0$ one gets$\, :$ 
\[
-bx_2 + (f_0 \vb H)x_2^2 = (g_2 \vb H)(x_0x_1 + cx_2^2)\, . 
\]
One deduces that $g_2 \vb H = x_2g_2^\prim$ with $g_2^\prim \in k[x_0,x_1,x_2]$ 
hence$\, :$ 
\[
-b + (f_0 \vb H)x_2 = g_2^\prim(x_0x_1 + cx_2^2)\, .
\]
Restricting this relation to the line $L_1$ of equations $x_2 = x_3 = 0$ one 
gets $-b = (g_2^\prim \vb L_1)x_0x_1$ whence the Claim. 

\vskip2mm

\noindent
{\bf Claim 2.}\quad \emph{If} $x_1\sco_{L_1}(-l-3) \subseteq \text{Im}\, \psi$ 
\emph{then} $x_0 \mid b$. 

\vskip2mm

\noindent
\emph{Indeed}, since $\tH^1_\ast(\sci_{L_1^{(1)} \cup C}) = 0$ (by 
Lemma~\ref{L:l1(1)cupc4}) there exists $f \in \tH^0(\sci_{X \cup C}(l+3))$ such 
that $\psi(f) = x_1 \in \tH^0(\sco_{L_1}(1))$. Since $\psi(f) = \eta(f)$ it 
follows that$\, :$ 
\[
f = x_1F_2 + f_0x_2^2 + f_1x_2x_3 + f_2x_3^2\, . 
\]
On the other hand, $f \in \tH^0(\sci_C(l+2))$ hence$\, :$ 
\[
f = g_1x_3 + g_2(x_0x_1 + cx_2^2)\, . 
\]
Restricting the two different expressions of $f$ to the plane $H$ of 
equation $x_3 = 0$ one gets$\, :$ 
\[
-bx_1x_2 + (f_0 \vb H)x_2^2 = (g_2 \vb H)(x_0x_1 + cx_2^2)\, . 
\]
One deduces that $g_2 \vb H = x_2g_2^\prim$ with $g_2^\prim \in k[x_0,x_1,x_2]$ 
hence$\, :$ 
\[
-bx_1 + (f_0 \vb H)x_2 = g_2^\prim(x_0x_1 + cx_2^2)\, .
\]
Restricting this relation to the line $L_1$ of equations $x_2 = x_3 = 0$ one 
gets $-b = (g_2^\prim \vb L_1)x_0$ whence the Claim. Similarly, one has$\, :$  

\vskip2mm

\noindent
{\bf Claim 3.}\quad \emph{If} $x_0\sco_{L_1}(-l-3) \subseteq \text{Im}\, \psi$ 
\emph{then} $x_1 \mid b$. 

\vskip2mm

(a) In this case, $F_2 - cb_2x_2^3 = -q_4b_2x_2 + ax_3$   
belongs to $\tH^0(\sci_{X \cup C}(l+2))$ 
and $\psi(F_2 - cb_2x_2^3) = \eta(F_2 - cb_2x_2^3) = 1 \in 
\tH^0(\sco_{L_1})$. It follows that $\psi$ is an epimorphism and that one has 
an exact sequence$\, :$ 
\[
0 \lra I(L_1^{(1)} \cup C) \lra I(X \cup C) 
\xra{\displaystyle \tH^0_\ast(\psi)} S(L_1)(-l-2) \lra 0
\]   
from which one gets the generators of $I(X \cup C)$ from the statement. 
Using this exact sequence one can also get a graded free resolution of 
$I(X \cup C)$.

(b) In this case, by Claim 1, $\text{Im}\, \psi \neq 
\sco_{L_1}(-l-2)$. On the other hand, $x_1F_2 - cb_0x_2^3 = -q_4b_0x_2 + ax_1x_3$  
belongs to $\tH^0(\sci_{X \cup C}(l+3))$ 
and $\psi(x_1F_2 - cb_0x_2^3) = \eta(x_1F_2 - cb_0x_2^3) = x_1 \in 
\tH^0(\sco_{L_1}(1))$. It follows that $\psi$ factorizes as$\, :$  
\[
\sci_{X \cup C} \overset{\displaystyle \psi_1^\prim}{\lra} \sco_{L_1}(-l-3) 
\overset{\displaystyle x_1}{\lra} \sco_{L_1}(-l-2)
\]
with $\psi_1^\prim$ an epimorphism and that one has an exact sequence$\, :$ 
\[
0 \lra I(L_1^{(1)} \cup C) \lra I(X \cup C) 
\xra{\displaystyle \tH^0_\ast(\psi_1^\prim)} S(L_1)(-l-3) \lra 0
\]   
from which one gets the generators of $I(X \cup C)$ from the statement. 
Using this exact sequence one can also get a graded free resolution of 
$I(X \cup C)$.

(c) In this case, by Claim 1, $\text{Im}\, \psi \neq 
\sco_{L_1}(-l-2)$. On the other hand, $x_0F_2 - cb_1x_2^3 = -q_4b_1x_2 + ax_0x_3$  
belongs to $\tH^0(\sci_{X \cup C}(l+3))$ 
and $\psi(x_0F_2 - cb_1x_2^3) = \eta(x_0F_2 - cb_1x_2^3) = x_0 \in 
\tH^0(\sco_{L_1}(1))$. It follows that $\psi$ factorizes as$\, :$  
\[
\sci_{X \cup C} \overset{\displaystyle \psi_0^\prim}{\lra} \sco_{L_1}(-l-3) 
\overset{\displaystyle x_0}{\lra} \sco_{L_1}(-l-2)
\]
with $\psi_0^\prim$ an epimorphism and that one has an exact sequence$\, :$ 
\[
0 \lra I(L_1^{(1)} \cup C) \lra I(X \cup C) 
\xra{\displaystyle \tH^0_\ast(\psi_0^\prim)} S(L_1)(-l-3) \lra 0
\]   
from which one gets the generators of $I(X \cup C)$ from the statement. 
Using this exact sequence one can also get a graded free resolution of 
$I(X \cup C)$.

(d) In this case, by Claims 2 and 3, $\text{Im}\, \psi = 
x_0x_1\sco_{L_1}(-l-4)$. $q_4F_2$ belongs to $\tH^0(\sci_{X \cup C}(l+4))$ 
and $\psi(q_4F_2) = \eta(q_4F_2) = x_0x_1 \in \tH^0(\sco_{L_1}(2))$. 
It follows that $\psi$ factorizes as$\, :$  
\[
\sci_{X \cup C} \overset{\displaystyle \psi^\secund}{\lra} \sco_{L_1}(-l-4) 
\xra{\displaystyle x_0x_1} \sco_{L_1}(-l-2)
\]
with $\psi^\secund$ an epimorphism and that one has an exact sequence$\, :$ 
\[
0 \lra I(L_1^{(1)} \cup C) \lra I(X \cup C) 
\xra{\displaystyle \tH^0_\ast(\psi^\secund)} S(L_1)(-l-4) \lra 0
\]   
from which one gets the generators of $I(X \cup C)$ from the statement. 
Using this exact sequence one can also get a graded free resolution of 
$I(X \cup C)$.  
\end{proof} 

\begin{prop}\label{P:reshoxcupc4} 
Under the hypothesis of Prop.~\ref{P:genixcupc4}$\, :$ 

\emph{(a)} If $x_0x_1 \mid b$, i.e., if $b = x_0x_1b_2$ then the graded 
$S$-module ${\fam0 H}^0_\ast(\sco_{X\cup C})$ admits the following 
free resolution$\, :$ 
\[
0 \lra \begin{matrix} S(-4)\\ \oplus\\ S(l-4) \end{matrix} 
\overset{\displaystyle \delta_2}{\lra} 
\begin{matrix} S(-1)\\ \oplus\\ S(-3)\\ \oplus\\ 2S(l-3) \end{matrix} 
\overset{\displaystyle \delta_1}{\lra} 
\begin{matrix} S\\ \oplus\\ S(l-2) \end{matrix} 
\overset{\displaystyle \delta_0}{\lra} 
{\fam0 H}^0_\ast(\sco_{X \cup C}) \lra 0 
\]
with $\delta_0 = (1\, ,\, x_0x_1e_1)$ and with $\delta_1$ and $\delta_2$ defined 
by the matrices$\, :$ 
\[
\begin{pmatrix} 
x_3 & x_2q_4 & 0 & 0\\
-b_2 & -a & x_2 & x_3 
\end{pmatrix}\, ,\  
\begin{pmatrix}
-x_2q_4 & 0\\
x_3 & 0\\
-q_4b_2 & -x_3\\
a & x_2
\end{pmatrix}\, .
\]

\emph{(b)} If $x_0 \mid b$ but $x_1 \nmid b$, i.e., if $b = x_0b_0$ 
with $x_1 \nmid b_0$ then the graded $S$-module 
${\fam0 H}^0_\ast(\sco_{X\cup C})$ admits the following free resolution$\, :$ 
\[
0 \lra \begin{matrix} S(-4)\\ \oplus\\ S(l-3) \end{matrix} 
\overset{\displaystyle \delta_2}{\lra} 
\begin{matrix} S(-1)\\ \oplus\\ S(-3)\\ \oplus\\ 2S(l-2) \end{matrix} 
\overset{\displaystyle \delta_1}{\lra} 
\begin{matrix} S\\ \oplus\\ S(l-1) \end{matrix} 
\overset{\displaystyle \delta_0}{\lra} 
{\fam0 H}^0_\ast(\sco_{X \cup C}) \lra 0 
\]
with $\delta_0 = (1\, ,\, x_0e_1)$ and with $\delta_1$ and $\delta_2$ defined 
by the matrices$\, :$ 
\[
\begin{pmatrix} 
x_3 & x_2q_4 & 0 & 0\\
-b_0 & -x_1a & x_2 & x_3 
\end{pmatrix}\, ,\  
\begin{pmatrix}
-x_2q_4 & 0\\
x_3 & 0\\
-q_4b_0 & -x_3\\
x_1a & x_2
\end{pmatrix}\, .
\]

\emph{(c)} If $x_1 \mid b$ but $x_0 \nmid b$, i.e., if $b = x_1b_1$ 
with $x_0 \nmid b_1$ then the graded $S$-module 
${\fam0 H}^0_\ast(\sco_{X\cup C})$ admits the following free resolution$\, :$ 
\[
0 \lra \begin{matrix} S(-4)\\ \oplus\\ S(l-3) \end{matrix} 
\overset{\displaystyle \delta_2}{\lra} 
\begin{matrix} S(-1)\\ \oplus\\ S(-3)\\ \oplus\\ 2S(l-2) \end{matrix} 
\overset{\displaystyle \delta_1}{\lra} 
\begin{matrix} S\\ \oplus\\ S(l-1) \end{matrix} 
\overset{\displaystyle \delta_0}{\lra} 
{\fam0 H}^0_\ast(\sco_{X \cup C}) \lra 0 
\]
with $\delta_0 = (1\, ,\, x_1e_1)$ and with $\delta_1$ and $\delta_2$ defined 
by the matrices$\, :$ 
\[
\begin{pmatrix} 
x_3 & x_2q_4 & 0 & 0\\
-b_1 & -x_0a & x_2 & x_3 
\end{pmatrix}\, ,\  
\begin{pmatrix}
-x_2q_4 & 0\\
x_3 & 0\\
-q_4b_1 & -x_3\\
x_0a & x_2
\end{pmatrix}\, .
\]

\emph{(d)} If $x_0 \nmid b$ and $x_1 \nmid b$ then the graded $S$-module 
${\fam0 H}^0_\ast(\sco_{X\cup C})$ admits the following free resolution$\, :$ 
\[
0 \lra \begin{matrix} S(-4)\\ \oplus\\ S(l-2) \end{matrix} 
\overset{\displaystyle \delta_2}{\lra} 
\begin{matrix} S(-1)\\ \oplus\\ S(-3)\\ \oplus\\ 2S(l-1) \end{matrix} 
\overset{\displaystyle \delta_1}{\lra} 
\begin{matrix} S\\ \oplus\\ S(l) \end{matrix} 
\overset{\displaystyle \delta_0}{\lra} 
{\fam0 H}^0_\ast(\sco_{X \cup C}) \lra 0 
\]
with $\delta_0 = (1\, ,\, e_1)$ and with $\delta_1$ and $\delta_2$ defined 
by the matrices$\, :$ 
\[
\begin{pmatrix} 
x_3 & x_2q_4 & 0 & 0\\
-b & -x_0x_1a & x_2 & x_3 
\end{pmatrix}\, ,\  
\begin{pmatrix}
-x_2q_4 & 0\\
x_3 & 0\\
-q_4b & -x_3\\
x_0x_1a & x_2
\end{pmatrix}\, . 
\]
\end{prop} 

\begin{proof}
One has $I(L_1 \cup C) = (x_3,\, x_2q_4)$ whence an exact sequence$\, :$ 
\[
0 \lra S(-4) \xra{\begin{pmatrix} -x_2q_4\\ x_3 \end{pmatrix}} 
\begin{matrix} S(-1)\\ \oplus\\ S(-3) \end{matrix} 
\xra{\displaystyle (x_3\, ,\, x_2q_4)} S \lra 
\tH^0_\ast(\sco_{L_1 \cup C}) \lra 0\, .
\]
Now, by Lemma~\ref{L:ycupt}, one has an exact sequence$\, :$  
\[
0 \lra \text{Im}\, \phi \times \{0\} \lra \sco_{X\cup C} \lra \sco_{L_1\cup C} 
\lra 0 
\]
where $\phi$ is the composite morphism$\, :$ 
\[
\sci_{L_1 \cup C} \lra \sci_{L_1} \lra \sci_{L_1}/\sci_{L_1}^2 \simeq 
2\sco_{L_1}(-1) \xra{\displaystyle (a\, ,\, b)} \sco_{L_1}(l)\, . 
\] 
Moreover, $\Cok \phi$ is an $\sco_{L_1 \cap C}= \sco_{\{P_0,P_1\}}$-module hence 
$x_0x_1\sco_{L_1}(l-2) \subseteq \text{Im}\, \phi \subseteq \sco_{L_1}(l)$. 
Notice that$\, :$ 
\[
\phi(x_3) = b\, ,\  \phi(x_2q_4) = x_0x_1a\, .
\]

(a) In this case $\text{Im}\, \phi = x_0x_1\sco_{L_1}(l-2)$ and one has an 
exact sequence$\, :$ 
\[
0 \lra S(L_1)(l-2) \lra \tH^0_\ast(\sco_{X \cup C}) \lra 
\tH^0_\ast(\sco_{L_1 \cup C}) \lra 0 
\]
where the left morphism maps $1 \in S(L_1)$ to the element of 
$\tH^0(\sco_{X \cup C}(-l+2))$ whose image into $\tH^0(\sco_X(-l+2)) \oplus 
\tH^0(\sco_C(-l+2))$ is $(x_0x_1e_1,\, 0)$. 

(b) In this case $\text{Im}\, \phi = x_0\sco_{L_1}(l-1)$ and one has an 
exact sequence$\, :$ 
\[
0 \lra S(L_1)(l-1) \lra \tH^0_\ast(\sco_{X \cup C}) \lra 
\tH^0_\ast(\sco_{L_1 \cup C}) \lra 0 
\]
where the left morphism maps $1 \in S(L_1)$ to the element of 
$\tH^0(\sco_{X \cup C}(-l+1))$ whose image into $\tH^0(\sco_X(-l+1)) \oplus 
\tH^0(\sco_C(-l+1))$ is $(x_0e_1,\, 0)$. 

(c) In this case $\text{Im}\, \phi = x_1\sco_{L_1}(l-1)$ and one has an 
exact sequence$\, :$ 
\[
0 \lra S(L_1)(l-1) \lra \tH^0_\ast(\sco_{X \cup C}) \lra 
\tH^0_\ast(\sco_{L_1 \cup C}) \lra 0 
\]
where the left morphism maps $1 \in S(L_1)$ to the element of 
$\tH^0(\sco_{X \cup C}(-l+1))$ whose image into $\tH^0(\sco_X(-l+1)) \oplus 
\tH^0(\sco_C(-l+1))$ is $(x_1e_1,\, 0)$.

(d) In this case $\text{Im}\, \phi = \sco_{L_1}(l)$ and one has an 
exact sequence$\, :$ 
\[
0 \lra S(L_1)(l) \lra \tH^0_\ast(\sco_{X \cup C}) \lra 
\tH^0_\ast(\sco_{L_1 \cup C}) \lra 0 
\]
where the left morphism maps $1 \in S(L_1)$ to the element of 
$\tH^0(\sco_{X \cup C}(-l))$ whose image into $\tH^0(\sco_X(-l)) \oplus 
\tH^0(\sco_C(-l))$ is $(e_1,\, 0)$.   
\end{proof}

\subsection{Union of two double lines}\label{SS:twodoublelines} 
Let $X$, $X^\prim$ and $X^\secund$ be double structures on the lines $L_1$, 
$L_1^\prim$ and $L_2$, respectively, defined by exact sequences$\, :$ 
\begin{gather*}
0 \lra \sci_X \lra \sci_{L_1} \overset{\displaystyle \pi}{\lra} 
\sco_{L_1}(l) \lra 0\, ,\\
0 \lra \sci_{X^\prim} \lra \sci_{L_1^\prim} 
\overset{\displaystyle \pi^\prim}{\lra} 
\sco_{L_1^\prim}(l^\prim) \lra 0\, ,\\
0 \lra \sci_{X^\secund} \lra \sci_{L_2} 
\overset{\displaystyle \pi^\secund}{\lra} 
\sco_{L_2}(l^\secund) \lra 0\, ,
\end{gather*} 
where $\pi$, $\pi^\prim$, $\pi^\secund$ are composite morphisms$\, :$ 
\begin{gather*}
\sci_{L_1} \lra \sci_{L_1}/\sci_{L_1}^2 \simeq 2\sco_{L_1}(-1) 
\xra{\displaystyle (a\, ,\, b)} \sco_{L_1}(l)\, ,\\
\sci_{L_1^\prim} \lra \sci_{L_1^\prim}/\sci_{L_1^\prim}^2 \simeq 2\sco_{L_1^\prim}(-1) 
\xra{\displaystyle (a^\prim,\, b^\prim)} \sco_{L_1^\prim}(l^\prim)\, ,\\
\sci_{L_2} \lra \sci_{L_2}/\sci_{L_2}^2 \simeq 2\sco_{L_2}(-1) 
\xra{\displaystyle (a^\secund,\, b^\secund)} \sco_{L_2}(l^\secund)\, . 
\end{gather*}
Putting $F_2 := -b(x_0,x_1)x_2 + a(x_0,x_1)x_3$, $F_2^\prim := 
-b^\prim(x_2,x_3)x_0 + a^\prim(x_2,x_3)x_1$ and $F_2^\secund := 
-b^\secund(x_0,x_2)x_1 + a^\secund(x_0,x_2)x_3$ one has$\, :$ 
\begin{gather*}
I(X) = (F_2,\, x_2^2,\, x_2x_2,\, x_3^2)\, ,\  
I(X^\prim) = (F_2^\prim,\, x_0^2,\, x_0x_1,\, x_1^2)\, ,\\
I(X^\secund) = (F_2^\secund,\, x_1^2,\, x_1x_3,\, x_3^2)\, .
\end{gather*} 
Recall, also, from the beginning of Subsection~\ref{SS:doublecupaline}, the 
exact sequences$\, :$ 
\begin{gather*}
0 \lra \sci_{L_1^{(1)}} \lra \sci_X \overset{\displaystyle \eta}{\lra} 
\sco_{L_1}(-l-2) \lra 0\, ,\\
0 \lra \sci_{L_1^{\prim (1)}} \lra \sci_{X^\prim} 
\overset{\displaystyle \eta^\prim}{\lra} 
\sco_{L_1^\prim}(-l^\prim -2) \lra 0\, ,\\
0 \lra \sci_{L_2^{(1)}} \lra \sci_{X^\secund}  
\overset{\displaystyle \eta^\secund}{\lra} 
\sco_{L_2}(-l^\secund -2) \lra 0\, . 
\end{gather*}
The following result is an immediate consequence of Lemma~\ref{L:zcupw}$\, :$ 

\begin{lemma}\label{L:l1(1)cupl1prim(1)} 
$I(L_1^{(1)} \cup L_1^{\prim (1)}) = (x_0^2,\, x_0x_1,\, x_1^2)(x_2^2,\, x_2x_3,\, 
x_3^2)$ and the tensor product of the complexes$\, :$ 
\[
2S(-3) \xra{\begin{pmatrix} -x_3 & 0\\ x_2 & -x_3\\ 0 & x_2 \end{pmatrix}} 
3S(-2)\, ,\  
2S(-3) \xra{\begin{pmatrix} -x_1 & 0\\ x_0 & -x_1\\ 0 & x_0 \end{pmatrix}} 
3S(-2) 
\]
is a minimal graded free resolution of this ideal. 
\qed
\end{lemma} 

\begin{lemma}\label{L:genixcupl1prim(1)} 
With the notation introduced at the beginning of this subsection$\, :$ 

\emph{(a)} If $l = -1$ then $I(X \cup L_1^{\prim (1)}) = Sx_0^2F_2 + Sx_0x_1F_2 
+ Sx_1^2F_2 + I(L_1^{(1)} \cup L_1^{\prim (1)})$. 

\emph{(b)} If $l = 0$ then $I(X \cup L_1^{\prim (1)}) = Sx_0F_2 + Sx_1F_2 + 
I(L_1^{(1)} \cup L_1^{\prim (1)})$. 

\emph{(c)} If $l \geq 1$ then $I(X \cup L_1^{\prim (1)}) = SF_2 +  
I(L_1^{(1)} \cup L_1^{\prim (1)})$. 
\end{lemma}

\begin{proof}
Tensorizing by $\sci_{L_1^{\prim (1)}}$ the exact sequence$\, :$ 
\[
0 \lra \sci_{L_1^{(1)}} \lra \sci_X \overset{\displaystyle \eta}{\lra} 
\sco_{L_1}(-l-2) \lra 0 
\]
one gets an exact sequence$\, :$ 
\[
0 \lra \sci_{L_1^{(1)} \cup L_1^{\prim (1)}} \lra \sci_{X \cup L_1^{\prim (1)}} 
\overset{\displaystyle \psi}{\lra} \sco_{L_1}(-l-2) \lra 0\, . 
\]
Lemma~\ref{L:l1(1)cupl1prim(1)} implies that 
$\tH^1(\sci_{L_1^{(1)} \cup L_1^{\prim (1)}}(i)) = 0$ for $i \geq 3$. On the other 
hand, one has $\tH^0(\sci_{L_1 \cup L_1^{\prim (1)}}(2)) = 0$ (by 
Lemma~\ref{L:zcupw}) hence $\tH^0(\sci_{X \cup L_1^{\prim (1)}}(2)) = 0$. 
One deduces that$\, :$ 
\[
\text{Im}\, \tH^0_\ast(\psi) = {\textstyle \bigoplus_{i \geq 3}} 
\tH^0(\sco_{L_1}(-l-2+i))\, . 
\] 

(a) In this case $X$ is the divisor $2L_1$ on the plane $H \supset L_1$ of 
equation $F_2 = 0$ hence a complete intersection. The assertion from the 
statement follows, now, from Lemma~\ref{L:zcupw}. Notice that, using the same 
lemma, one can get a minimal graded free resolution of 
$I(X \cup L_1^{\prim (1)})$. 

(b) In this case one has an exact sequence$\, :$ 
\[
0 \lra I(L_1^{(1)} \cup L_1^{\prim (1)}) \lra I(X \cup L_1^{\prim (1)})  
\xra{\displaystyle \tH^0_\ast(\psi)} S(L_1)_+(-l-2) \lra 0\, .
\]  
The assertion from the statement follows noticing that $x_0F_2$ and $x_1F_2$ 
belong to $I(X \cup L_1^{\prim (1)})$. Notice, also, that using the above exact 
sequence one can get a graded free resolution of $I(X \cup L_1^{\prim (1)})$ 
(a minimal free resolution of the graded $S$-module $S(L_1)_+$ can be found 
in the discussion following Prop.~\ref{P:genizprim}).   

(c) In this case one has an exact sequence$\, :$ 
\[
0 \lra I(L_1^{(1)} \cup L_1^{\prim (1)}) \lra I(X \cup L_1^{\prim (1)})  
\xra{\displaystyle \tH^0_\ast(\psi)} S(L_1)(-l-2) \lra 0\, .
\]  
The assertion from the statement follows noticing that $F_2$  
belongs to $I(X \cup L_1^{\prim (1)})$. Notice, also, that using the above exact 
sequence one can get a graded free resolution of $I(X \cup L_1^{\prim (1)})$.
\end{proof}

\begin{prop}\label{P:genixcupxprim}
With the notation from the beginning of this subsection, assume that 
$l \geq l^\prim$. 

\emph{(a)} If $l = -1$ $($hence $l^\prim = -1$$)$ then $I(X \cup X^\prim) = 
Sx_0^2F_2 + Sx_0x_1F_2 + Sx_1^2F_2 + I(L_1^{(1)} \cup X^\prim)$. 

\emph{(b)} If $l = 0$ and $l^\prim = -1$ then $I(X \cup X^\prim) = 
Sx_0F_2 + Sx_1F_2 + I(L_1^{(1)} \cup X^\prim)$. 

\emph{(c)} If $l = 0$, $l^\prim = 0$ and $kF_2 \neq kF_2^\prim$ then 
$I(X \cup X^\prim) = Sx_0F_2 + Sx_1F_2 + I(L_1^{(1)} \cup X^\prim)$. 

\emph{(d)} If $l = 0$, $l^\prim = 0$ and $kF_2 = kF_2^\prim$ then 
$I(X \cup X^\prim) = SF_2 + I(L_1^{(1)} \cup L_1^{\prim (1)})$, i.e., 
$X \cup X^\prim$ is the divisor $2L_1 + 2L_1^\prim$ on the quadric surface 
$Q \subset \piii$ of equation $F_2 = 0$. 

\emph{(e)} If $l \geq 1$ then $I(X \cup X^\prim) = SF_2 + I(L_1^{(1)} \cup 
X^\prim)$.     
\end{prop} 

\noindent
Notice that, by Lemma~\ref{L:genixcupl1prim(1)}, 
\[
I(L_1^{(1)} \cup X^\prim) =  
\begin{cases}
Sx_2^2F_2^\prim + Sx_2x_3F_2^\prim + Sx_3^2F_2^\prim + 
I(L_1^{(1)} \cup L_1^{\prim (1)})\, , &
\text{if $l^\prim = -1\, ;$}\\ 
Sx_2F_2^\prim + Sx_3F_2^\prim + I(L_1^{(1)} \cup L_1^{\prim (1)})\, , &  
\text{if $l^\prim = 0\, ;$}\\ 
SF_2^\prim + I(L_1^{(1)} \cup L_1^{\prim (1)})\, , & 
\text{if $l^\prim \geq 1\, .$}
\end{cases} 
\]

\begin{proof}
Tensorizing by $\sci_{X^\prim}$ the exact sequence$\, :$ 
\[
0 \lra \sci_{L_1^{(1)}} \lra \sci_X \overset{\displaystyle \eta}{\lra} 
\sco_{L_1}(-l-2) \lra 0 
\] 
one gets an exact sequence$\, :$ 
\[
0 \lra \sci_{L_1^{(1)} \cup X^\prim} \lra \sci_{X \cup X^\prim} 
\overset{\displaystyle {\widetilde \psi}}{\lra} \sco_{L_1}(-l-2) \lra 0\, .
\]
Recall, from the proof of Lemma~\ref{L:genixcupl1prim(1)}, that one has an 
exact sequence$\, :$ 
\[
0 \lra \sci_{L_1^{(1)} \cup L_1^{\prim (1)}} \lra \sci_{L_1^{(1)} \cup X^\prim} 
\overset{\displaystyle \psi^\prim}{\lra} \sco_{L_1^\prim}(-l^\prim -2) \lra 0\, . 
\]
Using Lemma~\ref{L:l1(1)cupl1prim(1)} one deduces that 
$\tH^1(\sci_{L_1^{(1)} \cup X^\prim}(i)) = 0$ for $i \geq \text{max}(3,\, l^\prim + 
1)$. 

\vskip2mm

(a) In this case $X$ and $X^\prim$ are complete intersections hence, by 
Lemma~\ref{L:zcupw}, $I(X \cup X^\prim) = I(X)I(X^\prim)$. From the same lemma 
one can get a minimal free resolution of $I(X \cup X^\prim)$. 

(b) In this case one has an exact sequence$\, :$ 
\[
0 \lra \sci_{L_1^{(1)} \cup X^\prim} \lra \sci_{X \cup X^\prim} 
\overset{\displaystyle {\widetilde \psi}}{\lra} \sco_{L_1}(-2) \lra 0 
\] 
and $\tH^1(\sci_{L_1^{(1)} \cup X^\prim}(i)) = 0$ for $i \geq 3$. We assert that 
$\tH^0(\sci_{X \cup X^\prim}(2)) = 0$. 

\emph{Indeed}, all the quadric surfaces containing $X$ are nonsingular.  
$X^\prim$ is the divisor $2L_1^\prim$ on the plane $H \supset L_1^\prim$ of 
equation $F_2^\prim = 0$. The intersection of $H$ with a nonsingular quadric 
containing $L_1^\prim$ is the union of $L_1^\prim$ and of another (different) 
line. It follows that no quadric surface containing $X$ can contain $X^\prim$. 

One deduces, now, that one has an exact sequence$\, :$ 
\[
0 \lra I(L_1^{(1)} \cup X^\prim) \lra I(X \cup X^\prim)  
\xra{\displaystyle \tH^0_\ast({\widetilde \psi})} S(L_1)_+(-2) \lra 0\, .
\]    
The assertion from the statement follows noticing that $x_0F_2$ and $x_1F_2$ 
belong to $I(X \cup X^\prim)$. Notice, also, that using the above exact 
sequence one can get a graded free resolution of $I(X \cup X^\prim)$ 
(a minimal free resolution of the graded $S$-module $S(L_1)_+$ can be found 
in the discussion following Prop.~\ref{P:genizprim}). 

(c) It follows, from Prop.~\ref{P:genixcupl1prim}, that 
$\tH^0(\sci_{X \cup L_1^\prim}(2)) = kF_2$ and that 
$\tH^0(\sci_{L_1 \cup X^\prim}(2)) = kF_2^\prim$. One deduces, from the hypothesis, 
that $\tH^0(\sci_{X \cup X^\prim}(2)) = 0$. One can use, now, the same argument 
as in case (b). 

(d) The argument used in case (c) shows that 
$\tH^0(\sci_{X \cup X^\prim}(2)) = kF_2 = kF_2^\prim$. One deduces the existence of 
an exact sequence$\, :$ 
\[
0 \lra I(L_1^{(1)} \cup X^\prim) \lra I(X \cup X^\prim)  
\xra{\displaystyle \tH^0_\ast({\widetilde \psi})} S(L_1)(-2) \lra 0 
\]  
from which the assertion from the statement follows. Notice, also, that using 
this exact sequence one can get a graded free resolution of 
$I(X \cup X^\prim)$. 

(e) In this case one has $\tH^1(\sci_{L_1^{(1)} \cup X^\prim}(l+2)) = 0$ (one takes 
into account that $l \geq l^\prim$). One deduces the existence of 
an exact sequence$\, :$ 
\[
0 \lra I(L_1^{(1)} \cup X^\prim) \lra I(X \cup X^\prim)  
\xra{\displaystyle \tH^0_\ast({\widetilde \psi})} S(L_1)(-l-2) \lra 0 
\]  
from which the assertion from the statement follows. Notice, also, that using 
this exact sequence one can get a graded free resolution of 
$I(X \cup X^\prim)$. 
\end{proof} 

\begin{lemma}\label{L:l1(1)cupl2(1)}
$I(L_1^{(1)} \cup L_2^{(1)}) = (x_3,\, x_1x_2)^2 = (x_3^2,\, x_1x_2x_3,\, 
x_1^2x_2^2)$ and admits the following minimal graded free resolution$\, :$ 
\[
0 \lra \begin{matrix} S(-4)\\ \oplus\\ S(-5) \end{matrix} 
\xra{\begin{pmatrix} -x_1x_2 & 0\\ x_3 & -x_1x_2\\ 0 & x_3 \end{pmatrix}} 
\begin{matrix} S(-2)\\ \oplus\\ S(-3)\\ \oplus\\ S(-4) \end{matrix} 
\lra I(L_1^{(1)} \cup L_2^{(1)}) \lra 0\, . 
\]
\end{lemma}

\begin{proof}
$I(L_1^{(1)} \cup L_2^{(1)}) = (x_2^2,\, x_2x_2,\, x_3^2) \cap 
(x_1^2,\, x_1x_3,\, x_3^2) = (x_3^2,\, x_1x_2x_3,\, x_1^2x_2^2)$. 
$L_1^{(1)} \cup L_2^{(1)}$ is directly linked to $L_1 \cup L_2$ by the complete 
intersection defined by $x_3^2$ and $x_1^2x_2^2$. Applying Ferrand's result 
about liaison one gets the minimal free resolution from the statement. 
\end{proof}

\begin{lemma}\label{L:genixcupl2(1)} 
With the notation introduced at the beginning of this subsection$\, :$  

\emph{(a)} If $x_1 \mid b$ then $I(X \cup L_2^{(1)}) = Sx_1F_2 + 
I(L_1^{(1)} \cup L_2^{(1)})$. 

\emph{(b)} If $x_1 \nmid b$ then $I(X \cup L_2^{(1)}) = Sx_1^2F_2 + 
I(L_1^{(1)} \cup L_2^{(1)})$.
\end{lemma}

\begin{proof}
According to Lemma~\ref{L:ycupt}, one has an exact sequence$\, :$ 
\[
0 \lra \sci_{L_1^{(1)} \cup L_2^{(1)}} \lra \sci_{X \cup L_2^{(1)}} 
\overset{\displaystyle \psi}{\lra} \sco_{L_1}(-l-2)
\]
where $\psi$ is the composite morphism$\, :$ 
\[
\sci_{X \cup L_2^{(1)}} \lra \sci_X \overset{\displaystyle \eta}{\lra}  
\sco_{L_1}(-l-2)\, .
\]
Moreover, $\Cok \psi$ is an $\sco_{L_1 \cap L_2^{(1)}}$-module. Since 
$I(L_1 \cap L_2^{(1)}) = (x_1^2,\, x_2,\, x_3)$ it follows that 
$x_1^2\sco_{L_1}(-l-4) \subseteq \text{Im}\, \psi \subseteq \sco_{L_1}(-l-2)$. 

\vskip2mm

\noindent
{\bf Claim 1.}\quad $\text{Im}\, \psi \subseteq x_1\sco_{L_1}(-l-3)$. 

\vskip2mm

\noindent
\emph{Indeed}, assume that $\text{Im}\, \psi = \sco_{L_1}(-l-2)$. Since, 
by Lemma~\ref{L:l1(1)cupl2(1)}, $\tH^1_\ast(\sci_{L_1^{(1)} \cup L_2^{(1)}}) = 0$ 
there exists $f \in \tH^0(\sci_{X \cup L_2^{(1)}}(l+2))$ such that 
$\psi(f) = 1 \in \tH^0(\sco_{L_1})$. Since $\psi(f) = \eta(f)$ it follows 
that$\, :$ 
\[
f = F_2 + f_0x_2^2 + f_1x_2x_3 + f_2x_3^2\, .
\] 
On the other hand, since $f \in I(L_2^{(1)})$ one has$\, :$ 
\[
f = g_0x_1^2 + g_1x_1x_3 + g_2x_3^2\, .
\]
Comparing the two different expressions of $f$ one deduces that each 
monomial appearing in $F_2 := -b(x_0,x_1)x_2 + a(x_0,x_1)x_3$ must be divisible 
by one of the monomials $x_1^2$, $x_1x_3$, $x_2^2$, $x_2x_3$, $x_3^2$. It 
follows that $x_1^2 \mid b$ and $x_1 \mid a$ which is \emph{not possible} 
because $a$ and $b$ are coprime. 

\vskip2mm

(a) In this case $b = x_1b_1$ hence $x_1F_2 = -b_1x_1^2x_2 + ax_1x_3$ belongs 
to $I(X \cup L_2^{(1)})$. Since $\psi(x_1F_2) = \eta(x_1F_2) = x_1 \in 
\tH^0(\sco_{L_1}(1))$ one deduces, taking into account Claim 1, that 
$\psi$ factorizes as$\, :$ 
\[
\sci_{X \cup L_2^{(1)}} \overset{\displaystyle \psi^\prim}{\lra} 
\sco_{L_1}(-l-3) \overset{\displaystyle x_1}{\lra} \sco_{L_1}(-l-2) 
\]
with $\psi^\prim$ and epimorphism and that one has an exact sequence$\, :$ 
\[
0 \lra I(L_1^{(1)} \cup L_2^{(1)}) \lra I(X \cup L_2^{(1)}) 
\xra{\displaystyle \tH^0_\ast(\psi^\prim)} S(L_1)(-l-3) \lra 0\, .
\]
Using this exact sequence one sees that $I(X \cup L_2^{(1)})$ is generated 
by the elements from the statement. One can also get, from this sequence, 
a graded free resolution of this ideal. 

(b) Using the same kind of argument as in the proof of Claim 1 one shows 
that$\, :$ 

\vskip2mm

\noindent
{\bf Claim 2.}\quad \emph{If} $x_1 \nmid b$ \emph{then} 
$\text{Im}\, \psi = x_1^2\sco_{L_1}(-l-4)$ (\emph{that is, it is not equal to} 
$x_1\sco_{L_1}(-l-3)$). 

\vskip2mm

\noindent
Since $x_1^2F_2 \in I(X \cup L_2^{(1)})$ and $\psi(x_1^2F_2) = \eta(x_1^2F_2) = 
x_1^2 \in \tH^0(\sco_{L_1}(2))$ one deduces, taking into account Claim 2, that 
$\psi$ factorizes as$\, :$ 
\[
\sci_{X \cup L_2^{(1)}} \overset{\displaystyle \psi^\secund}{\lra} 
\sco_{L_1}(-l-4) \overset{\displaystyle x_1^2}{\lra} \sco_{L_1}(-l-2) 
\]
with $\psi^\secund$ and epimorphism and that one has an exact sequence$\, :$ 
\[
0 \lra I(L_1^{(1)} \cup L_2^{(1)}) \lra I(X \cup L_2^{(1)}) 
\xra{\displaystyle \tH^0_\ast(\psi^\secund)} S(L_1)(-l-4) \lra 0\, .
\]
Using this exact sequence one sees that $I(X \cup L_2^{(1)})$ is generated 
by the elements from the statement. One can also get, from this sequence, 
a graded free resolution of this ideal.   
\end{proof}

\begin{prop}\label{P:genixcupxsecund}
With the notation introduced at the beginning of this subsection, assume that 
$l \geq l^\secund$. 

\emph{(a)} If $x_1 \mid b$ $($i.e., $b = x_1b_1$$)$, $x_2 \mid b^\secund$ (i.e., 
$b^\secund = x_2b_1^\secund$) and $\frac{b_1(x_0,0)}{a(x_0,0)} = 
\frac{b_1^\secund(x_0,0)}{a^\secund(x_0,0)}$ then$\, :$ 
\[
I(X \cup X^\secund) = S\left(F_2 + b_1(x_0,0)x_1x_2 - a(x_0,0)x_3 + 
{\textstyle \frac{a(x_0,0)}{a^\secund(x_0,0)}}F_2^\secund \right) + 
Sx_2F_2^\secund + I(L_1^{(1)} \cup L_2^{(1)})\, .
\]

\emph{(b)} If $x_1 \mid b$ $($i.e., $b = x_1b_1$$)$, $x_2 \mid b^\secund$ (i.e., 
$b^\secund = x_2b_1^\secund$) and $\frac{b_1(x_0,0)}{a(x_0,0)} \neq  
\frac{b_1^\secund(x_0,0)}{a^\secund(x_0,0)}$ then$\, :$ 
\[
I(X \cup X^\secund) = Sx_1F_2 + Sx_2F_2^\secund + I(L_1^{(1)} \cup L_2^{(1)})\, .
\]

\emph{(c)} If $x_1 \mid b$ and $x_2 \nmid b^\secund$ then$\, :$ 
\[
I(X \cup X^\secund) = Sx_1F_2 + Sx_2^2F_2^\secund + I(L_1^{(1)} \cup L_2^{(1)})\, .
\]

\emph{(d)} If $x_1 \nmid b$ and $x_2 \mid b^\secund$ then$\, :$ 
\[
I(X \cup X^\secund) = Sx_1^2F_2 + Sx_2F_2^\secund + I(L_1^{(1)} \cup L_2^{(1)})\, .
\]

\emph{(e)} If $x_1 \nmid b$ and $x_2 \nmid b^\secund$ then$\, :$ 
\[
I(X \cup X^\secund) = S\left(x_1F_2 + b(x_0,0)x_1x_2 + 
{\textstyle \frac{b(x_0,0)}{b^\secund(x_0,0)}}x_2F_2^\secund \right) + 
Sx_2^2F_2^\secund + I(L_1^{(1)} \cup L_2^{(1)})\, .  
\]
\end{prop} 

\begin{proof}
From Lemma~\ref{L:ycupt}, one has an exact sequence$\, :$ 
\[
0 \lra \sci_{L_1^{(1)} \cup X^\secund} \lra \sci_{X \cup X^\secund} 
\overset{\displaystyle {\widetilde \psi}}{\lra} \sco_{L_1}(-l-2) 
\]
with $\widetilde \psi$ the composite morphism$\, :$ 
\[
\sci_{X \cup X^\secund} \lra \sci_X \overset{\displaystyle \eta}{\lra} 
\sco_{L_1}(-l-2)\, . 
\]
By Lemma~\ref{L:genixcupl2(1)}$\, :$ 
\[
I(L_1^{(1)} \cup X^\secund) = 
\begin{cases}
Sx_2F_2^\secund + I(L_1^{(1)} \cup L_2^{(1)})\, , & 
\text{if $x_2 \mid b^\secund \, ;$}\\
Sx_2^2F_2^\secund + I(L_1^{(1)} \cup L_2^{(1)})\, , & 
\text{if $x_2 \nmid b^\secund \, .$}
\end{cases}
\]
Now, $\widetilde \psi$ can be also written as a composite morphism$\, :$ 
\[
\sci_{X \cup X^\secund} \lra \sci_{X \cup L_2} 
\overset{\displaystyle \overline{\psi}}{\lra} 
\sco_{L_1}(-l-2)
\]
with $\overline{\psi}$ the composite morphism $\sci_{X \cup L_2} \ra \sci_X 
\overset{\eta}{\lra} \sco_{L_1}(-l-2)$. By Lemma~\ref{L:ycupt} again, one has 
an exact sequence$\, :$ 
\[
0 \lra \sci_{X \cup X^\secund} \lra \sci_{X \cup L_2} 
\overset{\displaystyle \phi^\secund}{\lra} \sco_{L_2}(l^\secund) 
\]
where $\phi^\secund$ is the composite morphism$\, :$ 
\[
\sci_{X \cup L_2} \lra \sci_{L_2} \overset{\displaystyle \pi^\secund}{\lra} 
\sco_{L_2}(l^\secund)\, . 
\]
By Prop.~\ref{P:genixcupl} one has$\, :$ 
\[
I(X \cup L_2) = 
\begin{cases}
(F_2,\, x_2x_3,\, x_3^2,\, x_1x_2^2)\, , & \text{if $x_1 \mid b \, ;$}\\ 
(x_1F_2,\, x_2x_3,\, x_3^2,\, x_1x_2^2)\, , & \text{if $x_1 \nmid b \, .$}
\end{cases}
\]
Recalling, from the beginning of this subsection, the defintion of 
$\pi^\secund$ one gets$\, :$ 
\begin{gather*}
\phi^\secund (F_2) = -b_1(x_0,0)x_2a^\secund + a(x_0,0)b^\secund\, ,\  \  
\text{if $b = x_1b_1 \, ,$}\\
\phi^\secund (x_1F_2) = -b(x_0,0)x_2a^\secund\, ,\  \  
\text{if $x_1 \nmid b \, ,$}\\ 
\phi^\secund(x_2x_3) = x_2b^\secund \, ,\  \phi^\secund (x_3^2) = 0\, ,\  
\phi^\secund (x_1x_2^2) = x_2^2a^\secund \, .
\end{gather*} 

\vskip2mm

\noindent
{\bf Claim 1.}\quad (i) \emph{If} $x_1 \mid b$ \emph{then} 
$x_1\sco_{L_1}(-l-3) \subseteq \text{Im}\, {\widetilde \psi} \subseteq 
\sco_{L_1}(-l-2)$. 

(ii) \emph{If} $x_1 \nmid b$ \emph{then} 
$x_1^2\sco_{L_1}(-l-4) \subseteq \text{Im}\, {\widetilde \psi} \subseteq 
x_1\sco_{L_1}(-l-3)$. 

\vskip2mm

\noindent
\emph{Indeed}, in case (i) $b = x_1b_1$ hence $x_1F_2$ belongs to 
$I(X \cup L_2^{(1)}) \subseteq I(X \cup X^\secund)$ and 
${\widetilde \psi}(x_1F_2) = 
\eta(x_1F_2) = x_1 \in \tH^0(\sco_{L_1}(1))$. 

In case (ii), it follows from the proof of Prop.~\ref{P:genixcupl} that 
$\text{Im}\, \overline{\psi} = x_1 \sco_{L_1}(-l-3)$ hence 
$\text{Im}\, {\widetilde \psi} \subseteq x_1 \sco_{L_1}(-l-3)$. On the other 
hand, by Lemma~\ref{L:ycupt}, $\Cok {\widetilde \psi}$ is an 
$\sco_{L_1 \cap X^\secund}$-module. 
Since $I(X^\secund) = (F_2^\secund ,\, x_1^2,\, x_1x_3,\, x_3^2)$ one deduces that 
$x_1^2\sco_{L_1}(-l-4) \subseteq \text{Im}\, {\widetilde \psi}$. 

\vskip2mm

\noindent
{\bf Claim 2.}\quad \emph{If} $x_1 \mid b$ (\emph{i.e., if} $b = x_1b_1$) 
\emph{and} $\text{Im}\, {\widetilde \psi} = \sco_{L_1}(-l-2)$ \emph{then} 
$x_2 \mid b^\secund$ (\emph{i.e.,} $b^\secund = x_2b_1^\secund$) \emph{and} 
$\frac{b_1(x_0,0)}{a(x_0,0)} = \frac{b_1^\secund(x_0,0)}{a^\secund(x_0,0)}$. 

\vskip2mm

\noindent
\emph{Indeed}, if $\text{Im}\, {\widetilde \psi} = \sco_{L_1}(-l-2)$ then, for 
$m >> 0$, there exists $f \in \tH^0(\sci_{X \cup X^\secund}(l+2+m))$ such that 
${\widetilde \psi}(f) = x_0^m \in \tH^0(\sco_{L_1}(m))$. Since 
$f \in \tH^0(\sci_{X \cup L_2}(l+2+m))$ and ${\widetilde \psi}(f) = 
\overline{\psi}(f)$ it follows that$\, :$ 
\[
f = x_0^mF_2 + f_0x_2x_3 + f_1x_3^2 + f_2x_1x_2^2\, .
\]  
On the other hand, $\phi^\secund(f) = 0$ hence$\, :$ 
\[
-x_0^mb_1(x_0,0)x_2a^\secund + x_0^ma(x_0,0)b^\secund + (f_0 \vb L_2)x_2b^\secund + 
(f_2 \vb L_2)x_2^2a^\secund = 0\, .  
\]
Since $x_1 \mid b$ and $a$ and $b$ are coprime it follows that $x_1 \nmid a$ 
hence $a(x_0,0) \neq 0$. One deduces, now, from the last relation, that 
$x_2 \mid b^\secund$, i.e., $b^\secund = x_2b_1^\secund$ and then that$\, :$ 
\[
(x_0^ma(x_0,0) + (f_0 \vb L_2)x_2)b_1^\secund = 
(x_0^mb_1(x_0,0) - (f_2 \vb L_2)x_2)a^\secund \, .
\] 
Since $a^\secund$ and $b_1^\secund$ are coprime it follows that there exists 
$g \in k[x_0,x_2]$ such that$\, :$ 
\begin{gather*}
x_0^ma(x_0,0) + (f_0 \vb L_2)x_2 = ga^\secund \, ,\\ 
x_0^mb_1(x_0,0) - (f_2 \vb L_2)x_2 = gb_1^\secund \, .
\end{gather*}
Restricting these two relations to the line $L_1$ of equations $x_2 = x_3 = 0$ 
and then dividing the second relation by the first one one gets 
$\frac{b_1(x_0,0)}{a(x_0,0)} = \frac{b_1^\secund(x_0,0)}{a^\secund(x_0,0)}$. 

\vskip2mm

\noindent
{\bf Claim 3.}\quad \emph{If} $x_1 \nmid b$ \emph{and} 
$\text{Im}\, {\widetilde \psi} = x_1\sco_{L_1}(-l-3)$ \emph{then} $x_2 \nmid 
b^\secund$. 

\vskip2mm

\noindent
\emph{Indeed}, if $\text{Im}\, {\widetilde \psi} = x_1\sco_{L_1}(-l-3)$ 
then, for $m >> 0$, 
there exists $f \in \tH^0(\sci_{X \cup X^\secund}(l+3+m))$ such that 
${\widetilde \psi}(f) = x_0^mx_1 \in \tH^0(\sco_{L_1}(m+1))$. Since 
$f \in \tH^0(\sci_{X \cup L_2}(l+3+m))$ and 
${\widetilde \psi}(f) = \overline{\psi}(f)$ it follows that$\, :$ 
\[
f = x_0^mx_1F_2 + f_0x_2x_3 + f_1x_3^2 + f_2x_1x_2^2\, . 
\]
On the other hand, $\phi^\secund(f) = 0$ hence$\, :$ 
\[
-x_0^mb(x_0,0)x_2a^\secund + (f_0 \vb L_2)x_2b^\secund + 
(f_2 \vb L_2)x_2^2a^\secund = 0   
\] 
from which one deduces that$\, :$
\[
(f_0 \vb L_2)b^\secund = (x_0^mb(x_0,0) - (f_2 \vb L_2)x_2)a^\secund \, .
\]
Restricting the last relation to the line $L_1$ of equations $x_2 = x_3 = 0$ 
one gets$\, :$ 
\[
x_0^mb(x_0,0)a^\secund(x_0,0) = f_0(x_0,0,0,0)b^\secund(x_0,0)\, .
\]
Since $x_1 \nmid b$ one has $b(x_0,0) \neq 0$.  
If one would have $x_2 \mid b^\secund$, which 
is equivalent to $b^\secund(x_0,0) = 0$, it would follow that $a^\secund(x_0,0) 
= 0$, i.e., that $x_2 \mid a^\secund$ which would \emph{contradict} the fact 
that $a^\secund$ and $b^\secund$ are coprime. It thus remains that $x_2 \nmid 
b^\secund$. 

\vskip2mm

(a) Let us put$\, :$ 
\[
\Phi := F_2 + b_1(x_0,0)x_1x_2 - a(x_0,0)x_3 + 
{\textstyle \frac{a(x_0,0)}{a^\secund(x_0,0)}}F_2^\secund \, . 
\]  
One has$\, :$ 
\begin{gather*}
F_2 + b_1(x_0,0)x_1x_2 - a(x_0,0)x_3 \in (x_1x_3,\, x_1^2x_2) \subset 
I(L_1 \cup L_2^{(1)})\, ,\\
b_1(x_0,0)x_1x_2 - a(x_0,0)x_3 + 
{\textstyle \frac{a(x_0,0)}{a^\secund(x_0,0)}}F_2^\secund =\\
= {\textstyle \frac{a(x_0,0)}{a^\secund(x_0,0)}}(b_1^\secund(x_0,0)x_1x_2 - 
a^\secund(x_0,0)x_3 + F_2^\secund ) \in (x_2x_3,\, x_1x_2^2) \subset 
I(L_1^{(1)} \cup L_2)\, . 
\end{gather*}
It follows that $\Phi \in I(X \cup X^\secund)$ and that 
${\widetilde \psi}(\Phi) = 
\eta(\Phi) = \eta(F_2) = 1 \in \tH^0(\sco_{L_1})$. 
One deduces the exactness of the sequence$\, :$ 
\[
0 \lra I(L_1^{(1)} \cup X^\secund) \lra I(X \cup X^\secund) 
\xra{\displaystyle \tH^0_\ast({\widetilde \psi})} S(L_1)(-l-2) \lra 0\, . 
\]
This implies that $I(X \cup X^\secund)$ is generated by the elements from the 
statement. Using this exact sequence one can also get a graded free resolution 
of $I(X \cup X^\secund)$. 

Notice, also, that, as a consequence of the above relations$\, :$ 
\[
x_1\left(b_1(x_0,0)x_1x_2 - a(x_0,0)x_3 + 
{\textstyle \frac{a(x_0,0)}{a^\secund(x_0,0)}}F_2^\secund \right) \in 
(x_1x_2x_3,\, x_1^2x_2^2) \subset I(L_1^{(1)} \cup L_2^{(1)}) 
\]  
hence $x_1F_2 \equiv x_1\Phi \pmod{I(L_1^{(1)} \cup L_2^{(1)})}$. 

(b) It follows, from Claim 1(i) and Claim 2, that 
$\text{Im}\, {\widetilde \psi} = x_1\sco_{L_1}(-l-3)$
hence $\widetilde \psi$ factorizes as$\, :$ 
\[
\sci_{X \cup X^\secund} 
\overset{\displaystyle {\widetilde \psi}^{\, \prim}}{\lra} 
\sco_{L_1}(-l-3) \overset{\displaystyle x_1}{\lra} \sco_{L_1}(-l-2)
\] 
with ${\widetilde \psi}^{\, \prim}$ an epimorphism. Since$\, :$ 
\[
x_1F_2 = -b_1x_1^2x_2 + ax_1x_3 \in I(X \cup L_2^{(1)}) \subseteq 
I(X \cup X^\secund) 
\] 
and since ${\widetilde \psi}(x_1F_2) = \eta(x_1F_2) = x_1 \in 
\tH^0(\sco_{L_1}(1))$ it follows that ${\widetilde \psi}^{\, \prim}(x_1F_2) = 1 
\in \tH^0(\sco_{L_1})$ hence the sequence$\, :$ 
\[
0 \lra I(L_1^{(1)} \cup X^\secund) \lra I(X \cup X^\secund) 
\xra{\displaystyle \tH^0_\ast({\widetilde \psi}^{\, \prim})} 
S(L_1)(-l-3) \lra 0\, . 
\] 
is exact. One deduces that $I(X \cup X^\secund)$ is generated by the elements 
from the statement. One can also get, using this exact sequence, a graded free 
resolution of $I(X \cup X^\secund)$. 

(c) The argument for case (b) works verbatim for case (c), too. 

(d) It follows, from Claim 1(ii) and Claim 3, that 
$\text{Im}\, {\widetilde \psi} = x_1^2\sco_{L_1}(-l-4)$ 
hence $\widetilde \psi$ factorizes as$\, :$ 
\[
\sci_{X \cup X^\secund} 
\overset{\displaystyle {\widetilde \psi}^{\, \secund}}{\lra} 
\sco_{L_1}(-l-4) \overset{\displaystyle x_1^2}{\lra} \sco_{L_1}(-l-2)
\] 
with ${\widetilde \psi}^{\, \secund}$ an epimorphism. Since  
$x_1^2F_2 \in I(X)I(L_2^{(1)}) \subseteq I(X \cup X^\secund)$ 
and since ${\widetilde \psi}(x_1^2F_2) = \eta(x_1^2F_2) = x_1^2 \in 
\tH^0(\sco_{L_1}(2))$ it follows that ${\widetilde \psi}^{\, \secund}(x_1^2F_2) = 
1 \in \tH^0(\sco_{L_1})$ hence the sequence$\, :$ 
\[
0 \lra I(L_1^{(1)} \cup X^\secund) \lra I(X \cup X^\secund) 
\xra{\displaystyle \tH^0_\ast({\widetilde \psi}^{\, \secund})} 
S(L_1)(-l-4) \lra 0\, . 
\] 
is exact. One deduces that $I(X \cup X^\secund)$ is generated by the elements 
from the statement. One can also get, using this exact sequence, a graded free 
resolution of $I(X \cup X^\secund)$. 

(e) Put $\Psi := x_1F_2 + b(x_0,0)x_1x_2 + 
\frac{b(x_0,0)}{b^\secund(x_0,0)}x_2F_2^\secund$. One has$\, :$ 
\begin{gather*}
x_1F_2 + b(x_0,0)x_1x_2 \in (x_1x_3,\, x_1^2x_2) \subset 
I(L_1 \cup L_2^{(1)})\, ,\\ 
b(x_0,0)x_1x_2 + 
{\textstyle \frac{b(x_0,0)}{b^\secund(x_0,0)}}x_2F_2^\secund =\\
= {\textstyle \frac{b(x_0,0)}{b^\secund(x_0,0)}}(b^\secund(x_0,0)x_1x_2 + 
x_2F_2^\secund) \in (x_2x_3,\, x_1x_2^2) \subset 
I(L_1^{(1)} \cup L_2)\, . 
\end{gather*} 
It follows that $\Psi \in I(X \cup X^\secund)$ and that 
${\widetilde \psi}(\Psi) = 
\eta(\Psi) = \eta(x_1F_2) = x_1 \in \tH^0(\sco_{L_1}(1))$. Taking into account 
Claim 1(ii), one deduces that $\widetilde \psi$ factorizes as$\, :$ 
\[
\sci_{X \cup X^\secund} 
\overset{\displaystyle {\widetilde \psi}^{\, \prim}}{\lra} 
\sco_{L_1}(-l-3) \overset{\displaystyle x_1}{\lra} \sco_{L_1}(-l-2)
\] 
with ${\widetilde \psi}^{\, \prim}$ an epimorphism and that the 
sequence$\, :$     
\[
0 \lra I(L_1^{(1)} \cup X^\secund) \lra I(X \cup X^\secund) 
\xra{\displaystyle \tH^0_\ast({\widetilde \psi}^{\, \prim})} 
S(L_1)(-l-3) \lra 0\, . 
\] 
is exact. It follows that $I(X \cup X^\secund)$ is generated by the elements 
from the statement. One can also get, using this exact sequence, a graded free 
resolution of $I(X \cup X^\secund)$. 

Notice, also, that, as a consequence of the above relations$\, :$ 
\[
x_1\left(b(x_0,0)x_1x_2 + 
{\textstyle \frac{b(x_0,0)}{b^\secund(x_0,0)}}x_2F_2^\secund \right) \in 
(x_1x_2x_3,\, x_1^2x_2^2) \subset I(L_1^{(1)} \cup L_2^{(1)}) 
\]  
hence $x_1^2F_2 \equiv x_1\Psi \pmod{I(L_1^{(1)} \cup L_2^{(1)})}$. 
\end{proof}

\begin{cor}\label{C:genixcupxsecund} 
Under the hypothesis of Prop.~\ref{P:genixcupxsecund}, $X \cup X^\secund$ is 
locally complete intersection in the cases \emph{(a)} and \emph{(e)} and it 
is not locally complete intersection in the other cases. 
\end{cor}

\begin{proof}
We cut $X \cup X^\secund$ with the plane $H \subset \piii$ of equation 
$x_1 - x_2 = 0$. $H \cap (X \cup X^\secund)$ is a 0-dimensional subscheme of 
length 4 of $H$, concentrated at $P_0$.   
The composite map $k[x_0,x_1,x_3] \hookrightarrow S 
\twoheadrightarrow S(H)$ is bijective and allows one to identify $S(H)$ with 
$k[x_0,x_1,x_3]$. Under this identification, the restrictions to $H$ of the 
generators of $I(L_1^{(1)} \cup L_2^{(1)})$ are$\, :$ 
\[
x_3^2 \vb H = x_3^2\, ,\  x_1x_2x_3 \vb H = x_1^2x_3\, ,\  
x_1^2x_2^2 \vb H = x_1^4\, . 
\]

(a) Using the notation from the proof of Prop.~\ref{P:genixcupxsecund}(a), one 
has $\Phi \vb H \equiv F_2 \vb H \equiv a(x_0,0)x_3 \pmod{(x_1,\, x_3)^2}$. It 
follows that the 1-dimensional scheme $H \cap \{\Phi = 0\}$ is nonsingular at 
$P_0$. Since $H \cap (X \cup X^\secund)$ is a subscheme of 
$H \cap \{\Phi = 0\}$ one deduces that it is locally complete intersection in 
$H$ hence $X \cup X^\secund$ is locally complete intersection at $P_0$. 

(b) One has $x_1F_2 \vb H \equiv a(x_0,0)x_1x_3 \pmod{(x_1,\, x_3)^3}$ and 
$x_2F_2^\secund \vb H \equiv a^\secund(x_0,0)x_1x_3 \pmod{(x_1,\, x_3)^3}$. 
Moreover, as we saw above, $x_3^2 \vb H = x_3^2$ and $x_1x_2x_3 \vb H$ and 
$x_1^2x_2^2 \vb H$ belong to $(x_1,\, x_3)^3$. 
Using \cite[Chap.~I,~Ex.~5.4(a)]{hag} (in that exercise one has 
equality if and only if $Y$ and $Z$ have no common tangent direction at 
$P$) one deduces that $H \cap (X \cup X^\secund)$ is not locally complete 
intersection in $H$.  

(c) $x_1F_2 \vb H \equiv a(x_0,0)x_1x_3 \pmod{(x_1,\, x_3)^3}$ and 
$x_2^2F_2^\secund \vb H \in (x_1,\, x_3)^3$. 

(d) $x_1^2F_2 \in (x_1,\, x_3)^3$ and 
$x_2F_2^\secund \vb H \equiv a^\secund(x_0,0)x_1x_3 \pmod{(x_1,\, x_3)^3}$. 

(e) Using the notation from the proof of Prop.~\ref{P:genixcupxsecund}(e),  
one has$\, :$ 
\[
\Psi \vb H \equiv -b(x_0,0)\left(x_1^2 - 
\left({\textstyle \frac{a(x_0,0)}{b(x_0,0)}} + 
{\textstyle \frac{a^\secund (x_0,0)}{b^\secund (x_0,0)}}\right)x_1x_3\right) 
\pmod{(x_1,\, x_3)^3}\, .
\]    
If $Y \subset H$ (resp., $Z \subset H$) is the effective divisor of equation 
$(\Psi \vb H) = 0$ (resp., $x_3^2 = 0$) it follows, from the above mentioned 
exercise from Hartshorne's book, that the intersection multiplicity of 
$Y$ and $Z$ at $P_0$ is 4. Since $H \cap (X \cup X^\secund)$ is a subscheme of 
$Y \cap Z$ of degree 4 concentrated at $P_0$  
one deduces that $H \cap (X \cup X^\secund)$ is locally complete 
intersection in $H$ hence $X \cup X^\secund$ is locally complete intersection 
at $P_0$. 
\end{proof}

\begin{remark}\label{R:genixcupxsecund} 
Under the hypothesis of Prop.~\ref{P:genixcupxsecund}, recall the exact 
sequences$\, :$ 
\[
0 \lra \sci_{X \cup L_2} \lra \sci_{L_1 \cup L_2} 
\overset{\displaystyle \phi}{\lra} \sco_{L_1}(l)\, ,\  \  
0 \lra \sci_{L_1 \cup X^\secund} \lra \sci_{L_1 \cup L_2} 
\overset{\displaystyle \phi^\secund}{\lra} \sco_{L_2}(l^\secund)\, , 
\]
where $\phi$ and $\phi^\secund$ are the composite morphisms$\, :$ 
\[
\sci_{L_1 \cup L_2} \lra \sci_{L_1} \overset{\displaystyle \pi}{\lra} 
\sco_{L_1}(l)\, ,\  
\sci_{L_1 \cup L_2} \lra \sci_{L_2} \overset{\displaystyle \pi^\secund}{\lra} 
\sco_{L_2}(l^\secund)\, .
\]
One deduces an exact sequence$\, :$ 
\[
0 \lra \sci_{X \cup X^\secund} \lra \sci_{L_1 \cup L_2} 
\xra{\displaystyle (\phi \, ,\, \phi^\secund )} \sco_{L_1}(l) \oplus 
\sco_{L_2}(l^\secund )\, .
\]
Moreover, since $I(L_1 \cup L_2) = (x_3,\, x_1x_2)$ and since$\, :$ 
\[
\phi(x_3) = b\, ,\  \phi(x_1x_2) = x_1a\, ,\  
\phi^\secund(x_3) = b^\secund \, ,\  \phi^\secund(x_1x_2) = x_2a^\secund \, , 
\]
it follows that if $x_1 \mid b$ (resp., $x_2 \mid b^\secund$) then $\phi$ 
(resp., $\phi^\secund$) factorizes as$\, :$ 
\[
\sci_{L_1 \cup L_2} \overset{\displaystyle \phi_1}{\lra} \sco_{L_1}(l-1) 
\overset{\displaystyle x_1}{\lra} \sco_{L_1}(l)\  \  
(\text{resp.,}\  
\sci_{L_1 \cup L_2} \overset{\displaystyle \phi_1^\secund}{\lra} 
\sco_{L_2}(l^\secund -1) 
\overset{\displaystyle x_2}{\lra} \sco_{L_2}(l^\secund)\, . 
\] 

(a) If $x_1 \mid b$ (i.e., $b = x_1b_1$), $x_2 \mid b^\secund$ (i.e., 
$b^\secund = x_2b_1^\secund$) and $\frac{b_1(x_0,0)}{a(x_0,0)} = 
\frac{b_1^\secund(x_0,0)}{a^\secund(x_0,0)}$ then one has an exact sequence$\, :$ 
\[
0 \ra \sci_{X \cup X^\secund} \lra \sci_{L_1 \cup L_2} 
\xra{\begin{pmatrix} \phi_1\\ \phi_1^\secund \end{pmatrix}} 
\begin{matrix} \sco_{L_1}(l-1)\\ \oplus\\ \sco_{L_2}(l^\secund - 1) \end{matrix} 
\xra{\displaystyle (a^\secund(x_0,0)\, ,\, -a(x_0,0))} 
\sco_{\{P_0\}}(l + l^\secund ) \ra 0\, . 
\]

(b) If $x_1 \mid b$ (i.e., $b = x_1b_1$), $x_2 \mid b^\secund$ (i.e., 
$b^\secund = x_2b_1^\secund$) and $\frac{b_1(x_0,0)}{a(x_0,0)} \neq  
\frac{b_1^\secund(x_0,0)}{a^\secund(x_0,0)}$ then one has an exact sequence$\, :$ 
\[
0 \lra \sci_{X \cup X^\secund} \lra \sci_{L_1 \cup L_2} 
\xra{\begin{pmatrix} \phi_1\\ \phi_1^\secund \end{pmatrix}} 
\begin{matrix} \sco_{L_1}(l-1)\\ \oplus\\ \sco_{L_2}(l^\secund - 1) \end{matrix} 
\lra 0\, . 
\] 

(c) If $x_1 \mid b$ and $x_2 \nmid b^\secund$ then one has an exact 
sequence$\, :$ 
\[
0 \lra \sci_{X \cup X^\secund} \lra \sci_{L_1 \cup L_2} 
\xra{\begin{pmatrix} \phi_1\\ \phi^\secund \end{pmatrix}} 
\begin{matrix} \sco_{L_1}(l-1)\\ \oplus\\ \sco_{L_2}(l^\secund ) \end{matrix} 
\lra 0\, . 
\] 

(d) If $x_1 \nmid b$ and $x_2 \mid b^\secund$ then one has an exact 
sequence$\, :$ 
\[
0 \lra \sci_{X \cup X^\secund} \lra \sci_{L_1 \cup L_2} 
\xra{\begin{pmatrix} \phi\\ \phi_1^\secund \end{pmatrix}} 
\begin{matrix} \sco_{L_1}(l)\\ \oplus\\ \sco_{L_2}(l^\secund - 1) \end{matrix} 
\lra 0\, . 
\] 

(e) If $x_1 \nmid b$ and $x_2 \nmid b^\secund$ then one has an exact 
sequence$\, :$ 
\[
0 \ra \sci_{X \cup X^\secund} \lra \sci_{L_1 \cup L_2} 
\xra{\begin{pmatrix} \phi\\ \phi^\secund \end{pmatrix}} 
\begin{matrix} \sco_{L_1}(l)\\ \oplus\\ \sco_{L_2}(l^\secund ) \end{matrix} 
\xra{\displaystyle (b^\secund(x_0,0)\, ,\, -b(x_0,0))} 
\sco_{\{P_0\}}(l + l^\secund + 1) \ra 0\, . 
\] 

\vskip2mm 

\noindent
\emph{Indeed}, in case (a) let $H \subset \piii$ be the plane of equation 
$x_3 = 0$. $\Cok (\phi_1\, ,\, \phi_1^\secund )$ is concentrated in $P_0$ and 
coincides with the cokernel of the morphism$\, :$ 
\[
\sco_H(-1) \oplus \sco_H(-2) 
\xra{\begin{pmatrix} b_1 & a\\ b_1^\secund & a^\secund \end{pmatrix}} 
\sco_{L_1}(l-1) \oplus \sco_{L_2}(l^\secund -1)  
\]
hence with the cokernel of the morphism$\, :$ 
\[
\sco_H(-1) \oplus \sco_H(-2) \oplus \sco_H(l-2) \oplus \sco_H(l^\secund - 2)  
\xra{\begin{pmatrix} b_1 & a & x_2 & 0\\ 
b_1^\secund & a^\secund & 0 & x_1\end{pmatrix}} 
\sco_H(l-1) \oplus \sco_H(l^\secund -1)  
\]
The cokernel of the last morphism is annihilated by $x_1b_1$, $x_1a$, 
$x_2b_1^\secund$, $x_2a^\secund$. Since $b_1$ and $a$ (resp., $b_1^\secund$ and 
$a^\secund$) are coprime one deduces that $\Cok (\phi_1\, ,\, \phi_1^\secund )$ 
is annihilated by $x_1$ and $x_2$ hence it is an $\sco_{\{P_0\}}$-module. It 
follows that the sequence$\, :$ 
\[
\begin{matrix} \sco_{\{P_0\}}(-1)\\ \oplus\\ \sco_{\{P_0\}}(-2) \end{matrix}  
\xra{\begin{pmatrix} b_1(x_0,0) & a(x_0,0)\\ 
b_1^\secund(x_0,0) & a^\secund(x_0,0) \end{pmatrix}} 
\begin{matrix} \sco_{\{P_0\}}(l-1)\\ \oplus\\ \sco_{\{P_0\}}(l^\secund -1) 
\end{matrix}  
\lra \Cok (\phi_1\, ,\, \phi_1^\secund ) \lra 0
\]
is exact and assertion (a) follows. 

The assertions (b)-(e) can be proven similarly. 
\end{remark}

Recall, now, that the graded $S$-module $\tH^0_\ast(\sco_X)$ 
(resp., $\tH^0_\ast(\sco_{X^\secund})$) 
is generated by $1$ and by an element $e_1 \in \tH^0(\sco_X(-l))$ (resp., 
$e_1^\secund \in \tH^0(\sco_{X^\secund}(-l^\secund))$). 

\begin{prop}\label{P:reshoxcupxsecund}
Using the notation introduced at the beginning of this subsection, assume that 
$l \leq l^\secund$. Define $a_1$, $b_1$, $b_2$, $a_1^\secund$, $b_1^\secund$ and 
$b_2^\secund$ by the relations$\, :$ 
\begin{gather*}
a = a(x_0,0) + a_1x_1\, ,\  b = b(x_0,0) + b_1x_1\, ,\  
b_1 = b_1(x_0,0) + b_2x_1\, ,\\
a^\secund = a^\secund(x_0,0) + a_1^\secund x_2\, ,\  
b^\secund = b^\secund(x_0,0) + b_1^\secund x_2\, ,\  
b_1^\secund = b_1^\secund(x_0,0) + b_2^\secund x_2\, . 
\end{gather*}

\emph{(a)} If $x_1 \mid b$, $x_2 \mid b^\secund$ and 
$\frac{b_1(x_0,0)}{a(x_0,0)} = \frac{b_1^\secund(x_0,0)}{a^\secund(x_0,0)}$ 
then the graded $S$-module ${\fam0 H}^0_\ast(\sco_{X \cup X^\secund})$ 
admits a free resolution of the form$\, :$ 
\[
0 \ra \begin{matrix} S(-3)\\ \oplus\\ S(l-3)\\ \oplus\\ S(l^\secund - 4) 
\end{matrix} \overset{\displaystyle \delta_2}{\lra} 
\begin{matrix} S(-1)\\ \oplus\\ S(-2)\\ \oplus\\ 2S(l-2)\\ \oplus\\ 
2S(l^\secund - 3) \end{matrix} \overset{\displaystyle \delta_1}{\lra} 
\begin{matrix} S\\ \oplus\\ S(l-1)\\ \oplus\\ S(l^\secund - 2) 
\end{matrix} \overset{\displaystyle \delta_0}{\lra} 
{\fam0 H}^0_\ast(\sco_{X \cup X^\secund}) \ra 0  
\]
with $\delta_0 = (1\, ,\, x_1e_1 + 
\frac{a^\secund (x_0,0)}{a(x_0,0)}x_2e_1^\secund \, ,\, x_2^2e_1^\secund)$ 
and with $\delta_1$ and $\delta_2$ defined by the 
matrices$\, :$ 
\[
\begin{pmatrix} 
x_3 & x_1x_2 & 0 & 0 & 0 & 0\\
-b_1 & -a & x_2 & x_3 & 0 & 0\\
-b_2^\secund & -a_1^\secund & -{\textstyle \frac{a^\secund (x_0,0)}{a(x_0,0)}} & 
0 & x_1 & x_3
\end{pmatrix}\, ,\  
\begin{pmatrix}
-x_1x_2 & 0 & 0\\
x_3 & 0 & 0\\
-x_1b_1 & -x_3 & 0\\
a & x_2 & 0\\
-x_2b_2^\secund - {\textstyle \frac{a^\secund (x_0,0)}{a(x_0,0)}}b_1 & 
0 & -x_3\\
a_1^\secund & -{\textstyle \frac{a^\secund (x_0,0)}{a(x_0,0)}} & x_1
\end{pmatrix}\, .
\]

\emph{(b)} If $x_1 \mid b$, $x_2 \mid b^\secund$ and 
$\frac{b_1(x_0,0)}{a(x_0,0)} \neq \frac{b_1^\secund(x_0,0)}{a^\secund(x_0,0)}$ 
then the graded $S$-module ${\fam0 H}^0_\ast(\sco_{X \cup X^\secund})$ 
admits a free resolution of the form$\, :$ 
\[
0 \ra \begin{matrix} S(-3)\\ \oplus\\ S(l-3)\\ \oplus\\ S(l^\secund - 3) 
\end{matrix} \overset{\displaystyle \delta_2}{\lra} 
\begin{matrix} S(-1)\\ \oplus\\ S(-2)\\ \oplus\\ 2S(l-2)\\ \oplus\\ 
2S(l^\secund - 2) \end{matrix} \overset{\displaystyle \delta_1}{\lra} 
\begin{matrix} S\\ \oplus\\ S(l-1)\\ \oplus\\ S(l^\secund - 1) 
\end{matrix} \overset{\displaystyle \delta_0}{\lra} 
{\fam0 H}^0_\ast(\sco_{X \cup X^\secund}) \ra 0  
\]
with $\delta_0 = (1\, ,\, x_1e_1\, ,\, x_2e_1^\secund)$ and with $\delta_1$ and 
$\delta_2$ defined by the matrices$\, :$ 
\[
\begin{pmatrix}
x_3 & x_1x_2 & 0 & 0 & 0 & 0\\
-b_1 & -a & x_2 & x_3 & 0 & 0\\
-b_1^\secund & -a^\secund & 0 & 0 & x_1 & x_3 
\end{pmatrix}\, ,\  
\begin{pmatrix}
-x_1x_2 & 0 & 0\\
x_3 & 0 & 0\\
-x_1b_1 & -x_3 & 0\\
a & x_2 & 0\\
-x_2b_1^\secund & 0 & -x_3\\
a^\secund & 0 & x_1
\end{pmatrix}\, . 
\]

\emph{(c)} If $x_1 \mid b$ and $x_2 \nmid b^\secund$ 
then the graded $S$-module ${\fam0 H}^0_\ast(\sco_{X \cup X^\secund})$ 
admits a free resolution of the form$\, :$ 
\[
0 \ra \begin{matrix} S(-3)\\ \oplus\\ S(l-3)\\ \oplus\\ S(l^\secund - 2) 
\end{matrix} \overset{\displaystyle \delta_2}{\lra} 
\begin{matrix} S(-1)\\ \oplus\\ S(-2)\\ \oplus\\ 2S(l-2)\\ \oplus\\ 
2S(l^\secund - 1) \end{matrix} \overset{\displaystyle \delta_1}{\lra} 
\begin{matrix} S\\ \oplus\\ S(l-1)\\ \oplus\\ S(l^\secund ) 
\end{matrix} \overset{\displaystyle \delta_0}{\lra} 
{\fam0 H}^0_\ast(\sco_{X \cup X^\secund}) \ra 0  
\]
with $\delta_0 = (1\, ,\, x_1e_1\, ,\, e_1^\secund)$ and with $\delta_1$ and 
$\delta_2$ defined by the matrices$\, :$ 
\[
\begin{pmatrix}
x_3 & x_1x_2 & 0 & 0 & 0 & 0\\
-b_1 & -a & x_2 & x_3 & 0 & 0\\
-b^\secund & -x_2a^\secund & 0 & 0 & x_1 & x_3 
\end{pmatrix}\, ,\  
\begin{pmatrix}
-x_1x_2 & 0 & 0\\
x_3 & 0 & 0\\
-x_1b_1 & -x_3 & 0\\
a & x_2 & 0\\
-x_2b^\secund & 0 & -x_3\\
x_2a^\secund & 0 & x_1
\end{pmatrix}\, . 
\]

\emph{(d)} If $x_1 \nmid b$ then the graded $S$-module 
${\fam0 H}^0_\ast(\sco_{X \cup X^\secund})$ admits a free resolution of the 
form$\, :$ 
\[
0 \ra \begin{matrix} S(-3)\\ \oplus\\ S(l-2)\\ \oplus\\ S(l^\secund - 3) 
\end{matrix} \overset{\displaystyle \delta_2}{\lra} 
\begin{matrix} S(-1)\\ \oplus\\ S(-2)\\ \oplus\\ 2S(l-1)\\ \oplus\\ 
2S(l^\secund - 2) \end{matrix} \overset{\displaystyle \delta_1}{\lra} 
\begin{matrix} S\\ \oplus\\ S(l)\\ \oplus\\ S(l^\secund - 1) 
\end{matrix} \overset{\displaystyle \delta_0}{\lra} 
{\fam0 H}^0_\ast(\sco_{X \cup X^\secund}) \ra 0  
\]
with $\delta_0 = 
(1\, ,\, e_1 + \frac{b^\secund (x_0,0)}{b(x_0,0)}e_1^\secund \, ,\, x_2e_1^\secund)$ 
and with $\delta_1$ and $\delta_2$ defined by the matrices$\, :$ 
\[
\begin{pmatrix} 
x_3 & x_1x_2 & 0 & 0 & 0 & 0\\
-b & -x_1a & x_2 & x_3 & 0 & 0\\
-b_1^\secund & -a^\secund & -{\textstyle \frac{b^\secund (x_0,0)}{b(x_0,0)}} & 
0 & x_1 & x_3
\end{pmatrix}\, ,\  
\begin{pmatrix}
-x_1x_2 & 0 & 0\\
x_3 & 0 & 0\\
-x_1b & -x_3 & 0\\
x_1a & x_2 & 0\\
-x_2b_1^\secund - {\textstyle \frac{b^\secund (x_0,0)}{b(x_0,0)}}b & 
0 & -x_3\\
a^\secund & -{\textstyle \frac{b^\secund (x_0,0)}{b(x_0,0)}} & x_1
\end{pmatrix}\, .
\]
\end{prop} 

\begin{proof}
By Lemma~\ref{L:ycupt} one has an exact sequence$\, :$ 
\[
0 \lra \{0\} \times \text{Im}\, \phi^\secund \lra \sco_{X \cup X^\secund} \lra 
\sco_{X \cup L_2} \lra 0 
\]
where $\phi^\secund$ is the composite morphism$\, :$ 
\[
\sci_{X \cup L_2} \lra \sci_{L_2} \overset{\displaystyle \pi^\secund}{\lra} 
\sco_{L_2}(l^\secund)\, . 
\]
According to Prop.~\ref{P:genixcupl} one has$\, :$ 
\[
I(X \cup L_2) = 
\begin{cases}
(F_2,\, x_2x_3,\, x_3^2,\, x_1x_2^2)\, , & \text{if $x_1 \mid b\, ;$}\\
(x_1F_2,\, x_2x_3,\, x_3^2,\, x_1x_2^2)\, , & \text{if $x_1 \nmid b\, .$}
\end{cases}
\]
Taking into account the definition of $\pi^\secund$ one gets$\, :$ 
\begin{gather*}
\phi^\secund(F_2) = -b_1(x_0,0)x_2a^\secund + a(x_0,0)b^\secund =\\
= a(x_0,0)b^\secund(x_0,0) + \left(-b_1(x_0,0)a^\secund(x_0,0) + 
a(x_0,0)b_1^\secund(x_0,0)\right)x_2 +\\
+ \left(-b_1(x_0,0)a_1^\secund + a(x_0,0)b_2^\secund \right)x_2^2 \, , \  
\text{if}\  x_1 \mid b\, ,\\
\phi^\secund(x_1F_2) = -b(x_0,0)x_2a^\secund = -b(x_0,0)a^\secund(x_0,0)x_2 - 
b(x_0,0)a_1^\secund x_2^2 \, ,\  \text{if}\  x_1 \nmid b\, ,\\
\phi^\secund(x_2x_3) = x_2b^\secund \, ,\  \phi^\secund(x_3^2) = 0\, ,\  
\phi^\secund(x_1x_2^2) = x_2^2a^\secund\, .
\end{gather*}
We shall identify the element $e_1$ of $\tH^0_\ast(\sco_X)$ (resp., 
$e_1^\secund$ of $\tH^0_\ast(\sco_{X^\secund})$) to the element 
$(e_1,0)$ (resp., $(0,e_1^\secund)$) of 
$\tH^0_\ast(\sco_X) \oplus \tH^0_\ast(\sco_{X^\secund})$. 
Using the exact sequence$\, :$ 
\[
0 \lra \tH^0_\ast(\sco_{X \cup X^\secund}) \lra \tH^0_\ast(\sco_X) \oplus 
\tH^0_\ast(\sco_{X^\secund}) \lra \tH^0_\ast(\sco_{X \cap X^\secund})
\]   
and taking into account that $X \cap X^\secund$ is concentrated at $P_0$ one 
deduces that if an element $\xi$ of $\tH^0_\ast(\sco_X) \oplus 
\tH^0_\ast(\sco_{X^\secund})$ has the property that $x_0^m\xi$ belongs to 
$\tH^0_\ast(\sco_{X \cup X^\secund})$ for some $m \geq 0$ then $\xi$ itself 
belongs to $\tH^0_\ast(\sco_{X \cup X^\secund})$. 

Finally, let us recall that Prop.~\ref{P:reshoxcupl} provides a graded free 
resolution for $\tH^0_\ast(\sco_{X \cup L_2})$. 

\vskip2mm 

(a) In this case $\text{Im}\, \phi^\secund = x_2^2\sco_{L_2}(l^\secund - 2)$.  
By Prop.~\ref{P:reshoxcupl}(a), the graded $S$-module 
$\tH^0_\ast(\sco_{X \cup L_2})$ is generated by $1$ and by $x_1e_1 \in 
\tH^0(\sco_{X \cup L_2}(-l+1))$. Since $l \leq l^\secund$ it follows that   
$\tH^1((\text{Im}\, \phi^\secund)(-l+1)) \simeq 
\tH^1(\sco_{L_2}(l^\secund -l -1)) = 0$. One deduces an exact sequence$\, :$ 
\begin{equation}\label{E:hoxcupxsecunda}
0 \lra S(L_2)(l^\secund -2) \lra \tH^0_\ast(\sco_{X \cup X^\secund}) \lra 
\tH^0_\ast(\sco_{X \cup L_2}) \lra 0 
\end{equation}
where the left morphism maps $1 \in S(L_2)$ to 
$x_2^2e_1^\secund \in \tH^0(\sco_{X \cup X^\secund}(-l^\secund + 2))$ (identified, 
as we assumed above, to the element $(0,x_2^2e_1^\secund)$ of $\tH^0_\ast(\sco_X) 
\oplus \tH^0_\ast(\sco_{X^\secund})$). 

\vskip2mm

\noindent
{\bf Claim 1.}\quad $x_1^2e_1$ \emph{belongs to} 
$\tH^0_\ast(\sco_{X \cup X^\secund})$. 

\vskip2mm

\noindent
\emph{Indeed}, since $x_1e_1$ belongs to $\tH^0_\ast(\sco_{X \cup L_2})$, there 
exists an element of $\tH^0_\ast(\sco_{X \cup X^\secund})$ of the form 
$x_1e_1 + fe_1^\secund$. Since $x_1$ annihilates $e_1^\secund$ the claim 
follows. 

\vskip2mm

\noindent
{\bf Claim 2.}\quad $x_1e_1 + \frac{a^\secund (x_0,0)}{a(x_0,0)}x_2e_1^\secund$ 
\emph{belongs to} $\tH^0_\ast(\sco_{X \cup X^\secund})$. 

\vskip2mm

\noindent
\emph{Indeed}, one has in $\tH^0_\ast(\sco_X) \oplus 
\tH^0_\ast(\sco_{X^\secund})$$\, :$ 
\[
x_1x_2 \cdot 1 = ax_1e_1 + a^\secund x_2e_1^\secund = 
a(x_0,0)x_1e_1 + a_1x_1^2e_1 + a^\secund(x_0,0)x_2e_1^\secund +  
a_1^\secund x_2^2e_1^\secund \, . 
\]
Since $x_1^2e_1$ and $x_2^2e_1^\secund$ belong to 
$\tH^0_\ast(\sco_{X \cup X^\secund})$ it follows that $a(x_0,0)x_1e_1 + 
a^\secund(x_0,0)x_2e_1^\secund$ belongs to $\tH^0_\ast(\sco_{X \cup X^\secund})$ 
hence$\, :$ 
\[
a(x_0,0)\left(x_1e_1 + 
{\textstyle \frac{a^\secund (x_0,0)}{a(x_0,0)}}x_2e_1^\secund \right) \in 
\tH^0_\ast(\sco_{X \cup X^\secund})\, .
\]
The claim follows, now, using an observation from the beginning of the proof. 

\vskip2mm

\noindent
The assertion from case (a) of the statement is, now, a consequence of the 
exact sequence \eqref{E:hoxcupxsecunda}, of Prop.~\ref{P:reshoxcupl}(a) and 
of the following relations in 
$\tH^0_\ast(\sco_{X \cup X^\secund}) \subset \tH^0_\ast(\sco_X) \oplus 
\tH^0_\ast(\sco_{X^\secund})$$\, :$ 
\begin{gather*}
x_3 \cdot 1 = be_1 + b^\secund e_1^\secund = 
b_1x_1e_1 + b_1^\secund x_1e_1^\secund = 
b_1\left(x_1e_1 + 
{\textstyle \frac{a^\secund (x_0,0)}{a(x_0,0)}}x_2e_1^\secund \right) - 
b_1{\textstyle \frac{a^\secund (x_0,0)}{a(x_0,0)}}x_2e_1^\secund + 
b_1^\secund x_2e_1^\secund\\ 
= b_1\left(x_1e_1 + 
{\textstyle \frac{a^\secund (x_0,0)}{a(x_0,0)}}x_2e_1^\secund \right) - 
b_1(x_0,0){\textstyle \frac{a^\secund (x_0,0)}{a(x_0,0)}}x_2e_1^\secund + 
b_1^\secund x_2e_1^\secund =\\ 
b_1\left(x_1e_1 + 
{\textstyle \frac{a^\secund (x_0,0)}{a(x_0,0)}}x_2e_1^\secund \right) - 
b_1^\secund(x_0,0)x_2e_1^\secund + b_1^\secund x_2e_1^\secund =  
b_1\left(x_1e_1 + 
{\textstyle \frac{a^\secund (x_0,0)}{a(x_0,0)}}x_2e_1^\secund \right) + 
b_2^\secund x_2^2e_1^\secund \, ;\\ 
x_1x_2 \cdot 1 = ax_1e_1 + a^\secund x_2e_1^\secund = 
a\left(x_1e_1 + 
{\textstyle \frac{a^\secund (x_0,0)}{a(x_0,0)}}x_2e_1^\secund \right) - 
a{\textstyle \frac{a^\secund (x_0,0)}{a(x_0,0)}}x_2e_1^\secund + 
a^\secund x_2e_1^\secund =\\
a\left(x_1e_1 + 
{\textstyle \frac{a^\secund (x_0,0)}{a(x_0,0)}}x_2e_1^\secund \right) - 
a(x_0,0){\textstyle \frac{a^\secund (x_0,0)}{a(x_0,0)}}x_2e_1^\secund + 
a^\secund x_2e_1^\secund 
= a\left(x_1e_1 + 
{\textstyle \frac{a^\secund (x_0,0)}{a(x_0,0)}}x_2e_1^\secund \right) 
+ a_1^\secund x_2^2e_1^\secund \, ;\\
x_2 \cdot \left(x_1e_1 + 
{\textstyle \frac{a^\secund (x_0,0)}{a(x_0,0)}}x_2e_1^\secund \right) = 
{\textstyle \frac{a^\secund (x_0,0)}{a(x_0,0)}}x_2^2e_1^\secund \, ;\ \  
x_3 \cdot \left(x_1e_1 + 
{\textstyle \frac{a^\secund (x_0,0)}{a(x_0,0)}}x_2e_1^\secund \right) = 0\, .   
\end{gather*} 

(b) In this case $\text{Im}\, \phi^\secund = x_2\sco_{L_2}(l^\secund - 1)$. 
By Prop.~\ref{P:reshoxcupl}(a) the graded $S$-module 
$\tH^0_\ast(\sco_{X \cup L_2})$ is generated by 
$1 \in \tH^0(\sco_{X \cup L_2})$ and by $x_1e_1 \in \tH^0(\sco_{X \cup L_2}(-l+1))$. 
One deduces, as in case (a), the existence of an exact sequence$\, :$ 
\begin{equation}\label{E:hoxcupxsecundb}
0 \lra S(L_2)(l^\secund -1) \lra \tH^0_\ast(\sco_{X \cup X^\secund}) \lra 
\tH^0_\ast(\sco_{X \cup L_2}) \lra 0 
\end{equation}
where the left morphism maps $1 \in S(L_2)$ to $x_2e_1^\secund \in 
\tH^0(\sco_{X \cup X^\secund}(-l^\secund + 1))$. One shows, as in Claim 1, 
that $x_1^2e_1$ belongs to $\tH^0_\ast(\sco_{X \cup X^\secund})$ 
and then, as in Claim 2, that $x_1e_1$ belongs to 
$\tH^0_\ast(\sco_{X \cup X^\secund})$. The assertion from case (b) of the statement 
is, now, a consequence of the exact sequence \eqref{E:hoxcupxsecundb} and 
of Prop.~\ref{P:reshoxcupl}(a).  

(c) In this case $\text{Im}\, \phi^\secund = \sco_{L_2}(l^\secund )$. One deduces, 
as in the previous cases, the existence of an exact sequence$\, :$ 
\begin{equation}\label{E:hoxcupxsecundc}
0 \lra S(L_2)(l^\secund ) \lra \tH^0_\ast(\sco_{X \cup X^\secund}) \lra 
\tH^0_\ast(\sco_{X \cup L_2}) \lra 0\, , 
\end{equation} 
where the left morphism maps $1 \in S(L_2)$ to $e_1^\secund \in 
\tH^0(\sco_{X \cup X^\secund}(-l^\secund ))$, and the fact that $x_1e_1$ belongs to 
$\tH^0_\ast(\sco_{X \cup X^\secund})$. The assertion from case (c) of the statement 
is, now, a consequence of the exact sequence \eqref{E:hoxcupxsecundc} and 
of Prop.~\ref{P:reshoxcupl}(a).  

(d) In this case $\text{Im}\, \phi^\secund = x_2\sco_{L_2}(l^\secund - 1)$. 
By Prop.~\ref{P:reshoxcupl}(b) the graded $S$-module 
$\tH^0_\ast(\sco_{X \cup L_2})$ is generated by 
$1 \in \tH^0(\sco_{X \cup L_2})$ and by $e_1 \in \tH^0(\sco_{X \cup L_2}(-l))$. 
One deduces, as in case (a), the existence of an exact sequence$\, :$ 
\begin{equation}\label{E:hoxcupxsecundd}
0 \lra S(L_2)(l^\secund -1) \lra \tH^0_\ast(\sco_{X \cup X^\secund}) \lra 
\tH^0_\ast(\sco_{X \cup L_2}) \lra 0 
\end{equation}
where the left morphism maps $1 \in S(L_2)$ to $x_2e_1^\secund \in 
\tH^0(\sco_{X \cup X^\secund}(-l^\secund + 1))$. One shows, as in Claim 1, 
that $x_1e_1$ belongs to $\tH^0_\ast(\sco_{X \cup X^\secund})$. 

\vskip2mm

\noindent
{\bf Claim 3.}\quad $e_1 + \frac{b^\secund (x_0,0)}{b(x_0,0)}e_1^\secund$ 
\emph{belongs to} $\tH^0_\ast(\sco_{X \cup X^\secund})$. 

\vskip2mm

\noindent
\emph{Indeed}, one has in $\tH^0_\ast(\sco_X) \oplus 
\tH^0_\ast(\sco_{X^\secund})$$\, :$ 
\[
x_3 \cdot 1 = be_1 + b^\secund e_1^\secund = b(x_0,0)e_1 + b_1x_1e_1 + 
b^\secund(x_0,0)e_1^\secund + b_1^\secund x_2e_1^\secund \, . 
\] 
Since $x_1e_1$ and $x_2e_1^\secund$ belong to $\tH^0_\ast(\sco_{X \cup X^\secund})$ 
it follows that $b(x_0,0)e_1 + b^\secund(x_0,0)e_1^\secund$ belongs to 
$\tH^0_\ast(\sco_{X \cup X^\secund})$ hence$\, :$ 
\[
b(x_0,0)\left(e_1 + 
{\textstyle \frac{b^\secund (x_0,0)}{b(x_0,0)}}e_1^\secund \right) \in 
\tH^0_\ast(\sco_{X \cup X^\secund})\, .
\]
The claim follows, now, using an observation from the beginning of the proof. 

\vskip2mm 

\noindent
The assertion from case (d) of the statement is, now, a consequence of the 
exact sequence \eqref{E:hoxcupxsecundd}, of Prop.~\ref{P:reshoxcupl}(b) and 
of the following relations in 
$\tH^0_\ast(\sco_{X \cup X^\secund}) \subset \tH^0_\ast(\sco_X) \oplus 
\tH^0_\ast(\sco_{X^\secund})$$\, :$  
\begin{gather*}
x_3 \cdot 1 = be_1 + b^\secund e_1^\secund = 
b\left(e_1 + 
{\textstyle \frac{b^\secund (x_0,0)}{b(x_0,0)}}e_1^\secund \right) - 
b{\textstyle \frac{b^\secund (x_0,0)}{b(x_0,0)}}e_1^\secund + 
b^\secund e_1^\secund =\\
b\left(e_1 + 
{\textstyle \frac{b^\secund (x_0,0)}{b(x_0,0)}}e_1^\secund \right) - 
b(x_0,0){\textstyle \frac{b^\secund (x_0,0)}{b(x_0,0)}}e_1^\secund + 
b^\secund e_1^\secund = b\left(e_1 + 
{\textstyle \frac{b^\secund (x_0,0)}{b(x_0,0)}}e_1^\secund \right) + 
b_1^\secund x_2e_1^\secund \, ;\\
x_1x_2 \cdot 1 = x_1ae_1 + a^\secund x_2e_1^\secund = 
x_1a\left(e_1 + 
{\textstyle \frac{b^\secund (x_0,0)}{b(x_0,0)}}e_1^\secund \right) + 
a^\secund x_2e_1^\secund \, ;\\
x_2 \cdot \left(e_1 + 
{\textstyle \frac{b^\secund (x_0,0)}{b(x_0,0)}}e_1^\secund \right) = 
{\textstyle \frac{b^\secund (x_0,0)}{b(x_0,0)}}x_2e_1^\secund \, ;\ \  
x_3 \cdot \left(e_1 + 
{\textstyle \frac{b^\secund (x_0,0)}{b(x_0,0)}}e_1^\secund \right) = 0\, .
\qedhere
\end{gather*}
\end{proof}

\section{Double structures on a conic}\label{A:doubleconic}

Let $C \subset \piii$ be the conic of equations $x_3 = x_0x_2 - x_1^2 = 0$. 
Put $q := x_0x_2 - x_1^2$. Let $D$ be a double structure on $C$, i.e., a 
locally CM subscheme of $\piii$ of degree 4 with $D_{\text{red}} = C$. 
According to Ferrand \cite{f} there exists an exact sequence$\, :$ 
\[
0 \lra \sci_D \lra \sci_C \overset{\displaystyle \pi}{\lra} \scl \lra 0  
\] 
where $\scl$ is a line bundle on $C$ and where $\pi$ is a composite 
morphism$\, :$ 
\[
\sci_C \lra \sci_C/\sci_C^2 \simeq \sco_C(-1) \oplus \sco_C(-2) 
\xra{\displaystyle (\alpha \, ,\, \beta)} \scl 
\]
with $\alpha \in \tH^0(\scl(1))$ and $\beta \in \tH^0(\scl(2))$. Using the 
exact sequence$\, :$ 
\[
0 \lra \scl^{-1}(-3) \xra{\begin{pmatrix} -\beta\\ \alpha \end{pmatrix}} 
\sco_C(-1) \oplus \sco_C(-2) \xra{\displaystyle (\alpha \, ,\, \beta)} \scl 
\lra 0 
\]
one deduces an exact sequence$\, :$ 
\[
0 \lra \sci_{C^{(1)}} \lra \sci_D \overset{\displaystyle \eta}{\lra} 
\scl^{-1}(-3) \lra 0\, .
\]
 
Now, $C$ is the image of the embedding $\nu : \pj \ra \piii$, 
$\nu(t_0:t_1) = (t_0^2:t_0t_1:t_1^2:0)$. One has $\sco_C(1) \simeq 
\nu_\ast\sco_\pj(2)$ and if $t_0,\, t_1$ is the canonical basis of 
$\tH^0(\sco_\pj(1))$ then$\, :$ 
\[
x_0 \vb C = t_0^2\, ,\  x_1 \vb C = t_0t_1\, ,\  x_2 \vb C = t_1^2\, ,\  
x_3 \vb C = 0\, .
\]
One has to consider two cases$\, :$ 

\vskip2mm

(I) $\scl \simeq \sco_C(l)$ with $l \geq -2$$\, ;$ 

(II) $\scl \simeq \sco_C(l) \otimes \nu_\ast\sco_\pj(1)$ with 
$l \geq -1$. 
 
\vskip2mm

\noindent 
In case (I) one has an exact sequence$\, :$ 
\[
0 \lra \sco_C(l) \lra \sco_D \lra \sco_C \lra 0 
\]
and one denotes by $e \in \tH^0(\sco_D(-l))$ the image of $1 \in 
\tH^0(\sco_C)$, while in case (II) one has an exact sequence$\, :$ 
\[
0 \lra \sco_C(l) \otimes \nu_\ast\sco_\pj(1) \lra \sco_D \lra \sco_C \lra 0 
\]
and one denotes by $e_i \in \tH^0(\sco_D(-l))$ the image of $t_i \in 
\tH^0(\nu_\ast\sco_\pj(1)))$, $i = 0,\, 1$. 

\begin{lemma}\label{L:resic(1)} 
$I(C^{(1)})$ admits the following minimal graded free resolution$\, :$ 
\[
0 \ra \begin{matrix} S(-4)\\ \oplus\\ S(-5) \end{matrix} 
\xra{\begin{pmatrix} -q & 0\\ x_3 & -q\\ 0 & x_3 \end{pmatrix}} 
\begin{matrix} S(-2)\\ \oplus\\ S(-3)\\ \oplus\\ S(-4) \end{matrix} 
\xra{\displaystyle (x_3^2\, ,\, x_3q\, ,\, q^2)} I(C^{(1)}) \ra 0\, . 
\]
\qed 
\end{lemma} 

\begin{lemma}\label{L:resnuo(1)} 
$\nu_\ast\sco_\pj(1)$ admits the following resolution$\, :$ 
\[
0 \ra 2\sco_\piii(-2) \xra{\begin{pmatrix} 
\text{--}x_3 & 0\\ 0 & \text{--}x_3\\ \text{--}x_2 & \text{--}x_1\\ 
x_1 & x_0 \end{pmatrix}} 4\sco_\piii(-1) 
\xra{\begin{pmatrix} \text{--}x_2 & \text{--}x_1 & x_3 & 0\\ x_1 & x_0 & 0 & x_3 
\end{pmatrix}} 2\sco_\piii \xra{\displaystyle (t_0,t_1)} 
\nu_\ast\sco_\pj(1) \ra 0 
\]
\end{lemma}

\begin{proof} 
Let $H \subset \piii$ be the plane of equation $x_3 = 0$. One has an exact 
sequence$\, :$ 
\[
0 \ra 2\sco_H(-1) \xra{\begin{pmatrix} -x_2 & -x_1\\ x_1 & x_0 \end{pmatrix}} 
2\sco_H \xra{\displaystyle (t_0\, ,\, t_1)} \nu_\ast\sco_\pj(1) \ra 0\, .
\]
One deduces that the tensor product of the complexes$\, :$ 
\[
2\sco_\piii(-1) \xra{\begin{pmatrix} -x_2 & -x_1\\ x_1 & x_0 \end{pmatrix}} 
2\sco_\piii \, ,\  \  \sco_\piii(-1) \overset{\displaystyle x_3}{\lra} 
\sco_\piii
\]
is a resolution of $\nu_\ast\sco_\pj(1)$. 
\end{proof}

We treat, firstly, case (I). In this case, there exist $a \in 
k[x_0,x_1,x_2]_{l+1}$, $b \in k[x_0,x_1,x_2]_{l+2}$ such that $\alpha = 
a \vb C$ and $\beta = b\vb C$. 

\begin{lemma}\label{L:piepi1}
The morphism $\pi : \sci_C \ra \sco_C(l)$ defined by $\alpha = a \vb C$ 
and $\beta = b\vb C$ is an epimorphism if and only if $a$ and $b$ have no 
common zero on $C$. 
\qed
\end{lemma} 

\begin{prop}\label{P:genid1} 
Keeping the previously introduced notation, the homogeneous ideal $I(D)$ of 
the closed subscheme $D$ of $\piii$ defined by the kernel of an epimorphism 
$\pi$ as in Lemma~\ref{L:piepi1} is generated by$\, :$ 
\[
F := \begin{vmatrix} a & b\\ x_3 & q \end{vmatrix}\, ,\  
x_3^2\, ,\  x_3q\, ,\  q^2 
\]
and admits the following graded free resolution$\, :$ 
\[
0 \lra S(-l-6) \overset{\displaystyle d_2}{\lra} 
\begin{matrix} S(-l-4)\\ \oplus\\ S(-l-5)\\ \oplus\\ S(-4)\\ \oplus\\ S(-5) 
\end{matrix} \overset{\displaystyle d_1}{\lra} 
\begin{matrix} S(-l-3)\\ \oplus\\ S(-2)\\ \oplus\\ S(-3)\\ \oplus\\ S(-4) 
\end{matrix} \overset{\displaystyle d_0}{\lra} I(D) \lra 0
\]
with $d_1$ and $d_2$ defined by the matrices$\, :$ 
\[
\begin{pmatrix}
x_3 & q & 0 & 0\\
b & 0 & -q & 0\\
-a & b & x_3 & -q\\
0 & -a & 0 & x_3 
\end{pmatrix}\, ,\  
\begin{pmatrix} 
-q\\ x_3\\ -b\\ a 
\end{pmatrix}\, .
\]
This resolution is minimal for $l \geq 0$ but not for $l = -2$ where  
$I(D) = (x_3,\, q^2)$ and for $l = -1$ where $I(D) = (q-bx_3,\, x_3^2)$. 
\end{prop} 

\begin{proof}
The image of $F\in \tH^0(\sci_C(l+3))$ by the map$\, :$ 
\[
\sci_C(l+3) \lra (\sci_C/\sci_C^2)(l+3) \simeq \sco_C(l+2) \oplus 
\sco_C(l+1) 
\]
is $(-\beta,\, \alpha)$. It follows that $F \in \tH^0(\sci_D(l+3))$ and that 
$\eta(F) = 1 \in \tH^0(\sco_C)$. One deduces the exactness of the 
sequence$\, :$ 
\[
0 \lra I(C^{(1)}) \lra I(D) \xra{\displaystyle \tH^0_\ast(\eta)} 
S(C)(-l-3) \lra 0 
\]
and then one uses Lemma~\ref{L:resic(1)}.   
\end{proof} 

\begin{prop}\label{P:reshod1} 
Keeping the previously introduced notation, let $D$ be  
the closed subscheme of $\piii$ defined by the kernel of an epimorphism 
$\pi$ as in Lemma~\ref{L:piepi1}. Then the graded $S$-module 
${\fam0 H}^0_\ast(\sco_D)$ admits the following free resolution$\, :$ 
\[
0 \lra \begin{matrix} S(-3)\\ \oplus\\ S(l-3) \end{matrix} 
\overset{\displaystyle \delta_2}{\lra} 
\begin{matrix} S(-1)\\ \oplus\\ S(-2)\\ \oplus\\ S(l-1)\\ \oplus\\ 
S(l-2) \end{matrix} \overset{\displaystyle \delta_1}{\lra} 
\begin{matrix} S\\ \oplus\\ S(l) \end{matrix} 
\overset{\displaystyle \delta_0}{\lra} {\fam0 H}^0_\ast(\sco_D) \lra 0 
\]
with $\delta_0 = (1\, ,\, e)$ and with $\delta_1$ and $\delta_2$ defined by the 
matrices$\, :$ 
\[
\begin{pmatrix} 
x_3 & q & 0 & 0\\ 
-a & -b & x_3 & q 
\end{pmatrix}\, ,\  
\begin{pmatrix}
-q & 0\\ x_3 & 0\\ b & -q\\ -a & x_3 
\end{pmatrix}\, .
\]
This resolution is minimal for $l \geq 0$ but not for the special cases 
$l = -2$ and $l = -1$ where $D$ is a complete intersection. 
\end{prop}  

\begin{proof}
Since the graded $S$-module $\tH^0_\ast(\sco_C)$ is generated by 
$1 \in \tH^0(\sco_C)$ one deduces the exactness of the sequence$\, :$ 
\[
0 \lra S(C)(l) \lra \tH^0_\ast(\sco_D) \lra \tH^0_\ast(\sco_C) \lra 0\, .
\]
It suffices, now, to use the following relations in $\tH^0_\ast(\sco_D)$$\, :$ 
\[
x_3 \cdot 1 = \pi(x_3) = \alpha = ae\, ,\  
q \cdot 1 = \pi(q) = \beta = be\, . 
\qedhere
\]
\end{proof} 

We treat, finally, case (II). In this case there exist $a_0,\, a_1 \in 
k[x_0,x_1,x_2]_{l+1}$, $b_0,\, b_1 \in k[x_0,x_1,x_2]_{l+2}$ such that$\, :$ 
\[
\alpha = (a_0 \vb C)t_0 + (a_1 \vb C)t_1\, ,\  
\beta = (b_0 \vb C)t_0 + (b_1 \vb C)t_1\, .
\] 

\begin{lemma}\label{L:piepi2}
The morphism $\pi : \sci_C \ra \sco_C(l)\otimes \nu_\ast\sco_\pj(1)$ 
defined by $\alpha = (a_0 \vb C)t_0 + (a_1 \vb C)t_1$  
and $\beta = (b_0 \vb C)t_0 + (b_1 \vb C)t_1$ is an epimorphism if and only if 
\[
a_0x_0 + a_1x_1\, ,\  a_0x_1 + a_1x_2\, ,\  b_0x_0 + b_1x_1\, ,\  
b_0x_1 + b_1x_2 
\]
have no common zero on $C$. 
\end{lemma} 

\begin{proof}
One uses the fact that $\alpha$ and $\beta$ have no common zero on $C$ is and 
only if $\alpha t_0$, $\alpha t_1$, $\beta t_0$, $\beta t_1$ have no common 
zero on $C$. 
\end{proof}

\begin{prop}\label{P:genid2} 
Keeping the previously introduced notation, the homogeneous ideal $I(D)$ of 
the closed subscheme $D$ of $\piii$ defined by the kernel of an epimorphism 
$\pi$ as in Lemma~\ref{L:piepi2} is generated by$\, :$ 
\[
F_0 := \begin{vmatrix} a_0x_0 + a_1x_1 & b_0x_0 + b_1x_1\\ x_3 & q 
\end{vmatrix}\, ,\ 
F_1 := \begin{vmatrix} a_0x_1 + a_1x_2 & b_0x_1 + b_1x_2\\ x_3 & q 
\end{vmatrix}\, ,\   
x_3^2\, ,\  x_3q\, ,\  q^2 
\]
and admits the following graded free resolution$\, :$ 
\[
0 \lra 2S(-l-6) \overset{\displaystyle d_2}{\lra} 
\begin{matrix} 4S(-l-5)\\ \oplus\\ S(-4)\\ \oplus\\ S(-5) 
\end{matrix} \overset{\displaystyle d_1}{\lra} 
\begin{matrix} 2S(-l-4)\\ \oplus\\ S(-2)\\ \oplus\\ S(-3)\\ \oplus\\ S(-4) 
\end{matrix} \overset{\displaystyle d_0}{\lra} I(D) \lra 0
\]
with $d_1$ and $d_2$ defined by the matrices$\, :$ 
\[
\begin{pmatrix}
-x_2 & -x_1 & x_3 & 0 & 0 & 0\\
x_1 & x_0 & 0 & x_3 & 0 & 0\\
0 & 0 & b_0x_0 + b_1x_1 & b_0x_1 + b_1x_2 & -q & 0\\
-b_0 & b_1 & -a_0x_0 - a_1x_1 & -a_0x_1 -a_1x_2 & x_3 & -q\\
a_0 & -a_1 & 0 & 0 & 0 & x_3
\end{pmatrix}\, ,\  
\begin{pmatrix}
-x_3 & 0\\ 0 & -x_3\\ -x_2 & -x_1\\ x_1 & x_0\\ -b_0 & b_1\\ a_0 & -a_1 
\end{pmatrix}\, .
\] 
This resolution is minimal for $l \geq 0$ but not for $l = -1$ where the  
minimal resolution has the form 
$0 \ra S(-5) \lra 4S(-4) \lra S(-2) \oplus 3S(-3) \lra I(D) \ra 0$.  
\end{prop}

\begin{proof}
Recall the exact sequence$\, :$ 
\[
0 \ra \sco_C(-l-4) \otimes \nu_\ast\sco_\pj(1)  
\xra{\begin{pmatrix} -\beta\\ \alpha \end{pmatrix}} 
\sco_C(-1) \oplus \sco_C(-2) \xra{\displaystyle (\alpha \, ,\, \beta)} 
\sco_C(l) \otimes \nu_\ast\sco_\pj(1) \ra 0\, . 
\]
The image of $F_i \in \tH^0(\sci_C(l+4))$ by the map$\, :$ 
\[
\sci_C(l+4) \lra (\sci_C/\sci_C^2)(l+4) \simeq \sco_C(l+3) \oplus 
\sco_C(l+2) 
\]
is $(-\beta t_i,\, \alpha t_i)$, $i = 0,\, 1$. It follows that $F_i \in 
\tH^0(\sci_D(l+4))$ and that $\eta(F_i) = t_i \in \tH^0(\nu_\ast\sco_\pj(1))$, 
$i = 0, 1$. One deduces the exactness of the sequence$\, :$ 
\[
0 \lra I(C^{(1)}) \lra I(D) \xra{\displaystyle \tH^0_\ast(\eta)} 
\tH^0_\ast(\nu_\ast\sco_\pj(1))(-l-4) \lra 0 
\]
and one uses Lemma~\ref{L:resic(1)}, Lemma~\ref{L:resnuo(1)} and the 
following relations$\, :$ 
\begin{gather*}
-x_2F_0 + x_1F_1 = b_0x_3q - a_0q^2\, ,\  
-x_1F_0 + x_0F_1 = -b_1x_3q + a_1q^2\, ,\\
x_3F_0 = -(b_0x_0+b_1x_1)x_3^2 + (a_0x_0+a_1x_1)x_3q\, ,\\  
x_3F_1 = -(b_0x_1+b_1x_2)x_3^2 + (a_0x_1+a_1x_2)x_3q\, .  
\qedhere 
\end{gather*} 
\end{proof}

\begin{prop}\label{P:reshod2} 
Keeping the previously introduced notation, let $D$ be  
the closed subscheme of $\piii$ defined by the kernel of an epimorphism 
$\pi$ as in Lemma~\ref{L:piepi2}. Then the graded $S$-module 
${\fam0 H}^0_\ast(\sco_D)$ admits the following free resolution$\, :$ 
\[
0 \lra \begin{matrix} S(-3)\\ \oplus\\ 2S(l-2) \end{matrix} 
\overset{\displaystyle \delta_2}{\lra} 
\begin{matrix} S(-1)\\ \oplus\\ S(-2)\\ \oplus\\ 4S(l-1)  
\end{matrix} \overset{\displaystyle \delta_1}{\lra} 
\begin{matrix} S\\ \oplus\\ 2S(l) \end{matrix} 
\overset{\displaystyle \delta_0}{\lra} {\fam0 H}^0_\ast(\sco_D) \lra 0 
\]
with $\delta_0 = (1\, ,\, e_0\, ,\, e_1)$ and with $\delta_1$ and $\delta_2$ 
defined by the matrices$\, :$ 
\[
\begin{pmatrix}
x_3 & q & 0 & 0 & 0 & 0\\
-a_0 & -b_0 & -x_2 & -x_1 & x_3 & 0\\
-a_1 & -b_1 & x_1 & x_0 & 0 & x_3
\end{pmatrix}\, ,\  
\begin{pmatrix}
-q & 0 & 0\\ x_3 & 0 & 0\\ a_0x_0 + a_1x_1 & -x_3 & 0\\ 
-a_0x_1 -a_1x_2 & 0 & -x_3\\ b_0 & -x_2 & -x_1\\ b_1 & x_1 & x_0 
\end{pmatrix}\, .
\] 
This resolution is minimal for $l \geq 0$ but not for $l = -1$ where the 
minimal free resolution has the form 
$0 \ra 3S(-3) \lra 5S(-2) \lra S \oplus S(-1) \lra {\fam0 H}^0_\ast(\sco_D) 
\ra 0$. 
\end{prop}

\begin{proof}
One uses the exact sequence$\, :$ 
\[
0 \lra \tH^0_\ast(\nu_\ast\sco_\pj(1))(l) \lra \tH^0_\ast(\sco_D) \lra 
\tH^0_\ast(\sco_C) \lra 0
\]
and Lemma~\ref{L:resnuo(1)}. In order to determine the remaining entries of 
the matrix of $\delta_1$ one uses the following relations in 
$\tH^0_\ast(\sco_D)$$\, :$ 
\[
x_3 \cdot 1 = \pi(x_3) = \alpha = a_0e_0 + a_1e_1\, ,\  
q \cdot 1 = \pi(q) = \beta = b_0e_0 + b_1e_1\, ,
\]
and in order to determine the remaining entries of $\delta_2$ one uses the 
matrix relations$\, :$ 
\[
\begin{pmatrix} -x_2 & -x_1\\ x_1 & x_0 \end{pmatrix}
\begin{pmatrix} -x_0 & -x_1\\ x_1 & x_2 \end{pmatrix}
\begin{pmatrix} a_0\\ a_1 \end{pmatrix} = 
\begin{pmatrix} q & 0\\ 0 & q \end{pmatrix} 
\begin{pmatrix} a_0\\ a_1 \end{pmatrix} = 
\begin{pmatrix} qa_0\\ qa_1 \end{pmatrix}\, .
\qedhere 
\]
\end{proof}

\end{document}